\documentclass[11pt,letterpaper]{article}
\usepackage[T1]{fontenc}
\usepackage[margin=1in]{geometry}
\usepackage{amsmath,amssymb,amsthm}
\usepackage{newtxtext,newtxmath}
\usepackage{mathrsfs}
\usepackage{enumitem}
\usepackage{microtype}
\usepackage{titlesec}
\usepackage{xcolor}
\usepackage[authoryear,round]{natbib}
\usepackage{xurl}
\usepackage[colorlinks=true,linkcolor=blue,citecolor=blue,urlcolor=blue]{hyperref}
\hypersetup{
  pdftitle={Improved Berry--Esseen bounds for multivariate nonlinear statistics in convex distance},
  pdfauthor={Zhi-Jun Cai, Jing-Cai Yang, Zhuo-Song Zhang}
}
\newcommand{\doi}[1]{\href{https://doi.org/#1}{\nolinkurl{https://doi.org/#1}}}
\numberwithin{equation}{section}
\theoremstyle{plain}
\newtheorem{theorem}{Theorem}[section]
\newtheorem{lemma}[theorem]{Lemma}
\newtheorem{corollary}[theorem]{Corollary}
\newtheorem{proposition}[theorem]{Proposition}

\theoremstyle{definition}
\newtheorem{assumption}[theorem]{Assumption}

\newtheorem{remark}[theorem]{Remark}
\newcommand{\Var}{\mathop{\mathrm{Var}}}

\newcommand{\E}{\mathbb{E}}
\newcommand{\va}{\mathrm{Var}}
\newcommand{\op}{\mathrm{op}}

\renewcommand{\P}{\mathbb{P}}
\newcommand{\N}{\mathbb{N}}

\newcommand{\QQ}{\mathcal{Q}}

\newcommand{\R}{\mathbb{R}}
\newcommand{\caa}{\mathcal{A}}
\newcommand{\dist}{\mathrm{dist}}
\newcommand{\MM}{\mathcal{M}}
\newcommand{\Fm}{\mathrm{F}}

\newcommand{\la}{\langle}
\newcommand{\ra}{\rangle}
\newcommand{\na}{\nabla}
\newcommand{\norm}[1]{\lvert#1\rvert}
\newcommand{\dc}{d_{\mathrm c}}
\DeclareMathOperator*{\argmin}{arg\,min}

\title{\Large\bfseries Improved Berry--Esseen bounds for multivariate nonlinear statistics in convex distance}
\author{%
  Zhi-Jun Cai\textsuperscript{1}\qquad
  Jing-Cai Yang\textsuperscript{2}\qquad
  Zhuo-Song Zhang\textsuperscript{3}\\[1.2ex]
  \makebox[0pt][c]{%
  \begin{minipage}{\textwidth}
  \centering\small
  \textsuperscript{1}Department of Statistics and Data Science,
  The Chinese University of Hong Kong\enspace
  \href{mailto:zhijuncai@cuhk.edu.hk}{zhijuncai@cuhk.edu.hk}\\[0.6ex]
  \textsuperscript{2}Department of Statistics and Data Science,
  \mbox{Southern University of Science and Technology}\enspace
  \href{mailto:yangjc@sustech.edu.cn}{yangjc@sustech.edu.cn}\\[0.6ex]
  \textsuperscript{3}Department of Statistics and Data Science and
  Shenzhen International Center for Mathematics,
  \mbox{Southern University of Science and Technology}\enspace
  \href{mailto:zhangzs3@sustech.edu.cn}{zhangzs3@sustech.edu.cn}
  \end{minipage}%
  }%
}
\date{}
\begin{document}
\maketitle
\begin{abstract}
In this paper, we establish two nonasymptotic Berry--Esseen bounds
over convex sets for the Gaussian approximation of multivariate
nonlinear statistics. The statistics of interest can be written as a sum
of independent centered random vectors plus a remainder that may
depend on all observations. The first bound retains the classical factor
$d^{1/4}$ in the contribution of the independent sum, where $d$ is the
dimension, while controlling the remainder through its size and its
sensitivity to replacing a single observation. The second bound expresses
the contribution of the independent sum in terms of fourth moments and
can allow the dimension to grow faster with the sample size. For sums of
independent random vectors, it removes the logarithmic factor from an
existing fourth moment bound without imposing additional moment
assumptions. As applications, we apply these results to Polyak--Ruppert
averaging for nonsmooth stochastic approximation, temporal difference
learning with linear function approximation, and multivariate $U$-statistics.
The resulting bounds provide explicit Gaussian approximation errors and
sufficient conditions under which these errors converge to zero as the
dimension grows with the sample size.
\end{abstract}


\medskip

\section{Introduction}

Gaussian approximation over convex sets is a classical problem in multivariate probability. When the dimension grows with the sample size, the main question is not only whether a Gaussian limit exists, but also how the approximation error depends on the dimension. Many statistics have a leading fluctuation given by a sum of independent random
vectors, and for such sums sharp Gaussian approximation bounds are available.
The difficulty is to pass from the leading sum to the full statistic without
losing the dimension dependence of the independent case.

For $x\in\R^d$, write $|x|$ for its Euclidean norm, and let
$\mathcal A$ denote the collection of all measurable convex subsets of
$\R^d$. For $\R^d$-valued random vectors $U$ and $V$, define
\begin{align*}
	d_c(U,V)
	:=
	\sup_{A\in\mathcal A}
	\big|\P(U\in A)-\P(V\in A)\big|.
\end{align*}
The classical case is that of a sum $W=\sum_{i=1}^n X_i$ of independent centered
random vectors normalized so that $\sum_{i=1}^n\mathrm{Var}(X_i)=I_d$. If
$Z\sim\mathcal N(0,I_d)$, then  \cite{bentkus2003,bentkus2005} proved
\begin{align*}
	d_c(W,Z)
	\leq
	C d^{1/4}\sum_{i=1}^n \E|X_i|^3,
\end{align*}
where $C$ is an absolute constant. Later, \cite{raivc2019} obtained the explicit coefficient $42d^{1/4}+16$. The
factor $d^{1/4}$ is thus the standard dimension dependence for Gaussian
approximation over all convex sets.

For sums of independent random vectors, \citet{xiao2024}
obtained a bound in terms of fourth moments for a general
positive definite covariance matrix. In particular, under the
normalization $\Var(W)=I_d$, their result gives
\begin{align*}
	d_c(W,Z)
	\leq
	C d^{1/4}\delta
	\sqrt{|\log\delta|},
\end{align*}
where
$
\delta
:=
\min\{
(
\sum_{i=1}^n \E|X_i|^4
)^{1/2},
1/e
\}.
$
This estimate can give a weaker restriction on the growth of $d$. For example, suppose that $X_i=Y_i/\sqrt n$, where the $Y_i$'s are centered, independent and identically distributed, and isotropic, with $\E|Y_i|^3\asymp d^{3/2}$ and $\E|Y_i|^4\asymp d^2$. The first bound is then of order $d^{7/4}n^{-1/2}$, whereas the second is of order $d^{5/4}n^{-1/2}$ up to a logarithmic factor. Thus, when the dimension grows, the relevant issue is the actual size of $\E|X_i|^3$ and $\E|X_i|^4$, rather than their finiteness alone. Related Gaussian approximation results in growing dimension have also been obtained for hyperrectangles and for Wasserstein distance; see, for example, \cite{bonis2020,chernozhukov2017,fang2021,miles2022,zhai2018}.

Many statistics arising in practice are nonlinear, and there is a substantial literature on error bounds for their normal approximation.  Let $\xi_1,\ldots,\xi_n$ be independent random elements. For a given $\R^d$-valued statistic $T=T(\xi_1,\ldots,\xi_n)$, set
\begin{align} \label{eq-na-01}
	X_i:=g_i(\xi_i),
	\qquad
	W:=\sum_{i=1}^n X_i,
	\qquad
	D:=T-W,
\end{align}
where each $g_i$ is Borel measurable. Thus $T=W+D$, and the remainder $D$ may depend on all observations.  Assume that
\begin{align}	\label{eq-na-02}
	\E X_i=0,
	\qquad
	\sum_{i=1}^n\mathrm{Var}(X_i)
	=I_d.
\end{align}
A bound based only on the size of $D$ does not fully describe the nonlinear part. It is also useful to measure how much $D$ changes when one observation is replaced.
In the one-dimensional case,  \cite{louis2007} proved the bound
\begin{align*}
	\sup_{z\in\R}
	\big|\P(T\leq z)-\P(Z\leq z)\big|
	\leq
	6.1\sum_{i=1}^n\E|X_i|^3
	+\E|WD|
	+\sum_{i=1}^n\E\big|X_i(D-D^{(i)})\big|,
\end{align*}
where, for each $1\leq i\leq n$, $D^{(i)}$ is a random variable
such that $X_i$ is independent of the pair $(D^{(i)},W-X_i)$.
For $d>1$, \cite{Zhang2022} extended this approach to convex sets. Their result states that, for any $\Delta\geq|D|$ and suitable random variables $\Delta^{\{i\}}$ independent of $X_i$,
\begin{align*}
	d_c(T,Z)
	\leq
	259d^{1/2}\sum_{i=1}^n\E|X_i|^3
	+2\E\{|W|\Delta\}
	+2\sum_{i=1}^n
	\E\big\{|X_i|\,|\Delta-\Delta^{\{i\}}|\big\}.
\end{align*}
This bound has the optimal order in the sample size in many applications. Its leading dimension factor is, however, $d^{1/2}$, even though the independent sum admits the $d^{1/4}$ factor above.

The preceding results lead to two natural questions. Can the $d^{1/4}$ factor for the independent sum be retained after a general remainder is added? Can the estimate based on fourth moments also be extended to nonlinear statistics while removing the logarithmic factor from the contribution of the independent sum?

We answer both questions. First, we prove a general Berry--Esseen bound for $T=W+D$ in which the contribution of $W$ retains the factor $d^{1/4}$. The remainder is controlled through suitable $L_p$ norms of $D$ and the replacement differences $D-D^{\{i\}}$, without any independence assumption between $W$ and $D$. Second, we prove an alternative bound in terms of fourth moments for nonlinear statistics. For independent sums ($D=0$), it improves the bound of \citet[Theorem~2.1]{xiao2024} by removing the logarithmic factor, while retaining the $d^{1/4}$ dependence and requiring no additional moment assumptions. The resulting bound can allow the dimension to grow faster in common settings. Both bounds cover nonidentically distributed summands and general positive definite covariance matrices.
We apply the general results to Polyak--Ruppert averaging for nonsmooth stochastic approximation, temporal difference learning with linear function approximation, and multivariate $U$-Statistics. In each case, we obtain an explicit bound in convex distance and a condition relating the problem dimension to the sample size.

The proofs combine Stein's method with recursive arguments.
The proof of the second main result, stated in
Theorem~\ref{thm-main-2}, additionally employs an exchangeable
pair argument, together with a symmetry argument in the spirit
of \cite{xiao2024}.
We remove the logarithmic factor by treating small and large replacement increments separately at the natural smoothing scale and keeping this separation throughout the integration.

Throughout the paper, for a matrix or a multilinear map $A$,
$\|A\|_{\op}$ denotes its operator norm, while
$\|A\|_{\mathrm{F}}$ denotes the Frobenius norm when $A$ is a matrix.
For an $\R^d$-valued random vector $X$ and $p\geq1$, we write
$\|X\|_p:=(\E|X|^p)^{1/p}$ for its $L_p$ norm. For \(A\in\mathcal A\) and \(t\in\R\), let \(A^t:=\{x\in\R^d:\operatorname{dist}(x,A)\leq t\}\) if \(t>0\),  and \(A^t:=\{x\in A:\operatorname{dist}(x,A^c)>-t\}\) if \(t<0\), where  \(\operatorname{dist}(x,A):=\inf_{y\in A}|x-y|\). For a nonempty closed convex set \(Q\subset\R^d\), let \(P_Q(x):=\operatorname*{argmin}_{y\in Q}|y-x|\) denote the Euclidean projection of \(x\) onto \(Q\).

The rest of the paper is organized as follows. Section~\ref{sec:main-results} states the two general Berry--Esseen bounds. Section~\ref{sec:application} gives the applications to nonsmooth stochastic approximation, temporal difference learning, and multivariate $U$-Statistics. Section~\ref{sec: proof-of-main-results} proves the general results, and Section~\ref{sec:proof-of-other-results} proves the application theorems. Additional technical details are deferred to Appendices~\ref{app:supplementary-main}--\ref{app:additional-applications}.

\section{Main results}\label{sec:main-results}
\begin{theorem}\label{thm-main-1}
	Let $W$, $T$, and $D$ be as defined in \eqref{eq-na-01} and
	\eqref{eq-na-02}, and let $(\xi_1',\dots,\xi_n')$ be an independent
	copy of $(\xi_1,\dots,\xi_n)$. For each $1\leq i\leq n$, set
	$$
	D^{\{i\}}
	:=
	D(\xi_1,\dots,\xi_{i-1},\xi_i',
	\xi_{i+1},\dots,\xi_n).
	$$
Then,	for any exponents $q\geq 2$, $r\geq 3/2$, and $s>1$ satisfying
	$
	1/q+1/r+1/s=1,
	$
	we have
	\begin{align}\label{eq-thm-main-1}
		\begin{aligned}
			d_{c}(T,Z)\leq
			C\Big(
			d^{1/4}\sum_{i=1}^{n}\|X_{i}\|_{3}^3
			+qd^{1/(2q)}
			\sum_{i=1}^{n}
			\|D-D^{\{i\}}\|_{r}\|X_i\|_{s}
			+\|D\|_{3/2}
			\Big),
		\end{aligned}
	\end{align}
	where $C>0$ is an absolute constant.
\end{theorem}
\begin{remark}\label{rem-thm-main-1}
	In applications, it is often convenient to estimate the difference
	term $\|D^{\{i\}}-D\|_r$ in $L^2$, that is, to take $r=2$.
	For this choice, the relation
	$
	1/q+1/r+1/s=1
	$
	reduces to
	$
	1/q+1/s=1/2,
	$
	and hence
	$
	s=2q/(q-2)$,  $q>2.
	$
	As $q$ increases, the exponent $1/(2q)$ in the dimensional factor
	$d^{1/(2q)}$ decreases to zero, while the required moment order $s$
	decreases to $2$. In particular, for every finite $q\geq 6$,
	$
	2<s\leq 3.
	$
	Therefore, if $\mathbb{E}|X_i|^3<\infty$ for all
	$1\leq i\leq n$, then $\|X_i\|_s\leq\|X_i\|_3$, and
	Theorem~\ref{thm-main-1} may be applied with $r=2$ and any fixed
	finite $q\geq6$. Consequently, we have
	\begin{align*}
		d_c(T,Z)
		\leq
		C\Big(
		d^{1/4}\sum_{i=1}^{n}\|X_{i}\|_{3}^3
		+q d^{1/(2q)}
		\sum_{i=1}^{n}
		\|D-D^{\{i\}}\|_2\|X_i\|_3
		+\|D\|_{3/2}
		\Big).
	\end{align*}
Choosing $q=\max\{6,(\log d)/2\}$ gives
$q d^{1/(2q)}\leq C(1+\log d)$.
Thus the replacement term has only logarithmic dependence on $d$,
without requiring moments beyond the third order.
\end{remark}

Theorem~\ref{thm-main-1} admits a natural extension to a more general setting. By combining the result of Theorem~\ref{thm-main-1} with the observation in Remark~\ref{rem-thm-main-1}, we obtain the following corollary.

\begin{corollary}\label{coro-main-1}
	Let $W$, $T$ and $D$ be defined as in \eqref{eq-na-01}. For $1\leq i\leq n,$ assume that
	\begin{align*}
		\mathbb{E} X_{i}=0, \quad \sum_{i=1}^{n}\mathrm {Var}(X_i)=\Sigma,
	\end{align*}
	where $\Sigma$ is a positive definite matrix with $\lambda_{\min}(\Sigma)\geq\sigma>0$.  Then we have
	\begin{align*}
		&d_{c}(T, \Sigma^{1/2}Z)		\leq 	C\bigg(
		d^{1/4}\sum_{i=1}^{n}\frac{\|X_i\|_{3}^3}{\sigma^{3/2}}
		+(1+\log d)
		\sum_{i=1}^{n}
	\frac{	\|D-D^{\{i\}}\|_{2}\|X_i\|_{3}}{\sigma}
		+\frac{\|D\|_{3/2}}{\sigma^{1/2}}
		\bigg),
	\end{align*}
	where $C>0$ is an absolute constant.
\end{corollary}

In the preceding results, the contribution arising from the linear part \(W\) is controlled in terms of the third moments of the \(X_i\)'s. In applications where the nonlinear remainder \(D\) is of smaller order and does not dominate the overall approximation error, this contribution can be further sharpened under fourth-moment assumptions. Using an exchangeable pair construction, we obtain an alternative bound involving the fourth moments of the \(X_i\)'s. As illustrated by the multivariate $U$-statistic example, this alternative
bound may lead to less restrictive conditions on the growth of the
dimension $d$ relative to the sample size $n$ that still ensure that the
Gaussian approximation error over convex sets tends to zero. The precise result is stated below.

\begin{theorem}\label{thm-main-2}
	With the same notation as in Theorem~\ref{thm-main-1}, we have
	\begin{align}\label{eq-thm-main-2}
		d_{c}(T,Z)
		\leq
		C\big(
		d^{1/4}\gamma_1
		+(1+\log d)\gamma_2
		+d^{1/8}\gamma_3
		+\E|D|
		\big),
	\end{align}
	where $C>0$ is an absolute constant and
	\begin{align*}
		\gamma_1^2
		&=
		\sum_{i=1}^{n}\|X_i\|_{4}^4,\quad
		\gamma_2
		=
		\sum_{i=1}^{n}
		\|X_i\|_4\|D-D^{\{i\}}\|_2,\quad
		\gamma_3^2
		=
		\sum_{i=1}^{n}
		\|X_i\|_{4}^{2}\|D-D^{\{i\}}\|_2.
	\end{align*}
\end{theorem}

\begin{remark}
	Taking $D=0$ in Theorem~\ref{thm-main-2}, we obtain
	\[
	d_{c}(W,Z)
	\leq
	Cd^{1/4}
	\Big(
	\sum_{i=1}^{n}\E|X_{i}|^{4}
	\Big)^{1/2}.
	\]
	Under the normalization $\Var(W)=I_{d}$, this improves
	the fourth-moment bound of
	\citet[Theorem~2.1]{xiao2024}
	by removing the logarithmic factor, while retaining the
	same $d^{1/4}$ dependence and requiring no additional
	moment assumptions.
\end{remark}
The result of Theorem~\ref{thm-main-2} also extends to a more general setting. The precise statement is given in the following corollary.

\begin{corollary}\label{coro-main-2}
	Under the same notation and assumptions as in Corollary~\ref{coro-main-1}, we have
	\begin{align*}
		d_{c}(T, \Sigma^{1/2}Z)
		\leq
		C\big(
		\sigma^{-1}d^{1/4}\gamma_1
		+\sigma^{-1}(1+\log d)\gamma_2
		+\sigma^{-3/4}d^{1/8}\gamma_3
		+\sigma^{-1/2}\E|D|
		\big).
	\end{align*}

\end{corollary}

\section{Applications}\label{sec:application}
This section presents three applications of our main theorem: Polyak--Ruppert averaging for nonsmooth stochastic approximation, temporal difference learning with linear function approximation, and multivariate $U$-statistics. These applications demonstrate how the general result yields explicit Berry--Esseen bounds for nonlinear statistics arising in stochastic algorithms and classical statistical problems.
\subsection{Polyak--Ruppert averaging for nonsmooth stochastic approximation}
Let $F:\R^d\rightrightarrows\R^d$ be a set-valued map. We consider the
problem of finding a point $x_*$ satisfying
\begin{align}\label{eq-aa-01}
	0\in F(x_*).
\end{align}
Throughout this subsection, we fix one such solution $x_*$.
To approximate $x_{*}$, we use the stochastic approximation recursion
\begin{align}\label{eq-aa-02}
	x_{k+1}
	=x_k-\alpha_{k+1}
	G_{\alpha_{k+1}}(x_k,\zeta_{k+1}),
	\qquad k\geq0,
\end{align}
where
$G:\R_{++}\times\R^d\times\R^d\to\R^d$
is an update map, $\alpha_k>0$ are stepsizes, and $\zeta_k$ denotes
stochastic noise. The associated Polyak--Ruppert average is
\begin{align}\label{eq-aa-03}
	\bar x_n=\frac1n\sum_{k=0}^{n-1} x_k.
\end{align}
Our goal is to establish a nonasymptotic Berry--Esseen bound over
convex sets for the tangent-space component of the Polyak--Ruppert average
in \eqref{eq-aa-03}.
\subsubsection{Background}

Stochastic approximation is a recursive method for solving equations
from noisy observations. To explain the basic idea, suppose that the
target point $x_*$ satisfies $h(x_*)=0$, where
$h(x):=\E H(x,\zeta)$. If $h$ were available in closed form, one could
apply a deterministic root-finding method using $h(x_k)$. In many
applications, however, $h(x_k)$ cannot be evaluated directly, whereas
$H(x_k,\zeta_{k+1})$ can be computed from a new observation
$\zeta_{k+1}$ and serves as a noisy approximation to $h(x_k)$.
Stochastic approximation therefore replaces $h(x_k)$ by
$H(x_k,\zeta_{k+1})$ at each iteration. The recursion
\eqref{eq-aa-02} is a general version of this idea, in which the
update map may also depend on the stepsize. The method originates from
\cite{RobbinsMonro1951}; early distributional theory
was developed by \cite{Sacks1958} and
\cite{Fabian1968}. See also the review of
\cite{Lai2003}.

The ordinary sample mean is a simple example. To estimate
$\mu:=\E\zeta$, take $G_\alpha(x,\zeta)=x-\zeta$. If
$\alpha_{k+1}=1/(k+1)$, then \eqref{eq-aa-02} gives
$x_k=k^{-1}\sum_{j=1}^k\zeta_j$ for $k\geq1$. Thus, the usual sample
mean can itself be viewed as a stochastic approximation procedure.
Another standard example is stochastic optimization. Suppose that the
goal is to minimize the population loss
$f(x):=\E\ell(x,\zeta)$, where $\ell(x,\zeta)$ is the loss associated
with one random observation $\zeta$. When $f$ is smooth, its gradient is
given by $\nabla f(x)=\E\nabla_x\ell(x,\zeta)$ under standard
regularity conditions. Since this expectation is generally unavailable,
one replaces $\nabla f(x_k)$ by the stochastic gradient
$\nabla_x\ell(x_k,\zeta_{k+1})$. Accordingly, taking
$G_\alpha(x,\zeta)=\nabla_x\ell(x,\zeta)$ in
\eqref{eq-aa-02} yields stochastic gradient descent.
At each iteration, stochastic approximation updates the current estimate
using only one new observation. It therefore does not need to process all
available data at every step, which makes it useful when the data set is
large or when observations arrive one at a time; see, for example,
\cite{NemirovskiEtAl2009,ToulisAiroldi2017,ChenEtAl2020}.

Even when $x_k$ converges to $x_*$, the last iterate may still be
sensitive to recent stochastic errors and to the choice of stepsize.
Polyak--Ruppert averaging reduces these effects by averaging the
iterates. Suppose that $h$ is differentiable at $x_*$, and let
$J:=\nabla h(x_*)$ and $S:=\Var\{H(x_*,\zeta)\}$. Under standard
stability and moment conditions, the classical averaging theory gives
\begin{align*}
	\sqrt n(\bar x_n-x_*)
	\ \Longrightarrow\
	\mathcal N\bigl(0,J^{-1}SJ^{-\top}\bigr).
\end{align*}
See \cite{PolyakJuditsky1992,KushnerYang1993,Ruppert1988} for the classical theory
and \cite{ToulisAiroldi2017,ChenEtAl2020,DuchiRuan2021} for subsequent
developments in statistical inference and asymptotic optimality.

The smooth equation $h(x_*)=0$ does not cover many nonsmooth
optimization and estimation problems. As a simple example, consider the
$\ell_1$-regularized problem of minimizing $f(x)+\lambda|x|_1$, where
$f(x):=\E\ell(x,\zeta)$ is smooth, $\lambda>0$, and
$|x|_1:=\sum_{j=1}^d|x_j|$ denotes the $\ell_1$-norm of $x$. Its
first-order optimality condition is
$0\in\nabla f(x_*)+\lambda\partial|\cdot|_1(x_*)$.
Thus, the problem takes the form \eqref{eq-aa-01} with
$F(x):=\nabla f(x)+\lambda\partial|\cdot|_1(x)$. The map $F$ is
set-valued because the subdifferential of the absolute value at zero is
the interval $[-1,1]$. A natural stochastic algorithm for this problem is
the proximal stochastic gradient update
$x_{k+1}=\operatorname{prox}_{\alpha_{k+1}\lambda|\cdot|_1}
\{x_k-\alpha_{k+1}\nabla_x\ell(x_k,\zeta_{k+1})\}$; see, for example,
 \cite{DuchiSinger2009}. This is a
special case of \eqref{eq-aa-02}, with
$
G_\alpha(x,\zeta)
:=
\alpha^{-1}
\{
x-
\operatorname{prox}_{\alpha\lambda|\cdot|_1}
(x-\alpha\nabla_x\ell(x,\zeta))
\}.
$

Although $F$ may be nonsmooth in the ambient space, many problems have a
locally smooth structure once the active pattern of the solution is
fixed. In the $\ell_1$-regularized example, this pattern is determined by
the zero coordinates of $x_*$ and the signs of its nonzero coordinates.
More generally, under suitable regularity conditions, the active
structure can be represented by a smooth manifold $\MM$ containing
$x_*$; see \cite{Wright1993,Lewis2002,LewisZhang2013}.
In the framework of \cite{DDJ2024}, the tangential map
$F_{\MM}(y):=P_{T_{\MM}(y)}F(y)$ is locally single-valued and smooth,
while a separate proximal aiming condition controls the attraction
of the stochastic iterates toward $\MM$.

Using this local smooth structure,  \cite{DDJ2024} proved asymptotic normality of the
Polyak--Ruppert average and local asymptotic optimality of its limiting
covariance. \cite{JiangEtAl2025} subsequently studied online estimation of
the limiting covariance and established nonasymptotic bounds
for the covariance-estimation error. These results do not, however, provide a finite-sample bound
for the Gaussian approximation error of the averaged iterates
in convex distance.  For related finite-sample Gaussian approximation results, see
\cite{Zhang2022,SheshukovaEtAl2026} for averaged stochastic
gradient descent with smooth objectives and
\cite{SamsonovEtAl2024} for linear stochastic approximation.

\subsubsection{Smooth manifolds}\label{sec:smooth-manifold}
For the reader's convenience, we  first recall  a few definitions from smooth manifolds. Throughout this subsection, all manifolds are assumed to be embedded in $\R^d$, and their tangent spaces are regarded as linear subspaces of $\R^d$. For $p\geq1$,
a set $\MM\subset\R^d$ is called an $m$-dimensional $C^p$ embedded manifold if, for every $x\in\MM$, there exist an open neighborhood $\mathcal U\subset\R^d$ of $x$ and a $C^p$ map
$f:\mathcal U\to\R^{d-m}$ such that
\begin{align*}
	\MM\cap\mathcal U=f^{-1}(0),
	\qquad
	\operatorname{rank} \nabla f(x)=d-m,
\end{align*}
where $\nabla f(x):\R^d\to\R^{d-m}$ denotes the differential of $f$ at $x$. The tangent and normal spaces to $\MM$ at $x$ are defined, respectively, by
\begin{align*}
	T_{\MM}(x)=\ker \nabla f(x),
	\qquad
	N_{\MM}(x)=T_{\MM}(x)^\perp.
\end{align*}
For a $C^p$-smooth manifold $\MM$ with $p\geq 2$, the projection
$P_{\MM}$ is $C^{p-1}$-smooth on a sufficiently small neighborhood of
each point $x\in\MM$;
see \cite{MillerMalick2005}. Moreover,  for every $x\in\MM$,
\[
\nabla P_{\MM}(x)=P_{T_{\MM}(x)};
\]
see \citet[Lemma~2.1]{LewisMalick2008}. A map $h:\MM\to\R^r$ is said to be $C^p$ near $x\in\MM$ if there
exist an open neighborhood $\mathcal U$ of $x$ in $\R^d$ and a map
$\widehat h\in C^p(\mathcal U;\R^r)$ such that
$\widehat h=h$ on $\mathcal U\cap\MM$. For $p\geq1$, its tangential
Jacobian at $x$ is the linear map
$\na_{\MM}h(x):T_{\MM}(x)\to\R^r$ defined by
\[
\na_{\MM}h(x)v
:=
\na\widehat h(x)v,
\qquad v\in T_{\MM}(x).
\]

Throughout the remainder of this subsection, let $\MM$ be an
embedded $m$-dimensional $C^4$ submanifold of $\R^d$
containing $x_{*}$. Let $U\in\R^{d\times m}$ be a matrix whose
columns form an orthonormal basis of $T_{\MM}(x_{*})$.

Choose $\delta_{1}>0$ sufficiently small so that the nearest-point
projection $P_{\MM}$ is well-defined and $C^{3}$ on
$B_{2\delta_{1}}(x_{*})$. To quantify its local regularity, set
\begin{align}\label{eq-defi-cm}
	C_{\MM}
	:=
	1+\sum_{j=1}^{3}
	\sup_{x\in B_{\delta_{1}}(x_{*})}
	\|\nabla^{j}P_{\MM}(x)\|_{\op}.
\end{align}

\subsubsection{Assumptions and main result}

To state our main result, we first introduce several assumptions.
Classical results for stochastic approximation often assume that
$F$ is smooth in a neighborhood of $x_*$ in the whole space
$\mathbb{R}^d$. For the nonsmooth setting considered here,
we follow the local manifold framework used by \cite{DDJ2024}
in their study of asymptotic normality.
In particular, Assumption~\ref{assu-1} does not require $F$ to
be smooth in the whole space. It only requires the tangent component
$F_{\MM}$ to be smooth near $x_*$ along a smooth manifold $\MM$.
The assumption also requires $F_{\MM}$ to be locally Lipschitz and
its Riemannian Jacobian at $x_*$ to be positive and bounded on
the tangent space $T_{\MM}(x_*)$.

\begin{assumption}\label{assu-1}
	Suppose that there exists a $C^{4}$ manifold $\MM\subset\R^{d}$
	containing $x_{*}$ such that
	\begin{align*}
		F_{\MM}(x):=P_{T_{\MM}(x)}F(x),
		\qquad x\in\MM,
	\end{align*}
	is single valued and $C^{2}$ on a neighborhood of $x_{*}$ in $\MM$.
	Furthermore, $F_{\MM}$ is Lipschitz with constant $L_{\MM}>0$
	there. The local extension
	$\widehat{F}_{\MM}:=F_{\MM}\circ P_{\MM}$,
	defined on a sufficiently small neighborhood of $x_{*}$ in $\R^{d}$,
	has second derivative bounded in operator norm by a constant
	$M_{2}>0$ at every $y\in\MM$ sufficiently close to $x_{*}$.
	Moreover, there are constants $0<\lambda_{1}\leq\lambda_{2}$ such that
	\begin{align}\label{eq-aa-04}
		\lambda_{1}\norm{v}^{2}
		&\leq
		\bigl\langle
		\nabla_{\MM}F_{\MM}(x_{*})v,v
		\bigr\rangle,
		\qquad
		\|\nabla_{\MM}F_{\MM}(x_{*})\|_{\op}
		\leq\lambda_{2}
	\end{align}
	for every $v\in T_{\MM}(x_{*})$.
\end{assumption}

Assumption~\ref{assu-1} describes the smooth and stable dynamics
intrinsic to $\MM$. We next impose a local compatibility condition
between these dynamics and the stochastic update map $G_\alpha$.
Conditions (i) and (ii) below control the error in approximating the
tangential component of $G_\alpha(x,\zeta)$ by
$F_{\MM}(P_{\MM}(x))
+P_{T_{\MM}(P_{\MM}(x))}\zeta$, together with the local variation of
this error with respect to $(x,\zeta)$. Condition (iii) imposes a
quantitative proximal aiming condition in the normal direction, which
provides the mechanism for controlling the distance of the iterates
from $\MM$. These requirements are made precise in the following
assumption.
\begin{assumption}\label{assu-2}
	There are constants $c_1,c_2,c_3,c_4>0$, a neighborhood $\mathcal U$ of
	$x_*$, and $\mathcal U_F:=\mathcal U\cap\operatorname{dom}F$ such that the
	following statements hold for all $x,x'\in\mathcal U_F$, all
	$\zeta,\zeta'\in\R^d$, and all $\alpha>0$.  Define
	\begin{align*}
		\mathcal R_\alpha(x,\zeta)
		:=P_{T_{\MM}(P_{\MM}(x))}
		\{G_\alpha(x,\zeta)-F_{\MM}(P_{\MM}(x))-\zeta\}.
	\end{align*}
	\begin{enumerate}[label=(\roman*)]
		\item The residual $\mathcal{R}_{\alpha}$ satisfies
		\begin{align*}
			\norm{\mathcal R_\alpha(x,\zeta)}
			\leq c_1(1+|x-x_{*}|^{2}+\norm\zeta^2)
			\{\dist(x,\MM)+\alpha\}.
		\end{align*}
		\item The difference between two residuals satisfies
		\begin{align*}
			&\norm{\mathcal R_\alpha(x,\zeta)
				-\mathcal R_\alpha(x',\zeta')}\\
				&\quad\leq c_2(1+\norm\zeta^2+\norm{\zeta'}^{2})
			\{\alpha+\norm{x-x_*}+\norm{x'-x_*}\}
			\{\norm{x-x'}+\norm{\zeta-\zeta'}\}.
		\end{align*}
		\item The update map $G_{\alpha}$ satisfies
		\begin{align*}
			\la G_\alpha(x,\zeta)-\zeta,x-P_{\MM}(x)\ra
			&\geq c_3\dist(x,\MM)
			-c_4
			\{\alpha(1+\norm{\zeta}^{2})+\dist^2(x,\MM)\}.
		\end{align*}
	\end{enumerate}
\end{assumption}
To control the iterates globally, we adopt the following stability
assumption from \cite{JiangEtAl2025}. The first inequality ensures that the
update points toward $x_*$, while the second controls the growth of the
update. These conditions will be used to establish $L_p$-bounds for the
iterates.
\begin{assumption}\label{assu-3}
	There are constants $c_5,c_6,c_7>0$ such that, for every $x\in\R^d$,
	$\zeta\in\R^d$, and $\alpha>0$,
	\begin{align*}
		\la G_\alpha(x,\zeta)-\zeta,x-x_*\ra
		&\geq c_5\norm{x-x_*}^2
		-c_6\alpha\{1+\norm{x-x_*}^2+\norm\zeta^2\},\\
		\norm{G_\alpha(x,\zeta)}^2
		&\leq c_7\{1+\norm{x-x_*}^2+\norm\zeta^2\}.
	\end{align*}
\end{assumption}
The next assumption specifies conditions on the iterates and the
noise sequence. Condition (i) requires the iterates to remain in
$\operatorname{dom}F$ and bounds the twelfth moment of the initial
error.  Condition (ii) decomposes the noise $\zeta_k$ into an
independent innovation $\xi_k$ and a state-dependent
martingale difference component $\eta_k$, and specifies their moment,
covariance, and regularity properties. A similar assumption appears in
\cite{Zhang2022}.
\begin{assumption}\label{assu-4}
	The following conditions hold.
	\begin{enumerate}
		\item[(i)] The iterates satisfy
		$x_{k}\in\operatorname{dom}F$ almost surely for every $k\geq0$.
		Moreover, there exists a constant $\tau_{0}>0$ such that
		$\|x_{0}-x_{*}\|_{12}\leq\tau_{0}$.
		\item[(ii)] The innovation sequence $(\xi_k)_{k\geq1}$ is independent of
		$x_0$, and
		$
		\zeta_k=\xi_k+\eta_k,$
		$	\eta_k=g(x_{k-1},\xi_k).
		$
		The random vectors $(\xi_k)_{k\geq1}$ are independent and satisfy
		$\E(\xi_k)=0,$
		$\E(\xi_k\xi_k^\top)=\Sigma_k,$
		and
		$\lambda_3I_d\preceq\Sigma_k\preceq\lambda_4I_d
		$
		for some constants $0<\lambda_3\leq\lambda_4$.  Moreover, there exists a constant $\tau>0$ such that
		$
		\sup_{k\geq1}\|\xi_k\|_{12}\leq\tau.
		$
		With $\mathcal F_k=\sigma(x_0,\xi_1,\ldots,\xi_k)$, the state-dependent
		component satisfies $
		\E(\eta_k\mid\mathcal F_{k-1})=0.
		$
	Finally, there exists a constant $c_{8}>0$ such that,
	for all $x,x',\xi\in\R^{d}$,
		\begin{align*}
			\norm{g(x,\xi)-g(x',\xi)}
			\leq c_8\norm{x-x'},
			\qquad
			g(x_*,\xi)=0.
		\end{align*}
	\end{enumerate}
\end{assumption}

To control the effect of replacing a single innovation, we impose
the following conditions on two trajectories driven by the same
innovations. The first condition yields contraction of the sixth
moment of their difference for sufficiently small stepsizes.
The second provides a dissipation estimate for the difference
between their full proximal residuals.
\begin{assumption}\label{assu-5}
	For $x\in\operatorname{dom}F$, $\xi\in\R^d$, and $\alpha>0$, define
	\begin{align*}
		\Phi_\alpha(x,\xi)
		&:=x-\alpha G_\alpha\bigl(x,\xi+g(x,\xi)\bigr),\quad
		\widehat G_\alpha(x,\xi)
		:=G_\alpha\bigl(x,\xi+g(x,\xi)\bigr).
	\end{align*}
For each $k\geq1$, let $\mu_{k}$ denote the distribution of
$\xi_{k}$. There exist constants $c_{9},c_{11},c_{12}>0$,
$c_{10},c_{13}\geq0$, and $\bar{\alpha}>0$ such that,
for every $k\geq1$, every $x,x'\in\operatorname{dom}F$,
and every $0<\alpha\leq\bar{\alpha}$,
\begin{align}
	\int_{\R^{d}}
	\norm{\Phi_{\alpha}(x,\xi)
		-\Phi_{\alpha}(x',\xi)}^{6}
	\,\mu_{k}(d\xi)
	&\leq
	\bigl(1-c_{9}\alpha+c_{10}\alpha^{2}\bigr)
	\norm{x-x'}^{6}
	\label{eq:coupled-L6}
\end{align}
and
\begin{align}	\label{eq:coupled-residual}
	\begin{aligned}
		\int_{\R^{d}}\Big[
	&\norm{\Phi_{\alpha}(x,\xi)
		-\Phi_{\alpha}(x',\xi)}^{2}
	+c_{11}\alpha^{2}
	\norm{\widehat{G}_{\alpha}(x,\xi)
		-\widehat{G}_{\alpha}(x',\xi)}^{2}
	\Big]\,\mu_{k}(d\xi)
\\
	&\leq
	\bigl(1-c_{12}\alpha+c_{13}\alpha^{2}\bigr)
	\norm{x-x'}^{2}.
	\end{aligned}
\end{align}
\end{assumption}
For sufficiently small stepsizes,
condition~\eqref{eq:coupled-residual} gives a contraction estimate
in mean square together with a dissipation term for the residual.
A sufficient condition is available when the update consists of
an explicit step followed by a proximal step.
Suppose that, for every $x\in\operatorname{dom}F$,
$\xi\in\R^{d}$, and $0<\alpha\leq\bar{\alpha}$,
$
	\Phi_{\alpha}(x,\xi)
	=
	\operatorname{prox}_{\alpha f}
[x-\alpha\{\mathcal{A}(x,\xi)+\xi\}]
$
for a proper lower semicontinuous convex function $f$.
Suppose further that there exist constants $\mu_{0},L_{0}>0$
such that, for every $k\geq1$ and every
$x,x'\in\operatorname{dom}F$,
\begin{align*}
	\int_{\R^{d}}
	\bigl\langle
	\mathcal{A}(x,\xi)-\mathcal{A}(x',\xi),\,
	x-x'
	\bigr\rangle
	\,\mu_{k}(d\xi)
	&\geq
	\mu_{0}\norm{x-x'}^{2},\\
	\int_{\R^{d}}
	\norm{\mathcal{A}(x,\xi)-\mathcal{A}(x',\xi)}^{2}
	\,\mu_{k}(d\xi)
	&\leq
	L_{0}^{2}\norm{x-x'}^{2}.
\end{align*}
Recall that the proximal map of a proper lower semicontinuous
convex function is firmly nonexpansive; that is, for every
$u,v\in\R^{d}$,
\begin{align*}
	&\norm{
		\operatorname{prox}_{\alpha f}(u)
		-\operatorname{prox}_{\alpha f}(v)
	}^{2}+
	\norm{
		\{u-\operatorname{prox}_{\alpha f}(u)\}
		-\{v-\operatorname{prox}_{\alpha f}(v)\}
	}^{2}
	\leq
	\norm{u-v}^{2};
\end{align*}
see \citet[Proposition~12.27]{bauschke2011convex}.
Applying this inequality with
\begin{align*}
	u
	&=
	x-\alpha\{\mathcal{A}(x,\xi)+\xi\},\quad
	v
	=
	x'-\alpha\{\mathcal{A}(x',\xi)+\xi\},
\end{align*}
and using
$
\norm{a+b}^{2}/2
	\leq
	\norm{a}^{2}+\norm{b}^{2},
$
gives \eqref{eq:coupled-residual} with
$
	c_{11}
	=
	1/2,$
	$c_{12}
	=
	2\mu_{0},$
$	c_{13}
	=
	2L_{0}^{2}.$  To show that our assumptions are reasonable, we provide in Appendix~\ref{subsec:coupled-nonsmooth-example}
an example satisfying Assumptions \ref{assu-1}--\ref{assu-5}.
Throughout the remainder of this subsection, $C$ denotes a positive
constant that may change from line to line and depends only on the fixed
structural quantities
\[
\beta,\ \alpha_0,\ \bar\alpha,\ L_{\MM},\ M_{2},\,
\lambda_1,\lambda_2,\lambda_3,\lambda_4,\
c_1,\ldots,c_{13}.
\]
Choose $\delta_{2}>0$ sufficiently small so that
$B_{\delta_{2}}(x_{*})\subset\mathcal{U}$ and all the local
regularity conditions in Assumption~\ref{assu-1} hold on
$\MM\cap B_{\delta_{2}}(x_{*})$.
Then the inequalities in Assumption~\ref{assu-2} hold for all
$x,x'\in B_{\delta_{2}}(x_{*})\cap\operatorname{dom}F$,
all $\zeta,\zeta'\in\R^{d}$, and all $\alpha>0$.
Set
\[
\delta:=\frac{1}{8}\min\{\delta_{1},\delta_{2}\}.
\]
For convenience, let
$H:=U^\top\na_{\MM}F_{\MM}(x_*)U\in\R^{m\times m}.$
For $0\leq k\leq n$, define
\begin{align}\label{eq:Qkn}
	Q_{k}
	:=\alpha_k\sum_{j=k}^{n-1}
	\prod_{\ell=k+1}^j(I_m-\alpha_\ell H), \qquad \text{and}\qquad \Sigma_n
	:=\frac1n\sum_{k=1}^{n-1}
	Q_{k}U^\top\Sigma_kUQ_{k}^\top,
\end{align}
where an empty product equals $I_m$. Applying Corollary \ref{coro-main-1}, we have the following theorem.

\begin{theorem}\label{thm-app-nonsmooth}
	Suppose that Assumptions~\ref{assu-1}--\ref{assu-5} hold.
Let $\alpha_{0}>0$ and $\beta\in(2/3,1)$, and set
$\alpha_{k}=\alpha_{0}k^{-\beta}$ for $k\geq1$.
	For $m\geq1$, there exists an integer $n_{0}\geq 1$ such that, for every $n\geq n_0$,
	\begin{align}\label{eq-app-thm-01}
		\begin{aligned}
			&d_{c}\big(U^\top\sqrt n(\bar x_n-x_*),\Sigma_n^{1/2}Z_m\big)\\
			&\quad \leq C\Big(
			m^{1/4}\tau^3n^{-1/2}+\tau_0n^{-1/2}
			+C_{\MM}^{3}\tau_{1}^{5} (1+\log m) n^{1-3\beta/2}
			\Big),
		\end{aligned}
	\end{align}
	where $\tau_{1}=(1+\delta+\tau_{0}+\tau)(1+\delta^{-1}(1+\tau_{0}+\tau))$.
\end{theorem}

\begin{remark}
	\label{rem:tangent-coordinates}
	The factor $U^\top$ in Theorem~\ref{thm-app-nonsmooth} is essential.
	\cite{DDJ2024} showed that the limiting
	fluctuations of $\sqrt n(\bar x_n-x_*)$ are supported on
	$T_{\MM}(x_*)$.  Accordingly, the natural Gaussian approximation in the
	ambient space $\R^d$ is
	$
	U\Sigma_n^{1/2}Z_m,
	$
	whose covariance matrix is $U\Sigma_n U^\top$.  If $m<d$, then
	$
	\operatorname{rank}(U\Sigma_n U^\top)\leq m<d,
	$
	and hence this covariance matrix is singular, even though $\Sigma_n$ is
	positive definite.  Thus, the natural ambient-space Gaussian approximation
	does not satisfy the nondegeneracy condition required by
	Corollary~\ref{coro-main-1}.

	There is also a more basic obstruction coming from the convex-set distance.
	Set $T_*:=T_{\MM}(x_*)$.  Since $T_*$ is a convex subset of $\R^d$ and
	$U\Sigma_n^{1/2}Z_m\in T_*$ almost surely, we have
	\begin{align*}
		&d_c\big(
		\sqrt n(\bar x_n-x_*),
		U\Sigma_n^{1/2}Z_m
		\big)\geq
		1-\mathbb P\left\{
		\sqrt n(\bar x_n-x_*)\in T_*
		\right\}.
	\end{align*}
	Consequently, even if the normal component of
	$\sqrt n(\bar x_n-x_*)$ converges to zero in probability, this alone does
	not imply convergence in the convex-set distance.  In particular, if the
	finite-sample statistic lies outside $T_*$ almost surely, then the distance
	above equals one.  Passing to
	$U^\top\sqrt n(\bar x_n-x_*)$ removes this structural degeneracy and retains
	precisely the nondegenerate $m$-dimensional tangent fluctuations described
	by $\Sigma_n$.

	Finally, the conclusion does not depend on the particular orthonormal basis
	used to define $U$.  Indeed, if $\widetilde U=UO$ for an orthogonal matrix
	$O\in\R^{m\times m}$, then
	\[
	\widetilde U^\top\sqrt n(\bar x_n-x_*)
	=
	O^\top U^\top\sqrt n(\bar x_n-x_*),
	\qquad
	\widetilde\Sigma_n=O^\top\Sigma_nO.
	\]
	Since the collection of convex sets is invariant under orthogonal
	transformations, the corresponding convex-set distance is unchanged.
	When $m=d$, the matrix $U$ is square and orthogonal, so the tangent-coordinate
	and ambient-space formulations are equivalent.
\end{remark}

\begin{remark}\label{rem:nonsmooth-dimensional-rate}
Suppose that
$\delta\asymp d^{1/2}$, $\tau_{0}\asymp d^{1/2}$, and
$\tau\asymp d^{1/2}$, while $C_{\MM}$ is bounded uniformly in $d$.
Assume further that the structural quantities in
Theorem~\ref{thm-app-nonsmooth} are uniformly controlled in $d$
so that the constant $C$ and the threshold $n_{0}$ can be chosen
independently of $d$.
	Since \(m\leq d\), the first term on the right-hand side of
	\eqref{eq-app-thm-01} is bounded by
	\(C d^{7/4}n^{-1/2}\), whereas
	\(\tau_0n^{-1/2}\lesssim d^{1/2}n^{-1/2}\) is of smaller order.
	Moreover,
	\(\delta^{-1}(1+\tau_0+\tau)=O(1)\), and hence
	\(\tau_1\lesssim d^{1/2}\).  Therefore,
	Theorem~\ref{thm-app-nonsmooth} yields
	\begin{align*}
		d_c\bigl(
		U^\top\sqrt n(\bar x_n-x_*),
		\Sigma_n^{1/2}Z_m
		\bigr)
		\leq
		C\big\{
		d^{7/4}n^{-1/2}
		+
		d^{5/2}(1+\log d)n^{1-3\beta/2}
		\big\}.
	\end{align*}
	Since \((3\beta-2)/5<1/5<2/7\), the Gaussian approximation
	error converges to zero whenever \(d=O(n^\gamma)\) for some
	\(\gamma<(3\beta-2)/5\). By taking \(\beta\) sufficiently close to
	\(1\), the admissible exponent can be made arbitrarily close to \(1/5\).
\end{remark}

\subsection{Temporal difference learning with linear function approximation}
\label{subsec:TD-linear}

Temporal-difference (TD) learning, introduced by \citet{Sutton1988}, is a
stochastic-approximation method for evaluating a fixed policy. For a policy
$\pi$ and a discount factor $\gamma\in[0,1)$, its value function is
\[
V^\pi(s)
=\E_\pi\bigg\{
\sum_{k=0}^{\infty}\gamma^k r(s_k,a_k)
\,\bigg|\,s_0=s
\bigg\}.
\]
When the state space is large, one often uses the linear approximation
$V_\theta(s)=\phi(s)^\top\theta$, where
$\phi:\mathcal S\to\R^d$ is a feature map. The target coefficient
$\theta^\star$ is the unique solution of the projected Bellman equation,
equivalently $A\theta^\star=b$; see \citet{TsitsiklisVanRoy1997}.

We work under independent sampling. The tuples
$(s_t,a_t,s_t',r_t)$ are independent and identically distributed across $t$, with
$s_t\sim\mu$, $a_t\sim\pi(\cdot\mid s_t)$,
$s_t'\sim P(\cdot\mid s_t,a_t)$, and
$\E(r_t\mid s_t,a_t)=r(s_t,a_t)$. The distribution $\mu$ is stationary
for the transition kernel induced by $\pi$, so $s_t$ and $s_t'$ have the
same marginal distribution. Set
$B_\phi=\sup_{s\in\mathcal S}|\phi(s)|$,
$\Sigma_\phi=\E\{\phi(s_t)\phi(s_t)^\top\}$,
$\lambda_0=\lambda_{\min}(\Sigma_\phi)$, and
$\lambda_\phi=\lambda_{\max}(\Sigma_\phi)$. We assume that there is a
constant $K_\phi\geq1$, independent of $d$ and $T$, such that
\begin{align}
	K_\phi^{-1}I_d
	\preceq\Sigma_\phi\preceq K_\phi I_d,
	\qquad
	1-\gamma\geq K_\phi^{-1}.
	\label{eq-ab-01}
\end{align}

We remark that, unlike \cite{WuEtAl2026}, we do not impose the normalization $B_\phi\le 1$.
Indeed, under $\Sigma_\phi\succeq K_\phi^{-1}I_d$ we have
$
B_\phi^2\geq \E|\phi(s_t)|^2=\operatorname{Tr}(\Sigma_\phi)\geq	d/K_\phi $
so that the feature envelope must grow with the dimension.
Define
$A_t=\phi(s_t)\{\phi(s_t)-\gamma\phi(s_t')\}^\top$,
$b_t=r_t\phi(s_t)$, $A=\E A_t$, and $b=\E b_t$.
Lemma~\ref{lem:TD-feature-algebra} in Appendix~\ref{sec:appendix-TD} shows that $A$ is nonsingular.
Starting from a deterministic $\theta_0\in\R^d$, the TD recursion and its
Polyak--Ruppert average are
\begin{align}
	\theta_t
	&=\theta_{t-1}-\eta_t(A_t\theta_{t-1}-b_t),
	\qquad
	\overline\theta_T=\frac1T\sum_{t=1}^T\theta_t.
	\label{eq-ab-03}
\end{align}
The stepsizes satisfy
\begin{align}
	\eta_t=\eta_0t^{-\alpha},
	\qquad
	\alpha\in(1/2,1),
	\qquad
	0<\eta_0B_\phi^2\leq\frac14.
	\label{eq-ab-04}
\end{align}
Combining this restriction with the inequality $B_\phi^2 \ge d/K_\phi$ yields $\eta_0 \le K_\phi/(4d)$; in this case, therefore, $\eta_0$ cannot be chosen independently of the dimension.
Set
\begin{align}
	\Delta_t&=\theta_t-\theta^\star,
	&
	\overline\Delta_T&=\overline\theta_T-\theta^\star,
	&
	\varepsilon_t&=A_t\theta^\star-b_t,
	&
	\Gamma&=\E(\varepsilon_t\varepsilon_t^\top).
	\label{eq-ab-05}
\end{align}
Since $A\theta^\star=b$, we have $\E\varepsilon_t=0$.  For brevity, we write
$\nu_p=\|\varepsilon_1\|_p$ for $p\geq 1.$ The asymptotic covariance is
$\Lambda^\star=A^{-1}\Gamma A^{-\top}$. Under fixed-dimensional
regularity conditions,
$\sqrt T\,\overline\Delta_T\Rightarrow\mathcal N(0,\Lambda^\star)$;
see \citet{PolyakJuditsky1992,Fort2015}. Nonasymptotic analyses of linear
stochastic approximation and TD learning include
\citet{MouEtAl2020,SamsonovEtAl2024,ButyrinEtAl2025,WuEtAl2026}.
In particular, under a bounded-feature normalization,
\citet{WuEtAl2026} obtain an $O(T^{-\alpha/2})$ Berry--Esseen bound for a
Gaussian law with the finite-sample covariance and control the difference
between that covariance and $\Lambda^\star$ separately. Their bounds give
the fixed-dimensional rate $O(T^{-1/3})$ when $\alpha=2/3$.

Our aim is to establish a Berry--Esseen bound in which the dependence on the dimension $d$ is tracked explicitly.
For $1\leq t\leq T$, define
\[
Q_t
=\eta_t\sum_{j=t}^T
\prod_{k=t+1}^j(I_d-\eta_kA),
\qquad
Q_0
=\eta_0\sum_{j=1}^T
\prod_{k=1}^j(I_d-\eta_kA),
\]
where an empty product is $I_d$. The second definition is needed for the
contribution of the deterministic initial error. Put
\[
\overline\Lambda_T
=\frac1T\sum_{t=1}^TQ_t\Gamma Q_t^\top,
\qquad
V_T=A\overline\Lambda_TA^\top,
\qquad
\underline\lambda_T=\lambda_{\min}(V_T).
\]
If $\Gamma$ is positive definite, then $\underline\lambda_T>0$ because
$Q_T=\eta_TI_d$ and $A$ is nonsingular. Finally, let
$b_\eta=1+\eta_0^{-1/(1-\alpha)}$.  Throughout this subsection, \(C\) denotes a positive constant depending only on \(\alpha\) and \(K_\phi\), and is independent of \(d\), \(T\), \(\eta_0\), \(\|\Delta_0\|_2\), \(\nu_p\), \(\operatorname{Tr}(\Gamma)\), and \(\underline\lambda_T\).

\begin{theorem}\label{thm-app-TD}
	Suppose that \eqref{eq-ab-01} and \eqref{eq-ab-04} hold.
	Assume further that $\Gamma$ is positive definite and $\nu_3<\infty$. Then
	\begin{align}
		\begin{aligned}
			\dc\bigl(
			\sqrt T\,(\bar\theta_T-\E\bar\theta_T),
			\overline\Lambda_T^{1/2}Z
			\bigr)
			&\leq C\bigg[
			\frac{d^{1/4}\nu_3^3}{\underline\lambda_T^{3/2}}
			\bigg\{
			\frac1{T^{1/2}}
			+\frac{b_\eta^{3\alpha}}{T^{3/2}}
			\bigg\}+
			\frac{(1+\log d)\nu_3
				\operatorname{Tr}(\Gamma)^{1/2}}
			{\underline\lambda_TT^{\alpha/2}}
			\\
			&\quad{}+
			\frac{\eta_0^{-1/2}
			\bigl\{|\Delta_{0}|+
			(\eta_0\operatorname{Tr}(\Gamma))^{1/2}\bigr\}
			b_\eta^\alpha}
			{\underline\lambda_T^{1/2}T^{1/2}}
			\bigg].
		\end{aligned}
		\label{eq-app2-thm-01}
	\end{align}
\end{theorem}
As a consequence of Theorem 3.9, we obtain the following bound for \(\sqrt T\,\overline\Delta_T\).
\begin{corollary}\label{coro-app-TD-uncentered}
	Under the assumptions of Theorem~\ref{thm-app-TD}, we have
	\begin{align}
		\begin{aligned}
		&	\dc\bigl(
			\sqrt T\,\overline\Delta_T,
			\overline\Lambda_T^{1/2}Z
			\bigr)\\
			&\quad\leq C\bigg[
			\frac{d^{1/4}\nu_3^3}{\underline\lambda_T^{3/2}}
			\bigg\{
			\frac1{T^{1/2}}
			+\frac{b_\eta^{3\alpha}}{T^{3/2}}
			\bigg\}+
			\frac{(1+\log d)\nu_3
				\operatorname{Tr}(\Gamma)^{1/2}}
			{\underline\lambda_TT^{\alpha/2}}
			\\
			&\qquad\quad{}+
			\frac1{\underline\lambda_T^{1/2}T^{1/2}}
			\bigg\{
			\|\Delta_{0}\|_{2}b_\eta
			+\eta_0^{-1/2}
			\bigl\{|\Delta_{0}|+
			(\eta_0\operatorname{Tr}(\Gamma))^{1/2}\bigr\}
			b_\eta^\alpha
			\bigg\}
			\bigg].
		\end{aligned}
		\label{eq-app2-coro-01}
	\end{align}
\end{corollary}

\begin{remark}\label{rem:TD-high-dimensional-CLT}
	The theorem does not prescribe how its dimension-sensitive quantities
	scale. To obtain a rate involving only $d$ and $T$, suppose
	in addition that, for a constant $K\geq1$
	independent of $d$ and $T$,
	\begin{align}
		\begin{aligned}
			&B_\phi^2\leq Kd,\quad
			K^{-1}d^{-1}\leq\eta_0\leq Kd^{-1},\quad
			K^{-1}I_d\preceq\Gamma\preceq KI_d,
			\\
			&|\Delta_{0}|\leq Kd^{1/2},
			\quad
			\nu_3\leq Kd^{1/2}.
		\end{aligned}
		\label{eq-ab-06}
	\end{align}
		Here $\alpha\in(1/2,1)$ is fixed, and $c_K,C_K>0$ denote
		constants that may change from line to line and depend only on
		$\alpha$, $K_\phi$, and $K$.
	Put $s_\alpha=(1-\alpha)^{-1}$. Then
	$b_\eta\leq C_Kd^{s_\alpha}$ and
	$\operatorname{Tr}(\Gamma)\leq Kd$. Moreover, applying Lemma~6.12 of
	\citet{WuEtAl2026} to the rescaled feature map
	$\widetilde\phi=\phi/B_\phi$ and the corresponding stepsizes
	$\widetilde\eta_t=B_\phi^2\eta_t$ yields constants
	$c_K,C_K>0$ such that
	$\underline\lambda_T\geq c_K$ whenever
	$T\geq C_Kd^{s_\alpha}$.
		Combining this lower bound with Theorem~\ref{thm-app-TD} and
		using the trivial bound $d_c\leq1$ when necessary, we obtain,
		for every integer $T\geq1$,
	\begin{align}\label{eq-hh-01}
		\dc\bigl(
		\sqrt T\,(\bar\theta_T-\E\bar\theta_T),
		\overline\Lambda_T^{1/2}Z
		\bigr)
		\leq C_K\bigg\{
		\frac{d(1+\log d)}{T^{\alpha/2}}
		+\frac{d^{s_\alpha}}{T^{1/2}}
		\bigg\}.
	\end{align}
		The details are provided in Appendix~\ref{app:TD-rate-verification}.
	Since $(1-\alpha)/2<\alpha/2$, the Gaussian approximation error
	converges to zero if $d=O(T^\kappa)$ with
	$\kappa<(1-\alpha)/2$. In particular, $\alpha=2/3$ permits
	every $\kappa<1/6$, and
the sufficient threshold for $\kappa$ approaches $1/4$
	as $\alpha\downarrow1/2$.

	For the original estimation error,
		Corollary~\ref{coro-app-TD-uncentered} and the bound $d_c\leq1$
		similarly give, for every integer $T\geq1$,
	\begin{align}
		\dc\bigl(
		\sqrt T\,\overline\Delta_T,
		\overline\Lambda_T^{1/2}Z
		\bigr)
		\leq C_K\bigg\{
		\frac{d(1+\log d)}{T^{\alpha/2}}
		+\frac{d^{s_\alpha+1/2}}{T^{1/2}}
		\bigg\}.
		\label{eq-ab-08}
	\end{align}
	Since $(1-\alpha)/(3-\alpha)<\alpha/2$, it suffices that
	$d=O(T^\kappa)$ with $\kappa<(1-\alpha)/(3-\alpha)$.
	For example, $\alpha=2/3$ permits every $\kappa<1/7$, while
	the sufficient threshold for $\kappa$ approaches $1/5$
	as $\alpha\downarrow1/2$.
\end{remark}

\subsection{\texorpdfstring{Multivariate $U$-statistics}{Multivariate U-statistics}}\label{subsec:multiple-U}

A $U$-statistic is obtained by averaging a symmetric kernel over all
subsamples of a fixed size. Introduced by \cite{Hoeffding1948},
this class includes many familiar statistics, such as the sample mean,
the sample variance, and rank-based statistics. For a nondegenerate
scalar $U$-statistic, the Hoeffding decomposition reduces the leading
term to a sum of independent random variables, and the classical
Berry--Esseen rate of order $n^{-1/2}$ is available under a third-moment
condition; see, for example,
\cite{CallaertJanssen1978,louis2007}. Multivariate Berry--Esseen and Edgeworth
approximations were studied by
\cite{HelmersJanssen1982} and  \cite{Goetze1987}. More recently,
\cite{Chen2018} established high-dimensional Gaussian and bootstrap
approximations over hyperrectangles. Here we consider a vector of
$U$-statistics whose coordinates may have different orders and derive
normal-approximation bounds over the class of all convex sets.

Let $X_1,\dots,X_n$ be independent and identically distributed random
variables taking values in a measurable space $\mathcal X$. For each
$1\leq i\leq d$, let $2\leq m_i\leq n$ and let
$h_i:\mathcal X^{m_i}\to\mathbb R$ be a measurable symmetric kernel.
Define
\begin{align*}
	U_{i,n}
	=
	\binom{n}{m_i}^{-1}
	\sum_{1\leq j_1<\cdots<j_{m_i}\leq n}
	h_i(X_{j_1},\dots,X_{j_{m_i}}).
\end{align*}
Assume that for each $1\leq i\leq d$,
$
\mathbb E h_i(X_1,\dots,X_{m_i})=0,$
$\sigma_i^2
=
\mathbb E h^2_{i}(X_1,\dots,X_{m_i})
<\infty.
$
The first Hoeffding projection of $h_i$ is
$
h_{i,1}(x)
=
\mathbb E h_i(x,X_2,\dots,X_{m_i}),$
$\lambda_i^2
=
\mathbb E h_{i,1}(X_1)^2.
$
Throughout this subsection, we consider the nondegenerate case
$\lambda_i>0$ for every $1\leq i\leq d$. Put
\begin{align*}
	\widetilde h_i(x_1,\dots,x_{m_i})
	=
	h_i(x_1,\dots,x_{m_i})
	-
	\sum_{\ell=1}^{m_i}h_{i,1}(x_\ell).
\end{align*}
Then the first-order Hoeffding decomposition gives
\begin{align*}
	\frac{\sqrt n}{m_i\lambda_i}U_{i,n}
	=
	W_{i,n}+D_{i,n},
\end{align*}
where
\begin{align*}
	W_{i,n}
	&=
	\frac{1}{\sqrt n\lambda_i}
	\sum_{j=1}^n h_{i,1}(X_j),\quad
	D_{i,n}
	=
	\frac{\sqrt n}{m_i\lambda_i}
	\binom{n}{m_i}^{-1}
	\sum_{1\leq j_1<\cdots<j_{m_i}\leq n}
	\widetilde h_i(X_{j_1},\dots,X_{j_{m_i}}).
\end{align*}
Define
$
G(x)
=
\big(
h_{1,1}(x)/\lambda_1,
\cdots,
h_{d,1}(x)/\lambda_d
\big)^{\top},
$
$Y_j=G(X_j)/\sqrt{n}$,
and set
\begin{align}\label{eq-ac-01}
	W=(W_{1,n},\dots,W_{d,n})^\top,
	\qquad
	D=(D_{1,n},\dots,D_{d,n})^\top.
\end{align}
Thus
\begin{align*}
	T
	:=
	\Big(
	\frac{\sqrt n\,U_{1,n}}{m_1\lambda_1},
	\cdots,
	\frac{	\sqrt n\,U_{d,n}}{m_d\lambda_d}
	\Big)^{\top}
	=
	W+D
	=
	\sum_{j=1}^nY_j+D.
\end{align*}
The covariance matrix of the linear part is
$
\Sigma
=
\mathbb E WW^\top
=
\mathbb E\{G(X_1)G(X_1)^\top\}.
$
In particular, every diagonal entry of $\Sigma$ equals one.
Define
\begin{align*}
	\Lambda_{n}^{2}
	=
	\sum_{i=1}^d
	\frac{(m_i-1)^{2}\sigma_i^2}
	{m_i(n-m_i+1)\lambda_i^2}.
\end{align*}
Applying
Corollaries \ref{coro-main-1} and
 \ref{coro-main-2}, we have the following theorem.

\begin{theorem}\label{thm-app-U-statistics}
	Let $T$, $W$, $D$, and $\Sigma$ be defined as above, and suppose that
	$\Sigma$ is positive definite with
	$\lambda_{\min}(\Sigma)\geq\sigma>0$.
	\begin{itemize}
		\item [(i)]	There exists an absolute constant $C$ such that
		\begin{align}\label{eq-app3-thm-01}
			\begin{aligned}
				d_{c}(T,\Sigma^{1/2}Z)&\leq
				C\Big\{
				\sigma^{-3/2}d^{1/4}n^{-1/2}	\|G(X_1)\|_{3}^{3}
				\\&\quad+
				\sigma^{-1}(1+\log d)	\|G(X_1)\|_{3}\Lambda_n
				+
				\sigma^{-1/2}\Lambda_n
				\Big\}.
			\end{aligned}
		\end{align}
		\item [(ii)]
		There exists an absolute constant $C$ such that
		\begin{align}\label{eq-app3-thm-02}
			\begin{aligned}
				d_{c}(T,\Sigma^{1/2}Z)&\leq
				C\Big\{
				\sigma^{-1}	d^{1/4}n^{-1/2}	\|G(X_1)\|_{4}^{2}
				+
				\sigma^{-1}(1+\log d)
				\Lambda_{n}\|G(X_1)\|_{4}\\
				&\quad\quad\qquad+
				\sigma^{-3/4}d^{1/8}n^{-1/4}
				\Lambda_{n}^{1/2} 	\|G(X_1)\|_{4}
				+
				\sigma^{-1/2}\Lambda_n
				\Big\}.
			\end{aligned}
		\end{align}
	\end{itemize}

\end{theorem}

\begin{remark}\label{rem:U-simplified-rates}
	The assumptions involving the orders $m_i$, the variance ratios
	$\sigma_i^2/\lambda_i^2$, and the growth of the projection moments are
	needed only when one wants a simpler rate in terms of $n$ and $d$.
	For example, suppose that, for fixed constants $M,L_3,L_4>0$,
	\begin{align*}
		\max_{1\leq i\leq d}m_i\leq M,
		\qquad
		\max_{1\leq i\leq d}\sigma_i^2/\lambda_i^2\leq M,
		\qquad
		\|G(X_1)\|_{3}^3\leq L_3d^{3/2},
		\qquad
		\|G(X_1)\|_{4}^4\leq L_4d^2.
	\end{align*}
If $n+2>2M$, then
$
\Lambda_n^2
\leq
(M^2d)/(n-M+1)
\leq
(2M^2d)/n.
$
Since the diagonal entries of $\Sigma$ are one, $0<\sigma\leq1$.
Using $d(1+\log d)\leq Cd^{5/4}$ for $d\geq1$, part~(i) of
Theorem~\ref{thm-app-U-statistics} gives
$
d_c(T,\Sigma^{1/2}Z)
\leq
C_{M,L_3}\sigma^{-3/2}d^{7/4}n^{-1/2}.
$
Similarly, part~(ii) gives
$
d_c(T,\Sigma^{1/2}Z)
\leq
C_{M,L_4}\sigma^{-1}d^{5/4}n^{-1/2}.
$
Here $C_{M,L_3}$ and $C_{M,L_4}$ depend only on the
quantities indicated by their subscripts.
Thus, when $\sigma$ is bounded away from zero,
$d^{7/2}/n\to0$ is sufficient for the bound
\eqref{eq-app3-thm-01} to converge to zero, whereas
$d^{5/2}/n\to0$ is sufficient for the bound
\eqref{eq-app3-thm-02} to converge to zero.
\end{remark}

\section{Proof of main results}\label{sec: proof-of-main-results}
In this section, we prove the main results. We begin by establishing
several preliminary lemmas and propositions needed in the subsequent
proofs. The remaining two subsections are devoted, respectively, to the
proofs of Theorem~\ref{thm-main-1} and Corollary~\ref{coro-main-1}, and
to those of Theorem~\ref{thm-main-2} and
Corollary~\ref{coro-main-2}. Throughout this section, \(C\) denotes a
positive absolute constant whose value may vary from line to line.
\subsection{Preliminary lemmas and propositions}
Throughout this section, we use the following notation. For a sufficiently
smooth function $f:\R^d\to\R$, write
$\partial_i f=\partial f/\partial x_i$ and
$\partial_{ij}f=\partial^2f/(\partial x_i\partial x_j)$, and regard
$\nabla^r f(x)$ as the symmetric $r$-linear form defined by
\begin{align*}
	\big\langle \nabla^r f(x),u_1\otimes\cdots\otimes u_r\big\rangle
	:=
	\sum_{j_1,\ldots,j_r=1}^d
	\frac{\partial^r f(x)}
	{\partial x_{j_1}\cdots\partial x_{j_r}}
	\prod_{\ell=1}^r u_{\ell j_\ell},
\end{align*}
where $u_\ell=(u_{\ell1},\ldots,u_{\ell d})^\top\in\R^d$. When \(u_1=\cdots=u_r=u\), we write
\(u_1\otimes\cdots\otimes u_r=u^{\otimes r}\).
For an \(r\)-linear form \(T\), its injective norm is
\[
|T|_{\vee}
:=
\sup_{|u_1|,\ldots,|u_r|\leq 1}
\big|
\langle T,u_1\otimes\cdots\otimes u_r\rangle
\big|.
\] Finally, set $M_0(f)
:=
(
\sup_{w\in\R^d}f(w)-\inf_{w\in\R^d}f(w))/2,$ and
\[
M_r(f)
:=
\sup_{w\neq z}
\frac{|\nabla^{r-1}f(w)-\nabla^{r-1}f(z)|_{\vee}}{|w-z|},
\quad r\geq1,
\]
where $\nabla^0f:=f$. We set $M_r(f)=\infty$ if $f$ is not everywhere
$(r-1)$-times differentiable.

Let \(\phi_d\) denote the density of \(\mathcal N(0,I_d)\). We will repeatedly use the following consequence of rotational invariance: for every integer \(r\geq1\) and \(u\in\R^d\),
\begin{align}\label{eq-ss-01}
		\int_{\R^d}
	\big|
	\big\langle\nabla^r\phi_d(z),u^{\otimes r}\big\rangle
	\big|dz
	\leq c_r|u|^r,
\end{align}
where $c_r:=\int_{\R}|\phi_1^{(r)}(z)|dz.$

The following two lemmas are useful in the proof of our main results.
\begin{lemma}[\cite{bentkus2003}, Lemma~2.3]\label{lem4.1}
	For every \(A\in\mathcal A\) and \(\varepsilon>0\), there exists a
	differentiable function \(f_A^\varepsilon:\R^d\to\R\) such that
	\(f_A^\varepsilon=1\) on \(A\),
	\(f_A^\varepsilon=0\) on \(\R^d\setminus A^\varepsilon\), and
	\(0\leq f_A^\varepsilon\leq1\). Moreover,
	\[
	M_1(f_A^\varepsilon)\leq\frac{2}{\varepsilon},
	\qquad
	M_2(f_A^\varepsilon)\leq\frac{8}{\varepsilon^2}.
	\]
\end{lemma}

\begin{lemma}[\cite{FANG2015}, Lemma~4.2]\label{lem4.2}
	Let \(T\) be an \(\R^d\)-valued random vector and
	\(Z\sim\mathcal N(0,I_d)\). Then, for every \(\varepsilon>0\),
	\begin{align}\label{eq-ha-02}
		\sup_{A\in\mathcal A}
		\big|\P(T\in A)-\P(Z\in A)\big|
		\leq
		\sup_{A\in\mathcal A}
		\big|
		\E f_A^\varepsilon(T)-\E f_A^\varepsilon(Z)
		\big|
		+4d^{1/4}\varepsilon.
	\end{align}
\end{lemma}

Fix \(A\in\mathcal A\) and \(0<\varepsilon\leq1\), and let
\(f_A^\varepsilon\) be as in Lemma~\ref{lem4.1}. For notational
convenience, we write \(f=f_A^\varepsilon\) whenever there is no risk
of confusion.  We now turn to estimating
\(\lvert \E f(T)-\E f(Z)\rvert\). To this end, we follow the Gaussian
interpolation argument in \cite{raivc2019b}. For \(0\leq\alpha\leq\pi/2\), define
\begin{align}\label{eq-ha-03}
	g_\alpha(x)
	:=
	\int_{\R^d}
	f(x\cos\alpha+z\sin\alpha)\phi_d(z)\,dz.
\end{align}
Since \(M_2(f)<\infty\), differentiation under the integral sign,
together with Gaussian integration by parts, gives, for
\(0<\alpha<\pi/2\),
\begin{align}
	\nabla g_\alpha(x)
	&=
	\cos\alpha
	\int_{\mathbb{R}^d}
	\nabla f\bigl(x\cos\alpha+z\sin\alpha\bigr)
	\phi_d(z)\,dz
	\nonumber\\
	&=
	-\cot\alpha
	\int_{\mathbb{R}^d}
	f\bigl(x\cos\alpha+z\sin\alpha\bigr)
	\nabla\phi_d(z)\,dz,
	\label{eq-ha-04}\\
	\nabla^2 g_\alpha(x)
	&=
	\cos^2\alpha
	\int_{\mathbb{R}^d}
	\nabla^2 f\bigl(x\cos\alpha+z\sin\alpha\bigr)
	\phi_d(z)\,dz
	\nonumber\\
	&=
	\cot^2\alpha
	\int_{\mathbb{R}^d}
	f\bigl(x\cos\alpha+z\sin\alpha\bigr)
	\nabla^2\phi_d(z)\,dz
	\label{eq-ha-05}\\
	&=
	-\cos\alpha\cot\alpha
	\int_{\mathbb{R}^d}
	\nabla f\bigl(x\cos\alpha+z\sin\alpha\bigr)
	\otimes\nabla\phi_d(z)\,dz,
	\nonumber\\
	\nabla^3 g_\alpha(x)
	&=
	-\cos^2\alpha\cot\alpha
	\int_{\mathbb{R}^d}
	\nabla^2 f\bigl(x\cos\alpha+z\sin\alpha\bigr)
	\otimes\nabla\phi_d(z)\,dz
	\nonumber\\
	&=
	-\cot^3\alpha
	\int_{\mathbb{R}^d}
	f\bigl(x\cos\alpha+z\sin\alpha\bigr)
	\nabla^3\phi_d(z)\,dz
	\label{eq-ha-06}	\\
	&=
	\cos\alpha\cot^2\alpha
	\int_{\mathbb{R}^d}
	\nabla f\bigl(x\cos\alpha+z\sin\alpha\bigr)
	\otimes\nabla^2\phi_d(z)\,dz.\nonumber
\end{align}
Since \(g_0=f\) and \(g_{\pi/2}(x)=\E f(Z)\), we have
\begin{equation}\label{eq-ha-07}
	\begin{aligned}
		\E f(T)-\E f(Z)
		&=
		-\int_0^{\pi/2}
		\E\{\partial_\alpha g_\alpha(T)\}\,d\alpha,
	\end{aligned}
\end{equation}
where $	\partial_\alpha g_\alpha(x)=
(
\Delta g_\alpha(x)
-\langle\nabla g_\alpha(x),x\rangle
)\tan\alpha.$

The following lemma provides an upper bound for the integral involving \(\E\{\nabla^3 g_\alpha(T)\}\), which will be used in the proof of Theorem~\ref{thm-main-1}.
\begin{lemma}[\cite{raivc2019}, Lemma~2.7]\label{lem4.3}
	Let \(T\) be an \(\R^d\)-valued random vector. Suppose that there
	exist \(\mu\in\R^d\), a positive-definite matrix \(\Sigma\), and
	\(\delta\geq0\) such that
	\[
	\sup_{B\in\mathcal A}
	\big|
	\P(T\in B)
	-\P(\mu+\Sigma^{1/2}Z\in B)
	\big|
	\leq\delta.
	\]
	Let \(\sigma^2=\lambda_{\min}(\Sigma)\) and assume that
	\(0<\sigma^2\leq1\). Then, for every \(\varepsilon>0\),
	\[
	\int_0^{\pi/2}
	\big|
	\E\nabla^3g_\alpha(T)
	\big|_{\vee}
	\tan\alpha\,d\alpha
	\leq
	C_1
	\Big(
	\frac{1}{\sigma^3}
	+\frac{d^{1/4}}{\sigma}
	+\frac{\delta}{\varepsilon}
	\Big),
	\]
	where
	\(C_1:=\max\{c_3/6,\,8\sqrt{c_1c_3}\}\), and \(c_r:=\int_{\R}|\phi_1^{(r)}(z)|dz\) for \(r=1,3\).
\end{lemma}
We next state a proposition that plays a key role in the proofs of our two main theorems. The proof of this proposition is given in Appendix~\ref{sec:proof-of-proposition-4.5}.
\begin{proposition}\label{Prop-4.5}
	Let \(U\) be an \(\mathbb{R}^{d}\)-valued random vector and \(V\) be a random \(d\times d\) matrix. Let $q\geq 2$ and $r>1$ satisfy
	$
	1/q+1/r=1.
	$
	Then, for any $\varepsilon\in(0,1)$, we have
	\begin{align*}
		\begin{aligned}
			&\bigg|\int^{\frac{\pi}{2}}_{0}\mathbb{E}\langle \nabla^{2} g_{\alpha}(U), V\rangle\tan \alpha d\alpha\bigg| \leq C\Big(qd^{1/(2q)}+|\log \varepsilon| [d_{c}(U,Z)]^{1/q}\Big)(\mathbb{E}\|V\|_{\Fm}^{r})^{1/r},
		\end{aligned}
	\end{align*}
	where $C>0$ is an absolute constant.
\end{proposition}

\subsection{Proofs of Theorem \ref{thm-main-1} and Corollary \ref{coro-main-1}}
In this subsection, we provide the proofs of Theorem \ref{thm-main-1} and Corollary \ref{coro-main-1}.
\begin{proof}[Proof of Theorem \ref{thm-main-1}] 	We use a recursive argument to prove this theorem.  	To this end,  for each integer $m\ge1$,  we define
	\begin{align}\label{eq-pp-05}
		K_{m}:=\sup\frac{\sup_{A\in\mathcal{A}}|\mathbb{P}(T\in A)-\mathbb{P}(Z\in A)|}{\delta},
	\end{align}
	where $
	\delta=	d^{1/4}\sum_{i=1}^{n}\mathbb{E}|X_i|^3
	+qd^{1/(2q)}
	\sum_{i=1}^{n}
	\|D-D^{\{i\}}\|_{r}\|X_i\|_{s}
	+\|D\|_{3/2},
	$
and the supremum is taken over all choices of the independent
random elements $\xi_1,\ldots,\xi_n$, the Borel measurable
functions $g_1,\ldots,g_n$, and the statistic
$T=T(\xi_1,\ldots,\xi_n)$ satisfying
\eqref{eq-na-01} and \eqref{eq-na-02}, with $1\le n\le m$.
Here $d,q,r,s$ are held fixed. By \eqref{eq-na-02} and H\"older's inequality, we have
\[
\delta
\ge d^{1/4}\sum_{i=1}^n \mathbb E|X_i|^3
\ge \frac{d^{1/4}}{\sqrt n}
\Big(\sum_{i=1}^n \mathbb E|X_i|^2\Big)^{3/2}
= \frac{d^{7/4}}{\sqrt n}.
\]
Since $d_c(T,Z)\le1$, it follows that
$K_m\le \sqrt m\,d^{-7/4}<\infty$ for every $m\ge1$.  We shall prove that there exits an absolute constant $C^{\dagger}$ such that
\begin{align}\label{eq-ss-02}
K_{n}\leq C^{\dagger}.
\end{align}
 For any $1\leq i\leq n$,
	let $W^{(i)}=W-X_{i}$. Recalling that  $D^{\{i\}}=D(\xi_1,\dots, \xi_{i-1}, \xi'_i,\xi_{i+1}, \dots,\xi_n)$,  we have $(W^{(i)},D^{\{i\}})$ is independent of $X_i$. Hence, by the decomposition
	\eqref{eq-na-01} of $T$ and the fact that $\mathbb{E}X_i=0$, we have
	\begin{align}\label{eq-pp-01}
		\begin{aligned}
			\mathbb{E}\langle\nabla g_{\alpha}(T), T\rangle&=\E\la \na g_{\alpha}(T), D\ra+\sum_{i=1}^{n}\E \la \na g_{\alpha}(T), X_{i}\ra \\
			&=\mathbb{E}\langle\nabla g_{\alpha}(T), D\rangle+\sum_{i=1}^{n}\mathbb{E}\langle \nabla g_{\alpha}(T)-\nabla g_{\alpha}(W+D^{\{i\}}), X_{i}\rangle \\
			&\quad+\sum_{i=1}^{n}\mathbb{E}\langle \nabla g_{\alpha}(W+D^{\{i\}})-\nabla g_{\alpha}(W^{(i)}+D^{\{i\}}), X_{i}\rangle.
		\end{aligned}
	\end{align}
	Note that $I_{d}=\sum_{i=1}^{n}\Var(X_{i})$ and $\Delta g_{\alpha}(T)=\la \na^{2}g_{\alpha}(T), I_{d}\ra$. Combining this with \eqref{eq-pp-01} yields
	\begin{align}\label{eq-pp-02}
		\mathbb{E}\langle\nabla g_{\alpha}(T), T\rangle- \mathbb{E}\Delta g_{\alpha}(T)=I_{1,\alpha}+I_{2,\alpha}+I_{3,\alpha},
	\end{align}
	where
	\begin{align*}
		I_{1,\alpha}  &= \mathbb{E}\langle \nabla g_{\alpha}(T),D\rangle, \quad
		I_{2,\alpha}  =  \sum_{i=1}^{n}\mathbb{E}\langle \nabla g_{\alpha}(T)-\nabla g_{\alpha}(W+D^{\{i\}}), X_{i}\rangle,\\
		I_{3,\alpha}&=\sum_{i=1}^{n}\mathbb{E}\langle \nabla g_{\alpha}(W+D^{\{i\}})-\nabla g_{\alpha}(W^{(i)}+D^{\{i\}}), X_{i}\rangle-\sum_{i=1}^{n}\mathbb{E}\langle \nabla^{2} g_{\alpha}(T), \mathrm {Var}(X_{i})\rangle.
	\end{align*}
	Combining \eqref{eq-ha-07} and \eqref{eq-pp-02}, we have
	\begin{align} \label{eq-pp-03}
		\mathbb{E}f(T)-\mathbb{E}f(Z)=\int^{\pi/2}_{0}(I_{1,\alpha}+I_{2,\alpha}+I_{3,\alpha})\tan \alpha d\alpha.
	\end{align}
	Next, we deal with $I_{1,\alpha}$, $I_{2,\alpha}$, and $I_{3,\alpha}$ separately.
	For $I_{1,\alpha}$, by  \eqref{eq-ss-01} and \eqref{eq-ha-04}, we have
	\begin{align}\label{eq-pp-04}
		\begin{aligned}
			\Big|\int^{\pi/2}_{0}I_{1,\alpha}\tan \alpha d\alpha \Big|
			&=\Big|\mathbb{E}\int^{\pi/2}_{0}\int_{\mathbb{R}^d}f(T\cos \alpha+z\sin \alpha)\langle \nabla\phi_{d}(z),D\rangle dzd\alpha\Big|\\
			&\leq \frac{\pi}{2} \mathbb{E}c_{1}M_{0}(f)|D|\leq \mathbb{E}|D|,
		\end{aligned}
	\end{align}
	where  the last inequality  is derived by the fact $c_{1}=\sqrt{2/\pi}$ and $M_{0}(f)\leq 1/2$. In what follows, we
 give the upper bound for
	\begin{align*}
		I_{2}:=	\int^{\pi/2}_{0}I_{2,\alpha}\tan \alpha d\alpha\quad \text{and}\quad 	 	I_{3}:=\int^{\pi/2}_{0}I_{3,\alpha}\tan \alpha d\alpha.
	\end{align*}

	For  $I_{2,\alpha},$ applying Taylor expansion for $\nabla g_{\alpha}(W+D^{\{i\}})$ at $W+D$ yields
	\begin{align*}
		I_{2,\alpha}=\sum_{i=1}^{n}\mathbb{E}\langle \nabla^{2} g_{\alpha}(W+D+\theta(D^{\{i\}}-D)),  X_{i}\otimes(D-D^{\{i\}})\rangle,
	\end{align*}
	where $\theta$ is uniformly distributed over $[0,1]$ and is independent of all  other random variables. Let $(\eta_1,\dots,\eta_n)$ an independent copy of $(\xi_1,\dots,\xi_n)$ and is independent of  all other variable. For each $1\leq i \leq n$ and fixed $\theta$, define $Q_i(\theta)=D+\theta(D^{\{i\}}-D)=Q(\xi_1,\dots,\xi_n,\xi'_i,\theta),$ $Q^{\{j\}}_i=Q(\xi_1,\dots,\xi_{j-1},\eta_j,\xi_{j+1}\dots, \xi_n, \xi'_i,\theta),$ and
	\begin{align*}
		\delta^{\{i\}}(\theta)=&d^{1/4}\sum_{i=1}^{n}\E|X_{i}|^{3}+qd^{1/(2q)}\sum_{j=1}^{n}\|Q_i(\theta)-Q^{\{j\}}_i(\theta)\|_{r}\|X_j\|_{s}+\|Q_i(\theta)\|_{3/2}.
	\end{align*}
	By Minkowski's inequality, we have
	\begin{align}\label{eq-pp-5.1}
		\|Q_i(\theta)-Q^{\{j\}}_i(\theta)\|_{r}\leq (1-\theta)\|D-D^{\{j\}}\|_{r}+\theta\|D-D^{\{j\}}\|_{r}=\|D-D^{\{j\}}\|_{r}.
	\end{align}
	In addition,by the definitions of $Q_i$, a direct calculation gives $\|Q_i(\theta)\|_{3/2}\leq \|D\|_{3/2}.$
	Combining those two estimates, we have for each  $1\leq i \leq n$,
	$
	\delta^{\{i\}}(\theta)\leq \delta.
	$

	For each fixed $i$ and $\theta\in[0,1]$, condition on $\xi_i'$.
	The original inputs $\xi_1,\ldots,\xi_n$ remain independent
	with their distributions unchanged, and $W+Q_i(\theta)$ is an
	admissible statistic based on these $n$ inputs.
	Applying the definition of $K_n$ conditionally, averaging
	over $\xi_i'$, and using Jensen's inequality for the
	conditional moment terms, we obtain
\begin{align}\label{eq-ss-03}
		d_c(W+Q_i(\theta),Z)
	\le K_n\delta^{\{i\}}(\theta)
	\le K_n\delta.
\end{align}
Let $q'=q/(q-1)$.
	Applying Proposition \ref{Prop-4.5} with
	$U=W+Q_i(\theta)$ and
	$V=X_{i}\otimes(D-D^{\{i\}})$, and then using \eqref{eq-ss-03},
	we have, for each fixed $\theta$,
	\begin{align*}
		&\bigg|\int^{\frac{\pi}{2}}_{0}\sum_{i=1}^{n}\mathbb{E}\langle \nabla^{2} g_{\alpha}(W+D+\theta(D^{\{i\}}-D)), (D-D^{\{i\}})\otimes X_{i}\rangle\tan \alpha d\alpha\bigg|\\
			&\quad\leq \sum_{i=1}^{n}\bigg|\int^{\frac{\pi}{2}}_{0}\mathbb{E}\langle \nabla^{2} g_{\alpha}(W+D+\theta(D^{\{i\}}-D)), (D-D^{\{i\}})\otimes X_{i}\rangle\tan \alpha d\alpha\bigg|\\
			&\quad\leq C\sum_{i=1}^{n}\big(qd^{1/(2q)}+|\log \varepsilon| [d_{c}(W+Q_{i}(\theta),Z)]^{\frac{1}{q}}\big) \cdot \big(\E\|X_{i}\otimes(D-D^{\{i\}})\|_{\Fm}^{q'}\big)^{\frac{1}{q'}}\\
			&\quad\leq C\sum_{i=1}^{n}\big(qd^{1/(2q)}+|\log \varepsilon|[{K_{n}}\delta^{\{i\}}(\theta)]^{1/q}\big)\cdot\|D-D^{\{i\}}\|_{r}\|X_{i}\|_{s}\\
			&\quad\leq  C\Big(qd^{1/(2q)}+|\log \varepsilon|({K_{n}}\delta)^{1/q}\Big)\sum_{i=1}^{n}\|D-D^{\{i\}}\|_{r}\|X_{i}\|_{s},
	\end{align*}
	where we recall that $1/s+1/r=1-1/q.$ Hence,
	\begin{align}\label{eq-pp-07}
		|I_{2}|\leq C\Big(qd^{1/(2q)}+|\log \varepsilon|({K_{n}}\delta)^{1/q}\Big)\sum_{i=1}^{n}\|D-D^{\{i\}}\|_{r}\|X_{i}\|_{s}.
	\end{align}
	It remains to bound $I_3$.  Applying Taylor’s formula, we have
	\begin{align*}
		&\mathbb{E}\langle \nabla g_{\alpha}(W+D^{\{i\}})-\nabla g_{\alpha}(W^{(i)}+D^{\{i\}}), X_{i}\rangle\\
		&\qquad=\int_{0}^{1}\mathbb{E}\langle \nabla^{2} g_{\alpha}(W^{(i)}+D^{\{i\}}+tX_{i}),  X_{i}\otimes X_{i}\rangle dt.
	\end{align*}
	Since $W^{(i)}+X_i'+D^{\{i\}}$ has the same distribution as $T$,
	is independent of $X_i$, and $\mathbb{E}X_i=0$, we have
	\begin{align*}
		&	\sum_{i=1}^{n}\mathbb{E}\big\langle
		\nabla^{2}g_{\alpha}\bigl(W^{(i)}+D^{\{i\}}+X'_{i}\bigr),
		X_i\otimes X_{i}
		\big\rangle\\
		&\quad=	\sum_{i=1}^{n}\mathbb{E}\big\langle
		\nabla^{2}g_{\alpha}\bigl(W^{(i)}+D^{\{i\}}+X'_{i}\bigr),
		\va(X_{i})
		\big\rangle
		=\sum_{i=1}^{n}\mathbb{E}\big\langle
		\nabla^{2}g_{\alpha}\bigl(T\bigr),
		\va(X_{i})
		\big\rangle.
	\end{align*}
	Consequently, we have
	\begin{align*}
		I_{3,\alpha}&=\sum_{i=1}^{n}\mathbb{E}\langle \nabla g_{\alpha}(W+D^{\{i\}})-\nabla g_{\alpha}(W^{(i)}+D^{\{i\}}), X_{i}\rangle-\sum_{i=1}^{n}\mathbb{E}\langle \nabla^{2} g_{\alpha}(T), \mathrm {Var}(X_{i})\rangle\nonumber\\
		&=\sum_{i=1}^{n}\int_{0}^{1}\mathbb{E}\langle \nabla^{2} g_{\alpha}(W^{(i)}+D^{\{i\}}+tX_{i}),  X_{i}\otimes X_{i}\rangle dt\\
	&\quad	-
		\sum_{i=1}^{n}\mathbb{E}\big\langle
		\nabla^{2}g_{\alpha}\bigl(W^{(i)}+D^{\{i\}}+X'_{i}\bigr),
		X_i\otimes X_{i}
		\big\rangle \nonumber\\
		&=I_{31,\alpha}-I_{32,\alpha},
	\end{align*}
	where
	\begin{align*}
		I_{31,\alpha}&=\sum_{i=1}^{n}
		\mathbb{E}\big[
		(1-\theta)
		\big\langle
		\nabla^3 g_{\alpha}
		\big(W^{(i)}+D^{\{i\}}+\theta X_i\big),
		X_i\otimes X_i\otimes X_i
		\big\rangle
		\big],\\
		I_{32,\alpha}	&=
		\sum_{i=1}^{n}
		\mathbb{E}\big\langle
		\nabla^3 g_{\alpha}
		\big(W^{(i)}+D^{\{i\}}+\theta X_i'\big),
		X_i'\otimes X_i\otimes X_i
		\big\rangle,
	\end{align*}
	and where we recall that $\theta$ is uniformly distributed over $[0,1]$ and is independent of all  other random variables.
	Let $\Sigma_{i}=\mathrm {Var}(W^{(i)}).$
	Denoting $\sigma^{2}_{i}$ by the smallest eigenvalue of $\Sigma_{i}$. 	  We may assume that
	$d^{1/4}\sum_{i=1}^n\E|X_i|^3\leq0.1$,
	since otherwise $d_c(T,Z)\leq1\leq10\delta$. For every unit vector $u\in\mathbb{R}^d$, H\"older's inequality gives
	$
	u^\top\operatorname{Var}(X_i)u
	\leq \mathbb{E}|X_i|^2
	\leq \big(\mathbb{E}|X_i|^3\big)^{2/3}
	\leq 0.1^{2/3}d^{-1/6}.
	$
	Consequently, we have
	\begin{align}\label{eq-pp-9.1}
		\sigma_i
		=\sqrt{\lambda_{\min}\big(\operatorname{Var}(W^{(i)})\big)}
		=\sqrt{1-\lambda_{\max}\big(\operatorname{Var}(X_i)\big)}
		\geq \sqrt{1-0.1^{2/3}}>0.8.
	\end{align}
	For $I_{31,\alpha}$, we define
	\begin{align*}
		D_{i,A,1}:&=|\mathbb{P}(W^{(i)}+D^{\{i\}}+\theta X_{i}\in A\mid X_{i},\theta)-\mathcal{N}(\theta X_{i}, \Sigma_{i})( A)|\\
		&=\big|\mathbb{P}\big(\Sigma_{i}^{-1/2}(W^{(i)}+D^{\{i\}})\in \Sigma_{i}^{-1/2}(A-\theta X_{i})\mid X_{i},\theta\big)\\
		&\qquad{}-\mathcal{N}(0,I_d)(\Sigma_{i}^{-1/2}(A-\theta X_{i}))\big|.
	\end{align*}
	To apply the definition of $K_n$, regard $\xi_i'$ as the
	$i$th independent input and take its corresponding linear
	summand to be zero. Hence,
	\begin{align*}
		D_{i,A,1}\leq K_{n}u_{i,1},
	\end{align*}
	where
	\begin{align*}
		u_{i,1}=&d^{1/4}\sum_{j\neq i}\big\|\Sigma_{i}^{-1/2}X_{j}\big\|_{3}^{3}+\big\|\Sigma_{i}^{-1/2}D^{\{i\}}\big\|_{3/2}\\
	&\quad	+qd^{1/(2q)}\sum_{j\neq i}\big\|\Sigma_{i}^{-1/2}D^{\{i\}}-\Sigma_{i}^{-1/2}D^{\{i,j\}}\big\|_{r}\big\|\Sigma_{i}^{-1/2}X_j\big\|_{s},
	\end{align*}
	$D^{\{i\}}=D(\xi_1,\dots, \xi'_i,\dots, \xi_n),$ $D^{\{i,j\}}=D(\xi_1,\dots, \xi'_i,\dots,\xi_{j-1}, \tilde{\xi}_{j}, \xi_{j+1}, \dots, \xi_n),$ $\tilde{\xi}_{j}$ is an independent copy of  $\xi_j$ and is independent of  all  other variables.
By \eqref{eq-pp-9.1}, we have, for each $1\leq i\leq n$,
	$
	u_{i,1}\leq 2\delta.
	$ Combining this with  Lemma \ref{lem4.3}, we have
	\begin{align}\label{eq-pp-11}
		\begin{aligned}
			\Big|\int^{\pi/2}_{0}I_{31,\alpha}\tan \alpha d\alpha\Big|&\leq C_{1}\sum_{i=1}^{n}\E\Big\{\Big[\frac{1}{\sigma^{3}_{i}}+\frac{d^{1/4}}{\sigma_{i}}+\frac{K_{n}u_{i,1}}{\varepsilon}\Big]|X_{i}|^{3}\Big\}\\
			&\leq Cd^{1/4}\sum_{i=1}^{n}\E|X_{i}|^{3}+\frac{C}{\varepsilon}\sum_{i=1}^{n}\E|X_{i}|^{3}K_{n}\delta,
		\end{aligned}
	\end{align}
	where $C_{1}$ is defined as in Lemma \ref{lem4.3}.
	For $I_{32,\alpha}$, we define
	\begin{align*}
		D_{i,A,2}
		:&=|\mathbb{P}(W^{(i)}+D^{\{i\}}+\theta X'_{i}\in A\mid \xi_{i}, \xi'_{i},\theta)-\mathcal{N}(\theta X'_{i}, \Sigma_{i})( A)|\nonumber\\
		&=\big|\mathbb{P}\big(\Sigma_{i}^{-1/2}(W^{(i)}+D^{\{i\}})\in \Sigma_{i}^{-1/2}(A-\theta X'_{i})\mid \xi_{i}, \xi'_{i},\theta\big)\\
		&\qquad{}-\mathcal{N}(0,I_d)(\Sigma_{i}^{-1/2}(A-\theta X'_{i}))\big|,
	\end{align*}
	then by the definition of $K_{n-1}$, we  have
	\[
	D_{i,A,2}\leq K_{n-1}u_{i,2}\leq K_{n}u_{i,2},
	\]
	where
	\begin{align*}
		u_{i,2}=&d^{1/4}\sum_{j\neq i}\mathbb{E}\big|\Sigma_{i}^{-1/2}X_{j}\big|^{3}+\big(\mathbb{E}\big[\big|\Sigma_{i}^{-1/2}D^{\{i\}}\big|^{3/2}\mid \xi'_{i}\big]\big)^{2/3}\nonumber\\
		&+qd^{1/(2q)}\sum_{j\neq i}\big(\E\big[\big|\Sigma_{i}^{-1/2}D^{\{i\}}-\Sigma_{i}^{-1/2}D^{\{i,j\}}\big|^{r} \mid \xi'_{i} \big]\big)^{1/r}\big\|\Sigma_{i}^{-1/2}X_j\big\|_{s},
	\end{align*}
	From the definition of $u_{i,2}$, we have $u_{i,2}$ is a function of $\xi'_{i}$, and hence $u_{i,2}$ is independent of $X_{i}$.
	By Lemma \ref{lem4.3}, we have
	\begin{align}\label{eq-pp-12}
		&\Big|\int^{\pi/2}_{0}I_{32,\alpha}\tan \alpha d\alpha\Big|\leq C_{1}\sum_{i=1}^{n}\E\Big\{\Big[\frac{1}{\sigma^{3}_{i}}+\frac{d^{1/4}}{\sigma_{i}}+\frac{K_{n}u_{i,2}}{\varepsilon}\Big]|X'_{i}||X_{i}|^{2}\Big\}.
	\end{align}
	Since $r\geq 3/2$, we have $r/(r-1)\leq 3$. Combining this with H\"{o}lder's inequality, we have
	\begin{align}\label{eq-pp-13}
		\begin{aligned}
			&\E\Big\{ \big(\E\big[\big|\Sigma_{i}^{-1/2}D^{\{i\}}-\Sigma_{i}^{-1/2}D^{\{i,j\}}\big|^{r} \mid \xi'_{i} \big]\big)^{1/r} |X'_{i}|\Big\}\\
			&\quad\leq 2\|D-D^{\{j\}}\|_{r}\|X'_{i}\|_{r/(r-1)}\leq 2\|D-D^{\{j\}}\|_{r}\|X_{i}\|_{3}.
		\end{aligned}
	\end{align}
	In addition, H\"{o}lder's inequality gives
	\begin{align}\label{eq-pp-14}
		\E\Big\{\big(\mathbb{E}\big[\big|\Sigma_{i}^{-1/2}D^{\{i\}}\big|^{3/2}\mid \xi'_{i}\big]\big)^{2/3}|X'_{i}|\Big\}\leq 2\|D\|_{3/2}\|X_{i}\|_{3}.
	\end{align}
	Substituting \eqref{eq-pp-13} and \eqref{eq-pp-14} into \eqref{eq-pp-12}, we obtain
	\[
	\Big|\int^{\pi/2}_{0}I_{32,\alpha}\tan \alpha d\alpha\Big|\leq
	Cd^{1/4}\sum_{i=1}^{n}\E|X_{i}|^{3}+\frac{C}{\varepsilon}\sum_{i=1}^{n}\E|X_{i}|^{3}K_{n}\delta.
	\]
	Therefore, we have
	\begin{align}\label{eq-pp-15}
		|I_{3}|&\leq Cd^{1/4}\sum_{i=1}^{n}\E|X_{i}|^{3}+\frac{C}{\varepsilon}\sum_{i=1}^{n}\E|X_{i}|^{3}K_{n}\delta
	\end{align}

	Combining \eqref{eq-pp-03}, \eqref{eq-pp-04}, \eqref{eq-pp-07} and \eqref{eq-pp-15}, we have
	\begin{align}\label{eq-pp-16}
		\begin{aligned}
			&\sup_{A\in \mathcal{A}}\big|\mathbb{E}[f(W+D)]-\mathbb{E}[f(Z)]\big|\\
			&\quad\leq C'\left(qd^{1/(2q)}+|\log \varepsilon|(K_{n}\delta)^{1/q}\right)\sum_{i=1}^{n}\|D-D^{\{i\}}\|_{r}\|X_{i}\|_{s} \\
			&\qquad +C'\|D\|_{3/2}+C'd^{1/4}\sum_{i=1}^{n}\E|X_{i}|^{3}+\frac{C'}{\varepsilon}\sum_{i=1}^{n}\E|X_{i}|^{3}K_{n}\delta
		\end{aligned}
	\end{align}
	for some positive constant $C'$.
	Choosing $\varepsilon=2C'\delta/d^{1/4}$, we have
	\begin{align}\label{eq-pp-17}
		|\log(\varepsilon)|\delta^{1/q}\leq  |\log(2C')|\delta^{1/q}+|\log(\delta)|\delta^{1/q}+\log(d)\delta^{1/q}\leq Cqd^{1/(2q)}
	\end{align}
	provided that $0\leq \delta\leq 1$. Moreover, we may assume that $\delta\leq 1/(4C')$, and hence $0<\varepsilon<1.$  Substituting \eqref{eq-pp-17} into \eqref{eq-pp-16} and using \eqref{eq-ha-02}, we have
	\begin{align*}
		&\sup_{A\in \mathcal{A}}|\mathbb{P}(W+D \in A)-\mathbb{P}(Z \in A)|\leq (C+CK_{n}^{1/q}+0.5K_{n})\delta.
	\end{align*}
After enlarging $C$ to cover the trivial case, the same argument
applies to every admissible choice with $1\leq k\leq n$ inputs.
Using $K_k\leq K_n$ and taking the supremum in the definition
of $K_n$, we obtain
	\begin{align} \label{eq-pp-19}
	K_{n}\leq  C+CK_{n}^{1/q}+0.5K_{n}.
	\end{align}
	If $K_{n}<1$, the desired bound is immediate. Suppose now that $K_{n}\geq 1$.
	Since $q\geq 2$, we have $K_{n}^{1/q}\leq K_{n}^{1/2}$.
	Hence, by Young's inequality,
$
	C K_{n}^{1/q}\leq C\sqrt{K_{n}}\leq K_{n}/8+C.
$
	Combining this inequality with \eqref{eq-pp-19}, we have $K_{n}\leq C^{\dagger}$, as desired.

\end{proof}

\begin{proof}[Proof of Corollary \ref{coro-main-1}]
	For any $A\in\mathcal{A}$, set
	$
	B=\Sigma^{-1/2}A
	:=\{\Sigma^{-1/2}x:x\in A\}.
	$
	Since an invertible linear transformation preserves convexity,
	$B\in\mathcal{A}$. Moreover,
	\begin{align*}
		&\mathbb{P}(T\in A)
		-\mathbb{P}(\Sigma^{1/2}Z\in A)=
		\mathbb{P}(\Sigma^{-1/2}T\in B)
		-\mathbb{P}(Z\in B).
	\end{align*}
Apply Theorem~\ref{thm-main-1} to $\Sigma^{-1/2}T$ with
	$r=2$, $q=\max\{6,\tfrac12\log d\}$, and $s=2q/(q-2)\leq3$.
	Since $\|\Sigma^{-1/2}\|_{\op}\leq\sigma^{-1/2}$ and
	$q d^{1/(2q)}\leq C(1+\log d)$, the desired bound follows.

\end{proof}
\subsection{Proofs of Theorem \ref{thm-main-2} and Corollary \ref{coro-main-2}}
In this subsection, we provide the proofs of Theorem \ref{thm-main-2} and Corollary \ref{coro-main-2}.
\begin{proof}[Proof of Theorem \ref{thm-main-2}]
We first prove the auxiliary bound
	\[
	d_c(T,Z)\leq C\bigl\{
	d^{1/4}\gamma_1+q d^{1/(2q)}\gamma_2
	+d^{1/8}\gamma_3+\E|D|
	\bigr\},\qquad q\geq4,
	\]
	with $C$ independent of $q$.
	We again use a recursive argument to prove this bound.
	Fix the dimension
	$d$ and an exponent $q\geq4$. For each integer $m\geq1$, define
	\[
	K_m:=\sup\frac{d_c(T,Z)}{\delta},
	\]
	where
$
	\delta
	=d^{1/4}\gamma_1
	+qd^{1/(2q)}\gamma_2
	+d^{1/8}\gamma_3
	+\E|D|,
$
	and the supremum is taken over all choices of the independent
	random elements $\xi_1,\ldots,\xi_n$, the Borel measurable functions
	$g_1,\ldots,g_n$, and the statistic $T=T(\xi_1,\ldots,\xi_n)$
	satisfying \eqref{eq-na-01} and \eqref{eq-na-02}, with
	$1\leq n\leq m$.
	By \eqref{eq-na-02}, Jensen's inequality, and the
	Cauchy--Schwarz inequality,
	\[
	\gamma_1^2
	=\sum_{i=1}^n\E|X_i|^4
	\geq\sum_{i=1}^n(\E|X_i|^2)^2
	\geq\frac1n\Big(\sum_{i=1}^n\E|X_i|^2\Big)^2
	=\frac{d^2}{n}.
	\]
	Hence $\delta\geq d^{5/4}/\sqrt n>0$. Since $d_c(T,Z)\leq1$,
	it follows that
	\[
	K_m\leq\sqrt m\,d^{-5/4}<\infty.
	\]
	We shall prove that $K_n\leq C^{\ddagger}$ for an absolute constant $C^{\ddagger}$.
	We may assume that $d^{1/4}\gamma_1\leq0.01$, since otherwise
	$d_c(T,Z)\leq1\leq100\delta$.

	For each $1\leq i\leq n$,	let $W^{\{i\}}=W-X_{i}+X'_{i}$, $D^{\{i\}}=D(\xi_1,\dots, \xi'_i \dots,\xi_n)$ and
	$
	T^{\{i\}}=W^{\{i\}}+D^{\{i\}}.
	$
	It follows from exchangeability that for any $\alpha\in [0,\pi/2]$,
	\begin{align}\label{eq-t-01}
		\begin{aligned}
			0&=\E \la X_{i}-X'_{i}, \nabla g_{\alpha}(T)+\na  g_{\alpha}(T^{\{i\}})\ra\\
			&=\E \la  X_{i}-X'_{i}, \na  g_{\alpha}(T^{\{i\}})-\nabla g_{\alpha}(T)\ra+2\E \la  X_{i}-X'_{i}, \nabla g_{\alpha}(T)\ra\\
			&=\E \la  X_{i}-X'_{i}, \na  g_{\alpha}(T^{\{i\}})-\nabla g_{\alpha}(T)\ra+2\E\la X_{i}, \nabla g_{\alpha}(T)\ra.
		\end{aligned}
	\end{align}
	Rearranging the preceding identity, we have
	\begin{align}\label{eq-t-02}
		\begin{aligned}
			\E\la X_{i}, \nabla g_{\alpha}(T)\ra&=\frac{1}{2}\E \la X_{i}-X'_{i},  \nabla g_{\alpha}(T)- \nabla  g_{\alpha}(T^{\{i\}})\ra\\
			&=\frac{1}{2}\E   \int_{0}^{1} \langle \nabla^{2} g_{\alpha}(T-t(X_{i}-X'_{i})),  (X_{i}-X'_{i})\otimes (X_{i}-X'_{i})\ra dt\\
			&\quad+\frac{1}{2}\E \la X_{i}-X'_{i},  \nabla g_{\alpha}(W^{\{i\}}+D)-\nabla  g_{\alpha}(W^{\{i\}}+D^{\{i\}})\ra.
		\end{aligned}
	\end{align}
	By \eqref{eq-t-02}, we have
	\begin{align}\label{eq-t-03}
		\begin{aligned}
		&	\mathbb{E}\langle\nabla g_{\alpha}(T), T\rangle-\mathbb{E}\Delta g_{\alpha}(T)\\
			&\quad=\mathbb{E}\langle\nabla g_{\alpha}(T), D\rangle+\sum_{i=1}^{n}\mathbb{E}\langle \nabla g_{\alpha}(T), X_{i}\rangle-\sum_{i=1}^{n}\mathbb{E}\langle \nabla^{2} g_{\alpha}(T), \mathrm {Var}(X_{i})\rangle \\
			&\quad:=I_{1,\alpha}+I_{2,\alpha}+I_{3,\alpha},
		\end{aligned}
	\end{align}
	where
	\begin{align*}
		I_{1,\alpha}&=  \mathbb{E}\langle \nabla g_{\alpha}(T),D\rangle,\quad
		I_{2,\alpha}=\frac{1}{2}\sum_{i=1}^{n}\E \la X_{i}-X'_{i},  \nabla g_{\alpha}(W^{\{i\}}+D)-\nabla  g_{\alpha}(W^{\{i\}}+D^{\{i\}})\ra,\\
		I_{3,\alpha} & = \frac{1}{2}\sum_{i=1}^{n}\E   \int_{0}^{1} \langle \nabla^{2} g_{\alpha}(T-t(X_{i}-X'_{i})),  (X_{i}-X'_{i})\otimes (X_{i}-X'_{i})\ra dt\nonumber\\
		&\quad-\sum_{i=1}^{n}\mathbb{E}\langle \nabla^{2} g_{\alpha}(T), \mathrm {Var}(X_{i})\rangle.
	\end{align*}
	For $j=1,2,3$, let
$
	I_j=\int_0^{\pi/2}I_{j,\alpha}\tan\alpha\,d\alpha.
$
	By \eqref{eq-ha-02}, \eqref{eq-ha-07}, and \eqref{eq-t-03},
	we have
	\begin{align}\label{eq-t-04}
		d_c(T,Z)
		\leq4d^{1/4}\varepsilon
		+\sup_{A\in\caa}
		\bigl\{|I_1|+|I_2|+|I_3|\bigr\}.
	\end{align}
	To complete the proof, we shall use the following bounds,
	which hold uniformly over $A\in\caa$:
	\begin{align}
		|I_1|
		&\leq \E|D|, \label{eq-bound-I1}\\
		|I_2|
		&\leq C\big(qd^{1/(2q)}+|\log(\varepsilon)|(K_{n}\delta)^{1/q}\big)\sum_{i=1}^{n}\|X_{i}\|_{4} \|D-D^{\{i\}}\|_{2}, \label{eq-bound-I2}\\
		|I_{3}|&\leq C\big(d^{1/4}+|\log(\varepsilon)|(K_{n}\delta)^{1/2}\big)\Big(\sum_{i=1}^{n}\E|X_{i}|^{4}\Big)^{1/2}+\frac{C}{\varepsilon}\sum_{i=1}^{n}\|D-D^{\{i\}}\|_{2}\|X_{i}\|_{4}^{2}\nonumber\\
		&\quad+\frac{C}{\varepsilon^{2}}\sum_{i=1}^{n}\E |X_{i}|^{4} (d^{1/4}\varepsilon+K_{n}\delta). \label{eq-bound-I3}
	\end{align}
	The proofs of \eqref{eq-bound-I1}--\eqref{eq-bound-I3} are deferred to Appendix~\ref{sec:appen-pr-three-bounds}. Combining these estimates with \eqref{eq-t-04}, we have
	\begin{align*}
		d_{c}(T,Z)&\leq 4d^{1/4}\varepsilon+\E|D|+\frac{C}{\varepsilon}\sum_{i=1}^{n}\|D-D^{\{i\}}\|_{2}\|X_{i}\|_{4}^{2}\\
		&\quad+C\big(qd^{1/(2q)}+|\log(\varepsilon)|(K_{n}\delta)^{1/q}\big)\sum_{i=1}^{n}\|X_{i}\|_{4} \|D-D^{\{i\}}\|_{2}\\
		&\quad+C\big(d^{1/4}+|\log(\varepsilon)|(K_{n}\delta)^{1/2}\big)\Big(\sum_{i=1}^{n}\E|X_{i}|^{4}\Big)^{1/2}\\
		&\quad+\frac{C}{\varepsilon^{2}}\sum_{i=1}^{n}\E |X_{i}|^{4} (d^{1/4}\varepsilon+K_{n}\delta)\\
		&\leq \tilde{C}d^{1/4}\varepsilon+\tilde{C}\E|D|+\tilde{C}\frac{\gamma_{3}^{2}}{\varepsilon}+\tilde{C}\big(qd^{1/(2q)}+|\log(\varepsilon)|(K_{n}\delta)^{1/q}\big)\gamma_{2}\\
		&\quad+\tilde{C}\big(d^{1/4}+|\log(\varepsilon)|(K_{n}\delta)^{1/2}\big)\gamma_{1}+\frac{\tilde{C}}{\varepsilon^{2}}\gamma_{1}^{2}(d^{1/4}\varepsilon+K_{n}\delta)
	\end{align*}
	for some positive constant $\tilde{C}$.
	Now, choose
	$
	\varepsilon=100\tilde{C}\delta/d^{1/4}.
	$
	Without loss of generality, we assume that $0\leq \delta\leq \min\{1/4, 1/(200\tilde{C})\}$, then we  have
	\begin{align*}
		|\log(\varepsilon)|\delta^{1/q}\leq  |\log(100\tilde{C})|\delta^{1/q}+|\log(\delta)|\delta^{1/q}+\log(d)\delta^{1/q}\leq Cqd^{1/(2q)}.
	\end{align*}
	Therefore, we have
	\begin{align}\label{eq-t-13}
		d_{c}(T,Z)\leq (C+CK_{n}^{1/q}+CK_{n}^{1/2}+0.5K_{n})\delta.
	\end{align}
	After enlarging $C$ to cover the trivial case, the same argument
	applies to every admissible choice with $k\leq n$ inputs.
	Using $K_k\leq K_n$ and taking the supremum in the definition
	of $K_n$, \eqref{eq-t-13} yields
	\[
	K_n\leq C+CK_n^{1/q}+CK_n^{1/2}+\frac12K_n.
	\]
	Since $q\geq4$, we have $K_n^{1/q}\leq1+K_n^{1/2}$.
	Young's inequality therefore gives
$
	CK_n^{1/q}+CK_n^{1/2}\leq K_n/4+C.
$
	Consequently, we have
	\[
	K_n\leq C+\frac34K_n.
	\]
Since $K_n<\infty$, rearranging this inequality yields
	$K_n\leq C^{\ddagger}$ uniformly in $q\geq4$, proving the auxiliary
	bound. Taking $q=\max\{4,(\log d)/2\}$ and using
	$q d^{1/(2q)}\leq C(1+\log d)$ completes the proof.

\end{proof}
\begin{proof}[Proof of Corollary \ref{coro-main-2}]
	Observe that
	\begin{align*}
		&\mathbb{P}(T\in A)
		-\mathbb{P}(\Sigma^{1/2}Z\in A)=
		\mathbb{P}(\Sigma^{-1/2}T\in \Sigma^{-1/2}A)
		-\mathbb{P}(Z\in \Sigma^{-1/2}A).
	\end{align*}
	Since $\Sigma^{-1/2}A$ is convex whenever $A$ is convex, the conclusion follows by applying Theorem \ref{thm-main-2}.
\end{proof}
\section{Proofs of other results}\label{sec:proof-of-other-results}
In this section, we prove Theorems~\ref{thm-app-nonsmooth},
\ref{thm-app-TD}, and~\ref{thm-app-U-statistics}, corresponding
respectively to the applications to nonsmooth stochastic approximation,
temporal-difference learning with linear function approximation, and
multivariate \(U\)-statistics.
We also prove Corollary~\ref{coro-app-TD-uncentered},
	which gives the corresponding bound for the original TD estimation error.

\subsection{Proof of Theorem~\ref{thm-app-nonsmooth}}\label{sec:proof-app-nonsmooth}
In this subsection, we provide the proof of Theorem~\ref{thm-app-nonsmooth}. We begin with the following preliminary lemma,  whose proof is deferred to Appendix~\ref{sec:apdix-proof-nonsmooth}.

\begin{lemma}\label{lem-A4}
	Let $\Sigma_{n}$ be defined as in \eqref{eq:Qkn}.
	Suppose that Assumptions \ref{assu-1} and \ref{assu-4} hold. Then there exist constants
	$n_{0}\in\mathbb{N}_{+}$ and $C_{1},C_{2}>0$ such that, for all
	$n\ge n_{0}$,
	\begin{align}\label{eq-ka-01}
		\lambda_{\min}(\Sigma_{n})\ge C_{1},
	\end{align}
	and, for all $0\le i\le n$,
	\begin{align}
		\|Q_i\|_{\op}\le C_{2}. \label{eq-ka-02}
	\end{align}
\end{lemma}

We now outline the main idea of the proof of Theorem \ref{thm-app-nonsmooth}.
The proof of a key lemma, namely Lemma \ref{lem-ap_nonsmooth-1}, is provided in Appendix~\ref{sec:appendix-proof-of-lemma-nonsmooth}.
\begin{proof}[Proof of Theorem \ref{thm-app-nonsmooth}]
	We use Corollary \ref{coro-main-1} to prove Theorem \ref{thm-app-nonsmooth}. To this end, we first decompose $\bar{x}_{n}-x_{*}$ into an independent sum plus a remainder, following an idea inspired by \cite{DDJ2024}.
	Consider the shadow sequence
	$$
	y_k= \begin{cases}
		P_{\mathcal{M}}(x_k), & \text { if } x_k \in B_{2 \delta}(x_{*}), \\
		x_{*} & \text { otherwise}.
	\end{cases}
	$$
	For each  $k\geq 0$, define
	\begin{align*}
		E_{k}=[y_{k+1}-y_{k}+\alpha_{k+1}F_{\MM}(y_{k})+\alpha_{k+1}P_{T_{\MM}(y_{k})}(\zeta_{k+1})]/\alpha_{k+1}.
	\end{align*}
	It follows from the definition of $E_{k}$ that
	\begin{align} \label{eq-ma-01}
		\begin{aligned}
			&U^{\top}(y_{k+1}-x_{*})\\
			&\quad=U^{\top}(y_{k}-x_{*})-\alpha_{k+1}U^{\top}F_{\MM}(y_{k})-\alpha_{k+1}U^{\top}P_{T_{\MM}(y_{k})}(\zeta_{k+1})+\alpha_{k+1}U^{\top}E_{k},\\
			&\quad=U^{\top}(y_{k}-x_{*})-\alpha_{k+1}U^{\top}\nabla_{\mathcal{M}} F_{\MM}(x_{*})P_{T_{\MM}(x_{*})}(y_{k}-x_{*})\\
			&\qquad-\alpha_{k+1}\big(U^{\top}F_{\MM}(y_{k})-U^{\top}\nabla_{\mathcal{M}} F_{\MM}(x_{*})P_{T_{\MM}(x_{*})}(y_{k}-x_{*})\big)\\
			&\qquad-\alpha_{k+1}U^{\top}P_{T_{\MM}(x_{*})}(\zeta_{k+1})-\alpha_{k+1}\big(U^{\top}P_{T_{\MM}(y_{k})}(\zeta_{k+1})-U^{\top}P_{T_{\MM}(x_{*})}(\zeta_{k+1})\big)\\
			&\qquad+\alpha_{k+1}U^{\top}E_{k}.
		\end{aligned}
	\end{align}
	For convenience, we denote
	\begin{align*}
		z_{k}&=U^{\top}(y_{k}-x_{*}), && G_{k}=U^{\top}P_{T_{\MM}(y_{k})}(\zeta_{k+1})-U^{\top}P_{T_{\MM}(x_{*})}(\zeta_{k+1}),\\
		H&=U^{\top}\nabla_{\mathcal{M}} F_{\MM}(x_{*})U,
		&&R(y)=U^{\top}F_{\MM}(y)-U^{\top}\nabla_{\mathcal{M}} F_{\MM}(x_{*})P_{T_{\MM}(x_{*})}(y-x_{*}).
	\end{align*}
	Recalling that the columns of $U \in \mathbb{R}^{d \times m}$ form an orthonormal basis of
	$T_{\mathcal M}(x_{*})$, we have  $P_{T_{M}(x_{*})}(x)=UU^{\top}x$ for any $x\in \R^{d}$. Hence,
	$U^{\top}P_{T_{\MM}(x_{*})}(\zeta_{k+1})=U^{\top}UU^{\top}\zeta_{k+1}=U^{\top}\zeta_{k+1}$. Using the notation above, equation \eqref{eq-ma-01} can be written compactly as follows:
	\begin{align*}
		z_{k+1}&=z_{k}-\alpha_{k+1}Hz_{k}-\alpha_{k+1}U^{\top}\zeta_{k+1}-\alpha_{k+1}(R(y_{k})+G_{k}-U^{\top}E_{k})\nonumber\\
		&=(I-\alpha_{k+1}H)z_{k}-\alpha_{k+1}(U^{\top}\zeta_{k+1}+R(y_{k})+G_{k}-U^{\top}E_{k}).
	\end{align*}
	It follows from the above iterative representation that
	\begin{align}\label{eq-ma-02}
		\frac{1}{\sqrt{n}}\sum_{k=0}^{n-1}z_{k}
		&=
		\frac{1}{\sqrt{n}\,\alpha_{0}}Q_{0}z_{0}
		-
		\frac{1}{\sqrt{n}}\sum_{k=1}^{n-1}Q_k\bigl(U^{\top}\zeta_{k}+R(y_{k-1})+G_{k-1}-U^{\top}E_{k-1}\bigr),
	\end{align}
	where  we recall that
	\[
	Q_{k}=\alpha_{k}\sum_{j=k}^{n-1}\prod_{\ell=k+1}^{j}(I_m-\alpha_{\ell} H).
	\]
	By \eqref{eq-ma-02}, we have
	\begin{align*}
		\frac{U^{\top}}{\sqrt{n}}\sum_{k=0}^{n-1}(x_{k}-x_{*})&= \frac{U^{\top}}{\sqrt{n}}\sum_{k=0}^{n-1}(x_{k}-y_{k})+ \frac{U^{\top}}{\sqrt{n}}\sum_{k=0}^{n-1}(y_{k}-x_{*})\nonumber\\
		&=-\frac{1}{\sqrt{n}}\sum_{k=1}^{n-1}Q_kU^{\top}\xi_{k}+\frac{U^{\top}}{\sqrt{n}}\sum_{k=0}^{n-1}(x_{k}-y_{k})+\frac{1}{\sqrt{n}\,\alpha_{0}}Q_{0}z_{0}\nonumber\\
		&\quad
		-
		\frac{1}{\sqrt{n}}\sum_{k=1}^{n-1}Q_{k}\bigl(U^{\top}\eta_{k}+R(y_{k-1})+G_{k-1}-U^{\top}E_{k-1}\bigr).
	\end{align*}
	Recall that $\Sigma_{n}=n^{-1}\sum_{i=1}^{n-1}Q_iU^{\top}\Sigma_{i}UQ_{i}^{\top}.$
	To apply Corollary \ref{coro-main-1}, we define
	\begin{align*}
		T_{n}=n^{-1/2}\Sigma_{n}^{-1/2}U^{\top}\sum_{k=0}^{n-1}(x_{k}-x_{*}),\quad \varepsilon_{k}=-n^{-1/2}\Sigma_{n}^{-1/2}Q_kU^{\top}\xi_{k}, \quad W_{n}=\sum_{k=1}^{n-1} \varepsilon_{k},
	\end{align*}
	and
	\begin{align*}
		D_{n}&=\frac{\Sigma_{n}^{-1/2}U^{\top}}{\sqrt{n}}\sum_{k=0}^{n-1}(x_{k}-y_{k})+\frac{\Sigma_{n}^{-1/2}}{\sqrt{n}\,\alpha_{0}}Q_{0}z_{0}+
		\frac{\Sigma_{n}^{-1/2}}{\sqrt{n}}\sum_{k=1}^{n-1}Q_{k}U^{\top}E_{k-1}\nonumber\\
		&\quad
		-
		\frac{\Sigma_{n}^{-1/2}}{\sqrt{n}}\sum_{k=1}^{n-1}Q_{k}\bigl(U^{\top}\eta_{k}+R(y_{k-1})+G_{k-1}\bigr)\nonumber\\
		&:=D_{1,n}+D_{2,n}+D_{3,n}-D_{4,n}.
	\end{align*}
Note that each of $D_{1,n},D_{2,n},D_{3,n},D_{4,n}$ is measurable
with respect to $\sigma(x_0,\xi_1,\ldots,\xi_{n-1})$. In the following
replacement construction, we replace only the innovation $\xi_i$
and keep the initial value $x_0$ unchanged. Suppressing the dependence
on $x_0$ from the notation, for each $1\leq m\leq 4$ we write
$D_{m,n}=\Psi_m(\xi_1,\ldots,\xi_{n-1})$.
	We then define, for each $1\leq m\leq 4$ and $1\leq i\leq n-1$,
	\[
	D_{m,n}^{\{i\}}
	=
	\Psi_{m}(\xi_{1},\ldots,\xi_{i-1},\xi'_{i},\xi_{i+1},\ldots,\xi_{n-1})\quad \text{and}\quad D_{n}^{\{i\}}=D_{1,n}^{\{i\}}+D_{2,n}^{\{i\}}+D_{3,n}^{\{i\}}-D_{4,n}^{\{i\}}.
	\]
	We complete the proof of Theorem \ref{thm-app-nonsmooth} using the following lemma.
	Its proof is deferred to Appendix~\ref{sec:appendix-proof-of-lemma-nonsmooth}.
	\begin{lemma}\label{lem-ap_nonsmooth-1}
		There exist constants
		$n_{0}\in\mathbb{N}_{+}$ and $C>0$ such that, for all
		$n\ge n_{0}$,
		\begin{align}\label{eq:X-third-sum}
			\sum_{k=1}^{n-1}\E|\varepsilon_{k}|^3
			&\leq C\tau^3n^{-1/2},\\
			\|D_n\|_{3/2}
			&\leq C(\tau_0n^{-1/2}
			+C_{\MM}^{2}\tau_{1}^{4}n^{1/2-\beta}),
			\label{eq:D-norm}\\
			\|D_n-D_n^{\{i\}}\|_2
			&\leq CC_{\MM}^{3}\tau_{1}^{4}n^{-1/2}i^{1-3\beta/2},
			\qquad 1\leq i\leq n-1.
			\label{eq:D-LOO}
		\end{align}
		As a consequence, we have
		\begin{align*}
			\sum_{i=1}^{n-1}
			\|D_n-D_n^{\{i\}}\|_2\|\varepsilon_{i}\|_3
			\leq C C_{\MM}^3\tau_{1}^{5}
			n^{1-3\beta/2}.
		\end{align*}
	\end{lemma}
\end{proof}

\subsection{\texorpdfstring{Proofs of Theorem~\ref{thm-app-TD} and Corollary~\ref{coro-app-TD-uncentered}}{Proofs of the TD results}}
\label{subsec:proof-TD}

We apply Corollary~\ref{coro-main-1}. By independence of the current
sample from the past and $\E\varepsilon_t=0$,
\begin{align*}
	\theta_t-\E\theta_t
	&=(I_d-\eta_tA)(\theta_{t-1}-\E\theta_{t-1})
	-\eta_t\varepsilon_t-\eta_t(A_t-A)\Delta_{t-1}.
\end{align*}
Since $\theta_0$ is deterministic, iterating this identity and summing
over $t$ gives
\[
\sqrt T\,A(\bar\theta_T-\E\bar\theta_T)
=\sum_{t=1}^T X_t+D_T,
\]
where
\begin{align*}
	X_t&:=-\frac1{\sqrt T}AQ_t\varepsilon_t,\qquad
	D_T:=-\frac1{\sqrt T}\sum_{t=1}^T
	AQ_t(A_t-A)\Delta_{t-1}.
\end{align*}
The vectors $X_t$ are independent and centered, with
$\sum_{t=1}^T\operatorname{Var}(X_t)=V_T$, and $\E D_T=0$.
For each $1\leq i\leq T$, let $(A_i',b_i')$ be an independent copy of
$(A_i,b_i)$. Replace $(A_i,b_i)$ by $(A_i',b_i')$, recompute the TD
trajectory from time $i$ onward, and denote the resulting version of
$D_T$ by $D_T^{\{i\}}$. To apply Corollary~\ref{coro-main-1}, we need estimates for several terms appearing in its bound. The following lemma, whose proof is given in Appendix~\ref{sec:appendix-TD}, provides the required estimates.

\begin{lemma}\label{lem:TD-remainder-summary}
	Under the assumptions of Theorem~\ref{thm-app-TD}, the following
	bounds hold.
	\begin{itemize}
		\item[\emph{(a)}] The third-moment term of the independent linear part satisfies
		\begin{align}
			\sum_{t=1}^T\|X_t\|_3^3
			\leq C\nu_3^3
			\biggl\{
			\frac{1}{T^{1/2}}
			+\frac{b_\eta^{3\alpha}}{T^{3/2}}
			\biggr\}.
			\label{eq-uu-01}
		\end{align}

		\item[\emph{(b)}]
The centered remainder satisfies
		\begin{align}
			\|D_T\|_2
			\leq C\bigg\{
			\frac{\operatorname{Tr}(\Gamma)^{1/2}}{T^{\alpha/2}}
			+\frac{\eta_0^{-1/2}
				\bigl\{\|\Delta_0\|_2+
				(\eta_0\operatorname{Tr}(\Gamma))^{1/2}\bigr\}
				b_\eta^\alpha}
			{T^{1/2}}
			\bigg\}.
			\label{eq-uu-02}
		\end{align}

		\item[\emph{(c)}]
		The replacement-difference term satisfies
		\begin{align}\label{eq-uu-03}
			\begin{aligned}
				&\sum_{i=1}^T
				\|X_i\|_3
				\|D_T-D_T^{\{i\}}\|_{2}\\
				&\quad\leq C\nu_3\bigg\{
				\frac{\operatorname{Tr}(\Gamma)^{1/2}}
				{T^{\alpha/2}}
				+\frac{\eta_0^{-1/2}
					\bigl\{\|\Delta_0\|_2+
					(\eta_0\operatorname{Tr}(\Gamma))^{1/2}\bigr\}
					b_\eta^{2\alpha}}
				{T}
				\bigg\}.
			\end{aligned}
		\end{align}
	\end{itemize}
\end{lemma}
We now return to the proof of Theorem~\ref{thm-app-TD}. We first remove the multiplication by $A$. Since $A$ is
invertible and
$
V_T=A\overline\Lambda_TA^\top,
$
the Gaussian random vectors
$A\overline\Lambda_T^{1/2}Z$ and $V_T^{1/2}Z$ have the same
distribution. Moreover, invertible linear maps preserve convexity.
Therefore,
\begin{align*}
	&\dc\bigl(
	\sqrt T\,A(\bar\theta_T-\E\bar\theta_T),V_T^{1/2}Z
	\bigr)=\dc\bigl(
	\sqrt T\,(\bar\theta_T-\E\bar\theta_T),
	\overline\Lambda_T^{1/2}Z
	\bigr).
\end{align*}
By Corollary~\ref{coro-main-1} and Lemma~\ref{lem:TD-remainder-summary},
the resulting bound consists of terms already present in
\eqref{eq-app2-thm-01} and the following two additional terms:
\begin{align*}
	&\frac{(1+\log d)\nu_3}{\underline\lambda_T}
	\frac{\eta_0^{-1/2}
		\bigl\{\|\Delta_0\|_2+
		(\eta_0\operatorname{Tr}(\Gamma))^{1/2}\bigr\}
		b_\eta^{2\alpha}}
	{T},\quad
\frac{\operatorname{Tr}(\Gamma)^{1/2}}
	{\underline\lambda_T^{1/2}T^{\alpha/2}}.
\end{align*}
The two additional terms can be absorbed into the right-hand side of
\eqref{eq-app2-thm-01}, using the trivial bound $d_c\leq 1$ when necessary.
For completeness, we provide the details in
Appendix~\ref{sec:appen-lemmac4}. This proves \eqref{eq-app2-thm-01}.

\begin{proof}[Proof of Corollary~\ref{coro-app-TD-uncentered}]
Taking expectations in the TD recursion and using the definition of
$Q_0$ gives
\[
\sqrt T\,\overline\Delta_T
=\sqrt T\,(\bar\theta_T-\E\bar\theta_T)
+\frac{Q_0\Delta_0}{\sqrt T\,\eta_0}.
\]
By translation invariance of the convex-set distance, the triangle
inequality, and the Gaussian translation bound,
\begin{align*}
	\dc\bigl(\sqrt T\,\overline\Delta_T,
	\overline\Lambda_T^{1/2}Z\bigr)
	&\leq
	\dc\bigl(\sqrt T\,(\bar\theta_T-\E\bar\theta_T),
	\overline\Lambda_T^{1/2}Z\bigr)+\frac{|V_T^{-1/2}AQ_0\Delta_0|}{\sqrt T\,\eta_0},
\end{align*}
where $V_T=A\overline\Lambda_TA^\top$. By Lemma~\ref{lem:TD-kernel-trajectory-bounds} in Appendix~\ref{sec:appendix-TD},
$\|AQ_0\|_{\mathrm{op}}/\eta_0\le Cb_\eta$, and hence
\[
\frac{|V_T^{-1/2}AQ_0\Delta_0|}{\sqrt T\,\eta_0}
\leq \frac{C\|\Delta_0\|_2b_\eta}
{\underline\lambda_T^{1/2}\sqrt T}.
\]
Combining this with Theorem~\ref{thm-app-TD} proves
\eqref{eq-app2-coro-01}.
\end{proof}

\subsection{Proof of Theorem \ref{thm-app-U-statistics}}

We begin with the following lemma, which provides a general estimate
for the Hoeffding remainder term and will be used in the proof of Theorem \ref{thm-app-U-statistics}. Its proof is
deferred to Appendix~\ref{sec-appendix-pr-U}.
\begin{lemma}\label{lem:U-remainder-general}
	Let $D=(D_{1,n},\ldots,D_{d,n})^\top$ be defined as in
	\eqref{eq-ac-01}. Let $(X_1',\ldots,X_n')$ be an independent copy of
	$(X_1,\ldots,X_n)$. For each $1\leq j\leq n$, define
	$$
	D^{\{j\}}
	=
	D(X_1,\ldots,X_{j-1},X_j',X_{j+1},\ldots,X_n).
	$$
	Then,
	\[
	\E|D|^2
	\leq
	\sum_{i=1}^{d}
	\frac{(m_i-1)^2\sigma_i^2}
	{m_i(n-m_i+1)\lambda_i^2},
	\]
	and for every $1\leq j\leq n$,
	\[
	\E|D-D^{\{j\}}|^2
	\leq
	4\sum_{i=1}^{d}
	\frac{(m_i-1)^2\sigma_i^2}
	{n m_i(n-m_i+1)\lambda_i^2}.
	\]
\end{lemma}

\begin{proof}[Proof of Theorem \ref{thm-app-U-statistics}]
	We first prove \eqref{eq-app3-thm-01}. Applying
	Corollary~\ref{coro-main-1}, we only need to estimate the
	three terms on the right-hand side. By the definition of
	$Y_i$, $\sum_{i=1}^{n}\E|Y_i|^3
	=n^{-1/2}\|G(X_1)\|_3^3$, which gives the first term in
	\eqref{eq-app3-thm-01}. For the second term, by
	Lemma~\ref{lem:U-remainder-general} and the definition of
	$Y_i$, we have
	$\sum_{i=1}^{n}\|D-D^{\{i\}}\|_2\|Y_i\|_3
	\leq C\|G(X_1)\|_3\Lambda_n$.
	Moreover, Lemma~\ref{lem:U-remainder-general} also yields
	$\|D\|_{3/2}\leq(\E|D|^2)^{1/2}\leq C\Lambda_n$.
	Substituting these estimates into
	Corollary~\ref{coro-main-1}, we obtain
	\eqref{eq-app3-thm-01}.

	As for \eqref{eq-app3-thm-02}, we apply
	Corollary~\ref{coro-main-2}. The second and fourth terms
	are estimated as above. It remains only to estimate the
	first and third terms on the right-hand side of
	Corollary~\ref{coro-main-2}. The first term follows
	directly from the definition of $Y_i$ and a simple
	algebraic calculation. For the third term, by
	Lemma~\ref{lem:U-remainder-general}, we have
	$\gamma_3^2
	=\sum_{i=1}^{n}\|D-D^{\{i\}}\|_2\|Y_i\|_4^2
	\leq Cn^{-1/2}\|G(X_1)\|_4^2\Lambda_n$,
	and hence
	$\gamma_3\leq Cn^{-1/4}\Lambda_n^{1/2}\|G(X_1)\|_4$.
	This proves \eqref{eq-app3-thm-02}.
\end{proof}

\section*{Funding}
	Zhuo-Song Zhang was supported by the National Natural Science
	Foundation of China (Grant No.~12301183).

\appendix
\titleformat{\section}{\normalfont\Large\bfseries}{Appendix~\thesection.}{0.75em}{}
\section{Additional details for the proofs of the main results}
\label{app:supplementary-main}
In this section, we provide the proof of
Proposition~\ref{Prop-4.5} and the detailed derivations of
\eqref{eq-bound-I1}, \eqref{eq-bound-I2}, and
\eqref{eq-bound-I3}, which are used in the proof of
Theorem~\ref{thm-main-2}.
\subsection{Proof of Proposition \ref{Prop-4.5}}\label{sec:proof-of-proposition-4.5}
The proof is inspired by the treatment of \(R_{1}\) in the proof of Theorem 2.1 of  \cite{xiao2024}, particularly the argument leading to equation (5.30).
For convenience, let $V=(V_{jk})$. Using the expression of $\nabla^{2}g_{\alpha}(x)$ in \eqref{eq-ha-05} with $\sin \alpha_0=\varepsilon$, we have
\begin{align}\label{eq-bb-01}
	\int^{\pi/2}_{0}\mathbb{E}\langle \nabla^{2} g_{\alpha}(U), V\rangle\tan \alpha d\alpha=I_1+I_2,
\end{align}
where
\begin{align*}
	I_1&=- \sum_{j,k=1}^d\mathbb{E}\Big[V_{j k} \int_0^{\alpha_0} \cos \alpha \int_{\mathbb{R}^d} \partial_jf(U\cos\alpha+z\sin\alpha) \partial_k \phi(z) d z d \alpha\Big],\\
	I_2&= \sum_{j,k=1}^d\mathbb{E}\Big[V_{j k} \int^{\pi/2}_{\alpha_0} \cot \alpha \int_{\mathbb{R}^d} f(U\cos\alpha+z\sin\alpha) \partial_{jk} \phi(z) d z d \alpha\Big].
\end{align*}
For $I_1$, combining Cauchy--Schwarz incquality and Lemma \ref{lem4.1}, we have
\begin{align}\label{eq-bb-02}
	\begin{aligned}
		|I_1|&=\Big|\mathbb{E}\Big[ \int_0^{\alpha_0} \cos \alpha \int_{\mathbb{R}^d} \sum_{j=1}^d\partial_jf(U\cos\alpha+z\sin\alpha) \sum_{k=1}^dV_{j k}\partial_k\phi(z) d z d\alpha\Big]\Big| \\
		& \leq \frac{2}{\varepsilon}  \mathbb{E}\int_0^{\alpha_0} \cos \alpha \int_{\mathbb{R}^d}\Big\{\sum_{j=1}^d\Big(\sum_{k=1}^d V_{jk} \frac{\partial_k \phi(z)}{\phi(z)}\Big)^2\Big\}^{1 / 2} \phi(z) d zd \alpha\\
		& \leq \frac{2}{\varepsilon}  \mathbb{E}\int_0^{\alpha_0} \cos \alpha \bigg\{\int_{\mathbb{R}^d} \sum_{j=1}^d\Big(\sum_{k=1}^d V_{jk} \frac{\partial_k \phi(z)}{\phi(z)}\Big)^2 \phi(z) d z\bigg\}^{1 / 2}d \alpha\\
		&=\frac{2}{\varepsilon} \mathbb{E}\int_0^{\alpha_0} \cos \alpha\bigg\{\bigg[\sum_{j=1}^d \sum_{k=1}^d V^2_{jk}\bigg]\bigg\}^{1 / 2}d \alpha \\
		&=2\E\|V\|_{\Fm}.
	\end{aligned}
\end{align}
We next estimate $I_2$. For a deterministic matrix
$v=(v_{jk})$, define
\[
Q_v(z)
=
\sum_{j,k=1}^d
v_{jk}\frac{\partial_{jk}\phi(z)}{\phi(z)}
=
\sum_{j,k=1}^dv_{jk}(z_jz_k-\delta_{jk}).
\]
The Gaussian fourth-moment identity yields
\begin{align*}
	\int_{\mathbb R^d}Q_v(z)^2\phi(z)\,dz
	&=
	\sum_{j,k=1}^dv_{jk}^2
	+
	\sum_{j,k=1}^dv_{jk}v_{kj}\leq 2\|v\|_{\Fm}^2.
\end{align*}
Since $q\geq2$ implies $1<r\leq2$, it follows that
\begin{equation}\label{eq-bb-03}
	\bigg\{
	\int_{\mathbb R^d}|Q_v(z)|^r\phi(z)\,dz
	\bigg\}^{1/r}
	\leq
	\bigg\{
	\int_{\mathbb R^d}Q_v(z)^2\phi(z)\,dz
	\bigg\}^{1/2}
	\leq \sqrt{2}\|v\|_{\Fm}.
\end{equation}
For $I_{2}$,
we  decompose $I_{2}$ as follows:
\[
I_2=I_{21}+I_{22},
\]
where
\begin{align*}
	I_{21}
	&=
	\E\int_{\alpha_0}^{\pi/2}\cot\alpha
	\int_{\mathbb R^d}
	\big(f-\mathbf 1_A\big)
	(U\cos\alpha+z\sin\alpha)
	Q_V(z)\phi(z)\,dz\,d\alpha,\\
	I_{22}
	&=
	\E\int_{\alpha_0}^{\pi/2}\cot\alpha
	\int_{\mathbb R^d}
	\mathbf 1_A(U\cos\alpha+z\sin\alpha)
	Q_V(z)\phi(z)\,dz\,d\alpha.
\end{align*}
We first control $I_{21}$. For $u\in\mathbb R^d$, let
\[
p_\alpha(u)
=
\int_{\mathbb R^d}
\mathbf 1_{A^\varepsilon\setminus A}
(u\cos\alpha+z\sin\alpha)
\phi(z)\,dz.
\]
Note that $1_A\leq f\leq 1_{A^\varepsilon}$. Combining this with  \eqref{eq-bb-03} and H\"older's inequality, we have
\begin{align}		\label{eq-bb-04}
	\begin{aligned}
		|I_{21}|
		&\leq 	\E\int_{\alpha_0}^{\pi/2}\cot\alpha
		\int_{\mathbb R^d}
		\mathbf 1_{A^\varepsilon\setminus A}
		(U\cos\alpha+z\sin\alpha)
		|Q_V(z)|\phi(z)\,dz\,d\alpha\\
		&\leq
		C\int_{\alpha_0}^{\pi/2}
		\E\big\{
		p_\alpha(U)^{1/q}\|V\|_{\Fm}
		\big\}
		\cot\alpha\,d\alpha \\
		&\leq
		C\int_{\alpha_0}^{\pi/2}
		\big\{\E p_\alpha(U)\big\}^{1/q}
		\cot\alpha\,d\alpha \cdot \big(\E\|V\|_{\Fm}^{r}\big)^{1/r}.
	\end{aligned}
\end{align}
For \(0<\alpha<\pi/2\) and \(z\in\mathbb R^d\), define
\begin{align*}
	B_{\alpha,z}
	&:=\{u\in\mathbb R^d:u\cos\alpha+z\sin\alpha\in A\}
	=\frac{A-z\sin\alpha}{\cos\alpha},\\
	B_{\alpha,z}^{+}
	&:=\{u\in\mathbb R^d:u\cos\alpha+z\sin\alpha\in A^\varepsilon\}
	=\frac{A^\varepsilon-z\sin\alpha}{\cos\alpha}.
\end{align*}
Since \(A\) and \(A^\varepsilon\) are convex and
\(A\subseteq A^\varepsilon\), the sets \(B_{\alpha,z}\) and
\(B_{\alpha,z}^{+}\) are convex and satisfy
$
B_{\alpha,z}\subseteq B_{\alpha,z}^{+}.
$ Hence, by the definition of $d_{c}(U,Z)$, we have
\begin{align*}
	\mathbb P\big(
	U\cos\alpha+z\sin\alpha
	\in A^\varepsilon\setminus A
	\big)
	&=
	\mathbb P(U\in B_{\alpha,z}^{+})
	-\mathbb P(U\in B_{\alpha,z})\\
	&\leq
	\mathbb P(Z\in B_{\alpha,z}^{+})
	-\mathbb P(Z\in B_{\alpha,z})
	+2d_{c}(U,Z).
\end{align*}
Let $G$ be an $N(0,I_d)$ random vector independent of $(U,V,Z)$.
Using
$Z\cos\alpha+G\sin\alpha\sim N(0,I_d)$, we have
\begin{align}		\label{eq-bb-05}
	\begin{aligned}
		\E \{p_\alpha(U)\}
		&\leq
		\mathbb P\big(
		Z\cos\alpha+G\sin\alpha
		\in A^\varepsilon\setminus A
		\big)
		+2d_{c}(U,Z)\\
		&\leq
		Cd^{1/4}\varepsilon+2d_{c}(U,Z),
	\end{aligned}
\end{align}
where the last inequality follows from the Gaussian boundary estimate
for convex sets; see Theorem~4 of \cite{ball1993}. Therefore,
since
\[
\int_{\alpha_0}^{\pi/2}\cot\alpha\,d\alpha
=
|\log\varepsilon|\quad \text{and}\quad
\sup_{0<t<1}t^{1/q}|\log t|\leq q,
\]
equations \eqref{eq-bb-04} and
\eqref{eq-bb-05} imply
\begin{equation}\label{eq-bb-06}
	|I_{21}|
	\leq
	C\big\{
	q	d^{1/(4q)}
	+
	|\log\varepsilon|[d_{c}(U,Z)]^{1/q}
	\big\}\cdot  \big(\E\|V\|_{\Fm}^{r}\big)^{1/r}.
\end{equation}

It remains to estimate $I_{22}$. We first establish an elementary
localization estimate closely related to the convex-set reduction in
Lemma~2.6 of
\cite{SchulteYukich2019}. Let $B\subset\mathbb R^d$ be a measurable convex set with
$\varnothing\ne B\ne\mathbb R^d$
 and put
\[
\rho_B=\operatorname{dist}(0,\partial B).
\]
Then, for every deterministic matrix $v$,
\begin{equation}\label{eq-bb-07}
	\bigg|
	\int_BQ_v(z)\phi(z)\,dz
	\bigg|
	\leq
	C\|v\|_{\Fm}
	\big\{
	\mathbb P(|G|\geq\rho_B)
	\big\}^{1/q}.
\end{equation}
Indeed, when $\rho_B=0$, this follows directly from
\eqref{eq-bb-03}. If $\rho_B>0$ and $0\notin B$, convexity gives
$
B\subset\{z:|z|\geq\rho_B\}.
$
If $\rho_B>0$ and $0\in B$, then
$
B^c\subset\{z:|z|\geq\rho_B\}.
$
Moreover,
\[
\int_{\mathbb R^d}Q_v(z)\phi(z)\,dz=0,
\]
so in the second case the integral over $B$ equals the negative of the
integral over $B^c$. In either case,
\eqref{eq-bb-07} follows from H\"older's inequality and
\eqref{eq-bb-03}.

We now return to deriving an upper bound for $I_{22}$. The cases \(A=\varnothing\) and \(A=\mathbb R^d\) are immediate, so we assume that \(A\) is a nonempty proper convex set.
For $\alpha\in[\alpha_0,\pi/2)$ and $u\in\mathbb R^d$, set
\[
A_{\alpha,u}
=
\frac{A-u\cos\alpha}{\sin\alpha},
\qquad
\rho_{\alpha,u}
=
\operatorname{dist}(0,\partial A_{\alpha,u})
=
\frac{\operatorname{dist}(u\cos\alpha,\partial A)}
{\sin\alpha}.
\]
Applying \eqref{eq-bb-07} conditionally on $(U,V)$ and
then using H\"older's inequality gives
\begin{align}
	|I_{22}|
	&\leq
	C\big(\E\|V\|_{\Fm}^{r}\big)^{1/r}\cdot\int_{\alpha_0}^{\pi/2}
	\pi_\alpha^{1/q}\cot\alpha\,d\alpha,
	\label{eq-bb-08}
\end{align}
where
$
\pi_\alpha
=
\mathbb P\big(
\operatorname{dist}(U\cos\alpha,\partial A)
\leq |G|\sin\alpha
\big).
$
For $t\geq0$, write
\[
(\bar{A})^t
=
\{x:\operatorname{dist}(x,\bar{A})\leq t\},\quad \text{and}\quad
(A^\circ)^{-t}
=
\{x\in A^\circ:
\operatorname{dist}(x,(A^\circ)^c)>t\}.
\]
Then both sets defined above are convex, and
$
\{x:\operatorname{dist}(x,\partial A)\leq t\}
=
(\bar{A})^t\setminus(A^\circ)^{-t}.
$
Consequently, conditioning on $G$ and applying the definition of
$d_{c}(U,Z)$ to these two nested convex sets, we have
\begin{equation}\label{eq-bb-09}
	\pi_\alpha
	\leq
	\mathbb P\big(
	\operatorname{dist}(Z\cos\alpha,\partial A)
	\leq |G|\sin\alpha
	\big)
	+2d_{c}(U,Z).
\end{equation}
Conditionally on $G$, the Gaussian boundary estimate applied to
$A/\cos\alpha$ yields
\begin{align*}
	&\mathbb P\big(
	\operatorname{dist}(Z\cos\alpha,\partial A)
	\leq |G|\sin\alpha
	\,\big|\,G
	\big)\leq
	1\wedge
	\big\{
	Cd^{1/4}|G|\tan\alpha
	\big\}.
\end{align*}
Since $\E|G|\leq d^{1/2}$, it follows that
\begin{equation}\label{eq-bb-10}
	\mathbb P\big(
	\operatorname{dist}(Z\cos\alpha,\partial A)
	\leq |G|\sin\alpha
	\big)
	\leq
	1\wedge
	\big\{
	Cd^{3/4}\tan\alpha
	\big\}.
\end{equation}

To bound the $\alpha$-integral in \eqref{eq-bb-08}, put $a=Cd^{3/4}$, where $C$ may be
enlarged so that $a\geq1$. The substitution $t=\tan\alpha$ gives
\begin{align}	\label{eq-bb-11}
	\begin{aligned}
		&
		\int_0^{\pi/2}
		\big[
		1\wedge\{a\tan\alpha\}
		\big]^{1/q}
		\cot\alpha\,d\alpha =
		\int_0^\infty
		\frac{[1\wedge\{at\}]^{1/q}}
		{t(1+t^2)}\,dt \\
		&\quad\leq
		a^{1/q}\int_0^{a^{-1}}t^{1/q-1}\,dt
		+
		\int_{a^{-1}}^1\frac{dt}{t}
		+
		\int_1^\infty\frac{dt}{t(1+t^2)} \\
		&\quad\leq
		q+\log a+C
		\leq
		Cqd^{1/(2q)},
	\end{aligned}
\end{align}
where the last inequality follows from the elementary inequality
$1+\log d\leq Cqd^{1/(2q)}$ for $d\geq1$ and $q\geq 2$.
Combining \eqref{eq-bb-08}--\eqref{eq-bb-11} and using
$
\int_{\alpha_0}^{\pi/2}\cot\alpha\,d\alpha
=
|\log\varepsilon|,
$
we conclude that
\begin{equation}\label{eq-bb-12}
	|I_{22}|
	\leq
	C\big(
	qd^{1/(2q)}
	+
	|\log\varepsilon|[d_{c}(U,Z)]^{1/q}
	\big)\cdot  \big(\E\|V\|_{\Fm}^{r}\big)^{1/r}.
\end{equation}
Finally, combining \eqref{eq-bb-01}, \eqref{eq-bb-02},
\eqref{eq-bb-06}, and \eqref{eq-bb-12}, and using
$d\geq1$,  we obtain the desired result.

\subsection{Proofs of \eqref{eq-bound-I1}, \eqref{eq-bound-I2} and \eqref{eq-bound-I3}}
\label{sec:appen-pr-three-bounds}
In this subsection, we prove \eqref{eq-bound-I1}, \eqref{eq-bound-I2} and \eqref{eq-bound-I3}.

	\smallskip
\noindent
{\it (i) Proof of \eqref{eq-bound-I1}.} \eqref{eq-bound-I1} follows directly from \eqref{eq-pp-04}.

	\smallskip
\noindent
{\it (ii) Proof of \eqref{eq-bound-I2}.}
Exchanging $\xi_i$ and $\xi_i'$ in the definition of
$I_{2,\alpha}$, we obtain
\begin{align*}
	I_{2,\alpha}
	&=\frac12\sum_{i=1}^n
	\E\Big\langle
	X_i-X_i',
	\nabla g_\alpha(W+D)
	-\nabla g_\alpha(W+D^{\{i\}})
	\Big\rangle.
\end{align*}
For each $1\leq i\leq n$ and $\theta\in[0,1]$, set
$
Q_i(\theta)=(1-\theta)D+\theta D^{\{i\}}.
$
Taylor's integral formula gives
\begin{align*}
	I_{2,\alpha}
	&=\frac12\sum_{i=1}^n\int_0^1
	\E\Big\langle
	\nabla^2g_\alpha(W+Q_i(\theta)),
	(X_i-X_i')\otimes(D-D^{\{i\}})
	\Big\rangle\,d\theta.
\end{align*}
Let $(\eta_1,\ldots,\eta_n)$ be an independent copy of
$(\xi_1,\ldots,\xi_n)$, independent of all previously introduced
random variables. Denote by $Q_i^{\{j\}}(\theta)$ the version of
$Q_i(\theta)$ obtained by replacing the original input $\xi_j$
with $\eta_j$, while keeping the additional input $\xi_i'$ fixed.
By Minkowski's inequality and the replacement construction,
\[
\|Q_i(\theta)-Q_i^{\{j\}}(\theta)\|_2
\leq\|D-D^{\{j\}}\|_2,
\qquad
\E|Q_i(\theta)|\leq\E|D|.
\]
Consequently, the quantity
\begin{align*}
	\delta_{i,1}(\theta)
	&:=d^{1/4}\gamma_1
	+qd^{1/(2q)}
	\sum_{j=1}^n
	\|X_j\|_4
	\|Q_i(\theta)-Q_i^{\{j\}}(\theta)\|_2\\
	&\quad
	+d^{1/8}
	\Big(
	\sum_{j=1}^n
	\|X_j\|_4^2
	\|Q_i(\theta)-Q_i^{\{j\}}(\theta)\|_2
	\Big)^{1/2}
	+\E|Q_i(\theta)|
\end{align*}
satisfies $\delta_{i,1}(\theta)\leq\delta$.
For each fixed $i$ and $\theta$, condition on $\xi_i'$.
The original inputs $\xi_1,\ldots,\xi_n$ remain independent
with their distributions unchanged, and the linear part is
still $W$. Applying the definition of $K_n$ conditionally
and then averaging over $\xi_i'$, we obtain
\[
d_c(W+Q_i(\theta),Z)
\leq K_n\delta_{i,1}(\theta)
\leq K_n\delta.
\]
Let $q'=q/(q-1)$. Since $q\geq4$, we have $q'\leq4/3$.
By H\"{o}lder's inequality, we have
\begin{align*}
	&\big(
	\E\big\|
	(X_i-X_i')\otimes(D-D^{\{i\}})
	\big\|_{\mathrm F}^{q'}
	\big)^{1/q'}\leq
	\|X_i-X_i'\|_4\|D-D^{\{i\}}\|_2
	\leq2\|X_i\|_4\|D-D^{\{i\}}\|_2.
\end{align*}
Applying Proposition~\ref{Prop-4.5} to the $\alpha$-integral
for each fixed $\theta$, with
$
U=W+Q_i(\theta),$
$
V=(X_i-X_i')\otimes(D-D^{\{i\}}),
$
and then integrating over $\theta$, we have
\begin{align}\label{eq-t-05}
	\begin{aligned}
		|I_2|
	&\leq C\sum_{i=1}^n\int_0^1
	\Big(
	qd^{1/(2q)}
	+|\log\varepsilon|
	[d_c(W+Q_i(\theta),Z)]^{1/q}
	\Big)
	\|X_i\|_4\|D-D^{\{i\}}\|_2\,d\theta
\\
	&\leq
	C\Big(
	qd^{1/(2q)}
	+|\log\varepsilon|(K_n\delta)^{1/q}
	\Big)
	\sum_{i=1}^n
	\|X_i\|_4\|D-D^{\{i\}}\|_2.
	\end{aligned}
\end{align}
This proves \eqref{eq-bound-I2}.

	\smallskip
\noindent
{\it (iii) Proof of \eqref{eq-bound-I3}.}  For notational convenience,  for each $1\leq i\leq n$, let
$Y_{i}=X_{i}-X'_{i}$. We further decompose $I_{3,\alpha}$ as follows:
\begin{align*}
	I_{3,\alpha}&=I_{31,\alpha}+I_{32,\alpha},
\end{align*}
where
\begin{align*}
	I_{31,\alpha}&=\sum_{i=1}^{n}\E   \langle \nabla^{2} g_{\alpha}(T),  Y_{i}\otimes Y_{i}/2-\va(X_{i})\ra, \\
	I_{32,\alpha}&=\frac{1}{2}\sum_{i=1}^{n}\E   \int_{0}^{1} \langle \nabla^{2} g_{\alpha}(T-tY_{i})-\nabla^{2} g_{\alpha}(T),  Y_{i}\otimes Y_{i}\ra dt.
\end{align*}
For $I_{31,\alpha}$,  noting that $\E Y_{i}\otimes Y_{i}/2=\E X_{i}\otimes X_{i}=\Var(X_{i})$, we have
\begin{align*}
	\E\Big\|\sum_{i=1}^{n}\Big\{\frac{1}{2} (X_{i}-X'_{i})\otimes (X_{i}-X'_{i})-\va(X_{i})\Big\}\Big\|_{\Fm}^{2}\leq \sum_{i=1}^{n}\E|X_{i}|^{4}.
\end{align*}
Combining this with Proposition \ref{Prop-4.5}, we have
\begin{align}\label{eq-t-06}
	\Big|\int_{0}^{\pi/2} I_{31,\alpha} \tan\alpha d\alpha\Big|\leq
	C\big(d^{1/4}+|\log(\varepsilon)|(K_{n}\delta)^{1/2}\big)\Big(\sum_{i=1}^{n}\E|X_{i}|^{4}\Big)^{1/2}.
\end{align}
For $I_{32,\alpha}$,  Taylor's expansion gives
\begin{align*}
	I_{32,\alpha}&=	-\frac{1}{2}\sum_{i=1}^{n}\E\int_{0}^{1}(1-t)
	\left\langle
	\nabla^{3}g_{\alpha}(T-tY_i),
	Y_i^{\otimes 3}
	\right\rangle dt.
\end{align*}
For any fixed $t$, by exchangeability, we have
\begin{align*}
	\E\big\langle
	\nabla^{3}g_{\alpha}(T-tY_i),
	Y_i^{\otimes 3}
	\big\rangle&=-\E\big\langle
	\nabla^{3}g_{\alpha}(T^{\{i\}}+tY_i),
	Y_i^{\otimes 3}
	\big\rangle=-\E\big\langle
	\nabla^{3}g_{\alpha}(W+D^{\{i\}}+(t-1)Y_i),
	Y_i^{\otimes 3}
	\big\rangle.
\end{align*}
Based on this identity, we can further decompose \(I_{32,\alpha}\) as follows:
\begin{align*}
	I_{32,\alpha}=I_{321,\alpha}+I_{322,\alpha},
\end{align*}
where
\begin{align*}
	I_{321,\alpha}&=-\frac{1}{4}\sum_{i=1}^{n}\E\int_{0}^{1}(1-t)
	\big\langle
	\nabla^{3}g_{\alpha}(T+(t-1)Y_i)-\nabla^{3}g_{\alpha}(W+D^{\{i\}}+(t-1)Y_i),
	Y_i^{\otimes 3}
	\big\rangle dt,\\
	I_{322,\alpha}&=-\frac{1}{4}\sum_{i=1}^{n}\E\int_{0}^{1}(1-t)
	\big\langle
	\nabla^{3}g_{\alpha}(T-tY_i)-	\nabla^{3}g_{\alpha}(T+(t-1)Y_i),
	Y_i^{\otimes 3}
	\big\rangle dt.
\end{align*}
We first bound $I_{321,\alpha}.$
Let $\alpha_{1}\in(0,\pi/2)$ be defined by
$
\tan\alpha_{1}=(\sqrt{2}\varepsilon)/4.
$
By \eqref{eq-ss-01},  \eqref{eq-ha-06} and Lemma \ref{lem4.1}, for $0<\alpha<\alpha_{1}$, we have
\begin{align*}
	|\nabla^{3}g_{\alpha}(x)|_{\vee}
	&=
	\sup_{|u|\leq1}
	\left|
	\left\langle
	\nabla^{3}g_{\alpha}(x),
	u^{\otimes3}
	\right\rangle
	\right|
	\\
	&=
	\sup_{|u|\leq1}
	\left|
	\cos^{2}\alpha\cot\alpha
	\int_{\mathbb{R}^{d}}
	\left\langle
	\nabla^{2}f(x\cos\alpha+z\sin\alpha),
	u^{\otimes2}
	\right\rangle
	\left\langle
	u,\nabla\phi_{d}(z)
	\right\rangle dz
	\right|
	\\
	&\leq
	\frac{8c_{1}}{\varepsilon^{2}}
	\cos^{2}\alpha\cot\alpha.
\end{align*}
Similarly, for $\alpha_{1}<\alpha<\frac{\pi}{2}$, we have
\begin{align*}
	|\nabla^{3}g_{\alpha}(x)|_{\vee}
	&=
	\sup_{|u|\leq1}
	\left|
	\left\langle
	\nabla^{3}g_{\alpha}(x),
	u^{\otimes3}
	\right\rangle
	\right|
	\\
	&=
	\sup_{|u|\leq1}
	\Big|
	\cot^{3}\alpha
	\int_{\mathbb{R}^{d}}
	f(x\cos\alpha+z\sin\alpha)
	\left\langle
	u^{\otimes3},
	\nabla^{3}\phi_{d}(z)
	\right\rangle dz
	\Big|
	\\
	&\leq
	c_{3}\cot^{3}\alpha.
\end{align*}
Hence,
\begin{align}\label{eq-t-07}
	\begin{aligned}
		\int_{0}^{\frac{\pi}{2}}
		|\nabla^{3}g_{\alpha}(x)|_{\vee}
		\tan\alpha\,d\alpha
		&\leq
		\frac{8c_{1}}{\varepsilon^{2}}
		\int_{0}^{\alpha_{1}}
		\cos^{2}\alpha\,d\alpha
		+
		c_{3}
		\int_{\alpha_{1}}^{\frac{\pi}{2}}
		\cot^{2}\alpha\,d\alpha
		\\
		&\leq
		\frac{8c_{1}}{\varepsilon^{2}}
		\tan\alpha_{1}
		+
		c_{3}\cot\alpha_{1}
		\\
		&=
		\frac{2\sqrt{2}c_{1}}{\varepsilon}
		+
		\frac{2\sqrt{2}c_{3}}{\varepsilon}
		\\
		&\leq
		\frac{3(c_{1}+c_{3})}{\varepsilon}.
	\end{aligned}
\end{align}
By a direct algebraic calculation, we have
\begin{align*}
	&-\frac{1}{4}\E\int_{0}^{1}(1-t)
	\big\langle
	\nabla^{3}g_{\alpha}(W+D+(t-1)Y_i)
	-\nabla^{3}g_{\alpha}(W+D^{\{i\}}+(t-1)Y_i),
	Y_i^{\otimes 3}
	\big\rangle dt
	\\
	&\quad=
	\frac{1}{4}\E
	\big\langle
	\nabla^{2}g_{\alpha}(W+D-Y_i)
	-\nabla^{2}g_{\alpha}(W+D^{\{i\}}-Y_i),
	Y_i^{\otimes 2}
	\big\rangle
	\\
	&\qquad-
	\frac{1}{4}\E\int_{0}^{1}
	\big\langle
	\nabla^{2}g_{\alpha}(W+D+(t-1)Y_i)
	-\nabla^{2}g_{\alpha}(W+D^{\{i\}}+(t-1)Y_i),
	Y_i^{\otimes 2}
	\big\rangle dt
	\\
	&\quad=
	\frac{1}{4}\E\int_{0}^{1}
	\big\langle
	\nabla^{3}g_{\alpha}
	\bigl(W+D^{\{i\}}-Y_i
	+s(D-D^{\{i\}})\bigr),
	(D-D^{\{i\}})\otimes Y_i^{\otimes 2}
	\big\rangle ds
	\\
	&\qquad-
	\frac{1}{4}\E\int_{0}^{1}\int_{0}^{1}
	\big\langle
	\nabla^{3}g_{\alpha}
	\bigl(W+D^{\{i\}}+(t-1)Y_i
	+s(D-D^{\{i\}})\bigr),
	(D-D^{\{i\}})\otimes Y_i^{\otimes 2}
	\big\rangle ds\,dt.
\end{align*}
Combining this with \eqref{eq-t-07}, we have
\begin{align}\label{eq-t-08}
	\Big|\int_{0}^{\pi/2} I_{321,\alpha}\tan \alpha d\alpha\Big|\leq \frac{C}{\varepsilon}\sum_{i=1}^{n}\|D-D^{\{i\}}\|_{2}\|X_{i}\|_{4}^{2}.
\end{align}
To bound $I_{322,\alpha}$, let
$(\xi_1^\dagger,\xi_2^\dagger,\ldots,\xi_n^\dagger)$
be an independent copy of $(\xi_1,\xi_2,\ldots,\xi_n)$,
independent of all previously introduced random variables. For each $1\leq i\leq n$, define
\[
D^{\dagger,\{i\}}=D(\xi_{1},\cdots, \xi_{i-1}, \xi^{\dagger}_{i}, \xi_{i+1},\cdots, \xi_{n}).
\]
We further decompose $I_{322,\alpha}$ as follows:
\[
I_{322,\alpha}=I_{323,\alpha}+I_{324,\alpha}+I_{325,\alpha},
\]
where
\begin{align*}
	I_{323,\alpha}&=-\frac{1}{4}\sum_{i=1}^{n}\E\int_{0}^{1}(1-t)
	\big\langle
	\nabla^{3}g_{\alpha}(W+D-tY_i)-	\nabla^{3}g_{\alpha}(W+D^{\dagger,\{i\}}-tY_i),
	Y_i^{\otimes 3}
	\big\rangle dt,\\
	I_{324,\alpha}&=-\frac{1}{4}\sum_{i=1}^{n}\E\int_{0}^{1}(1-t)
	\big\langle
	\nabla^{3}g_{\alpha}(W+D^{\dagger,\{i\}}-tY_i)-	\nabla^{3}g_{\alpha}(W+D^{\dagger,\{i\}}+(t-1)Y_i),
	Y_i^{\otimes 3}
	\big\rangle dt,\\
	I_{325,\alpha}&=-\frac{1}{4}\sum_{i=1}^{n}\E\int_{0}^{1}(1-t)
	\big\langle
	\nabla^{3}g_{\alpha}(W+D^{\dagger,\{i\}}+(t-1)Y_i)-	\nabla^{3}g_{\alpha}(T+(t-1)Y_i),
	Y_i^{\otimes 3}
	\big\rangle dt.
\end{align*}
For $I_{323,\alpha}$ and $I_{325,\alpha}$, following a similar argument as that leading to \eqref{eq-t-08}, we have
\begin{align}\label{eq-t-09}
	\Big|\int_{0}^{\pi/2} I_{323,\alpha}\tan \alpha d\alpha\Big|+\Big|\int_{0}^{\pi/2} I_{325,\alpha}\tan \alpha d\alpha\Big|\leq \frac{C}{\varepsilon}\sum_{i=1}^{n}\|D-D^{\{i\}}\|_{2}\|X_{i}\|_{4}^{2}.
\end{align}

It remains to bound $I_{324,\alpha}$. Recall that
$\gamma_{1}^{2}=\sum_{i=1}^{n}\|X_{i}\|_{4}^{4}$.
The smoothing function $f$ constructed in the proof of
\citet[Lemma 2.3]{bentkus2003} is given explicitly by
\[
f(x)=
\begin{cases}
	1-2\dist(x,A)^{2}/\varepsilon^{2},
	& \dist(x,A)\leq\varepsilon/2,\\[2pt]
	2\{1-\dist(x,A)/\varepsilon\}^{2},
	& \varepsilon/2<\dist(x,A)\leq\varepsilon,\\[2pt]
	0,
	& \dist(x,A)>\varepsilon.
\end{cases}
\]
The cases $A=\varnothing$ and $A=\mathbb{R}^{d}$ are immediate.
All superlevel sets of $f$ are convex. Moreover, $0\leq f\leq1$,
$M_{1}(f)\leq2/\varepsilon$, and $M_{2}(f)\leq8/\varepsilon^{2}$.
Let $g_{\alpha}(x)$ be defined by \eqref{eq-ha-03}.
For every $0<\alpha<\pi/2$, the change of variables
$y=x\cos\alpha+z\sin\alpha$ gives
\[
g_{\alpha}(x)
=
\frac{1}{(\sin\alpha)^{d}}
\int_{\mathbb{R}^{d}}
f(y)\phi_{d}\biggl(
\frac{y-x\cos\alpha}{\sin\alpha}
\biggr)\,dy.
\]
Since $f$ is bounded, differentiation of the Gaussian kernel shows
that $g_{\alpha}\in C^{\infty}(\mathbb{R}^{d})$ and
\begin{align}\label{eq-nn-02}
	\begin{aligned}
		\nabla^{4}g_{\alpha}(x)
		&=
		\cot^{4}\alpha
		\int_{\mathbb{R}^{d}}
		f(x\cos\alpha+z\sin\alpha)
		\nabla^{4}\phi_{d}(z)\,dz\\
		&=
		\cos^{2}\alpha\cot^{2}\alpha
		\int_{\mathbb{R}^{d}}
		\nabla^{2}f(x\cos\alpha+z\sin\alpha)
		\otimes\nabla^{2}\phi_{d}(z)\,dz.
	\end{aligned}
\end{align}
where
the second identity follows by Gaussian integration by parts.
Changing $t$ to $1-t$ in the second term of $I_{324,\alpha}$
and then integrating by parts in $t$, we obtain
\begin{align*}
	I_{324,\alpha}
	&=
	-\frac{1}{4}\sum_{i=1}^{n}
	\int_{0}^{1}(1-2t)
	\E\bigl\{
	\bigl\langle
	\nabla^{3}g_{\alpha}
	\bigl(W+D^{\dagger,\{i\}}-tY_{i}\bigr),
	Y_{i}^{\otimes3}
	\bigr\rangle
	\bigr\}\,dt\\
	&=
	-\frac{1}{4}\sum_{i=1}^{n}
	\int_{0}^{1}t(1-t)
	\E\bigl\{
	\bigl\langle
	\nabla^{4}g_{\alpha}
	\bigl(W+D^{\dagger,\{i\}}-tY_{i}\bigr),
	Y_{i}^{\otimes4}
	\bigr\rangle
	\bigr\}\,dt.
\end{align*}
These identities also hold after inserting either
$\mathbf{1}_{\{|Y_{i}|\leq\sin\alpha\}}$ or
$\mathbf{1}_{\{|Y_{i}|>\sin\alpha\}}$ in each $i$th summand.
Choose $\beta\in(0,\pi/2)$ such that
$\sin\beta=\varepsilon$. Write
\[
\int_{0}^{\pi/2}I_{324,\alpha}\tan\alpha\,d\alpha
=J_{1}+J_{2}+J_{3},
\]
where
\begin{align*}
	J_{1}
	={}&
	-\frac{1}{4}\sum_{i=1}^{n}
	\int_{0}^{\beta}
	\int_{0}^{1}
	t(1-t)
	\E\biggl\{
	\mathbf{1}_{\{|Y_{i}|\leq\sin\alpha\}}
	\bigl\langle
	\nabla^{4}g_{\alpha}
	\bigl(W+D^{\dagger,\{i\}}-tY_{i}\bigr),
	Y_{i}^{\otimes4}
	\bigr\rangle
	\biggr\}
	\,dt\,\tan\alpha\,d\alpha,
	\\
	J_{2}
	={}&
	-\frac{1}{4}\sum_{i=1}^{n}
	\int_{0}^{\beta}
	\int_{0}^{1}
	t(1-t)
	\E\biggl\{
	\mathbf{1}_{\{|Y_{i}|>\sin\alpha\}}
	\bigl\langle
	\nabla^{4}g_{\alpha}
	\bigl(W+D^{\dagger,\{i\}}-tY_{i}\bigr),
	Y_{i}^{\otimes4}
	\bigr\rangle
	\biggr\}
	\,dt\,\tan\alpha\,d\alpha,
	\\
	J_{3}
	={}&
	-\frac{1}{4}\sum_{i=1}^{n}
	\int_{\beta}^{\pi/2}
	\int_{0}^{1}
	t(1-t)
	\E\biggl\{
	\bigl\langle
	\nabla^{4}g_{\alpha}
	\bigl(W+D^{\dagger,\{i\}}-tY_{i}\bigr),
	Y_{i}^{\otimes4}
	\bigr\rangle
	\biggr\}
	\,dt\,\tan\alpha\,d\alpha.
\end{align*}

We shall use the following Gaussian comparison.  Let
$W^{(i)}=W-X_{i}$ and $\Sigma_{i}=\Var(W^{(i)})$.
Since $d^{1/4}\gamma_{1}\leq0.01$, we have
\[
\Sigma_{i}
=I_{d}-\Var(X_{i})
\succeq
\bigl(1-\|X_{i}\|_{2}^{2}\bigr)I_{d}
\succeq
(1-\gamma_{1})I_{d},
\qquad
\|\Sigma_{i}^{-1/2}\|_{\op}\leq2.
\]
Moreover, $W^{(i)}+D^{\dagger,\{i\}}$ is independent of
$(X_{i},X_{i}')$. To apply the definition of $K_{n}$ to
$
\Sigma_{i}^{-1/2}
\bigl(W^{(i)}+D^{\dagger,\{i\}}\bigr),
$
we regard $\xi_{i}^{\dagger}$ as the $i$th independent input
and the corresponding linear summand as zero. The covariance
matrix of the linear part is then $I_{d}$. For $j\neq i$,
the replacement difference of the remainder, after applying
the linear transformation $\Sigma_{i}^{-1/2}$, has $L^{2}$ norm
at most
$
2\|D-D^{\{j\}}\|_{2}.
$
The $i$th linear summand is zero and hence contributes nothing
to the terms involving replacement differences. Moreover,
$
\E\bigl|
\Sigma_{i}^{-1/2}D^{\dagger,\{i\}}
\bigr|
\leq2\E|D|.
$
It follows that all the quantities entering the definition of
$K_{n}$ are bounded by an absolute constant times the corresponding
quantities for $T$. Consequently,
\begin{align}\label{eq-nn-01}
	d_{c}\bigl(
	\Sigma_{i}^{-1/2}
	\{W^{(i)}+D^{\dagger,\{i\}}\},Z
	\bigr)
	\leq CK_{n}\delta.
\end{align}
In what follows, we estimate \(J_{1}\), \(J_{2}\), and \(J_{3}\) in turn. For $J_{1}$, using the Hessian representation of $\nabla^{4}g_{\alpha}$ and
$
\langle
\nabla^{2}\phi_{d}(z),Y_{i}^{\otimes2}
\rangle
=
\bigl\{\langle Y_{i},z\rangle^{2}-|Y_{i}|^{2}\bigr\}
\phi_{d}(z),
$
we obtain
\begin{align*}
	J_{1}
	={}&
	-\frac{1}{4}\sum_{i=1}^{n}
	\int_{0}^{\beta}\cos^{2}\alpha\cot\alpha
	\int_{\mathbb{R}^{d}}
	\E\biggl\{
	\mathbf{1}_{\{|Y_{i}|\leq\sin\alpha\}}
	\bigl\{\langle Y_{i},z\rangle^{2}-|Y_{i}|^{2}\bigr\}\\
	&\qquad\times
	\int_{0}^{1}t(1-t)
	\Bigl\langle
	\nabla^{2}f\bigl(
	\{W+D^{\dagger,\{i\}}-tY_{i}\}\cos\alpha
	+z\sin\alpha
	\bigr),
	Y_{i}^{\otimes2}
	\Bigr\rangle\,dt
	\biggr\}
	\phi_{d}(z)\,dz\,d\alpha.
\end{align*}
Define
\[
H(x)
:=
\int_{0}^{\dist(x,A)}
\min\biggl\{\frac{r}{\varepsilon},1\biggr\}\,dr
=
\begin{cases}
	\dist(x,A)^{2}/(2\varepsilon),
	& \dist(x,A)\leq\varepsilon,\\[2pt]
	\dist(x,A)-\varepsilon/2,
	& \dist(x,A)>\varepsilon.
\end{cases}
\]
Then  $H(x)$ is  convex.
Since $\dist(x,A)=\dist(x,\overline{A})$, the gradient formula for
the squared distance and the displayed piecewise expression give
\begin{align*}
	\nabla H(x)
	&=
	\frac{x-P_{\overline{A}}(x)}
	{\max\{\varepsilon,\dist(x,A)\}}.
\end{align*}
In particular, $\nabla H(x)=0$ on $\overline{A}$, and
$
|\nabla H(x)|
\leq1.
$
By the nonexpansiveness of the projection residual
(see \citealp[Corollary~4.10]{bauschke2011convex}),
we have
\[
\bigl|
\{x-P_{\overline{A}}(x)\}
-\{y-P_{\overline{A}}(y)\}
\bigr|
\leq |x-y|,
\qquad x,y\in\mathbb{R}^{d}.
\]
Moreover, the map $u\mapsto u/\max\{1,|u|\}$ is the Euclidean
projection onto the closed unit ball and is therefore
$1$-Lipschitz.
Combining these two properties with the formula for $\nabla H$,
we obtain
\begin{align*}
	|\nabla H(x)-\nabla H(y)|
	&\leq
	\frac{1}{\varepsilon}
	\bigl|
	\{x-P_{\overline{A}}(x)\}
	-\{y-P_{\overline{A}}(y)\}
	\bigr|\leq
	\frac{|x-y|}{\varepsilon}.
\end{align*}
We next verify the directional Hessian comparison. On
$\{x:\dist(x,A)<\varepsilon/2\}$, the definitions give
$f(x)=1-4H(x)/\varepsilon$, and hence
\[
\bigl|
\langle\nabla^{2}f(x),u^{\otimes2}\rangle
\bigr|
=
\frac{4}{\varepsilon}
\langle\nabla^{2}H(x),u^{\otimes2}\rangle.
\]
On $\{x:\varepsilon/2<\dist(x,A)<\varepsilon\}$,
differentiation gives, almost everywhere,
\begin{align*}
	\nabla^{2}f(x)
	&=
	\frac{4}{\varepsilon^{2}}
	\Bigl\{
	\nabla\dist(x,A)\otimes\nabla\dist(x,A)
	-\{\varepsilon-\dist(x,A)\}\nabla^{2}\dist(x,A)
	\Bigr\},\\
	\nabla^{2}H(x)
	&=
	\frac{1}{\varepsilon}
	\Bigl\{
	\nabla\dist(x,A)\otimes\nabla\dist(x,A)
	+\dist(x,A)\nabla^{2}\dist(x,A)
	\Bigr\}.
\end{align*}
Here $\nabla^{2}\dist(x,A)$ is positive semidefinite almost
everywhere outside $\overline{A}$, by convexity of the distance.
Since $0\leq\varepsilon-\dist(x,A)\leq\dist(x,A)$ on this region,
\begin{align*}
	\bigl|\langle\nabla^{2}f(x),u^{\otimes2}\rangle\bigr|
	&\leq
	\frac{4}{\varepsilon^{2}}
	\Bigl\{
	\langle\nabla\dist(x,A),u\rangle^{2}
	+\{\varepsilon-\dist(x,A)\}
	\langle\nabla^{2}\dist(x,A),u^{\otimes2}\rangle
	\Bigr\}\\
	&\leq
	\frac{4}{\varepsilon^{2}}
	\Bigl\{
	\langle\nabla\dist(x,A),u\rangle^{2}
	+\dist(x,A)
	\langle\nabla^{2}\dist(x,A),u^{\otimes2}\rangle
	\Bigr\}\\
	&=
	\frac{4}{\varepsilon}
	\langle\nabla^{2}H(x),u^{\otimes2}\rangle.
\end{align*}
On $\{x:\dist(x,A)>\varepsilon\}$, $f$ is zero and the same
inequality holds. Thus, for almost every $x$ and every $u$,
\begin{align}\label{eq-vv-01}
	\bigl|\langle\nabla^{2}f(x),u^{\otimes2}\rangle\bigr|
	\leq
	\frac{4}{\varepsilon}
	\langle\nabla^{2}H(x),u^{\otimes2}\rangle.
\end{align}
Applying \eqref{eq-vv-01} and using the inequality $t(1-t)\leq 1/4$ for $t\in (0,1)$, we have
\begin{align*}
	|J_{1}|
	\leq{}&
	\frac{1}{4\varepsilon}\sum_{i=1}^{n}
	\int_{0}^{\beta}\frac{\cos^{3}\alpha}{\sin\alpha}
	\int_{\mathbb{R}^{d}}
	\E\Biggl\{
	\mathbf{1}_{\{|Y_{i}|\leq\sin\alpha\}}
	\bigl|\langle Y_{i},z\rangle^{2}-|Y_{i}|^{2}\bigr|\\
	&\qquad\times
	\int_{0}^{1}
	\Bigl\langle
	\nabla^{2}H\bigl(
	\{W+D^{\dagger,\{i\}}-tY_{i}\}\cos\alpha+z\sin\alpha
	\bigr),Y_{i}^{\otimes2}
	\Bigr\rangle\,dt
	\Biggr\}\phi_{d}(z)\,dz\,d\alpha\\
	=&
	\frac{1}{4\varepsilon}\sum_{i=1}^{n}
	\int_{0}^{\beta}\frac{\cos^{2}\alpha}{\sin\alpha}
	\int_{\mathbb{R}^{d}}
	\E\Biggl\{
	\mathbf{1}_{\{|Y_{i}|\leq\sin\alpha\}}
	\bigl|\langle Y_{i},z\rangle^{2}-|Y_{i}|^{2}\bigr|\\
	&\quad\times
	\Bigl\langle Y_{i},
	\nabla H\bigl(
	\{W+D^{\dagger,\{i\}}\}\cos\alpha+z\sin\alpha
	\bigr)\\
	&\hspace{92pt}
	-\nabla H\bigl(
	\{W+D^{\dagger,\{i\}}-Y_{i}\}\cos\alpha+z\sin\alpha
	\bigr)
	\Bigr\rangle
	\Biggr\}\phi_{d}(z)\,dz\,d\alpha.
\end{align*}
Under the exchange of $\xi_{i}$ and $\xi_{i}'$, the joint law is
unchanged and
$
W\longmapsto W-Y_{i}$,
$Y_{i}\longmapsto-Y_{i},$
$D^{\dagger,\{i\}}\longmapsto D^{\dagger,\{i\}}.
$ Hence, for fixed $\alpha,z$,
\begin{align*}
	&\E\biggl\{
	\mathbf{1}_{\{|Y_{i}|\leq\sin\alpha\}}
	\bigl|\langle Y_{i},z\rangle^{2}-|Y_{i}|^{2}\bigr|
	\Bigl\langle Y_{i},
	\nabla H\bigl(
	\{W+D^{\dagger,\{i\}}-Y_{i}\}\cos\alpha+z\sin\alpha
	\bigr)
	\Bigr\rangle
	\biggr\}\\
	&\quad=
	-\E\biggl\{
	\mathbf{1}_{\{|Y_{i}|\leq\sin\alpha\}}
	\bigl|\langle Y_{i},z\rangle^{2}-|Y_{i}|^{2}\bigr|
	\Bigl\langle Y_{i},
	\nabla H\bigl(
	\{W+D^{\dagger,\{i\}}\}\cos\alpha+z\sin\alpha
	\bigr)
	\Bigr\rangle
	\biggr\}.
\end{align*}
Therefore,
\begin{align*}
	|J_{1}|
	\leq{}&
	\frac{1}{2\varepsilon}
	\int_{0}^{\beta}\frac{\cos^{2}\alpha}{\sin\alpha}
	\int_{\mathbb{R}^{d}}
	\sum_{i=1}^{n}
	\E\biggl\{
	\mathbf{1}_{\{|Y_{i}|\leq\sin\alpha\}}
	\bigl|\langle Y_{i},z\rangle^{2}-|Y_{i}|^{2}\bigr|\\
	&\qquad\times
	\Bigl\langle Y_{i},
	\nabla H\bigl(
	\{W+D^{\dagger,\{i\}}\}\cos\alpha+z\sin\alpha
	\bigr)
	\Bigr\rangle
	\biggr\}\phi_{d}(z)\,dz\,d\alpha.
\end{align*}
Using $T=W+D$, we now write, in each summand,
\begin{align*}
	&\nabla H\bigl(
	\{W+D^{\dagger,\{i\}}\}\cos\alpha+z\sin\alpha
	\bigr)\\
	&\,=
	\nabla H(T\cos\alpha+z\sin\alpha)+
	\Bigl[
	\nabla H\bigl(
	\{W+D^{\dagger,\{i\}}\}\cos\alpha+z\sin\alpha
	\bigr)
	-\nabla H(T\cos\alpha+z\sin\alpha)
	\Bigr].
\end{align*}
For the first term, condition on the full original sample.
Since $X_{i}'$ is independent of that sample and $X_{i}$ is a
function of $\xi_{i}$,
\begin{align*}
	&\E\Bigl\{
	Y_{i}\bigl|\langle Y_{i},z\rangle^{2}-|Y_{i}|^{2}\bigr|
	\mathbf{1}_{\{|Y_{i}|\leq\sin\alpha\}}
	\,\Big|\,\xi_{1},\ldots,\xi_{n}
	\Bigr\}=
	\E\Bigl\{
	Y_{i}\bigl|\langle Y_{i},z\rangle^{2}-|Y_{i}|^{2}\bigr|
	\mathbf{1}_{\{|Y_{i}|\leq\sin\alpha\}}
	\,\Big|\,X_{i}
	\Bigr\}.
\end{align*}
The common gradient $\nabla H(T\cos\alpha+z\sin\alpha)$ is
measurable with respect to the original sample. Thus
\begin{align}\label{eq-vv-02}
	\begin{aligned}
		&\sum_{i=1}^{n}
		\E\Bigl\{
		\mathbf{1}_{\{|Y_{i}|\leq\sin\alpha\}}
		\bigl|\langle Y_{i},z\rangle^{2}-|Y_{i}|^{2}\bigr|
		\langle Y_{i},\nabla H(T\cos\alpha+z\sin\alpha)\rangle
		\Bigr\}\\
		&\,\,\,\,\,\,=
		\E\biggl\{
		\biggl\langle
		\sum_{i=1}^{n}
		\E\Bigl\{
		Y_{i}\bigl|\langle Y_{i},z\rangle^{2}-|Y_{i}|^{2}\bigr|
		\mathbf{1}_{\{|Y_{i}|\leq\sin\alpha\}}
		\,\Big|\,X_{i}
		\Bigr\},\,
		\nabla H(T\cos\alpha+z\sin\alpha)
		\biggr\rangle
		\biggr\}\\
		&\,\,\,\,\,\,\leq
		\E\biggl\{
		\biggl|
		\sum_{i=1}^{n}
		\E\Bigl\{
		Y_{i}\bigl|\langle Y_{i},z\rangle^{2}-|Y_{i}|^{2}\bigr|
		\mathbf{1}_{\{|Y_{i}|\leq\sin\alpha\}}
		\,\Big|\,X_{i}
		\Bigr\}
		\biggr|
		\biggr\},
	\end{aligned}
\end{align}
where the inequality uses $|\nabla H|\leq1$.
Cauchy--Schwarz with respect to the product of the sample probability
measure and $\phi_{d}(z)\,dz$ gives
\begin{align}\label{eq-vv-03}
	\begin{aligned}
		&\int_{\mathbb{R}^{d}}
		\E\Biggl\{
		\biggl|
		\sum_{i=1}^{n}
		\E\Bigl\{
		Y_{i}\bigl|\langle Y_{i},z\rangle^{2}-|Y_{i}|^{2}\bigr|
		\mathbf{1}_{\{|Y_{i}|\leq\sin\alpha\}}
		\,\Big|\,X_{i}
		\Bigr\}
		\biggr|
		\Biggr\}\phi_{d}(z)\,dz\\
		&\quad\leq
		\Biggl[
		\int_{\mathbb{R}^{d}}
		\E\Biggl\{
		\biggl|
		\sum_{i=1}^{n}
		\E\Bigl\{
		Y_{i}\bigl|\langle Y_{i},z\rangle^{2}-|Y_{i}|^{2}\bigr|
		\mathbf{1}_{\{|Y_{i}|\leq\sin\alpha\}}
		\,\Big|\,X_{i}
		\Bigr\}
		\biggr|^{2}
		\Biggr\}\phi_{d}(z)\,dz
		\Biggr]^{1/2}.
	\end{aligned}
\end{align}
The sum is kept inside the norm throughout this estimate.
For the gradient difference, the Lipschitz bound gives
\begin{align}\label{eq-vv-04}
	\begin{aligned}
		&\Bigl|
		\nabla H\bigl(
		\{W+D^{\dagger,\{i\}}\}\cos\alpha+z\sin\alpha
		\bigr)
		-\nabla H(T\cos\alpha+z\sin\alpha)
		\Bigr|\\
		&\qquad\leq
		\frac{1}{\varepsilon}
		\bigl|
		\{W+D^{\dagger,\{i\}}\}\cos\alpha
		-T\cos\alpha
		\bigr|
		=
		\frac{\cos\alpha}{\varepsilon}
		|D-D^{\dagger,\{i\}}|.
	\end{aligned}
\end{align}
Also, for each fixed realization of $Y_{i}$,
\begin{align}\label{eq-vv-05}
	\begin{aligned}
		\int_{\mathbb{R}^{d}}
		\bigl|\langle Y_{i},z\rangle^{2}-|Y_{i}|^{2}\bigr|
		\phi_{d}(z)\,dz&\leq
		\biggl[
		\int_{\mathbb{R}^{d}}
		\bigl\{\langle Y_{i},z\rangle^{2}-|Y_{i}|^{2}\bigr\}^{2}
		\phi_{d}(z)\,dz
		\biggr]^{1/2}\\
		&=
		\bigl\{3|Y_{i}|^{4}-2|Y_{i}|^{4}+|Y_{i}|^{4}\bigr\}^{1/2}
		=\sqrt{2}|Y_{i}|^{2}.
	\end{aligned}
\end{align}
Combining \eqref{eq-vv-02}--\eqref{eq-vv-05}, we have
\begin{align*}
	|J_{1}|
	\leq{}&
	\frac{1}{2\varepsilon}
	\int_{0}^{\beta}\frac{\cos^{2}\alpha}{\sin\alpha}
	\biggl[
	\int_{\mathbb{R}^{d}}
	\E\biggl\{
	\biggl|
	\sum_{i=1}^{n}
	\E\Bigl\{
	Y_{i}\bigl|\langle Y_{i},z\rangle^{2}-|Y_{i}|^{2}\bigr|
	\mathbf{1}_{\{|Y_{i}|\leq\sin\alpha\}}
	\,\Big|\,X_{i}
	\Bigr\}
	\biggr|^{2}
	\biggr\}\phi_{d}(z)\,dz
	\biggr]^{1/2}\,d\alpha\\
	&+
	\frac{\sqrt{2}}{2\varepsilon^{2}}
	\sum_{i=1}^{n}
	\int_{0}^{\beta}\frac{\cos^{3}\alpha}{\sin\alpha}
	\E\Bigl\{
	|D-D^{\dagger,\{i\}}|\,|Y_{i}|^{3}
	\mathbf{1}_{\{|Y_{i}|\leq\sin\alpha\}}
	\Bigr\}\,d\alpha.
\end{align*}
For each fixed $\alpha$ and $z$, the conditional expectations in
the sum are independent, since the $i$th one depends only on
$X_{i}$. They are centered because exchanging $X_{i}$ and $X_{i}'$
changes the sign of the vector factor and preserves both scalar
weights. Thus the cross terms vanish after expanding the squared
norm. Moreover, we have
$
\int_{\mathbb{R}^{d}}
\bigl\{\langle Y_{i},z\rangle^{2}-|Y_{i}|^{2}\bigr\}^{2}
\phi_{d}(z)\,dz
=2|Y_{i}|^{4}.
$
Conditional Jensen's inequality therefore bounds the expression
inside the square brackets by
\begin{align*}
	2\sum_{i=1}^{n}
	\E\Bigl\{
	|Y_{i}|^{6}\mathbf{1}_{\{|Y_{i}|\leq\sin\alpha\}}
	\Bigr\}
	&\leq
	2\sin^{2}\alpha\sum_{i=1}^{n}\E\{|Y_{i}|^{4}\}\leq32\sin^{2}\alpha\,\gamma_{1}^{2}.
\end{align*}
For the second term, use
$|Y_{i}|^{3}\mathbf{1}_{\{|Y_{i}|\leq\sin\alpha\}}
\leq\sin\alpha\,|Y_{i}|^{2}$ and
$
\E\bigl\{|D-D^{\dagger,\{i\}}|\,|Y_{i}|^{2}\bigr\}
\leq
4\|X_{i}\|_{4}^{2}\|D-D^{\{i\}}\|_{2}.
$
Since both $\int_{0}^{\beta}\cos^{2}\alpha\,d\alpha$ and
$\int_{0}^{\beta}\cos^{3}\alpha\,d\alpha$ are at most
$\sin\beta=\varepsilon$, it follows that
\begin{align}\label{eq-t-010}
	|J_{1}|
	\leq C\gamma_{1}+\frac{C\gamma_{3}^{2}}{\varepsilon}.
\end{align}

We next bound $|J_{2}|.$ Integrating by parts in $t$ in the definition of $J_{2}$, and then
using the first identity in~\eqref{eq-ha-06}, we obtain
\begin{align*}
	J_{2}
	={}&
	-\frac{1}{4}\sum_{i=1}^{n}
	\int_{0}^{\beta}
	\int_{0}^{1}
	(1-2t)
	\E\biggl\{
	\mathbf{1}_{\{|Y_{i}|>\sin\alpha\}}
	\biggl\langle
	\nabla^{3}g_{\alpha}
	\bigl(W+D^{\dagger,\{i\}}-tY_{i}\bigr),
	Y_{i}^{\otimes3}
	\biggr\rangle
	\biggr\}
	\,dt\,\tan\alpha\,d\alpha
	\\
	={}&
	\frac{1}{4}\sum_{i=1}^{n}
	\int_{0}^{\beta}
	\cos^{2}\alpha
	\int_{0}^{1}(1-2t)
	\int_{\mathbb{R}^{d}}
	\E\biggl\{
	\mathbf{1}_{\{|Y_{i}|>\sin\alpha\}}
	\\
	&\qquad\qquad\times
	\Bigl\langle
	\nabla^{2}f\bigl(
	\{W+D^{\dagger,\{i\}}-tY_{i}\}\cos\alpha
	+z\sin\alpha
	\bigr),
	Y_{i}^{\otimes2}
	\Bigr\rangle
	\bigl\langle
	\nabla\phi_{d}(z),Y_{i}
	\bigr\rangle
	\biggr\}
	\,dz\,dt\,d\alpha.
\end{align*}
For almost every $x$, we have
$
|\langle\nabla^{2}f(x),Y_{i}^{\otimes2}\rangle|
\leq
8\varepsilon^{-2}|Y_{i}|^{2}
\mathbf{1}_{A^{\varepsilon}\setminus A}(x).
$
For fixed $0<\alpha<\pi/2$, $z\in\mathbb{R}^{d}$ and
$t\in[0,1]$, the Gaussian comparison above, affine invariance
of convexity, and the Gaussian boundary estimate yield
\begin{align*}
	&\P\Bigl\{
	\bigl(W+D^{\dagger,\{i\}}-tY_{i}\bigr)\cos\alpha
	+z\sin\alpha\in A^{\varepsilon}\setminus A
	\,\Big|\,X_{i},X_{i}'
	\Bigr\}\leq
	\frac{Cd^{1/4}\varepsilon}{\cos\alpha}+CK_{n}\delta.
\end{align*}
Moreover, we have
$
\int_{\mathbb{R}^{d}}
|\langle Y_{i},z\rangle|\phi_{d}(z)\,dz
\leq |Y_{i}|.
$
Hence,
\begin{align}\label{eq-t-10.1}
	\begin{aligned}
		|J_{2}|
		\leq{}&
		\frac{C}{\varepsilon^{2}}
		\bigl(d^{1/4}\varepsilon+K_{n}\delta\bigr)
		\sum_{i=1}^{n}
		\E\biggl\{
		|Y_{i}|^{3}
		\int_{0}^{\beta}\cos\alpha\,
		\mathbf{1}_{\{|Y_{i}|>\sin\alpha\}}\,d\alpha
		\biggr\}\\
		\leq{}&
		\frac{C}{\varepsilon^{2}}
		\bigl(d^{1/4}\varepsilon+K_{n}\delta\bigr)
		\sum_{i=1}^{n}\E|Y_{i}|^{4}.
	\end{aligned}
\end{align}

	Finally, we bound $J_{3}.$
Recall that
\[
W+D^{\dagger,\{i\}}-tY_{i}
=
W^{(i)}+D^{\dagger,\{i\}}+(1-t)X_{i}+tX_{i}'.
\]
We first compare $W^{(i)}+D^{\dagger,\{i\}}$ with
$\Sigma_{i}^{1/2}Z$. Fix $X_{i}$, $X_{i}'$, $t\in[0,1]$,
$\alpha\in(\beta,\pi/2)$, and $z\in\mathbb{R}^{d}$.
For every $r\in[0,1]$, the set
\[
\Bigl\{
x\in\mathbb{R}^{d}:
f\bigl(
\{x+(1-t)X_{i}+tX_{i}'\}\cos\alpha
+z\sin\alpha
\bigr)>r
\Bigr\}
\]
is convex, since every superlevel set of $f$ is convex.
Hence, by \eqref{eq-nn-01}, we have
\begin{align*}
	&\Bigl|
	\E\Bigl\{
	f\bigl(
	\{W^{(i)}+D^{\dagger,\{i\}}+(1-t)X_{i}+tX_{i}'\}
	\cos\alpha+z\sin\alpha
	\bigr)
	\,\Big|\,X_{i},X_{i}'
	\Bigr\}
	\\
	&\qquad-
	\E\Bigl\{
	f\bigl(
	\{\Sigma_{i}^{1/2}Z+(1-t)X_{i}+tX_{i}'\}
	\cos\alpha+z\sin\alpha
	\bigr)
	\,\Big|\,X_{i},X_{i}'
	\Bigr\}
	\Bigr|
	\leq CK_{n}\delta.
\end{align*}
Combining this with \eqref{eq-nn-02}, we have
\begin{align*}
	&\Bigl|
	\E\Bigl\{
	\bigl\langle
	\nabla^{4}g_{\alpha}
	\bigl(W+D^{\dagger,\{i\}}-tY_{i}\bigr),
	Y_{i}^{\otimes4}
	\bigr\rangle
	\,\Big|\,X_{i},X_{i}'
	\Bigr\}\\
	&\qquad-
	\E\Bigl\{
	\bigl\langle
	\nabla^{4}g_{\alpha}
	\bigl(\Sigma_{i}^{1/2}Z+(1-t)X_{i}+tX_{i}'\bigr),
	Y_{i}^{\otimes4}
	\bigr\rangle
	\,\Big|\,X_{i},X_{i}'
	\Bigr\}
	\Bigr|
	\\
	&\quad\leq
	CK_{n}\delta\,\cot^{4}\alpha
	\int_{\mathbb{R}^{d}}
	\bigl|
	\bigl\langle
	\nabla^{4}\phi_{d}(z),
	Y_{i}^{\otimes4}
	\bigr\rangle
	\bigr|\,dz
	\\
	&\quad\leq
	CK_{n}\delta\,\cot^{4}\alpha\,|Y_{i}|^{4}.
\end{align*}
It remains to estimate the Gaussian reference term
\[
\E\Bigl\{
\bigl\langle
\nabla^{4}g_{\alpha}
\bigl(\Sigma_{i}^{1/2}Z+(1-t)X_{i}+tX_{i}'\bigr),
Y_{i}^{\otimes4}
\bigr\rangle
\,\Big|\,X_{i},X_{i}'
\Bigr\}.
\]
Take the Gaussian vector in the definition of $g_{\alpha}$
independently of the Gaussian vector appearing above.
The sum of the two Gaussian contributions then has covariance
$
\cos^{2}\alpha\,\Sigma_{i}
+\sin^{2}\alpha\,I_{d}.
$
Consequently, we have
\begin{align*}
	&\E\Bigl\{
	\bigl\langle
	\nabla^{4}g_{\alpha}
	\bigl(\Sigma_{i}^{1/2}Z+(1-t)X_{i}+tX_{i}'\bigr),
	Y_{i}^{\otimes4}
	\bigr\rangle
	\,\Big|\,X_{i},X_{i}'
	\Bigr\}
	\\
	&\quad=
	\cos^{4}\alpha
	\int_{\mathbb{R}^{d}}
	f\Bigl(
	\{(1-t)X_{i}+tX_{i}'\}\cos\alpha
	+
	\{\cos^{2}\alpha\,\Sigma_{i}
	+\sin^{2}\alpha\,I_{d}\}^{1/2}z
	\Bigr)
	\\
	&\qquad\qquad\times
	\Bigl\langle
	\nabla^{4}\phi_{d}(z),
	\Bigl[
	\{\cos^{2}\alpha\,\Sigma_{i}
	+\sin^{2}\alpha\,I_{d}\}^{-1/2}Y_{i}
	\Bigr]^{\otimes4}
	\Bigr\rangle
	\,dz.
\end{align*}
Since  $\gamma_{1}\leq0.01$ and
\begin{align*}
	\cos^{2}\alpha\,\Sigma_{i}
	+\sin^{2}\alpha\,I_{d}
	&\succeq
	\cos^{2}\alpha(1-\gamma_{1})I_{d}
	+\sin^{2}\alpha\,I_{d}
	=
	\{1-\gamma_{1}\cos^{2}\alpha\}I_{d}
	\succeq
	(1-\gamma_{1})I_{d}.
\end{align*}
it follows that
\[
\Bigl\|
\{\cos^{2}\alpha\,\Sigma_{i}
+\sin^{2}\alpha\,I_{d}\}^{-1/2}
\Bigr\|_{\op}
\leq C.
\]
Using $0\leq f\leq1$ and the Gaussian derivative bound once more,
we obtain
\begin{align*}
	&\Bigl|
	\E\Bigl\{
	\bigl\langle
	\nabla^{4}g_{\alpha}
	\bigl(\Sigma_{i}^{1/2}Z+(1-t)X_{i}+tX_{i}'\bigr),
	Y_{i}^{\otimes4}
	\bigr\rangle
	\,\Big|\,X_{i},X_{i}'
	\Bigr\}
	\Bigr|\leq
	C\cos^{4}\alpha\,|Y_{i}|^{4}.
\end{align*}
Combining the preceding two estimates, and using
$t(1-t)\leq1/4$, we have
\begin{align}\label{eq-t-11}
	\begin{aligned}
		|J_{3}|
		\leq{}&
		C\sum_{i=1}^{n}\E\{|Y_{i}|^{4}\}
		\int_{\beta}^{\pi/2}
		\bigl\{
		\cos^{4}\alpha
		+
		K_{n}\delta\,\cot^{4}\alpha
		\bigr\}
		\tan\alpha\,d\alpha
		\\
		={}&
		C\sum_{i=1}^{n}\E\{|Y_{i}|^{4}\}
		\Biggl\{
		\int_{\beta}^{\pi/2}
		\cos^{3}\alpha\sin\alpha\,d\alpha
		+
		K_{n}\delta
		\int_{\beta}^{\pi/2}
		\frac{\cos^{3}\alpha}{\sin^{3}\alpha}\,d\alpha
		\Biggr\}\\
		\leq &
		C\gamma_{1}^{2}
		+
		\frac{C\gamma_{1}^{2}}{\varepsilon^{2}}K_{n}\delta.
	\end{aligned}
\end{align}
Combining \eqref{eq-t-010}, \eqref{eq-t-10.1} and
\eqref{eq-t-11}, and using $\gamma_{1}\leq1$, we obtain
\[
\Biggl|
\int_{0}^{\pi/2}I_{324,\alpha}\tan\alpha\,d\alpha
\Biggr|
\leq
C\gamma_{1}
+\frac{C\gamma_{3}^{2}}{\varepsilon}
+\frac{C\gamma_{1}^{2}}{\varepsilon^{2}}
\bigl(d^{1/4}\varepsilon+K_{n}\delta\bigr).
\]
Finally, combining this estimate with \eqref{eq-t-06},
\eqref{eq-t-08} and \eqref{eq-t-09}, and recalling
$I_{3,\alpha}=I_{31,\alpha}+I_{321,\alpha}
+I_{323,\alpha}+I_{324,\alpha}+I_{325,\alpha}$, gives
\begin{align}\label{eq-t-12}
	\begin{aligned}
		|I_{3}|
		\leq{}&
		C\bigl\{
		d^{1/4}+|\log\varepsilon|(K_{n}\delta)^{1/2}
		\bigr\}\gamma_{1}
		+\frac{C\gamma_{3}^{2}}{\varepsilon}
		+
		\frac{C\gamma_{1}^{2}}{\varepsilon^{2}}
		\bigl(d^{1/4}\varepsilon+K_{n}\delta\bigr).
	\end{aligned}
\end{align}
This proves \eqref{eq-bound-I3}.

\section{Additional details on nonsmooth stochastic approximation}\label{sec:appendix-nonsmooth}
In this section, we provide additional details for the application to
nonsmooth stochastic approximation. We first present a coupled
nonsmooth example and then prove
Lemmas~\ref{lem-A4} and~\ref{lem-ap_nonsmooth-1}.

\subsection{A coupled nonsmooth example}
\label{subsec:coupled-nonsmooth-example}

We now give a simple example showing that
Assumptions~\ref{assu-1}--\ref{assu-5} can hold simultaneously.  The
example is a stochastic forward--backward method with a fusion penalty.
Consider the problem of finding $x_{*}\in\R^{d}$ such that
\begin{align*}
	0\in A(x_{*})+\partial f(x_{*}),
\end{align*}
where $A:\R^{d}\to\R^{d}$ is a single valued smooth operator and
$f:\R^{d}\to(-\infty,+\infty]$ is a proper lower semicontinuous convex
function.
For $\alpha>0$, the proximal map of $f$ is defined by
\begin{align*}
	\operatorname{prox}_{\alpha f}(z)
	:=
	\argmin_{y\in\R^{d}}
	\Big\{
	f(y)+\frac{1}{2\alpha}\norm{y-z}^{2}
	\Big\},
	\qquad z\in\R^{d}.
\end{align*}
The stochastic forward--backward method uses a noisy evaluation
$A(x_{k})+\zeta_{k+1}$ of the smooth operator and updates the iterate by
\begin{align}\label{eq-oo-01}
	x_{k+1}
	=
	\operatorname{prox}_{\alpha_{k+1}f}
	\bigl[
	x_{k}-\alpha_{k+1}
	\{A(x_{k})+\zeta_{k+1}\}
	\bigr].
\end{align}

The deterministic forward--backward splitting framework can be traced
back at least to \cite{LionsMercier1979}; see
\cite{CombettesWajs2005} for an influential treatment of proximal
forward--backward splitting in convex optimization. Stochastic
forward--backward methods for monotone inclusion problems were studied
in \cite{RosascoVillaVu2016,CombettesPesquet2016}.

For $m\geq1$, let $d=2m$, and write $x=(u,v)\in\R^{m}\times\R^{m}$.  Fix
constants
$
\lambda>0,$
$\kappa>0$ and
$
\rho>0.
$
Consider the smooth quadratic function
$
h(u,v):=(\norm{u}^{2}+\norm{v}^{2})/2,
$
and its gradient
$
A(u,v):=\na h(u,v)=(u,v).
$
Define
$
f(u,v)
:=\lambda\norm{u-v}
+\kappa(\norm{u}^{2}+\norm{v}^{2})/2,
$
and set
\begin{align*}
	F(x):=A(x)+\partial f(x).
\end{align*}
The term $\norm{u-v}$ penalizes the difference between the two blocks
and vanishes when $u=v$. This naturally leads to the diagonal
manifold
\begin{align*}
	\MM:=\{(z,z):z\in\R^{m}\}.
\end{align*}
Since
$
h(x)+f(x)
\geq (1+\kappa)\norm{x}^{2}/2,
$
the  problem $0\in F(x)$ has the unique solution
$
x_{*}=0.
$
The orthogonal projections onto $\MM$ and $\MM^\perp$ are denoted by
$P$ and $Q:=I_{d}-P$, respectively.  More explicitly, for $x=(u,v)$,
\begin{align*}
	Px
	&=\bigl(\frac{u+v}{2},\frac{u+v}{2}\bigr),\quad
	Qx
	=\bigl(\frac{u-v}{2},\frac{v-u}{2}\bigr),
	\qquad
	\dist(x,\MM)=\norm{Qx}=\frac{\norm{u-v}}{\sqrt{2}}.
\end{align*}
Because $\MM$ is a linear subspace, $T_{\MM}(y)=\MM$ and
$P_{T_{\MM}(y)}=P$ for every $y\in\MM$.
For every $\alpha>0$, define the stochastic forward--backward step by
\begin{align*}
	x_{\alpha}^+(x,\zeta)
	&:=\operatorname{prox}_{\alpha f}
	\{x-\alpha(A(x)+\zeta)\},\qquad
	G_{\alpha}(x,\zeta)
	:=\alpha^{-1}\{x-x_{\alpha}^+(x,\zeta)\}.
\end{align*}
Combining the above definition of $G$ with \eqref{eq-oo-01}, we have
\begin{align*}
	x_{k+1}=x_{\alpha_{k+1}}^+( x_{k},\zeta_{k+1})=x_{k}-\alpha_{k+1}G_{\alpha_{k+1}}(x_{k},\zeta_{k+1}),
\end{align*}
which is exactly the iterative update in \eqref{eq-aa-02}.
For $t\geq0$, let
\begin{align*}
	\mathsf S_{t}(w)
	:=
	\begin{cases}
		\bigl(1-t/\norm{w}\bigr)_{+}w, & w\neq0,\\
		0, & w=0,
	\end{cases}
\end{align*}
where $a_{+}:=\max\{a,0\}$.  The orthogonal decomposition
$\R^{d}=\MM\oplus\MM^\perp$ and
$
f(x)=\kappa\norm{x}^{2}/2+\sqrt{2}\lambda\norm{Qx}
$
show that, for every $z\in\R^{d}$,
\begin{align*}
	\operatorname{prox}_{\alpha f}(z)
	=\frac{1}{1+\kappa\alpha}
	\bigl\{Pz+\mathsf S_{\sqrt{2}\lambda\alpha}(Qz)\bigr\}.
\end{align*}
Since $\mathsf S_{t}$ is nonexpansive for every $t\geq0$, the above representation and the orthogonality of $\MM$ and $\MM^\perp$ imply that
\begin{align}\label{eq-oo-02}
	\norm{
		\operatorname{prox}_{\alpha f}(z)
		-\operatorname{prox}_{\alpha f}(z')
	}
	\leq \frac{1}{1+\kappa\alpha}\norm{z-z'}
	\leq \norm{z-z'}.
\end{align}

We next specify the noise.  Let $(\xi_{k})_{k\geq1}$ be independent and
identically distributed $\R^{d}$-valued random vectors, with common law
$\mu$, and suppose that they are independent of $x_{0}$.    We assume that
\begin{align}\label{eq-oo-03}
	\E\xi_{1}&=0,
	\qquad
	\E(\xi_{1}\xi_{1}^\top)=\Sigma,
	\qquad
	\lambda_{3} I_{d}\preceq\Sigma\preceq\lambda_{4} I_{d},
	\quad
	\|\xi_{1}\|_{12}\leq\tau,
	\qquad
	\|x_{0}\|_{12}\leq\tau_{0},
\end{align}
for some constants $0<\lambda_{3}\leq\lambda_{4}$ and
$\tau,\tau_{0}>0$.
Let $\theta:\R^{d}\to[-1,1]$ be a bounded Borel function such that
\begin{align}\label{eq-oo-04}
	\E\theta(\xi_{1})=0,
	\qquad
	\E\{\theta(\xi_{1})^{2}\}>0.
\end{align}
To introduce a nonzero noise component that depends on the state, let
$
g(x,\xi):=\rho\theta(\xi)x,
$
and suppose that the noise sequence takes the form
\begin{align}\label{eq-oo-05}
	\zeta_{k}
	=\xi_{k}+g(x_{k-1},\xi_{k})
	=\xi_{k}+\rho\theta(\xi_{k})x_{k-1}.
\end{align}
The map for one update is
\begin{align*}
	\Phi_{\alpha}(x,\xi)
	:=x_{\alpha}^{+}\bigl(x,\xi+g(x,\xi)\bigr).
\end{align*}

We verify below that this model satisfies all five assumptions.\\

{\it (i) Verification of Assumption~\ref{assu-1}.}
For $y=(z,z)\in\MM$,
$
\partial f(y)
=\kappa y
+\lambda\{(s,-s):s\in\R^{m},\ \norm{s}\leq1\}.
$
The second term on the right-hand side belongs to $\MM^\perp$ and is
therefore removed by $P$.  It follows that
\begin{align}\label{eq-oo-06}
	F_{\MM}(y)=P F(y)=(1+\kappa)y,
	\qquad y\in\MM.
\end{align}
Hence $F_{\MM}$ is single valued and $C^\infty$ on $\MM$, and
$
\na_{\MM}F_{\MM}(0)w=(1+\kappa)w, w\in T_{\MM}(0).
$
Moreover, the extension specified in Assumption~\ref{assu-1} is
\begin{align*}
	\widehat{F}_{\MM}(x)
	=F_{\MM}(Px)=(1+\kappa)Px,
	\qquad x\in\R^{d}.
\end{align*}
This map is linear, so its second derivative is zero.
Consequently, Assumption~\ref{assu-1} holds with
$L_{\MM}=\lambda_{1}=\lambda_{2}=1+\kappa$ and, for example, $M_{2}=1$.
\\

{\it (ii) Verification of Assumption~\ref{assu-2}.}
Fix $x,\zeta\in\R^{d}$ and write
$
p:=x_{\alpha}^+(x,\zeta)=(p_{u},p_{v}).
$
The optimality condition for the proximal step implies that there exists
$s\in\partial\norm{\cdot}(p_{u}-p_{v})$ such that
\begin{align}\label{eq-oo-07}
	G_{\alpha}(x,\zeta)=\frac{x-p}{\alpha}
	=x+\zeta+\kappa p+\lambda(s,-s),
	\qquad
	\norm{s}\leq1.
\end{align}
Since $\operatorname{prox}_{\alpha f}(0)=0$,
\eqref{eq-oo-02} implies that
\begin{align}\label{eq-oo-08}
	\norm{p}
	&\leq
	\frac{|1-\alpha|\norm{x}+\alpha\norm{\zeta}}
	{1+\kappa\alpha}
	\leq
	C_{\kappa}\bigl(\norm{x}+\norm{\zeta}\bigr),
	\qquad \alpha>0,
\end{align}
where $C_{\kappa}:=\max\{1,\kappa^{-1}\}$.
Combining \eqref{eq-oo-07} and
\eqref{eq-oo-08}, we obtain
\begin{align}\label{eq-oo-09}
	\norm{G_{\alpha}(x,\zeta)}
	&\leq C\{1+\norm{x}+\norm{\zeta}\},\qquad
	\norm{p-x}
	\leq C\alpha\{1+\norm{x}+\norm{\zeta}\},
\end{align}
where here and below $C$ denotes a finite constant depending only on
$\lambda$, $\kappa$, $\rho$.
Fix $\delta>0$ and take $\mathcal U=B_{\delta}(0)$.  Since
$\operatorname{dom}F=\R^{d}$, we have $\mathcal U_{F}=\mathcal U$.  By
\eqref{eq-oo-06}, \eqref{eq-oo-07}, and
$P(s,-s)=0$,
\begin{align*}
	\mathcal R_{\alpha}(x,\zeta)
	&=P\{G_{\alpha}(x,\zeta)-F_{\MM}(Px)-\zeta\}
	=\kappa P(p-x).
\end{align*}
Taking the tangent component in the proximal update gives
$
(1+\kappa\alpha)Pp
=(1-\alpha)Px-\alpha P\zeta.
$
Therefore,
\begin{align*}
	\mathcal R_{\alpha}(x,\zeta)
	=-\frac{\kappa\alpha}{1+\kappa\alpha}
	\{(1+\kappa)Px+P\zeta\}.
\end{align*}
For $x,x'\in\mathcal U$ and $\zeta,\zeta'\in\R^{d}$, this identity yields
\begin{align*}
	\norm{\mathcal R_{\alpha}(x,\zeta)}
	&\leq
	C(1+|x-x_{*}|^{2}+\norm{\zeta}^{2})
	\{\dist(x,\MM)+\alpha\},
	\nonumber\\
	\norm{
		\mathcal R_{\alpha}(x,\zeta)
		-\mathcal R_{\alpha}(x',\zeta')
	}
	&\leq
	C(1+\norm{\zeta}^{2}+\norm{\zeta'}^{2})
	\{\alpha+\norm{x}+\norm{x'}\}
	\{\norm{x-x'}+\norm{\zeta-\zeta'}\}.
\end{align*}
These are conditions (i) and (ii) of Assumption~\ref{assu-2}.

It remains to verify the normal aiming condition.  Put
$
q:=Qx=x-Px
$
and write $r:=p_{u}-p_{v}$.  Since
$s\in\partial\norm{\cdot}(r)$, we have
$
\la s,r\ra=\norm{r},$
$
\norm{s}\leq1.
$
Consequently, we have
\begin{align}\label{eq-oo-10}
	\begin{aligned}
		\la s,u-v\ra
		&=
		\norm{r}+\la s,(u-v)-r\ra
		\geq
		\norm{r}-\norm{(u-v)-r}
		\\
		&\geq
		\norm{u-v}-2\norm{(u-v)-r}
		=
		\norm{u-v}-2\sqrt{2}\norm{Q(x-p)}.
	\end{aligned}
\end{align}
Note that
$
\la p,q\ra
=
\la Qp,q\ra
=
\norm{q}^{2}-\la Q(x-p),q\ra
\geq
\norm{q}^{2}-\norm{q}\,\norm{Q(x-p)},
$
and
$
\la(s,-s),q\ra
=
\la s,u-v\ra.
$
Combining these estimates with \eqref{eq-oo-07} and
\eqref{eq-oo-10}, we obtain
\begin{align*}
	\la G_{\alpha}(x,\zeta)-\zeta,q\ra
	&=
	\norm{q}^{2}+\kappa\la p,q\ra
	+\lambda\la(s,-s),q\ra
	\nonumber\\
	&\geq
	(1+\kappa)\norm{q}^{2}
	+\sqrt{2}\lambda\norm{q}
	-
	\bigl(\kappa\norm{q}+2\sqrt{2}\lambda\bigr)
	\norm{Q(x-p)}.
\end{align*}
For $\norm{Q(x-p)}$,  applying $Q$ to
\eqref{eq-oo-07} and using
$
Qp=q-Q(x-p)
,
Q(s,-s)=(s,-s),
$
we have
\begin{align*}
	\frac{Q(x-p)}{\alpha}
	&=
	q+Q\zeta+\kappa Qp+\lambda(s,-s)=
	(1+\kappa)q+Q\zeta
	-\kappa Q(x-p)+\lambda(s,-s).
\end{align*}
Therefore,
\begin{align}
	\norm{Q(x-p)}
	&\leq
	\frac{\alpha}{1+\kappa\alpha}
	\bigl\{
	(1+\kappa)\norm{q}
	+\norm{\zeta}+\sqrt{2}\lambda
	\bigr\}
	\leq
	C\alpha
	\bigl\{
	1+\norm{q}+\norm{\zeta}
	\bigr\}.
	\label{eq-oo-12}
\end{align}
By the first inequality in \eqref{eq-oo-12}, we have
\begin{align*}
	&\bigl(\kappa\norm{q}+2\sqrt{2}\lambda\bigr)
	\norm{Q(x-p)}
\leq
	C\frac{\alpha}{1+\kappa\alpha}
	\bigl(1+\norm{q}^{2}+\norm{\zeta}^{2}\bigr)
	\leq
	C\norm{q}^{2}+C\alpha\bigl(1+\norm{\zeta}^{2}\bigr),
\end{align*}
where the last inequality uses
$\alpha/(1+\kappa\alpha)\leq\kappa^{-1}$ and
$\alpha/(1+\kappa\alpha)\leq\alpha$.
It follows that
\begin{align*}
	\la G_{\alpha}(x,\zeta)-\zeta,q\ra
	&\geq
	(1+\kappa)\norm{q}^{2}
	+\sqrt{2}\lambda\norm{q}
	-C\norm{q}^{2}
	-C\alpha\bigl(1+\norm{\zeta}^{2}\bigr)
	\\
	&\geq
	\sqrt{2}\lambda\norm{q}
	-C\bigl\{\norm{q}^{2}
	+\alpha\bigl(1+\norm{\zeta}^{2}\bigr)\bigr\}.
\end{align*}
Since $\norm{q}=\dist(x,\MM)$, condition (iii) of
Assumption~\ref{assu-2} holds for every $\alpha>0$, with
$c_{3}=\sqrt{2}\lambda$ and a suitable constant $c_{4}>0$.
\\

{\it (iii) Verification of Assumption~\ref{assu-3}.}
Let
$
w:=\kappa p+\lambda(s,-s)\in\partial f(p).
$
Then $G_{\alpha}(x,\zeta)-\zeta=x+w$, and
\begin{align*}
	\la w,p\ra
	=\kappa\norm{p}^{2}+\lambda\norm{p_{u}-p_{v}}
	\geq0.
\end{align*}
Consequently, we have
\begin{align*}
	\la G_{\alpha}(x,\zeta)-\zeta,x\ra
	&=\norm{x}^{2}+\la w,p\ra+\la w,x-p\ra
	\geq
	\norm{x}^{2}-\norm{w}\norm{x-p}	\\
	&\geq
	\norm{x}^{2}
	-C\alpha\{1+\norm{x}^{2}+\norm{\zeta}^{2}\},
\end{align*}
where the last inequality follows from
\eqref{eq-oo-08}--\eqref{eq-oo-09}.  The first
bound in \eqref{eq-oo-09} also gives
\begin{align*}
	\norm{G_{\alpha}(x,\zeta)}^{2}
	\leq C\{1+\norm{x}^{2}+\norm{\zeta}^{2}\}.
\end{align*}
Thus Assumption~\ref{assu-3} holds globally, with $c_{5}=1$ and suitable
finite constants $c_{6}$ and $c_{7}$.\\

{\it (iv) Verification of Assumption~\ref{assu-4}.}
Since $\operatorname{dom}F=\R^{d}$, the iterates belong to
$\operatorname{dom}F$ for every $k\geq0$.
Together with the bound on $\|x_{0}\|_{12}$ in \eqref{eq-oo-03},
this verifies Assumption~\ref{assu-4}(i).
In this example, $\mu_{k}=\mu$ for every $k\geq1$.
The innovation conditions in Assumption~\ref{assu-4}(ii) follow directly
from \eqref{eq-oo-03}, with
$\Sigma_{k}=\Sigma$ for every $k\geq1$.  Moreover, because $x_{k-1}$ is
$\mathcal F_{k-1}$-measurable and $\xi_{k}$ is independent of
$\mathcal F_{k-1}$, \eqref{eq-oo-04} gives
\begin{align*}
	\E\{g(x_{k-1},\xi_{k})\mid\mathcal F_{k-1}\}
	=\rho x_{k-1}\E\theta(\xi_{k})=0.
\end{align*}
For every $x,x',\xi\in\R^{d}$, we also have
\begin{align*}
	\norm{g(x,\xi)-g(x',\xi)}
	\leq\rho\norm{x-x'},
	\qquad
	g(0,\xi)=0.
\end{align*}
Thus Assumption~\ref{assu-4} holds with $c_{8}=\rho$.\\

{\it (v) Verification of Assumption~\ref{assu-5}.}
Set
$
	b(\xi):=1+\rho\theta(\xi).
$
By \eqref{eq-oo-04}, we have
\begin{align*}
	\E b(\xi_{1})=1,
	\qquad
	|b(\xi)|\leq1+\rho.
\end{align*}
Using $A(x)=x$ and \eqref{eq-oo-05}, we obtain
\begin{align*}
	\Phi_{\alpha}(x,\xi)
	=\operatorname{prox}_{\alpha f}
	\bigl(\{1-\alpha b(\xi)\}x-\alpha\xi\bigr).
\end{align*}
Take $\bar{\alpha}=1$.
Combining this identity with \eqref{eq-oo-02}, and expanding the sixth
power, gives, for $0<\alpha\leq\bar{\alpha}$,
\begin{align*}
	\int_{\R^{d}}
	\norm{\Phi_{\alpha}(x,\xi)-\Phi_{\alpha}(y,\xi)}^{6}
	\,\mu(d\xi)
	&\leq \E|1-\alpha b(\xi_{1})|^{6}\norm{x-y}^{6}\\
	&\leq \bigl(1-c_{9}\alpha+c_{10}\alpha^{2}\bigr)\norm{x-y}^{6},
\end{align*}
where one may take $c_{9}=6$ and
$
	c_{10}:=\sum_{j=2}^{6}\binom{6}{j}(1+\rho)^{j}.
$
Indeed, the linear term in the binomial expansion is $-6\alpha$
because $\E b(\xi_{1})=1$, while the sum of all terms of order at
least two is bounded above by $c_{10}\alpha^{2}$ on
$0<\alpha\leq1$.
This proves condition~\eqref{eq:coupled-L6}.

It remains to verify condition~\eqref{eq:coupled-residual}.
The optimality condition for the proximal map gives
\begin{align*}
	\frac{
		\{1-\alpha b(\xi)\}x-\alpha\xi-\Phi_{\alpha}(x,\xi)
	}{\alpha}
	&\in
	\partial f\bigl(\Phi_{\alpha}(x,\xi)\bigr),\\
	\frac{
		\{1-\alpha b(\xi)\}y-\alpha\xi-\Phi_{\alpha}(y,\xi)
	}{\alpha}
	&\in
	\partial f\bigl(\Phi_{\alpha}(y,\xi)\bigr).
\end{align*}
By the definition of the subdifferential,
\begin{align*}
	f\bigl(\Phi_{\alpha}(y,\xi)\bigr)
	&\geq
	f\bigl(\Phi_{\alpha}(x,\xi)\bigr)
	+
	\frac{1}{\alpha}
	\bigl\langle
	\{1-\alpha b(\xi)\}x-\alpha\xi-\Phi_{\alpha}(x,\xi),
	\Phi_{\alpha}(y,\xi)-\Phi_{\alpha}(x,\xi)
	\bigr\rangle,
\end{align*}
and
\begin{align*}
	f\bigl(\Phi_{\alpha}(x,\xi)\bigr)
	&\geq
	f\bigl(\Phi_{\alpha}(y,\xi)\bigr)
	+
	\frac{1}{\alpha}
	\bigl\langle
	\{1-\alpha b(\xi)\}y-\alpha\xi-\Phi_{\alpha}(y,\xi),
	\Phi_{\alpha}(x,\xi)-\Phi_{\alpha}(y,\xi)
	\bigr\rangle.
\end{align*}
Adding these two inequalities yields
\begin{align*}
	\bigl\langle
	&\{1-\alpha b(\xi)\}(x-y)
	-\{\Phi_{\alpha}(x,\xi)-\Phi_{\alpha}(y,\xi)\}, \Phi_{\alpha}(x,\xi)-\Phi_{\alpha}(y,\xi)
	\bigr\rangle
	\geq0.
\end{align*}
Consequently, we have
\begin{align*}
	&\norm{\Phi_{\alpha}(x,\xi)-\Phi_{\alpha}(y,\xi)}^{2}+
	\norm{
		\{1-\alpha b(\xi)\}(x-y)
		-\{\Phi_{\alpha}(x,\xi)-\Phi_{\alpha}(y,\xi)\}
	}^{2}\\
	&\leq
	\bigl(1-\alpha b(\xi)\bigr)^{2}\norm{x-y}^{2}.
\end{align*}
Since
$
	x-\Phi_{\alpha}(x,\xi)
	=
	\alpha\widehat G_{\alpha}(x,\xi),
$
we have
\begin{align*}
	&\{1-\alpha b(\xi)\}(x-y)
	-\{\Phi_{\alpha}(x,\xi)-\Phi_{\alpha}(y,\xi)\}=
	\alpha\bigl\{
	\widehat G_{\alpha}(x,\xi)
	-\widehat G_{\alpha}(y,\xi)
	-b(\xi)(x-y)
	\bigr\}.
\end{align*}
Hence,
\begin{align*}
	&\norm{\Phi_{\alpha}(x,\xi)-\Phi_{\alpha}(y,\xi)}^{2}+
	\alpha^{2}
	\norm{
		\widehat G_{\alpha}(x,\xi)
		-\widehat G_{\alpha}(y,\xi)
		-b(\xi)(x-y)
	}^{2}\\
	&\leq
	\bigl(1-\alpha b(\xi)\bigr)^{2}\norm{x-y}^{2}.
\end{align*}
Using
$
\norm{u+v}^{2}/2
	\leq
	\norm{u}^{2}+\norm{v}^{2},
$
we further obtain
\begin{align*}
	\frac{\alpha^{2}}{2}
	\norm{
		\widehat G_{\alpha}(x,\xi)
		-\widehat G_{\alpha}(y,\xi)
	}^{2}&\leq
	\alpha^{2}
	\norm{
		\widehat G_{\alpha}(x,\xi)
		-\widehat G_{\alpha}(y,\xi)
		-b(\xi)(x-y)
	}^{2}\\
	&\quad+
	\alpha^{2}b(\xi)^{2}\norm{x-y}^{2}.
\end{align*}
Combining the preceding two inequalities gives
\begin{align*}
	&\norm{\Phi_{\alpha}(x,\xi)-\Phi_{\alpha}(y,\xi)}^{2}
	+\frac{\alpha^{2}}{2}
	\norm{
		\widehat G_{\alpha}(x,\xi)
		-\widehat G_{\alpha}(y,\xi)
	}^{2}\\
	&\quad\leq
	\bigl\{
	\bigl(1-\alpha b(\xi)\bigr)^{2}
	+\alpha^{2}b(\xi)^{2}
	\bigr\}\norm{x-y}^{2}\\
	&\quad=
	\bigl\{
	1-2\alpha b(\xi)
	+2\alpha^{2}b(\xi)^{2}
	\bigr\}\norm{x-y}^{2}.
\end{align*}
Since
$
	\E\{b(\xi_{1})\}=1,$
	$\E\{b(\xi_{1})^{2}\}
	=
	1+\rho^{2}\E\{\theta(\xi_{1})^{2}\},
$
integration with respect to $\mu$ gives
\begin{align*}
	\int_{\R^{d}}\Big[
	&\norm{\Phi_{\alpha}(x,\xi)-\Phi_{\alpha}(y,\xi)}^{2}
	+\frac{\alpha^{2}}{2}
	\norm{
		\widehat G_{\alpha}(x,\xi)
		-\widehat G_{\alpha}(y,\xi)
	}^{2}
	\Big]\,\mu(d\xi)\\
	&\leq
	\Big[
	1-2\alpha
	+2\bigl\{
	1+\rho^{2}\E\{\theta(\xi_{1})^{2}\}
	\bigr\}\alpha^{2}
	\Big]\norm{x-y}^{2}.
\end{align*}
Thus condition~\eqref{eq:coupled-residual} holds with
\begin{align*}
	c_{11}
	&=
	\frac{1}{2},
	\qquad
	c_{12}
	=
	2,
	\qquad
	c_{13}
	=
	2\bigl\{
	1+\rho^{2}\E\{\theta(\xi_{1})^{2}\}
	\bigr\}.
\end{align*}
Together with the preceding verification of
\eqref{eq:coupled-L6}, this proves that
Assumption~\ref{assu-5} holds with $\bar{\alpha}=1$.

\subsection{Proof of Lemma \ref{lem-A4} }\label{sec:apdix-proof-nonsmooth}
We first prove \eqref{eq-ka-01}. By Assumption \ref{assu-4}, we have for any $x\in \R^{m}$,  $x^{\top}Q_iU^{\top}\Sigma_{i}UQ_{i}^{\top}x\geq \lambda_{3}|UQ_{i}^{\top}x|^{2}\geq 0$. Therefore,
\begin{align}\label{eq-ka-03}
	\begin{aligned}
		\frac{1}{n}\sum_{i=1}^{n-1}x^{\top}Q_iU^{\top}\Sigma_{i}UQ_{i}^{\top}x
		&\geq \frac{1}{n}\sum_{i=\lfloor n/4\rfloor}^{\lfloor n/2\rfloor}x^{\top}
		Q_iU^{\top}\Sigma_{i}UQ_{i}^{\top}x\\
		&\geq \frac{\lambda_{3}}{n}\sum_{i=\lfloor n/4\rfloor}^{\lfloor n/2\rfloor}
		x^{\top}Q_iU^{\top}UQ_{i}^{\top}x=\frac{\lambda_{3}}{n}
		\sum_{i=\lfloor n/4\rfloor}^{\lfloor n/2\rfloor} x^{\top}Q_iQ_{i}^{\top}x\\
		&\geq \frac{\lambda_{3}}{n}\sum_{i=\lfloor n/4\rfloor}^{\lfloor n/2\rfloor}
		\sigma_{\min}^{2}(Q_{i})|x|^{2},
	\end{aligned}
\end{align}
where \(\sigma_{\min}(Q_i)\) denotes the smallest singular value of \(Q_i\), and
the equality in the second line follows from the fact that
$U^\top U=I_m$.
We claim that, for all sufficiently large $n$,
\begin{align}\label{eq-ka-04}
	\sigma_{\min}(Q_i)\geq \frac{1}{2\lambda_{2}},
	\quad \text{for all } \lfloor n/4\rfloor\leq i\leq \lfloor n/2\rfloor .
\end{align}
Combining this claim and \eqref{eq-ka-03}, we obtain \eqref{eq-ka-01}. In what follows, we prove \eqref{eq-ka-04}.
For convenience, we denote  $P_{i,j}:=\prod_{k=i+1}^{j}(I_m-\alpha_k H)$, where an empty product is understood to be \(I_m\), so that
\(P_{i,i}=I_m\). Then $  P_{i,j+1}=P_{i,j}(I_{m}-\alpha_{j+1} H)$, and hence
\begin{align*}
	P_{i,j}-P_{i,j+1}=\alpha_{j+1}P_{i,j}H,\quad \text{for any }j \geq i.
\end{align*}
Summing the above equality from $j=i$ to $j=n-2$ yields
\[
I_{m}-P_{i,n-1}=\sum_{j=i}^{n-2}\alpha_{j+1}P_{i,j}H.
\]
Using the above equality, we can rewrite $Q_i$ as follows:
\begin{align*}
	Q_{i}H&=\alpha_{i}\sum_{j=i}^{n-1} P_{i, j} H=\sum_{j=i}^{n-2}(\alpha_{i}-\alpha_{j+1}) P_{i, j} H+\sum_{j=i}^{n-2}\alpha_{j+1} P_{i, j}H+\alpha_{i}P_{i,n-1}H\nonumber\\
	&:=I_{m}-P_{i,n-1}+R_{i},
\end{align*}
where $R_{i}=\sum_{j=i}^{n-2}(\alpha_{i}-\alpha_{j+1}) P_{i, j} H+\alpha_{i}P_{i,n-1}H$. Hence,
\begin{align}\label{eq-ka-05}
	Q_{i}=(I_{m}-P_{i,n-1}+R_{i})H^{-1}.
\end{align}
Note that $\|H\|_{\op}\leq \|U^{\top}\|_{\op}\|\nabla_{\MM} F_{\MM}(x_{*})\|_{\op}\|U\|_{\op}\leq \lambda_{2}$. By \eqref{eq-ka-05}, the standard singular-value inequalities
$\sigma_{\min}(AB)\geq \sigma_{\min}(A)\sigma_{\min}(B)$ and
$\sigma_{\min}(A+B)\geq \sigma_{\min}(A)-\sigma_{\max}(B)$,
together with $\sigma_{\min}(H^{-1})=\sigma_{\max}(H)^{-1}\geq \lambda_2^{-1}$,
imply that
\begin{align}\label{eq-ka-06}
	\sigma_{\min}(Q_i)
	&\geq
	\sigma_{\min}(I_m-P_{i,n-1}+R_i)\sigma_{\min}(H^{-1}) \nonumber\\
	&\geq
	\lambda_2^{-1}\{1-\|R_i-P_{i,n-1}\|_{\op}\}.
\end{align}
Now, we provide an upper bound for $\|R_{i}-P_{i,n-1}\|_{\op}$. By the inequality $\|A+B\|_{\op}\leq \|A\|_{\op}+\|B\|_{\op}$, we have
\[
\|R_{i}-P_{i,n-1}\|_{\op}\leq \|P_{i, n-1}\|_{\op}+\|R_{i}\|_{\op}.
\]
For $\|P_{i, n-1}\|_{\op}$, by \eqref{eq-aa-04}, we have
\begin{align*}
	\left|\left(I_m-\alpha_k H\right) x\right|^2&=|x|^2-2 \alpha_k\langle H x, x\rangle+\alpha_k^2|H x|^2\nonumber\\
	&\leq (1-2\lambda_{1}\alpha_{k}+\lambda_{2}^{2}\alpha_{k}^{2})|x|^2\nonumber\\
	&\leq (1-2\lambda_{1}\alpha_{k}+(\lambda_{1}^{2}+\lambda_{2}^{2})\alpha_{k}^{2})|x|^2.
\end{align*}
Since the coefficient of $|x|^2$ is positive, it follows that
\begin{align}\label{eq-ka-07}
	\|I_m-\alpha_k H\|_{\op}\leq (1-2\lambda_{1}\alpha_{k}+(\lambda_{1}^{2}+\lambda^{2}_{2})\alpha_{k}^{2})^{1/2}.
\end{align}
Combining this with  the inequalities $\|AB\|_{\op}\leq \|A\|_{\op}\|B\|_{\op}$ and $1+x\leq e^{x} $ for any $x\geq -1$, we have for any $r\geq 1$,
\begin{align}\label{eq-ka-08}
	\begin{aligned}
		\|P_{i,i+r}\|_{\op}&\leq \prod_{k=i+1}^{i+r}\|I_m-\alpha_k H\|_{\op}\leq \prod_{k=i+1}^{i+r}
		\exp\Big(-\lambda_{1}\alpha_{k}+(\lambda_{1}^{2}+\lambda^{2}_{2})\alpha_{k}^{2}/2  \Big)
		\\
		&\leq \exp\Big(\sum_{k=i+1}^{i+r}-\lambda_{1}\alpha_{k}+(\lambda_{1}^{2}+
		\lambda^{2}_{2})\alpha_{k}^{2}/2  \Big)\\
		&\leq C\exp\Big(\sum_{k=i+1}^{i+r}-\lambda_{1}\alpha_{k}\Big)\leq C\exp\Big(-\lambda_{1}r\alpha_{0}(i+r)^{-\beta}\Big).
	\end{aligned}
\end{align}
Consequently, we have for any $\lfloor n/4\rfloor\leq i\leq \lfloor n/2\rfloor$,
\begin{align*}
	\|P_{i,n-1}\|_{\op}\leq C\exp\Big(-\lambda_{1}(n-1-i)\alpha_{0}(n-1)^{-\beta}\Big)\leq C\exp\big(-\lambda_{1}\alpha_{0}n^{1-\beta}/4\big)
\end{align*}
for any $n\geq 4.$
Hence, for all sufficiently large $n$, we have
\begin{align}\label{eq-ka-09}
	\|P_{i, n-1}\|_{\op}\leq 1/4.
\end{align}
As for $\|R_{i}\|_{\op}$, by the definition of $R_{i}$, we have
\begin{align*}
	\|R_{i}\|_{\op}&\leq \lambda_{2}\sum_{j=i}^{n-2}\left(\alpha_i-\alpha_{j+1}\right) \|P_{i, j}\|_{\op}+\lambda_{2}\alpha_i\|P_{i,n-1}\|_{\op}\nonumber\\
	&=\lambda_{2}\sum_{r=0}^{n-2-i}\left(\alpha_i-\alpha_{i+r+1}\right) \|P_{i, i+r}\|_{\op}+\lambda_{2}\alpha_i\|P_{i,n-1}\|_{\op}.
\end{align*}
For $0\leq r\leq i$, we have
\begin{align*}
	\alpha_i-\alpha_{i+r+1} =
	\alpha_0\big({i^{-\beta}-(i+r+1)^{-\beta}}\big) \leq
	\alpha_0\beta(r+1)i^{-\beta-1}
	\leq
	\frac{(r+1)\alpha_i}{i}.
\end{align*}
Combining this with \eqref{eq-ka-08}, we have
\begin{align}\label{eq-ka-10}
	\begin{aligned}
		\sum_{r=0}^i\left|\alpha_i-\alpha_{i+r+1}\right|\left\|P_{i, i+r}\right\|_{ \mathrm{op}}
		& \leq C \frac{\alpha_i}{i} \sum_{r=0}^i(r+1)
		\exp\Big(-\lambda_{1}r\alpha_{0}(i+r)^{-\beta}\Big) \\
		&\leq C \frac{\alpha_i}{i} \sum_{r=0}^{i}(r+1) \exp\Big(-\lambda_{1}2^{-\beta}r\alpha_{0}
		i^{-\beta}\Big)\\
		&\leq C \frac{\alpha_i}{i}\cdot \frac{1}{[1-\exp(-\lambda_{1}2^{-\beta}\alpha_{i})]^{2}},
	\end{aligned}
\end{align}
where the last inequality follows from $\sum_{r=0}^{\infty} (r+1)q^{r}=1/(1-q)^{2}, \forall |q|<1$. Note that $\lambda_{1}2^{-\beta}\alpha_{i}\leq \lambda_{1}\alpha_{0}$. Using the inequality $1-e^{-x}\geq e^{-a}x$ for any $0\leq x\leq a$, we have
$
[1-\exp(-\lambda_{1}2^{-\beta}\alpha_{i})]^{2}\geq C\alpha_{i}^{2}.
$
Combining this with \eqref{eq-ka-10}, we have for any $\lfloor n/4\rfloor\leq i\leq \lfloor n/2\rfloor$
\begin{align*}
	\sum_{r=0}^i\left|\alpha_i-\alpha_{i+r+1}\right|\left\|P_{i, i+r}\right\|_{ \mathrm{op}}\leq C\frac{\alpha_{i}}{i}\cdot \frac{1}{\alpha_{i}^{2}}=\frac{C}{i\alpha_{i}}\leq Cn^{\beta-1}.
\end{align*}
If $r>i$, then $|\alpha_i-\alpha_{i+r+1}|\leq \alpha_{i}$, which implies that
\begin{align*}
	\sum_{r>i}\left(\alpha_i-\alpha_{i+r+1}\right) \|P_{i, i+r}\|_{\op}&\leq C\sum_{r>i} \alpha_{i}\exp\Big(-\lambda_{1}r\alpha_{0}(i+r)^{-\beta}\Big)\nonumber\\
	&\leq C\sum_{r>i} \alpha_{i}\exp(-cr^{1-\beta})\nonumber\\
	&\leq C\exp(-ci^{1-\beta})\leq C\exp(-cn^{1-\beta}).
\end{align*}
Therefore, we have for any $\lfloor n/4\rfloor\leq i\leq \lfloor n/2\rfloor$
\[
\|R_{i}\|_{\op}\leq C\exp(-cn^{1-\beta})+Cn^{\beta-1}.
\]
Hence, for all sufficiently large $n$, we have
$
\|R_{i}\|_{\op}\leq 1/4.
$
Combining this with \eqref{eq-ka-06} and \eqref{eq-ka-09}, we complete the proof of \eqref{eq-ka-04}.

As for \eqref{eq-ka-02}, for $1\leq i\leq n-1$, by \eqref{eq-ka-07} and \eqref{eq-ka-08}, we have
\begin{align*}
	\|Q_{i}\|_{\op}&\leq \alpha_{i}\sum_{j=i}^{n-1}\prod_{\ell=i+1}^{j}\|I_m-\alpha_{\ell} H\|_{\op}\leq \alpha_{i}\sum_{r=0}^{\infty}\prod_{\ell=i+1}^{i+r}\|I_m-\alpha_{\ell} H\|_{\op}\nonumber\\
	&\leq C\alpha_{i}\sum_{r=0}^{\infty}\exp\Big(-\lambda_{1}r\alpha_{0}(i+r)^{-\beta}\Big)\nonumber\\
	&:=J_{1}+J_{2},
\end{align*}
where
\[
J_{1}=C\alpha_{i}\sum_{r=0}^{i}\exp\Big(-\lambda_{1}r\alpha_{0}(i+r)^{-\beta}\Big),\quad  J_{2}=C\alpha_{i}\sum_{r=i+1}^{\infty}\exp\Big(-\lambda_{1}r\alpha_{0}(i+r)^{-\beta}\Big).
\]
For $J_{1}$, using the equality $\sum_{r=0}^{\infty} q^{r}=1/(1-q)$ for $0<q<1$ and the inequality $1-e^{-x}\geq e^{-a}x$ for any $0\leq x\leq a$, we have
\begin{align*}
	J_{1}&\leq C\alpha_{i}\sum_{r=0}^{\infty}\exp\Big(-\lambda_{1}2^{-\beta}r\alpha_{0}i^{-\beta}\Big)
	=C\alpha_{i}\cdot \frac{1}{1-\exp(-\lambda_{1}2^{-\beta}\alpha_{0}i^{-\beta})}
	\leq C\alpha_{i}\cdot i^{\beta}\leq C.
\end{align*}
For $J_{2}$, we have
\begin{align*}
	J_{2}&\leq C\alpha_{i}\sum_{r=i+1}^{\infty}\exp\Big(-2^{-\beta}\lambda_{1}\alpha_{0}r^{1-\beta}\Big)\leq C\alpha_{i}\leq C.
\end{align*}
Combining the above bounds for $J_{1}$ and $J_{2}$ proves
\eqref{eq-ka-02} for $1\leq i\leq n-1$.
Since $\alpha_{1}=\alpha_{0}$, the definition of $Q_{i}$ gives
$
	Q_{0}
	=\alpha_{0}I_{m}+(I_{m}-\alpha_{1}H)Q_{1},$ and hence
$	\|Q_{0}\|_{\op}
	\leq
	\alpha_{0}
	+(1+\alpha_{0}\lambda_{2})\|Q_{1}\|_{\op}
	\leq C.$
Finally, $Q_{n}=0$ by definition. This completes the proof.
\subsection{Proof of Lemma \ref{lem-ap_nonsmooth-1}}
\label{sec:appendix-proof-of-lemma-nonsmooth}
In this subsection, we prove Lemma \ref{lem-ap_nonsmooth-1}.
The proof is organized as follows. We first establish several auxiliary lemmas
that will be used repeatedly. We then derive moment and
replacement bounds for the main components of $D_n$, and
finally combine these componentwise estimates to prove Lemma \ref{lem-ap_nonsmooth-1}.
\subsubsection{Auxiliary lemmas}
In this subsection, we introduce some preliminary lemmas that will be used in the proof of Lemma~\ref{lem-ap_nonsmooth-1}. The following lemma provide an upper bound for  $\E|x_{k}-x_{*}|^{12}$.

\begin{lemma}\label{lem-3} Suppose that  Assumptions \ref{assu-3}--\ref{assu-4} hold. Then,
	\begin{align}\label{eq-tc-01}
		\E|x_{k}-x_{*}|^{12}\leq C(1+\tau_{0}^{12}+\tau^{12})k^{-6\beta}\quad \text{for all } k\geq 1.
	\end{align}
\end{lemma}
\begin{proof}
	By the binomial expansion of
	$|a-b|^{12}=(|a|^2-2\langle a,b\rangle+|b|^2)^6$
	and the Cauchy--Schwarz inequality, there exists a constant $C>0$ such that for all $a,b\in\mathbb R^d$,
	\begin{align}\label{eq-tr-01}
		|a-b|^{12}
		\leq
		|a|^{12}
		-12|a|^{10}\langle a,b\rangle
		+
		C\sum_{j=2}^{12}|a|^{12-j}|b|^{j}.
	\end{align}
	Applying this inequality with $a=x_{k-1}-x_{*}$ and $b=\alpha_{k}G_{\alpha_{k}}(x_{k-1},\zeta_{k})$, we have
	\begin{align}\label{eq-tr-03}
		\begin{aligned}
			&|x_{k-1}-\alpha_{k}G_{\alpha_{k}}(x_{k-1},\zeta_{k})-x_{*}|^{12}\\
			&\quad\leq |x_{k-1}-x_{*}|^{12}-12|x_{k-1}-x_{*}|^{10}\la x_{k-1}-x_{*},
			\alpha_{k}G_{\alpha_{k}}(x_{k-1},\zeta_{k})\ra\\
			&\quad\quad+C\sum_{j=2}^{12}|x_{k-1}-x_{*}|^{12-j}|\alpha_{k}G_{\alpha_{k}}(x_{k-1},\zeta_{k})|^{j}.
		\end{aligned}
	\end{align}
	For $2\leq j\leq 12$, by Young's inequality and Assumptions \ref{assu-3} and \ref{assu-4},
	we have
	\begin{align}\label{eq-tr-04}
		\begin{aligned}
			&C|x_{k-1}-x_{*}|^{12-j}
			|\alpha_{k}G_{\alpha_{k}}(x_{k-1},\zeta_{k})|^{j}
			\\
			&\quad \leq
			0.1c_{5}\alpha_{k}|x_{k-1}-x_{*}|^{12}
			+
			C\alpha_{k}^{-(12-j)/j}
			|\alpha_{k}G_{\alpha_{k}}(x_{k-1},\zeta_{k})|^{12}\\
			&\quad\leq 0.1c_{5}\alpha_{k}|x_{k-1}-x_{*}|^{12}
			+
			C\alpha_{k}^{7}
			|G_{\alpha_{k}}(x_{k-1},\zeta_{k})|^{12}\\
			&\quad\leq 0.1c_{5}\alpha_{k}|x_{k-1}-x_{*}|^{12}
			+
			C\alpha_{k}^{7}
			(1+|x_{k-1}-x_{*}|^{12}+|\xi_{k}|^{12}).
		\end{aligned}
	\end{align}
	For the second term of \eqref{eq-tr-03}, by Assumptions \ref{assu-3} and \ref{assu-4}, we have
	\begin{align}\label{eq-tr-05}
		\begin{aligned}
			&-12|x_{k-1}-x_{*}|^{10}\la x_{k-1}-x_{*}, \alpha_{k}G_{\alpha_{k}}(x_{k-1},\zeta_{k})-\alpha_{k}\zeta_{k}\ra\\
			&\qquad\leq -12c_{5}\alpha_{k}|x_{k-1}-x_{*}|^{12}\\
			&\qquad\quad+12c_{6}\alpha_{k}^{2}|x_{k-1}-x_{*}|^{10}(1+(1+2c_{8}^{2})|x_{k-1}-x_{*}|^{2}+2|\xi_{k}|^{2})\\
			&\qquad\leq -11c_{5}\alpha_{k}|x_{k-1}-x_{*}|^{12}+C\alpha_{k}^{7}(1+|\xi_{k}|^{12})+C\alpha_{k}^{2}|x_{k-1}-x_{*}|^{12},
		\end{aligned}
	\end{align}
	where the last inequality follows from
	\[
	24c_{6}\alpha_{k}^{2}|x_{k-1}-x_{*}|^{10}(1+|\xi_{k}|^{2})\leq c_{5}\alpha_{k}|x_{k-1}-x_{*}|^{12}+C\alpha_{k}^{7}(1+|\xi_{k}|^{12}).
	\]
	Substituting \eqref{eq-tr-04} and \eqref{eq-tr-05} into \eqref{eq-tr-03}, we have
	\begin{align*}
		|x_{k}-x_{*}|^{12}&= |x_{k-1}-\alpha_{k}G_{\alpha_{k}}(x_{k-1},\zeta_{k})-x_{*}|^{12}\\
		&\leq (1-9c_{5}\alpha_{k}+C\alpha_{k}^{2})|x_{k-1}-x_{*}|^{12}
		+C\alpha_{k}^{7}(1+|\xi_{k}|^{12})\\
		&\quad-12\alpha_{k} |x_{k-1}-x_{*}|^{10}\la x_{k-1}-x_{*}, \zeta_{k}\ra.
	\end{align*}
	By Assumption \ref{assu-4}, we have $\E\{\zeta_{k}\mid x_{k-1}\}=0$.  Taking expectations on both sides of the above inequality, we have
	\begin{align*}
		\E|x_{k}-x_{*}|^{12}&\leq (1-9c_{5}\alpha_{k}+C\alpha_{k}^{2})\E|x_{k-1}-x_{*}|^{12}+C(1+\tau^{12})\alpha_{k}^{7}\nonumber\\
		&\leq (1-9c_{5}\alpha_{k}+(C+25c_{5}^{2})\alpha_{k}^{2})\E|x_{k-1}-x_{*}|^{12}+C(1+\tau^{12})\alpha_{k}^{7}.
	\end{align*}
	Note that  the coefficient
	in front of $\E|x_{k-1}-x_{*}|^{12}$ is always positive.  By the above recursive inequality, we have
	\begin{align}\label{eq-tr-5.1}
		\begin{aligned}
			\E|x_{k}-x_{*}|^{12}&\leq \tau_{0}^{12}\prod_{i=1}^{k}(1-9c_{5}\alpha_{i}+c'_{5}\alpha_{i}^{2})\\
			&\quad+ C(1+\tau^{12})\sum_{i=1}^{k}\prod_{j=i+1}^{k} (1-9c_{5}\alpha_{j}+c'_{5}\alpha_{j}^{2})\alpha_{i}^{7},
		\end{aligned}
	\end{align}
	where $c'_{5}=(C+25c_{5}^{2}).$
	Let $n_{0}=\inf\{n\in \N, c'_{5}\alpha_{n}\leq c_{5}\}$ and suppose that $k>2n_{0}$. Observe that
	\begin{align}\label{eq-tr-06}
		\begin{aligned}
			&\sum_{i=1}^{k}\prod_{j=i+1}^{k} (1-9c_{5}\alpha_{j}+c'_{5}\alpha_{j}^{2})\alpha_{i}^{7}\\
			&\quad= \sum_{i=n_{0}+1}^{k}\prod_{j=i+1}^{k} (1-9c_{5}\alpha_{j}+c'_{5}\alpha_{j}^{2})\alpha_{i}^{7}+\sum_{i=1}^{n_{0}}\prod_{j=i+1}^{k} (1-9c_{5}\alpha_{j}+c'_{5}\alpha_{j}^{2})\alpha_{i}^{7}.
		\end{aligned}
	\end{align}
	Let $k_{1}=\lfloor k/2 \rfloor$. For the first term of \eqref{eq-tr-06}, by the inequality $1+x\leq e^{x}$ for any $x\in \R$, we have
	\begin{align*}
		&\sum_{i=n_{0}+1}^{k}\prod_{j=i+1}^{k} (1-9c_{5}\alpha_{j}+c'_{5}\alpha_{j}^{2})\alpha_{i}^{7}\nonumber\\
		&\quad=\sum_{i=n_{0}+1}^{k_{1}}\prod_{j=i+1}^{k} (1-9c_{5}\alpha_{j}+c'_{5}\alpha_{j}^{2})
		\alpha_{i}^{7}+\sum_{i=k_{1}+1}^{k}\prod_{j=i+1}^{k} (1-9c_{5}\alpha_{j}+c'_{5}\alpha_{j}^{2})
		\alpha_{i}^{7}\nonumber\\
		&\quad\leq\sum_{i=n_{0}+1}^{k_{1}} \exp\Big( -9c_{5}\sum_{j=i+1}^{k}\alpha_{j}
		+c'_{5}\sum_{j=i+1}^{k}\alpha_{j}^{2} \Big) \alpha_{i}^{2}+\alpha_{k_{1}}^{6}\sum_{i=k_{1}+1}^{k}\prod_{j=i+1}^{k}(1-8c_{5}\alpha_{j})\alpha_{i}\nonumber\\
		&\quad\leq \exp\Big( -8c_{5}\sum_{j=k_{1}+1}^{k}\alpha_{j} \Big) \sum_{i=1}^{k}\alpha_{i}^{2}+\frac{\alpha_{k_{1}}^{6}}{c_{5}}\sum_{i=k_{1}+1}^{k}\Big[\prod_{j=i+1}^{k}(1-c_{5}\alpha_{j})-\prod_{j=i}^{k}(1-c_{5}\alpha_{j})\Big]\nonumber\\
		&\quad\leq C\exp\big( -ck^{1-\beta} \big)+C\alpha_{k_{1}}^{6}\nonumber\\
		&\quad\leq Ck^{-6\beta}.
	\end{align*}
	For the second term of \eqref{eq-tr-06}, we have
	\begin{align*}
		\prod_{j=n_{0}+1}^{k} (1-8c_{5}\alpha_{j})&\leq \exp\Big(-8c_{5}\sum_{j=n_{0}+1}^{k}\alpha_{j}\Big)\leq \exp\Big(-8c_{5}\sum_{j=1}^{k}\alpha_{j}\Big)\exp\Big(8c_{5}\sum_{j=1}^{n_{0}}\alpha_{j}\Big)\nonumber\\
		&\leq C\exp\Big(-c_{5}\sum_{j=1}^{k}\alpha_{j}\Big)
		\leq C\exp\Big(-ck^{1-\beta}\Big)\leq Ck^{-6\beta},
	\end{align*}
	and hence
	\begin{align*}
		\sum_{i=1}^{n_{0}}\prod_{j=i+1}^{k} (1-9c_{5}\alpha_{j}+c'_{5}\alpha_{j}^{2})\alpha_{i}^{7}&\leq \sum_{i=1}^{n_{0}}\alpha_{i}^{7}\prod_{j=i+1}^{n_{0}} (1-9c_{5}\alpha_{j}+c'_{5}\alpha_{j}^{2})\cdot \prod_{j=n_{0}+1}^{k} (1-8c_{5}\alpha_{j})\nonumber\\
		&\leq Ck^{-6\beta}.
	\end{align*}
	Consequently, we have
	\begin{align}\label{eq-tr-07}
		\sum_{i=1}^{k}\prod_{j=i+1}^{k} (1-9c_{5}\alpha_{j}+c'_{5}\alpha_{j}^{2})\alpha_{i}^{7}\leq Ck^{-6\beta}.
	\end{align}
	For \(k\leq 2n_0\), the above inequality remains valid by enlarging \(C\), since there are only finitely many such \(k\).
	Substituting \eqref{eq-tr-07} into \eqref{eq-tr-5.1} and using the
	inequality $1+x\leq e^x$, we have
	\begin{align*}
		\E|x_{k}-x_{*}|^{12}&\leq C\tau_{0}^{12}\exp\Big(-9c_{5}\sum_{i=1}^{k}\alpha_{i}+c'_{5}\sum_{i=1}^{k}\alpha_{i}^{2}\Big)+C(1+\tau^{12})k^{-6\beta}\nonumber\\
		&\leq C(1+\tau_{0}^{12}+\tau^{12})k^{-6\beta},
	\end{align*}
	as desired.
\end{proof}

The following lemma establishes an upper bound on $\E|x_k-x_k^{\{i\}}|^6$,
which will be used to control $\E|D_n-D_n^{\{i\}}|^2$.
\begin{lemma}\label{lem-UpperB-xn-xni-fourth}
	Suppose that Assumptions \ref{assu-3}--\ref{assu-5} hold.  Then,
	\begin{align}\label{eq-oa-01}
		\E|x_k-x_k^{\{i\}}|^{6} \le C(1+\tau_0^6+\tau^6)k^{-6\beta}\quad \text{for all } i,k\geq 1.
	\end{align}
\end{lemma}
\begin{proof}
	Recall that $\Phi_\alpha(x,\xi)
	=x-\alpha G_\alpha(x,\xi+g(x,\xi)).$ By the definitions of $x_k$ and $x_k^{\{i\}}$, we have
	\[
	x_k-x_k^{\{i\}}
	=
	\begin{cases}
		0, & k<i,\\[1mm]
		-\alpha_i\big\{G_{\alpha_i}(x_{i-1},\zeta_{i})
		-G_{\alpha_i}\big(x_{i-1},\zeta_{i}^{\{i\}}\big)\big\}, & k=i,\\[1mm]
		\Phi_{\alpha_{k}}\big(x_{k-1},\xi_{k}\big)-\Phi_{\alpha_{k}}\big(x_{k-1}^{\{i\}}, \xi_{k}\big),
		& k>i,
	\end{cases}
	\]
	where $\zeta_{i}^{\{i\}}=\xi'_{i}+g(x_{i-1}, \xi'_{i})$.
	We first bound the initial difference created by the replacement at time $i$. By the elementary inequality
	$|a-b|^{6}\le C|a|^{6}+C|b|^{6}$, we have
	\[
	\begin{aligned}
		\E|x_i-x_i^{\{i\}}|^{6}
		&\le C\alpha_i^{6}
		\E|G_{\alpha_i}(x_{i-1},\zeta_i)|^{6}
		+ C\alpha_i^{6}
		\E|G_{\alpha_i}(x_{i-1},\zeta_i^{\{i\}})|^{6} .
	\end{aligned}
	\]
	By Assumptions \ref{assu-3} and \ref{assu-4}, we have
	$|G_{\alpha_i}(x_{i-1},\zeta_i)|^{6}
	\le
	C\bigl(1+|x_{i-1}-x_*|^{6}+|\xi_i|^{6}\bigr).$ Combining this with Lemma \ref{lem-3}, we have
	\[
	\begin{aligned}
		\E|x_i-x_i^{\{i\}}|^{6}
		&\le
		C\alpha_{i}^{6}
		(1+\E|x_{i-1}-x_*|^{6}+\E|\xi_i|^{6}) \\
		&\le C(1+\tau_{0}^{6}+\tau^{6})\alpha_{i}^{6}
		\le C(1+\tau_{0}^{6}+\tau^{6}) i^{-6\beta}.
	\end{aligned}
	\]
	Choose an integer $k_{0}\geq 1$ sufficiently large such that
	$\alpha_{k_{0}}\leq\bar{\alpha}$,
	$c_{10}\alpha_{k_{0}}\leq c_{9}/2$, and
	$c_{9}\alpha_{k_{0}}\leq 1$.
	We first consider the case $i\geq k_{0}$ and $k>i$. Note that $\xi_{k}$ is independent of
	$(x_{k-1}, x_{k-1}^{\{i\}})$. Combining this with  Assumption \ref{assu-5}, we have
	\begin{align*}
		\E|x_k-x_k^{\{i\}}|^{6}
		&\le
		(1-c_9\alpha_k+c_{10}\alpha_k^2)\E|x_{k-1}-x_{k-1}^{\{i\}}|^{6}\nonumber\\
		&\leq \prod_{j=i+1}^{k} (1-c_9\alpha_j+c_{10}\alpha_j^2)\E|x_{i}-x_{i}^{\{i\}}|^{6}\\
		&\leq C(1+\tau_{0}^{6}+\tau^{6}) i^{-6\beta}
		\exp\Big\{-\frac{c_9}{2}\sum_{j=i+1}^{k}\alpha_j\Big\}.
	\end{align*}
	Taking sixth roots and using $i^{-\beta}=\alpha_i/\alpha_0$, we have
	\begin{align}\label{eq:loo-exponential}
		\|x_k-x_k^{\{i\}}\|_6
		\leq
		C(1+\tau_0+\tau)\alpha_i
		\exp\Big\{-c\sum_{j=i+1}^{k}\alpha_j\Big\},
	\qquad k\geq i,\quad i\geq k_0.
	\end{align}
For $i<k_{0}$, by Lemma \ref{lem-3}, we have
$
	\|x_{k_{0}}-x_{k_{0}}^{\{i\}}\|_{6}
	\leq C(1+\tau_{0}+\tau).
$
Applying Assumption~\ref{assu-5} from time $k_{0}+1$ onward gives, for
$k>k_{0}$,
\begin{align*}
	\|x_k-x_k^{\{i\}}\|_{6}
	\leq
	C(1+\tau_{0}+\tau)
	\exp\bigg\{-c\sum_{j=k_{0}+1}^{k}\alpha_j\bigg\}\leq C(1+\tau_0+\tau)\alpha_i
	\exp\Big\{-c\sum_{j=i+1}^{k}\alpha_j\Big\},
\end{align*}
where the last inequality follows since $k_{0}$ is fixed and there
are only finitely many $1\leq i<k_{0}$. Thus,
\eqref{eq:loo-exponential} also holds for $i<k_{0}$ and $k>k_{0}$. For the remaining finitely many pairs
$1\leq i\leq k\leq k_{0}$, Lemma~\ref{lem-3} and equality in
distribution give
\begin{align*}
	\|x_{k}-x_{k}^{\{i\}}\|_{6}
	&\leq
	\|x_{k}-x_{*}\|_{6}
	+
	\|x_{k}^{\{i\}}-x_{*}\|_{6}\leq C(1+\tau_{0}+\tau).
\end{align*}
Since $k_{0}$ depends only on the fixed structural quantities and
$
\min_{1\leq i\leq k\leq k_{0}}
\alpha_{i}
\exp\{
-c\sum_{j=i+1}^{k}\alpha_{j}
\}>0,
$
after enlarging $C$, \eqref{eq:loo-exponential} also holds for
these finitely many pairs. Consequently,
\eqref{eq:loo-exponential} holds for all $k\geq i\geq1$.

	We distinguish two cases. If $k<2i$, then $i>k/2$, and hence
	$i^{-6\beta}\leq 2^{6\beta}k^{-6\beta}.$
	Since the exponential factor is bounded by one, \eqref{eq-oa-01}
	follows in this case.
	If $k\geq 2i$, then $i\leq k/2$, and therefore
	\begin{align*}
		\sum_{j=i+1}^{k}j^{-\beta}
		&\geq
		\sum_{j=\lfloor k/2\rfloor+1}^{k}j^{-\beta}
		\geq c k^{1-\beta}
	\end{align*}
	for some constant $c>0$. Consequently, we have
	\begin{align*}
		i^{-6\beta}
		\exp\Big\{
		-0.5c_{9}\alpha_0\sum_{j=i+1}^{k}j^{-\beta}
		\Big\}
		&\leq
		\exp\{-c'k^{1-\beta}\}\leq Ck^{-6\beta}.
	\end{align*}
	Combining the two cases, we obtain
	\eqref{eq-oa-01}.
\end{proof}

A key step in the proof of Theorem~\ref{thm-app-nonsmooth} is to replace the original iterate
$x_k$ by the auxiliary iterate $y_k$. It is therefore important to
quantify the resulting approximation error. The following lemma gives
pointwise high-moment and weighted cumulative bounds for this error.
\begin{lemma}\label{lem-A6} 	Suppose that Assumptions \ref{assu-2}--\ref{assu-5} hold.    There exists a positive constant $C$ such that
	\begin{align}\label{eq-ya-01}
		\E|x_{k}-y_{k}|^{6}\leq C\tau_{2}k^{1-6\beta}\quad \text{for each } k\geq 1
	\end{align}
	and
	\begin{align}\label{eq-ya-02}
		\sum_{k=1}^n k^{4\beta}\E |x_{k}-y_{k}|^{5}
		\leq C\tau_{2}n^{1-\beta},
	\end{align}
	where $\tau_{2}=(1+\tau_{0}^{10}+\tau^{10})(1+\delta^{-6}(1+\tau_{0}^{6}+\tau^{6})).$
\end{lemma}
\begin{proof}
	We first prove \eqref{eq-ya-01}.
	By the definition of $y_{k}$, we have
	\begin{align}\label{eq-ya-03}
		\begin{aligned}
		&	|x_{k+1}-y_{k+1}|^{6}\\
			&\quad=
			|x_{k+1}-y_{k+1}|^{6}
			\mathbf{1}(x_{k+1}\in B_{\delta}(x_{*}))
			+
			|x_{k+1}-y_{k+1}|^{6}
			\mathbf{1}(x_{k+1}\notin B_{\delta}(x_{*}))\\
			&\quad\leq
			\dist^{6}(x_{k+1},\MM)
			\mathbf{1}(x_{k+1}\in B_{\delta}(x_{*}))
			+\delta^{-6}|x_{k+1}-x_{*}|^{12},
		\end{aligned}
	\end{align}
	where the second inequality follows from the fact that
	$|x_{k}-y_{k}|\leq |x_{k}-x_{*}|$.
	Noting that $y_{k}\in\MM$, we have
	$\dist^{6}(x_{k+1},\MM)\leq |x_{k+1}-y_{k}|^{6}$,
	which further implies
	\begin{align}\label{eq-ya-04}
		\begin{aligned}
			&\dist^{6}(x_{k+1},\MM)
			\mathbf{1}(x_{k+1}\in B_{\delta}(x_{*}))\\
			&\quad\leq
			|x_{k+1}-y_{k}|^{6}
			\mathbf{1}(x_{k+1}\in B_{\delta}(x_{*}))\\
			&\quad=
			|x_{k}-\alpha_{k+1}G_{\alpha_{k+1}}(x_{k},\zeta_{k+1})
			-y_{k}|^{6}
			\mathbf{1}(x_{k+1}\in B_{\delta}(x_{*})).
		\end{aligned}
	\end{align}
	By writing
	$|a-b|^{6}
	=(|a|^{2}-2\langle a,b\rangle+|b|^{2})^{3}$
	and using the Cauchy--Schwarz inequality, we obtain,
	for any $a,b\in\R^{d}$,
	\begin{align}\label{eq-ya-05}
		|a-b|^{6}
		\leq
		|a|^{6}
		-6|a|^{4}\langle a,b\rangle
		+|b|^{6}
		+C\sum_{j=1}^{4}|a|^{j}|b|^{6-j}.
	\end{align}
	For any $j\in[4]$, by Young's inequality and
	Assumptions~\ref{assu-3} and~\ref{assu-4}, we have
	\begin{align}\label{eq-ya-06}
		\begin{aligned}
			&C|x_{k}-y_{k}|^{j}
			|\alpha_{k+1}G_{\alpha_{k+1}}(x_{k},\zeta_{k+1})|^{6-j}\\
			&\quad\leq
			0.1c_{3}\alpha_{k+1}|x_{k}-y_{k}|^{5}
			+C\alpha_{k+1}^{6}
			|G_{\alpha_{k+1}}(x_{k},\zeta_{k+1})|
			^{\frac{5(6-j)}{5-j}}\\
			&\quad\leq
			0.1c_{3}\alpha_{k+1}|x_{k}-y_{k}|^{5}
			+C\alpha_{k+1}^{6}
			(1+|G_{\alpha_{k+1}}(x_{k},\zeta_{k+1})|^{10})\\
			&\quad\leq
			0.1c_{3}\alpha_{k+1}|x_{k}-y_{k}|^{5}
			+C\alpha_{k+1}^{6}
			(1+|\xi_{k+1}|^{10}+|x_{k}-x_{*}|^{10}).
		\end{aligned}
	\end{align}
	Applying \eqref{eq-ya-05} with
	$a=x_{k}-y_{k}$ and
	$b=\alpha_{k+1}G_{\alpha_{k+1}}(x_{k},\zeta_{k+1})$,
	and then using \eqref{eq-ya-06} and the growth bound
	$|G_{\alpha_{k+1}}(x_{k},\zeta_{k+1})|^{2}
	\leq C(1+|x_{k}-x_{*}|^{2}+|\xi_{k+1}|^{2})$,
	we have
	\begin{align}\label{eq-ya-07}
		\begin{aligned}
			&|x_{k}-\alpha_{k+1}G_{\alpha_{k+1}}(x_{k},\zeta_{k+1})
			-y_{k}|^{6}
			\mathbf{1}(x_{k+1}\in B_{\delta}(x_{*}))\\
			&\quad\leq
			|x_{k}-y_{k}|^{6}
			-6|x_{k}-y_{k}|^{4}
			\langle x_{k}-y_{k},
			\alpha_{k+1}G_{\alpha_{k+1}}(x_{k},\zeta_{k+1})\rangle
			\mathbf{1}(x_{k+1}\in B_{\delta}(x_{*}))\\
			&\qquad+
			c_{3}\alpha_{k+1}|x_{k}-y_{k}|^{5}
			+C\alpha_{k+1}^{6}
			(1+|\xi_{k+1}|^{10}+|x_{k}-x_{*}|^{10}).
		\end{aligned}
	\end{align}

	In what follows, we derive an upper bound for
	$-\langle G_{\alpha_{k+1}}(x_{k},\zeta_{k+1}),
	x_{k}-y_{k}\rangle
	\mathbf{1}(x_{k+1}\in B_{\delta}(x_{*}))$.
	We rewrite
	\begin{align*}
		&-\langle G_{\alpha_{k+1}}(x_{k},\zeta_{k+1}),
		x_{k}-y_{k}\rangle
		\mathbf{1}(x_{k+1}\in B_{\delta}(x_{*}))
		=I_{1}+I_{2},
	\end{align*}
	where
	\begin{align*}
		I_{1}
		&:=
		-\langle G_{\alpha_{k+1}}(x_{k},\zeta_{k+1}),
		x_{k}-y_{k}\rangle
		\mathbf{1}(x_{k+1}\in B_{\delta}(x_{*}),
		x_{k}\in B_{\delta}(x_{*})),\\
		I_{2}
		&:=
		-\langle G_{\alpha_{k+1}}(x_{k},\zeta_{k+1}),
		x_{k}-y_{k}\rangle
		\mathbf{1}(x_{k+1}\in B_{\delta}(x_{*}),
		x_{k}\notin B_{\delta}(x_{*})).
	\end{align*}
	For $I_{1}$, by Assumption~\ref{assu-2},
	if $x_{k}\in B_{\delta}(x_{*})$, then
	\begin{align*}
		\langle
		G_{\alpha_{k+1}}(x_{k},\zeta_{k+1})-\zeta_{k+1},
		x_{k}-y_{k}
		\rangle
		\geq
		c_{3}|x_{k}-y_{k}|
		-c_{4}\big\{
		\alpha_{k+1}(1+|\zeta_{k+1}|^{2})
		+|x_{k}-y_{k}|^{2}
		\big\}.
	\end{align*}
	Consequently, we have
	\begin{align*}
		I_{1}
		&=
		-\langle
		G_{\alpha_{k+1}}(x_{k},\zeta_{k+1})-\zeta_{k+1},
		x_{k}-y_{k}
		\rangle
		\mathbf{1}(x_{k+1}\in B_{\delta}(x_{*}),
		x_{k}\in B_{\delta}(x_{*}))\\
		&\quad-
		\langle\zeta_{k+1},x_{k}-y_{k}\rangle
		\mathbf{1}(x_{k+1}\in B_{\delta}(x_{*}),
		x_{k}\in B_{\delta}(x_{*}))\\
		&\leq I_{11}+I_{12}+I_{13},
	\end{align*}
	where
	\begin{align*}
		I_{11}
		&=
		-c_{3}|x_{k}-y_{k}|
		\mathbf{1}(x_{k+1}\in B_{\delta}(x_{*}),
		x_{k}\in B_{\delta}(x_{*})),\\
		I_{12}
		&=
		c_{4}\big\{
		|x_{k}-y_{k}|^{2}
		+\alpha_{k+1}(1+|\zeta_{k+1}|^{2})
		\big\},\\
		I_{13}
		&=
		-\langle\zeta_{k+1},x_{k}-y_{k}\rangle
		\mathbf{1}(x_{k+1}\in B_{\delta}(x_{*}),
		x_{k}\in B_{\delta}(x_{*})).
	\end{align*}
	For $I_{11}$, by Markov's inequality, we have
	\begin{align}\label{eq-yh-01}
		I_{11}
		\leq
		-c_{3}|x_{k}-y_{k}|
		+C\delta^{-1}|x_{k}-x_{*}|^{2}
		+C\delta^{-1}|x_{k+1}-x_{*}|^{2}.
	\end{align}
	By Assumption~\ref{assu-4}, we have
	$|\eta_{k}|\leq c_{8}|x_{k-1}-x_{*}|$
	for any $k\geq1$.
	Combining this with the fact that
	$|x_{k}-y_{k}|\leq |x_{k}-x_{*}|$, we have
	\begin{align}\label{eq-yh-02}
		|I_{12}|
		\leq
		c_{4}|x_{k}-x_{*}|^{2}
		+c_{4}\alpha_{k+1}
		(1+2|\xi_{k+1}|^{2}+2c_{8}^{2}|x_{k}-x_{*}|^{2}).
	\end{align}
	Combining the facts that
	$|\eta_{k}|\leq c_{8}|x_{k-1}-x_{*}|$ and
	$|x_{k}-y_{k}|\leq |x_{k}-x_{*}|$ once again, we have
	\begin{align}\label{eq-yh-03}
		I_{13}
		&=
		-\langle\zeta_{k+1},x_{k}-y_{k}\rangle
		+\langle\zeta_{k+1},x_{k}-y_{k}\rangle
		\mathbf{1}(x_{k+1}\in B_{\delta}(x_{*}),
		x_{k}\notin B_{\delta}(x_{*}))\nonumber\\
		&\quad+
		\langle\zeta_{k+1},x_{k}-y_{k}\rangle
		\mathbf{1}(x_{k+1}\notin B_{\delta}(x_{*}))\nonumber\\
		&\leq
		-\langle\zeta_{k+1},x_{k}-y_{k}\rangle
		+|\langle\zeta_{k+1},x_{k}-y_{k}\rangle|
		\min\{1,\delta^{-1}|x_{k}-x_{*}|\}\nonumber\\
		&\quad+
		|\langle\zeta_{k+1},x_{k}-y_{k}\rangle|
		\min\{1,\delta^{-1}|x_{k+1}-x_{*}|\}\nonumber\\
		&\leq
		-\langle\zeta_{k+1},x_{k}-y_{k}\rangle
		+C(1+\delta^{-1}|\xi_{k+1}|)
		\big\{
		|x_{k+1}-x_{*}|^{2}+|x_{k}-x_{*}|^{2}
		\big\}.
	\end{align}
	Combining the above upper bounds for
	$I_{11}$, $I_{12}$ and $I_{13}$, we have
	\begin{align*}
		I_{1}
		&\leq
		-c_{3}|x_{k}-y_{k}|
		-\langle\xi_{k+1},x_{k}-y_{k}\rangle
	+
		C\alpha_{k+1}(1+|\xi_{k+1}|^{2})\\
		&\quad+
		C(1+\delta^{-1}+\delta^{-1}|\xi_{k+1}|)
		\big\{
		|x_{k+1}-x_{*}|^{2}+|x_{k}-x_{*}|^{2}
		\big\}.
	\end{align*}
	For $I_{2}$, by Assumption~\ref{assu-3}, we have
	\begin{align*}
		I_{2}
		&\leq
		|G_{\alpha_{k+1}}(x_{k},\zeta_{k+1})|
		|x_{k}-x_{*}|
		\mathbf{1}(x_{k+1}\in B_{\delta}(x_{*}),
		x_{k}\notin B_{\delta}(x_{*}))\\
		&\leq
		C(1+|x_{k}-x_{*}|+|\xi_{k+1}|)
		|x_{k}-x_{*}|
		\min\{1,\delta^{-1}|x_{k}-x_{*}|\}\\
		&\leq
		C(1+\delta^{-1}+\delta^{-1}|\xi_{k+1}|)
		\big\{
		|x_{k+1}-x_{*}|^{2}+|x_{k}-x_{*}|^{2}
		\big\}.
	\end{align*}
	Combining the above upper bounds for $I_{1}$ and $I_{2}$,
	we have
	\begin{align}\label{eq-yy-4.1}
		\begin{aligned}
			&-\langle G_{\alpha_{k+1}}(x_{k},\zeta_{k+1}),
			x_{k}-y_{k}\rangle
			\mathbf{1}(x_{k+1}\in B_{\delta}(x_{*}))\\
			&\quad\leq
			-c_{3}|x_{k}-y_{k}|
			-\langle\xi_{k+1},x_{k}-y_{k}\rangle
			+C\alpha_{k+1}(1+|\xi_{k+1}|^{2})\\
			&\qquad+
			C(1+\delta^{-1}+\delta^{-1}|\xi_{k+1}|)
			\big\{
			|x_{k+1}-x_{*}|^{2}+|x_{k}-x_{*}|^{2}
			\big\}.
		\end{aligned}
	\end{align}
		Now, by \eqref{eq-yy-4.1}, we have
		\begin{align}\label{eq-ya-07.1}
			\begin{aligned}
				&-6\alpha_{k+1}|x_{k}-y_{k}|^{4}
				\langle G_{\alpha_{k+1}}(x_{k},\zeta_{k+1}),
				x_{k}-y_{k}\rangle
				\mathbf{1}(x_{k+1}\in B_{\delta}(x_{*}))\\
				&\quad\leq
				-6c_{3}\alpha_{k+1}|x_{k}-y_{k}|^{5}
				-6\alpha_{k+1}|x_{k}-y_{k}|^{4}
				\langle\xi_{k+1},x_{k}-y_{k}\rangle\\
				&\qquad+
				C\alpha_{k+1}|x_{k}-y_{k}|^{4}
				(1+\delta^{-1}+\delta^{-1}|\xi_{k+1}|)
				\big\{
				|x_{k+1}-x_{*}|^{2}+|x_{k}-x_{*}|^{2}
				\big\}\\
				&\qquad+
				C\alpha_{k+1}^{2}|x_{k}-y_{k}|^{4}
				(1+|\xi_{k+1}|^{2}).
			\end{aligned}
		\end{align}
		To estimate the two error terms, we first condition on
		$\mathcal{F}_{k}$ before applying Young's inequality.
		By the update rule, Assumptions~\ref{assu-3}--\ref{assu-4},
		and $\alpha_{k+1}\leq\alpha_{0}$,
		\begin{align*}
			|x_{k+1}-x_{*}|^{2}+|x_{k}-x_{*}|^{2}
			\leq
			C|x_{k}-x_{*}|^{2}
			+C\alpha_{k+1}^{2}(1+|\xi_{k+1}|^{2}).
		\end{align*}
		Since $\xi_{k+1}$ is independent of $\mathcal{F}_{k}$ and
		$\|\xi_{k+1}\|_{12}\leq\tau$, it follows that
		\begin{align*}
			&\E\Big[
			(1+\delta^{-1}+\delta^{-1}|\xi_{k+1}|)
			\big\{
			|x_{k+1}-x_{*}|^{2}+|x_{k}-x_{*}|^{2}
			\big\}
			\mid\mathcal{F}_{k}
			\Big]\\
			&\quad\leq
			C\{1+\delta^{-1}(1+\tau)\}
			\big\{
			|x_{k}-x_{*}|^{2}
			+\alpha_{k+1}^{2}(1+\tau)^{2}
			\big\}.
		\end{align*}
		As $x_{k}$ and $y_{k}$ are $\mathcal{F}_{k}$-measurable,
		Young's inequality gives
		\begin{align*}
			&C\alpha_{k+1}\E\Big[
			|x_{k}-y_{k}|^{4}
			(1+\delta^{-1}+\delta^{-1}|\xi_{k+1}|)
			\big\{
			|x_{k+1}-x_{*}|^{2}+|x_{k}-x_{*}|^{2}
			\big\}
			\Big]\\
			&\quad\leq
			C\alpha_{k+1}\{1+\delta^{-1}(1+\tau)\}
			\E\Big[
			|x_{k}-y_{k}|^{4}
			\big\{
			|x_{k}-x_{*}|^{2}
			+\alpha_{k+1}^{2}(1+\tau)^{2}
			\big\}
			\Big]\\
			&\quad\leq
			\frac{c_{3}}{2}\alpha_{k+1}\E|x_{k}-y_{k}|^{5}
			+C\alpha_{k+1}\{1+\delta^{-1}(1+\tau)\}^{5}
			\big\{
			\E|x_{k}-x_{*}|^{10}
			+\alpha_{k+1}^{10}(1+\tau)^{10}
			\big\}\\
			&\quad\leq
			\frac{c_{3}}{2}\alpha_{k+1}\E|x_{k}-y_{k}|^{5}
			+C\tau_{2}\alpha_{k+1}^{6},
			\qquad k\geq1,
		\end{align*}
		where the last inequality follows from Lemma~\ref{lem-3},
		$\alpha_{k+1}\leq\alpha_{0}$, and
		$	\{1+\delta^{-1}(1+\tau)\}^{5}
			\leq C\tau_{2}.
$
		For the other error term, Young's inequality yields
		\begin{align*}
			C\alpha_{k+1}^{2}
			\E\big\{
			|x_{k}-y_{k}|^{4}(1+|\xi_{k+1}|^{2})
			\big\}
			&\leq
			\frac{c_{3}}{2}\alpha_{k+1}\E|x_{k}-y_{k}|^{5}
			+C\alpha_{k+1}^{6}(1+\tau^{10})\\
			&\leq
			\frac{c_{3}}{2}\alpha_{k+1}\E|x_{k}-y_{k}|^{5}
			+C\tau_{2}\alpha_{k+1}^{6}.
		\end{align*}
		Moreover, independence and $\E\xi_{k+1}=0$ imply
	$
			\E\big\{
			|x_{k}-y_{k}|^{4}
			\langle\xi_{k+1},x_{k}-y_{k}\rangle
			\big\}=0.
	$
		Combining \eqref{eq-ya-03}, \eqref{eq-ya-04},
		\eqref{eq-ya-07}, \eqref{eq-ya-07.1},
		the preceding estimates, and Lemma~\ref{lem-3},
		we obtain, for every $k\geq1$,
	\begin{align}\label{eq-yy-11}
		\begin{aligned}
			&\E|x_{k+1}-y_{k+1}|^{6}\\
			&\quad\leq
			\E\dist^{6}(x_{k+1},\MM)
			\mathbf{1}(x_{k+1}\in B_{\delta}(x_{*}))
			+\delta^{-6}\E|x_{k+1}-x_{*}|^{12}\\
			&\quad\leq
			\E|x_{k}-\alpha_{k+1}G_{\alpha_{k+1}}(x_{k},\zeta_{k+1})
			-y_{k}|^{6}
			\mathbf{1}(x_{k+1}\in B_{\delta}(x_{*}))
			+C\tau_{2}\alpha_{k+1}^{6}\\
			&\quad\leq
			\E|x_{k}-y_{k}|^{6}
			-4c_{3}\alpha_{k+1}\E|x_{k}-y_{k}|^{5}
			+C\tau_{2}\alpha_{k+1}^{6}\\
			&\quad\leq
			(1-31k^{-\beta})\E|x_{k}-y_{k}|^{6}
			-3c_{3}\alpha_{k+1}\E|x_{k}-y_{k}|^{5}
			+C\tau_{2}\alpha_{k+1}^{6},
		\end{aligned}
	\end{align}
		where the last inequality follows from Young's inequality,
		\begin{align*}
			31k^{-\beta}|x_{k}-y_{k}|^{6}
			&\leq
			C\alpha_{k+1}|x_{k}-y_{k}|^{4}|x_{k}-x_{*}|^{2}\\
			&\leq
			c_{3}\alpha_{k+1}|x_{k}-y_{k}|^{5}
			+C\alpha_{k+1}|x_{k}-x_{*}|^{10},
		\end{align*}
		and Lemma~\ref{lem-3}.	Let $a_{k}=k^{5\beta}$, $k\geq0$.
	Noting that
	$a_{k+1}/a_{k}
	=(1+k^{-1})^{5\beta}
	\leq(1+k^{-1})^{5}
	\leq1+31/k$,
	we have
	\begin{align}
		a_{k+1}(1-31k^{-\beta})
		\leq a_{k}
		\quad\text{for all }k\geq1.
	\end{align}
	After multiplying both sides of \eqref{eq-yy-11} by $a_{k+1}$,
	applying the above inequality, and summing over
	$k=1,\ldots,n-1$, we obtain
	\begin{align*}
		&a_{n}\E|x_{n}-y_{n}|^{6}
		+3c_{3}\sum_{k=1}^{n-1}
		\alpha_{k+1}a_{k+1}\E|x_{k}-y_{k}|^{5}\\
		&\quad\leq
		\E|x_{1}-y_{1}|^{6}
		+C\tau_{2}\sum_{k=1}^{n-1}a_{k+1}\alpha_{k+1}^{6}\\
		&\quad\leq
		\E|x_{1}-x_{*}|^{6}
		+C\tau_{2}n^{1-\beta}
		\leq C\tau_{2}n^{1-\beta}.
	\end{align*}
		This proves \eqref{eq-ya-01}.
		Replacing $n$ by $n+1$ in the preceding bound and using
		$\alpha_{k+1}a_{k+1}
		=\alpha_{0}(k+1)^{4\beta}$
		also proves \eqref{eq-ya-02}.
\end{proof}

	The next lemma provides a local bound for the geometric remainder
	arising from replacing the chord $z-y$ by its projection onto
	$T_{\MM}(y)$. Its component in the fixed tangent coordinates
	determined by $U$ is quadratic in $|z-y|$ and gains an additional
	factor $|y-x_{*}|$.

\begin{lemma}\label{lem-gamma-local}
	For any $y,z\in\MM\cap B_{4\delta}(x_{*})$, define
	\begin{align*}
		\Gamma(y,z)
		=
		U^{\top}\big\{z-y-P_{T_{\MM}(y)}(z-y)\big\}.
	\end{align*}
Then there exists an absolute constant $C>0$ such that
	\begin{align}\label{eq-kb-00}
		|\Gamma(y,z)|
		\leq CC_{\MM}^{2}|y-x_{*}||y-z|^{2}.
	\end{align}
\end{lemma}

\begin{proof}
		Fix $y,z\in\MM\cap B_{4\delta}(x_{*})$.
		Since $4\delta\leq\delta_{1}$, the line segments joining $y$
		to $z$ and $x_{*}$ to $y$ lie in $B_{\delta_{1}}(x_{*})$.
		Hence, all second derivatives of $P_{\MM}$ used below are
		bounded in operator norm by $C_{\MM}$.

		Because $P_{\MM}(y)=y$, $P_{\MM}(z)=z$, and
		$\nabla P_{\MM}(y)=P_{T_{\MM}(y)}$, Taylor's formula with
		integral remainder gives
		\begin{align*}
			z-y
			&=P_{\MM}(z)-P_{\MM}(y)\\
			&=P_{T_{\MM}(y)}(z-y)
			+\int_{0}^{1}(1-t)
			\nabla^{2}P_{\MM}\big(y+t(z-y)\big)
			[z-y,z-y]\,dt.
		\end{align*}
		Multiplying this identity by
		$U^{\top}(I_{d}-P_{T_{\MM}(y)})$ and using
		$(I_{d}-P_{T_{\MM}(y)})P_{T_{\MM}(y)}=0$, we obtain
		\begin{align*}
			\Gamma(y,z)
			&=U^{\top}(I_{d}-P_{T_{\MM}(y)})
			\int_{0}^{1}(1-t)
			\nabla^{2}P_{\MM}\big(y+t(z-y)\big)
			[z-y,z-y]\,dt.
		\end{align*}

		We next bound the matrix factor in this expression.
		Since the columns of $U$ form an orthonormal basis of
		$T_{\MM}(x_{*})$, we have
		$U^{\top}P_{T_{\MM}(x_{*})}=U^{\top}$ and
		$\|U^{\top}\|_{\op}\leq1$.
		Using $P_{T_{\MM}(y)}=\nabla P_{\MM}(y)$ and the
		fundamental theorem of calculus, we have
		\begin{align*}
			\|U^{\top}(I_{d}-P_{T_{\MM}(y)})\|_{\op}
			&=\big\|U^{\top}
			\{P_{T_{\MM}(x_{*})}-P_{T_{\MM}(y)}\}\big\|_{\op}\\
			&\leq
			\|\nabla P_{\MM}(x_{*})-\nabla P_{\MM}(y)\|_{\op}\\
			&\leq
			\int_{0}^{1}
			\|\nabla^{2}P_{\MM}\big(x_{*}+t(y-x_{*})\big)\|_{\op}
			|y-x_{*}|\,dt\\
			&\leq C_{\MM}|y-x_{*}|.
		\end{align*}
		Combining the preceding identity and bound, we obtain
		\begin{align*}
			|\Gamma(y,z)|
			&\leq
			C_{\MM}|y-x_{*}|
			\int_{0}^{1}(1-t)
			\|\nabla^{2}P_{\MM}\big(y+t(z-y)\big)\|_{\op}
			|z-y|^{2}\,dt\\
			&\leq
			\frac{1}{2}C_{\MM}^{2}|y-x_{*}||z-y|^{2}.
		\end{align*}
		This proves \eqref{eq-kb-00}.
\end{proof}

The following lemma shows that the tangential component of the first-order projection remainder
is quadratic in the increment and is locally stable with respect to both the base point
and the increment. It will be used to control $\|D_{3,n}\|_{2}$ and
$\|D_{3,n}-D_{3,n}^{\{i\}}\|_{2}$, where $D_{3,n}^{\{i\}}$ denotes the version of $D_{3,n}$ obtained
by replacing the $i$-th observation with an independent copy.
\begin{lemma}\label{lem-projection-remainder-stability}
	Define
	\begin{align*}
		\mathcal{Q}(y,h)
		=
		P_{T_{\MM}(y)}\{P_{\MM}(y+h)-y-P_{T_{\MM}(y)}h\}.
	\end{align*}
	For any $y,y'\in\MM\cap B_{4\delta}(x_{*})$
	and any $h,h'$ satisfying
	$|h|\leq 4\delta$ and
	$|h'|\leq 4\delta$, we have
	\begin{align}
		|\QQ(y,h)|
		&\leq CC_{\MM}|h|^{2},
		\label{eq-ui-01}\\
		|\QQ(y,h)-\QQ(y',h')|
		&\leq CC_{\MM}^{2}|y-y'||h|^{2}
		+CC_{\MM}(|h|+|h'|)|h-h'|.
		\label{eq-ui-02}
	\end{align}
	Moreover, if $\pi=P_{\MM}$, $P_{y}=P_{T_{\MM}(y)}$, and
	$N_{y}=I_{d}-P_{y}$, then
	\begin{align}\label{eq-ui-tn}
		\QQ(y,h)
		=
		\na^{2}\pi(y)[P_{y}h,N_{y}h]+\mathcal{T}(y,h),
	\end{align}
	where
	\begin{align}
		|\mathcal{T}(y,h)|
		&\leq CC_{\MM}|h|^{3},
		\label{eq-ui-t-size}\\
		|\mathcal{T}(y,h)-\mathcal{T}(y',h')|
		&\leq CC_{\MM}^{2}(|h|^{2}+|h'|^{2})
		\{|y-y'|+|h-h'|\}.
		\label{eq-ui-t-diff}
	\end{align}
\end{lemma}

\begin{proof}
		Since $8\delta\leq\delta_{1}$, all points and line segments at which
		we evaluate derivatives of $P_{\MM}$ below lie in
		$B_{\delta_{1}}(x_{*})$. Hence the operator norms of these derivatives
		up to order three are bounded by $C_{\MM}$.
	For notational convenience, we let
	$R(y,h):=P_{\MM}(y+h)-y-P_{y}h$.
	Then
	\begin{align*}
		\QQ(y,h)=P_{y}R(y,h).
	\end{align*}
	Since $P_{\MM}(y)=y$ and $\na P_{\MM}(y)=P_{y}$ for every
	$y\in\MM$, Taylor's formula with integral remainder gives
	\begin{align}\label{eq-ui-03}
		R(y,h)
		=
		\int_{0}^{1}(1-t)\na^{2}P_{\MM}(y+th)[h,h]\,dt.
	\end{align}
	Therefore, by the definition of $C_{\MM}$, we have
	\begin{align*}
		|R(y,h)|
		&\leq
		\int_{0}^{1}(1-t)
		\|\na^{2}P_{\MM}(y+th)\|_{\op}|h|^{2}\,dt
		\leq CC_{\MM}|h|^{2}.
	\end{align*}
	Since $\|P_{y}\|_{\op}\leq1$, it follows that
	$|\QQ(y,h)|\leq CC_{\MM}|h|^{2}$,
	which proves \eqref{eq-ui-01}.

	We next prove \eqref{eq-ui-02}. First observe that
	\begin{align}\label{eq-ui-04}
		\QQ(y,h)-\QQ(y',h')
		=
		(P_{y}-P_{y'})R(y,h)
		+P_{y'}\{R(y,h)-R(y',h')\}.
	\end{align}
	We estimate the two terms on the right-hand side separately.
	Because $P_{y}=\na P_{\MM}(y)$ and
	$P_{y'}=\na P_{\MM}(y')$, the fundamental theorem of calculus
implies that
	\begin{align*}
		\|P_{y}-P_{y'}\|_{\op}
		\leq CC_{\MM}|y-y'|.
	\end{align*}
	Combining this estimate with the bound for $R(y,h)$ obtained above
	yields
	\begin{align}\label{eq-ui-05}
		|(P_{y}-P_{y'})R(y,h)|
		\leq CC_{\MM}^{2}|y-y'||h|^{2}.
	\end{align}
	It remains to control $R(y,h)-R(y',h')$.
	We decompose it as
	\begin{align*}
		R(y,h)-R(y',h')
		=
		\{R(y,h)-R(y',h)\}
		+\{R(y',h)-R(y',h')\}.
	\end{align*}
	For the first difference, by \eqref{eq-ui-03}, we have
	\begin{align}\label{eq:projection-rem-proof-08}
		\begin{aligned}
			|R(y,h)-R(y',h)|
			&=
			\Big|\int_{0}^{1}(1-t)
			\big\{\na^{2}P_{\MM}(y+th)-\na^{2}P_{\MM}(y'+th)\big\}
			(h,h)\,dt\Big|\\
			&\leq CC_{\MM}|y-y'||h|^{2}.
		\end{aligned}
	\end{align}
	For the second difference, set $\widetilde h:=h-h'$.
	Using the definition of $R(y',\cdot)$, we have
	\begin{align*}
		R(y',h)-R(y',h')
		&=
		P_{\MM}(y'+h)-P_{\MM}(y'+h')
		-\na P_{\MM}(y')\widetilde h\\
		&=
		\int_{0}^{1}
		\big\{\na P_{\MM}(y'+h'+s\widetilde h)-\na P_{\MM}(y')\big\}
		\widetilde h\,ds.
	\end{align*}
	Consequently, we have
	\begin{align*}
		|R(y',h)-R(y',h')|
		&\leq
		CC_{\MM}\int_{0}^{1}\big\{(1-s)|h'|+s|h|\big\}\,ds\,|h-h'|\\
		&\leq CC_{\MM}(|h|+|h'|)|h-h'|.
	\end{align*}
	Combining \eqref{eq-ui-04} and \eqref{eq-ui-05} with the estimates
	above, and noting that $\|P_{y'}\|_{\op}\leq1$, we have
	\begin{align*}
		|\QQ(y,h)-\QQ(y',h')|
		&\leq
		CC_{\MM}^{2}|y-y'||h|^{2}
		+CC_{\MM}(|h|+|h'|)|h-h'|.
	\end{align*}
	This proves \eqref{eq-ui-02}.

Now, we  prove \eqref{eq-ui-tn}--\eqref{eq-ui-t-diff}.
	Since $\pi\circ\pi=\pi$, differentiating this identity twice at
	$y\in\MM$ gives, for $u,v\in\R^{d}$,
	\begin{align}\label{eq-ui-proj-id}
		\na^{2}\pi(y)[P_{y}u,P_{y}v]
		+P_{y}\na^{2}\pi(y)[u,v]
		=
		\na^{2}\pi(y)[u,v].
	\end{align}
		For every fixed $n\in N_{\MM}(y)$, the nearest-point projection
		satisfies $\pi(y+tn)=y$ for all sufficiently small $|t|$.
		Differentiating twice at $t=0$ gives
		$\na^{2}\pi(y)[n,n]=0$.
		By symmetry and polarization, it follows that
		$\na^{2}\pi(y)[n_{1},n_{2}]=0$
		for all $n_{1},n_{2}\in N_{\MM}(y)$.
		For any $h\in\R^{d}$, taking $u=v=P_yh$ in
		\eqref{eq-ui-proj-id} gives
		\begin{align*}
			P_y\na^{2}\pi(y)[P_yh,P_yh]=0.
		\end{align*}
		Taking $u=P_yh$ and $v=N_yh$ gives
		\begin{align*}
			P_y\na^{2}\pi(y)[P_yh,N_yh]
			=
			\na^{2}\pi(y)[P_yh,N_yh].
		\end{align*}
		Since $\na^{2}\pi(y)[N_yh,N_yh]=0$, expanding
		$h=P_yh+N_yh$ yields
		\begin{align}
			\frac{1}{2}P_y\na^{2}\pi(y)[h,h]
			=
			\na^{2}\pi(y)[P_yh,N_yh].
			\label{eq-ui-cross}
		\end{align}
By Taylor's formula, we have
	\begin{align*}
		\pi(y+h)-y-P_{y}h
		=
		\frac12\na^{2}\pi(y)[h,h]+\mathcal{R}_{3}(y,h),
	\end{align*}
	where
	\begin{align*}
		\mathcal{R}_{3}(y,h)
		=
		\frac12\int_{0}^{1}(1-t)^{2}
		\na^{3}\pi(y+th)[h,h,h]\,dt.
	\end{align*}
	Multiplying by $P_{y}$ and using \eqref{eq-ui-cross}, we  prove
	\eqref{eq-ui-tn}, with
	$\mathcal{T}(y,h)=P_{y}\mathcal{R}_{3}(y,h)$, and also proves
	\eqref{eq-ui-t-size}.

	Finally, we prove \eqref{eq-ui-t-diff}.
		We first fix $h$ and compare the base
		points $y$ and $y'$. By the definition of $\mathcal{R}_{3}$, we have
	$
			\mathcal{R}_{3}(y,h)
			=
			R(y,h)-\na^{2}\pi(y)[h,h]/2.
$
		Consequently, the bound for $R(y,h)$ gives
		\begin{align*}
			|\mathcal{R}_{3}(y,h)|
			&\leq
			|R(y,h)|+\frac12\|\na^{2}\pi(y)\|_{\op}|h|^{2}\leq CC_{\MM}|h|^{2}.
		\end{align*}
		Moreover, using \eqref{eq:projection-rem-proof-08} and the bound on
		$\na^{3}\pi$, we obtain
		\begin{align*}
			|\mathcal{R}_{3}(y,h)-\mathcal{R}_{3}(y',h)|
			&\leq
			|R(y,h)-R(y',h)|
		+
			\frac12\|\na^{2}\pi(y)-\na^{2}\pi(y')\|_{\op}|h|^{2}\\
			&\leq CC_{\MM}|y-y'||h|^{2}.
		\end{align*}
		Since
		$\mathcal{T}(y,h)=P_{y}\mathcal{R}_{3}(y,h)$, it follows that
		\begin{align*}
			|\mathcal{T}(y,h)-\mathcal{T}(y',h)|
			&\leq
			\|P_{y}-P_{y'}\|_{\op}|\mathcal{R}_{3}(y,h)|
		+
			\|P_{y'}\|_{\op}
			|\mathcal{R}_{3}(y,h)-\mathcal{R}_{3}(y',h)|\\
			&\leq CC_{\MM}^{2}|y-y'||h|^{2}.
		\end{align*}
		We next fix $y'$ and compare the increments $h$ and $h'$.
		Differentiating
		\begin{align*}
			\mathcal{R}_{3}(y',h)
			=
			\pi(y'+h)-\pi(y')-\na\pi(y')h
			-\frac12\na^{2}\pi(y')[h,h]
		\end{align*}
		with respect to $h$, and applying Taylor's formula to $\na\pi$, gives,
		for every $v\in\R^{d}$,
		\begin{align*}
			\na_{h}\mathcal{R}_{3}(y',h)[v]
			&=
			\{\na\pi(y'+h)-\na\pi(y')\}v
			-\na^{2}\pi(y')[h,v]\\
			&=
			\int_{0}^{1}(1-t)\na^{3}\pi(y'+th)[h,h,v]\,dt.
		\end{align*}
		Hence
		\begin{align*}
			\|\na_{h}\mathcal{R}_{3}(y',h)\|_{\op}
			\leq CC_{\MM}|h|^{2}.
		\end{align*}
		Integrating along the segment from $h'$ to $h$ and using
		$\|P_{y'}\|_{\op}\leq1$, we obtain
		\begin{align*}
			|\mathcal{T}(y',h)-\mathcal{T}(y',h')|
			&\leq
			|\mathcal{R}_{3}(y',h)-\mathcal{R}_{3}(y',h')|\\
			&\leq
			CC_{\MM}|h-h'|
			\int_{0}^{1}|h'+s(h-h')|^{2}\,ds\\
			&\leq
			CC_{\MM}(|h|^{2}+|h'|^{2})|h-h'|.
		\end{align*}
		Combining the last estimate with the bound for the change in the
		base point, and using $C_{\MM}\geq1$, we conclude that
		\begin{align*}
			|\mathcal{T}(y,h)-\mathcal{T}(y',h')|
			&\leq
			|\mathcal{T}(y,h)-\mathcal{T}(y',h)|
			+|\mathcal{T}(y',h)-\mathcal{T}(y',h')|\\
			&\leq
			CC_{\MM}^{2}|y-y'||h|^{2}
			+CC_{\MM}(|h|^{2}+|h'|^{2})|h-h'|\\
			&\leq
			CC_{\MM}^{2}(|h|^{2}+|h'|^{2})
			\{|y-y'|+|h-h'|\}.
		\end{align*}
		This proves \eqref{eq-ui-t-diff}.
\end{proof}

We next apply Lemma~\ref{lem-projection-remainder-stability} to the
projection error along the iterates. The following lemma provides a
moment bound for its replacement difference and an accumulated bound
for the corresponding conditional expectations.
For $k\geq0$, let $T_{k}=T_{\MM}(y_{k})$ and define
\begin{align}\label{eq-qq-01}
	\Pi_{k}
	=\alpha_{k+1}^{-1}U^{\top}P_{T_{k}}
	(y_{k+1}-x_{k+1}+x_{k}-y_{k}).
\end{align}
For each $i\geq1$, the superscript $\{i\}$ denotes the corresponding
quantity obtained by replacing $\xi_{i}$ with an independent copy
$\xi_{i}'$, independent of $x_{0}$ and all the original innovations,
while keeping $x_{0}$ and the other innovations unchanged. For $k\geq i$,
let
\begin{align*}
	\mathscr{F}_{k}^{\{i\}}
	=\sigma(x_{0},\xi_{1},\ldots,\xi_{k},\xi_{i}').
\end{align*}

\begin{lemma}\label{lem-pi-replacement}
	Suppose that Assumptions~\ref{assu-1}--\ref{assu-5} hold, and let
	$\alpha_{k}=\alpha_{0}k^{-\beta}$ for $k\geq1$, where
	$\alpha_{0}>0$ and $\beta\in(2/3,1)$.
	Then there exists a constant $C>0$ such that
	\begin{align}\label{eq-qq-03}
		\|\Pi_{k}-\Pi_{k}^{\{i\}}\|_{2}
		\leq CC_{\MM}^{3}\tau_{1}^{4}k^{-\beta},
		\qquad 1\leq i\leq k,
	\end{align}
	and, for every integer $n\geq2$ and every $1\leq i\leq n-2$,
	\begin{align}\label{eq-qq-04}
		\sum_{k=i}^{n-2}
		\big\|\E\big(\Pi_{k}-\Pi_{k}^{\{i\}}
		\mid\mathscr{F}_{k}^{\{i\}}\big)
		\big\|_{2}
		\leq CC_{\MM}^{3}\tau_{1}^{4}i^{1-3\beta/2},
	\end{align}
	where $\tau_{1}=(1+\delta+\tau_{0}+\tau)
	\{1+\delta^{-1}(1+\tau_{0}+\tau)\}$,
	$C_{\MM}$ is defined in \eqref{eq-defi-cm}.
\end{lemma}

\begin{proof}
	Fix $i\geq1$. We first obtain estimates for the individual terms
	that will be used to prove both \eqref{eq-qq-03} and
	\eqref{eq-qq-04}.
	Since $x_{*}\in\MM$, the definition of $y_{k}$ gives
	$|x_{k}-y_{k}|\leq|x_{k}-x_{*}|$ and
	$|y_{k}-x_{*}|\leq2|x_{k}-x_{*}|$.
	If $x_{k}\in B_{2\delta}(x_{*})$, the latter inequality gives
	$|y_{k}-x_{*}|<4\delta$; otherwise, $y_{k}=x_{*}$.
	Consequently, we have
	\begin{align}\label{eq-qq-05}
		y_{k}\in\MM\cap B_{4\delta}(x_{*}),\qquad
		|x_{k}-y_{k}|\leq|x_{k}-x_{*}|,\qquad
		|y_{k}-x_{*}|\leq2|x_{k}-x_{*}|.
	\end{align}
	The same statements hold for the replaced trajectory.
	For $k\geq i$, set
	\begin{align*}
		\mathcal{E}_{1}
		&=\{x_{k},x_{k}^{\{i\}}\in B_{\delta}(x_{*})\},\qquad
		\mathcal{E}_{2}
		=\mathcal{E}_{1}\cap
		\{x_{k+1},x_{k+1}^{\{i\}}\in B_{\delta}(x_{*})\},
	\end{align*}
	where, for notational convenience, the dependence on $k$ and $i$
	is suppressed. The event $\mathcal{E}_{1}$ is
	$\mathscr{F}_{k}^{\{i\}}$-measurable, whereas $\mathcal{E}_{2}$
	need not be.
	By Assumptions~\ref{assu-3}--\ref{assu-4} and
	Lemma~\ref{lem-3}, we have
	\begin{align}\label{eq-qq-06}
		\begin{aligned}
			\|\zeta_{k+1}\|_{12}
			+\|G_{\alpha_{k+1}}(x_{k},\zeta_{k+1})\|_{12}
			\leq C(1+\tau_{0}+\tau).
		\end{aligned}
	\end{align}
	Combining the local derivative bounds for $P_{\MM}$ with
	Lemma~\ref{lem-UpperB-xn-xni-fourth}, we obtain
	\begin{align}\label{eq-qq-08}
		\begin{aligned}
			\|(y_{k}-y_{k}^{\{i\}})
			\mathbf{1}(\mathcal{E}_{1})\|_{6}+
			\big\|\{(x_{k}-y_{k})
			-(x_{k}^{\{i\}}-y_{k}^{\{i\}})\}
			\mathbf{1}(\mathcal{E}_{1})\big\|_{6}
			&\leq CC_{\MM}\tau_{1}k^{-\beta},\\
			\big\|\{(x_{k+1}-y_{k})
			-(x_{k+1}^{\{i\}}-y_{k}^{\{i\}})\}
			\mathbf{1}(\mathcal{E}_{2})\big\|_{6}
			&\leq CC_{\MM}\tau_{1}k^{-\beta},
		\end{aligned}
	\end{align}
	where, in the second inequality, we used the identity
	$
	(x_{k+1}-y_{k})-(x_{k+1}^{\{i\}}-y_{k}^{\{i\}})
	=(x_{k+1}-x_{k+1}^{\{i\}})
	-(y_{k}-y_{k}^{\{i\}}).
	$
	All moment bounds for a single trajectory also hold for the replaced
	trajectory by equality in distribution.
	Since the current innovation is the same for $k\geq i$,
	$
	|\zeta_{k+1}-\zeta_{k+1}^{\{i\}}|
	\leq c_{8}|x_{k}-x_{k}^{\{i\}}|.
	$

	We first control the contribution of $\mathcal{E}_{2}^{c}$.
	By  Markov's inequality and Lemma \ref{lem-3}, we have
	\begin{align}\label{eq-qq-09}
		\P(\mathcal{E}_{2}^{c})
		\leq C\delta^{-12}(1+\tau_{0}+\tau)^{12}k^{-6\beta}.
	\end{align}
	By \eqref{eq-qq-01} and \eqref{eq-qq-05}, we have
	\begin{align*}
		|\Pi_{k}|
		&\leq\alpha_{k+1}^{-1}
		\{|x_{k+1}-y_{k+1}|+|x_{k}-y_{k}|\}
		\leq\alpha_{k+1}^{-1}
		\{|x_{k+1}-x_{*}|+|x_{k}-x_{*}|\}.
	\end{align*}
		Applying H\"older's inequality with exponents $6$ and $3$
		to each of the four state terms, we obtain
		\begin{align}\label{eq-qq-10}
			\begin{aligned}
			\|(\Pi_{k}-\Pi_{k}^{\{i\}})
				\mathbf{1}(\mathcal{E}_{2}^{c})\|_{2}
				&\leq C\alpha_{k+1}^{-1}
				(1+\tau_{0}+\tau)k^{-\beta/2}
				\P(\mathcal{E}_{2}^{c})^{1/3}\\
				&\quad\leq
				C\delta^{-4}(1+\tau_{0}+\tau)^{5}k^{-3\beta/2}
				\leq C\tau_{1}^{4}k^{-3\beta/2},
			\end{aligned}
		\end{align}
		where we used
		$\tau_{1}=\delta^{-1}(1+\delta+\tau_{0}+\tau)^{2}$ and
		$\delta^{-4}(1+\tau_{0}+\tau)^{5}\leq\tau_{1}^{4}$.

	We next consider $\mathcal{E}_{2}$. Recall that, in Lemma~\ref{lem-projection-remainder-stability}, we defined
	\begin{align*}
		\mathcal{Q}(y,h)
		=
		P_{T_{\MM}(y)}
		\big\{
		P_{\MM}(y+h)-y-P_{T_{\MM}(y)}h
		\big\}.
	\end{align*}
	Therefore, on the event
	$\{x_{k},x_{k+1}\in B_{\delta}(x_{*})\}$, we have
	\begin{align*}
		P_{T_{k}}
		(y_{k+1}-x_{k+1}+x_{k}-y_{k})&=
		P_{T_{k}}(y_{k+1}-x_{k+1})
		=
		P_{T_{k}}\big\{
		y_{k+1}-y_{k}
		-P_{T_{k}}(x_{k+1}-y_{k})
		\big\}\\
		&=
		P_{T_{k}}\big\{
		P_{\MM}(y_{k}+x_{k+1}-y_{k})
		-y_{k}
		-P_{T_{k}}(x_{k+1}-y_{k})
		\big\}\\
		&=
		\mathcal{Q}(y_{k},x_{k+1}-y_{k}),
	\end{align*}
	where the second equality follows from $P_{T_{k}}^{2}=P_{T_{k}}$. Hence, on $\mathcal{E}_{2}$, we have
	\begin{align*}
		\Pi_{k}=\alpha_{k+1}^{-1}U^{\top}
		\QQ(y_{k},x_{k+1}-y_{k}).
	\end{align*}
	On this event, $P_{T_{k}}(x_{k}-y_{k})=0$ and
	\begin{align*}
		P_{T_{k}}(x_{k+1}-y_{k})
		&=-\alpha_{k+1}P_{T_{k}}
		G_{\alpha_{k+1}}(x_{k},\zeta_{k+1}),\\
		(I_{d}-P_{T_{k}})(x_{k+1}-y_{k})
		&=x_{k}-y_{k}-\alpha_{k+1}(I_{d}-P_{T_{k}})
		G_{\alpha_{k+1}}(x_{k},\zeta_{k+1}).
	\end{align*}
	Also, $|y_{k}-x_{*}|\leq2\delta$ and
	$|x_{k+1}-y_{k}|\leq3\delta$. Hence, by \eqref{eq-ui-tn}, we have
	\begin{align}\label{eq-cc-j2-05}
		\begin{aligned}
			\Pi_{k}
			&=
			-U^{\top}\nabla^{2}P_{\MM}(y_{k})
			\big[
			P_{T_{k}}G_{\alpha_{k+1}}(x_{k},\zeta_{k+1}),
			x_{k}-y_{k}
			\big]\\
			&\quad+
			\alpha_{k+1}U^{\top}\nabla^{2}P_{\MM}(y_{k})
			\big[
			P_{T_{k}}G_{\alpha_{k+1}}(x_{k},\zeta_{k+1}),
			(I_{d}-P_{T_{k}})
			G_{\alpha_{k+1}}(x_{k},\zeta_{k+1})
			\big]\\
			&\quad+
			\alpha_{k+1}^{-1}
			U^{\top}\mathcal{T}(y_{k},x_{k+1}-y_{k})\\
			&:=\Pi_{k,1}+\Pi_{k,2}+\Pi_{k,3},
		\end{aligned}
	\end{align}
	where $\mathcal{T}(y,h)$ is defined in
	Lemma~\ref{lem-projection-remainder-stability}.
	By \eqref{eq-qq-10}, \eqref{eq-cc-j2-05}, the triangle inequality,
	and conditional Jensen's inequality, it suffices to establish the
	following estimates to prove \eqref{eq-qq-03} and
	\eqref{eq-qq-04}:
	\begin{align}
		\|(\Pi_{k,1}-\Pi_{k,1}^{\{i\}})
		\mathbf{1}(\mathcal{E}_{2})\|_{2}
		&\leq CC_{\MM}^{3}\tau_{1}^{4}k^{-\beta}, \label{eq-qa-01}\\
		\big\|\E\big[(\Pi_{k,1}-\Pi_{k,1}^{\{i\}})
		\mathbf{1}(\mathcal{E}_{2})
		\mid\mathscr{F}_{k}^{\{i\}}\big]\big\|_{2}
		&\leq CC_{\MM}^{3}\tau_{1}^{4}k^{-3\beta/2}, \label{eq-qa-02}\\
		\|(\Pi_{k,3}-\Pi_{k,3}^{\{i\}})
		\mathbf{1}(\mathcal{E}_{2})\|_{2}
		&\leq CC_{\MM}^{3}\tau_{1}^{4}k^{-\beta},\label{eq-qa-03}\\
\sum_{k=i}^{\infty}
			\|(\Pi_{k,3}-\Pi_{k,3}^{\{i\}})
			\mathbf{1}(\mathcal{E}_{2})\|_{2}
		&\leq CC_{\MM}^{3}\tau_{1}^{4}i^{1-3\beta/2},
		\label{eq-qq-30}\\
		\|(\Pi_{k,2}-\Pi_{k,2}^{\{i\}})
		\mathbf{1}(\mathcal{E}_{2})\|_{2}
		&\leq CC_{\MM}^{3}\tau_{1}^{4}k^{-\beta},\label{eq-qa-04}\\
		\sum_{k=i}^{\infty}
		\big\|\E\big[(\Pi_{k,2}-\Pi_{k,2}^{\{i\}})
		\mathbf{1}(\mathcal{E}_{2})
		\mid\mathscr{F}_{k}^{\{i\}}\big]\big\|_{2}
		& \leq CC_{\MM}^{3}\tau_{1}^{4}i^{1-3\beta/2}. \label{eq-qa-05}
	\end{align}
We prove these estimates in five steps.

	\smallskip
	\noindent
	{\it (i) Proof of \eqref{eq-qa-01}}.
	On $\mathcal{E}_{1}$, the definition
	of $\mathcal{R}_{\alpha}$ gives
	\begin{align}\label{eq-qq-14}
		\begin{aligned}
			\Pi_{k,1}
			&=-U^{\top}\nabla^{2}P_{\MM}(y_{k})
			[P_{T_{k}}\zeta_{k+1},x_{k}-y_{k}]\\
			&\quad-U^{\top}\nabla^{2}P_{\MM}(y_{k})
			\big[F_{\MM}(y_{k})
			+\mathcal{R}_{\alpha_{k+1}}(x_{k},\zeta_{k+1}),
			x_{k}-y_{k}\big].
		\end{aligned}
	\end{align}
	We first estimate the difference of the second terms in this
	identity. Observe that
	\begin{align*}
		J_{1}:&=U^{\top}\nabla^{2}P_{\MM}(y_{k})
		\big[
		F_{\MM}(y_{k})
		+\mathcal{R}_{\alpha_{k+1}}(x_{k},\zeta_{k+1}),
		x_{k}-y_{k}
		\big]\\
		&\quad-
		U^{\top}\nabla^{2}P_{\MM}(y_{k}^{\{i\}})
		\big[
		F_{\MM}(y_{k}^{\{i\}})
		+\mathcal{R}_{\alpha_{k+1}}
		(x_{k}^{\{i\}},\zeta_{k+1}^{\{i\}}),
		x_{k}^{\{i\}}-y_{k}^{\{i\}}
		\big]\\
		&=J_{11}+J_{12}+J_{13}+J_{14},
	\end{align*}
	where
	\begin{align*}
		J_{11}&=U^{\top}
		\big\{
		\nabla^{2}P_{\MM}(y_{k})
		-\nabla^{2}P_{\MM}(y_{k}^{\{i\}})
		\big\}
		\big[
		F_{\MM}(y_{k})
		+\mathcal{R}_{\alpha_{k+1}}(x_{k},\zeta_{k+1}),
		x_{k}-y_{k}
		\big]\\
		J_{12}&=		U^{\top}\nabla^{2}P_{\MM}(y_{k}^{\{i\}})
		\big[
		F_{\MM}(y_{k})-F_{\MM}(y_{k}^{\{i\}}),
		x_{k}-y_{k}
		\big]\\
		J_{13}&=U^{\top}\nabla^{2}P_{\MM}(y_{k}^{\{i\}})
		\big[
		\mathcal{R}_{\alpha_{k+1}}(x_{k},\zeta_{k+1})
		-\mathcal{R}_{\alpha_{k+1}}
		(x_{k}^{\{i\}},\zeta_{k+1}^{\{i\}}),
		x_{k}-y_{k}
		\big]\\
		J_{14}&=U^{\top}\nabla^{2}P_{\MM}(y_{k}^{\{i\}})
		\Big[
		F_{\MM}(y_{k}^{\{i\}})
		+\mathcal{R}_{\alpha_{k+1}}
		(x_{k}^{\{i\}},\zeta_{k+1}^{\{i\}}),
		(x_{k}-y_{k})
		-(x_{k}^{\{i\}}-y_{k}^{\{i\}})
		\Big].
	\end{align*}
	For $J_{11}$, we first establish the following residual bound.
	By Assumption~\ref{assu-2}(i), \eqref{eq-qq-05}--\eqref{eq-qq-06},
	Lemma~\ref{lem-3}, and H\"older's inequality, we have
	\begin{align}\label{eq-qq-15}
		\begin{aligned}
			&\|\mathcal{R}_{\alpha_{k+1}}(x_{k},\zeta_{k+1})
			\mathbf{1}(\mathcal{E}_{1})\|_{3}\\
			&\quad\leq C\|1+|x_{k}-x_{*}|^{2}
			+|\zeta_{k+1}|^{2}\|_{6}
			\{\|x_{k}-y_{k}\|_{6}+\alpha_{k+1}\}
			\leq C\tau_{1}^{3}k^{-\beta/2}.
		\end{aligned}
	\end{align}
	Moreover, on $\mathcal{E}_{1}$,
	$
	F_{\MM}(y_{k})
	+\mathcal{R}_{\alpha_{k+1}}(x_{k},\zeta_{k+1})
	=P_{T_{k}}\big\{
	G_{\alpha_{k+1}}(x_{k},\zeta_{k+1})-\zeta_{k+1}
	\big\}.
	$
	Hence, by \eqref{eq-qq-06}, we have
	\begin{align*}
		&\big\|
		\{F_{\MM}(y_{k})
		+\mathcal{R}_{\alpha_{k+1}}(x_{k},\zeta_{k+1})\}
		\mathbf{1}(\mathcal{E}_{1})
		\big\|_{6}\leq
		\|G_{\alpha_{k+1}}(x_{k},\zeta_{k+1})\|_{6}
		+\|\zeta_{k+1}\|_{6}
		\leq C\tau_{1}.
	\end{align*}
	Using this bound, the local Lipschitz property of
	$\nabla^{2}P_{\MM}$, H\"older's inequality,
	\eqref{eq-qq-05}, \eqref{eq-qq-08}, and Lemma~\ref{lem-3}, we have
	\begin{align*}
		\|J_{11}\mathbf{1}(\mathcal{E}_{1})\|_{2}
		&\leq
		CC_{\MM}
		\|(y_{k}-y_{k}^{\{i\}})
		\mathbf{1}(\mathcal{E}_{1})\|_{6}
		\big\|
		\{F_{\MM}(y_{k})
		+\mathcal{R}_{\alpha_{k+1}}(x_{k},\zeta_{k+1})\}
		\mathbf{1}(\mathcal{E}_{1})
		\big\|_{6}
		\|x_{k}-y_{k}\|_{6}\\
		&\leq
		CC_{\MM}^{2}\tau_{1}^{3}k^{-3\beta/2}.
	\end{align*}
	For $J_{12}$, by the local Lipschitz property of $F_{\MM}$,
	H\"older's inequality, \eqref{eq-qq-08} and Lemma \ref{lem-3}, we have
	\begin{align*}
		\|J_{12}\mathbf{1}(\mathcal{E}_{1})\|_{2}
		&\leq
		CC_{\MM}
		\|(y_{k}-y_{k}^{\{i\}})
		\mathbf{1}(\mathcal{E}_{1})\|_{3}
		\|x_{k}-y_{k}\|_{6}\leq
		CC_{\MM}^{2}\tau_{1}^{2}k^{-3\beta/2}.
	\end{align*}
	For $J_{13}$, 	by Assumption~\ref{assu-2}(ii), Lemmas \ref{lem-3} and \ref{lem-UpperB-xn-xni-fourth},  and H\"older's inequality with
	exponents $(6,12,6,12)$, we have
	\begin{align}\label{eq-qq-16}
		\begin{aligned}
			&\Big\|\big|
			\mathcal{R}_{\alpha_{k+1}}(x_{k},\zeta_{k+1})
			-\mathcal{R}_{\alpha_{k+1}}
			(x_{k}^{\{i\}},\zeta_{k+1}^{\{i\}})
			\big|\,|x_{k}-y_{k}|
			\mathbf{1}(\mathcal{E}_{1})\Big\|_{2}\\
			&\quad\leq
			C\|1+|\zeta_{k+1}|^{2}
			+|\zeta_{k+1}^{\{i\}}|^{2}\|_{6}
			\|\alpha_{k+1}+|x_{k}-x_{*}|
			+|x_{k}^{\{i\}}-x_{*}|\|_{12}\\
			&\qquad\quad\times
			\|x_{k}-x_{k}^{\{i\}}\|_{6}
			\|x_{k}-y_{k}\|_{12}\\
			&\quad\leq C\tau_{1}^{5}k^{-2\beta}.
		\end{aligned}
	\end{align} Hence, we have
	\begin{align*}
		\|J_{13}\mathbf{1}(\mathcal{E}_{1})\|_{2}
		&\leq
		CC_{\MM}
		\Big\|
		\big|
		\mathcal{R}_{\alpha_{k+1}}(x_{k},\zeta_{k+1})
		-\mathcal{R}_{\alpha_{k+1}}
		(x_{k}^{\{i\}},\zeta_{k+1}^{\{i\}})
		\big|
		|x_{k}-y_{k}|
		\mathbf{1}(\mathcal{E}_{1})
		\Big\|_{2}\\
		&\leq
		CC_{\MM}\tau_{1}^{5}k^{-2\beta}.
	\end{align*}
		Moreover, the definition of $\mathcal{R}_{\alpha}$,
		\eqref{eq-qq-05}--\eqref{eq-qq-06}, the local Lipschitz property
		of $F_{\MM}$, and $F_{\MM}(x_{*})=0$ give
		\begin{align*}
			&\|\mathcal{R}_{\alpha_{k+1}}(x_{k},\zeta_{k+1})
			\mathbf{1}(\mathcal{E}_{1})\|_{6}
			+\|\mathcal{R}_{\alpha_{k+1}}
			(x_{k}^{\{i\}},\zeta_{k+1}^{\{i\}})
			\mathbf{1}(\mathcal{E}_{1})\|_{6}
			\leq C\tau_{1}.
		\end{align*}
		Hence, by H\"older's inequality with exponents $6$ and $3$,
		\begin{align*}
			\|J_{13}\mathbf{1}(\mathcal{E}_{1})\|_{2}
			&\leq CC_{\MM}
			\Big\{
			\|\mathcal{R}_{\alpha_{k+1}}(x_{k},\zeta_{k+1})
			\mathbf{1}(\mathcal{E}_{1})\|_{6}+
			\|\mathcal{R}_{\alpha_{k+1}}
			(x_{k}^{\{i\}},\zeta_{k+1}^{\{i\}})
			\mathbf{1}(\mathcal{E}_{1})\|_{6}
			\Big\}\|x_{k}-y_{k}\|_{3}\\
			&\leq CC_{\MM}\tau_{1}^{2}k^{-\beta/2}.
		\end{align*}
		Since both this bound and the bound obtained from
		\eqref{eq-qq-16} hold for the same quantity, we obtain
		\begin{align*}
			\|J_{13}\mathbf{1}(\mathcal{E}_{1})\|_{2}
			&\leq CC_{\MM}
			(\tau_{1}^{2}k^{-\beta/2})^{1/3}
			(\tau_{1}^{5}k^{-2\beta})^{2/3}=CC_{\MM}\tau_{1}^{4}k^{-3\beta/2}.
		\end{align*}
	For $J_{14}$, using the local Lipschitz property of $F_{\MM}$,
	$F_{\MM}(x_{*})=0$, \eqref{eq-qq-08}, \eqref{eq-qq-15},
	Lemma~\ref{lem-3}, and H\"older's inequality, we have
	\begin{align*}
		\|J_{14}\mathbf{1}(\mathcal{E}_{1})\|_{2}
		&\leq
		CC_{\MM}
		\big\|
		\{F_{\MM}(y_{k}^{\{i\}})
		+\mathcal{R}_{\alpha_{k+1}}
		(x_{k}^{\{i\}},\zeta_{k+1}^{\{i\}})\}
		\mathbf{1}(\mathcal{E}_{1})
		\big\|_{3}\\
		&\qquad\times
		\big\|
		\{(x_{k}-y_{k})
		-(x_{k}^{\{i\}}-y_{k}^{\{i\}})\}
		\mathbf{1}(\mathcal{E}_{1})
		\big\|_{6}\\
		&\leq
		CC_{\MM}^{2}\tau_{1}^{4}k^{-3\beta/2}.
	\end{align*}
Combining the preceding four estimates, we obtain
	\begin{align}\label{eq-qq-17}
		\big\|
		J_{1}
		\mathbf{1}(\mathcal{E}_{1})
		\big\|_{2}
		\leq
		CC_{\MM}^{2}\tau_{1}^{4}k^{-3\beta/2}.
	\end{align}

	For the difference of the first terms in \eqref{eq-qq-14},
	the local Lipschitz bounds for $\nabla^{2}P_{\MM}$ and the
	tangent projections give
	\begin{align*}
		&\Big\|\Big\{
		U^{\top}\nabla^{2}P_{\MM}(y_{k})
		[P_{T_{k}}\zeta_{k+1},x_{k}-y_{k}]-
		U^{\top}\nabla^{2}P_{\MM}(y_{k}^{\{i\}})
		[P_{T_{k}^{\{i\}}}\zeta_{k+1}^{\{i\}},
		x_{k}^{\{i\}}-y_{k}^{\{i\}}]
		\Big\}\mathbf{1}(\mathcal{E}_{1})\Big\|_{2}\\
		&\quad\leq CC_{\MM}^{2}
		\|(y_{k}-y_{k}^{\{i\}})
		\mathbf{1}(\mathcal{E}_{1})\|_{6}
		\|\zeta_{k+1}\|_{6}\|x_{k}-y_{k}\|_{6}\\
		&\qquad+CC_{\MM}
		\|\zeta_{k+1}-\zeta_{k+1}^{\{i\}}\|_{3}
		\|x_{k}-y_{k}\|_{6}\\
		&\qquad+CC_{\MM}\|\zeta_{k+1}^{\{i\}}\|_{3}
		\big\|\{(x_{k}-y_{k})
		-(x_{k}^{\{i\}}-y_{k}^{\{i\}})\}
		\mathbf{1}(\mathcal{E}_{1})\big\|_{6}\\
		&\quad\leq CC_{\MM}^{3}\tau_{1}^{3}k^{-3\beta/2}
		+CC_{\MM}^{2}\tau_{1}^{2}k^{-\beta}.
	\end{align*}
	Combining this inequality with  \eqref{eq-qq-17}, we obtain \eqref{eq-qa-01}.

	\smallskip
	\noindent
	{\it (ii) Proof of \eqref{eq-qa-02}.}
	In view of \eqref{eq-qq-17} and conditional Jensen's inequality,
	it suffices to consider the difference of the first terms in
	\eqref{eq-qq-14}.
	By independence of $\xi_{k+1}$ from the original past and
	Assumption~\ref{assu-4}, the deterministic function
	$x\mapsto\E g(x,\xi_{k+1})$ vanishes almost surely at $x_{k}$.
	It also vanishes almost surely at $x_{k}^{\{i\}}$, because the two
	states have the same distribution. Since $\xi_{k+1}$ is independent
	of $\mathscr{F}_{k}^{\{i\}}$ and has mean zero, we obtain
	\begin{align*}
		\E(\zeta_{k+1}\mid\mathscr{F}_{k}^{\{i\}})
		=\E(\zeta_{k+1}^{\{i\}}
		\mid\mathscr{F}_{k}^{\{i\}})=0.
	\end{align*}
	Since $\mathcal{E}_{1}$ and the states of both trajectories at
	time $k$ are $\mathscr{F}_{k}^{\{i\}}$-measurable, we have
	\begin{align}\label{eq-qq-18.1}
		&\E\Big[
		U^{\top}\nabla^{2}P_{\MM}(y_{k})
		[P_{T_{k}}\zeta_{k+1},x_{k}-y_{k}]
		\mathbf{1}(\mathcal{E}_{1})
		\mid\mathscr{F}_{k}^{\{i\}}\Big]=0,
	\end{align}
	and likewise for the replaced trajectory.
	To pass from $\mathcal{E}_{1}$ to $\mathcal{E}_{2}$, use
	$\mathbf{1}(\mathcal{E}_{2})
	=\mathbf{1}(\mathcal{E}_{1})
	-\mathbf{1}(\mathcal{E}_{1}\setminus\mathcal{E}_{2})$.
	By H\"older's inequality and \eqref{eq-qq-09}, we have
	\begin{align}\label{eq-qq-19}
		\begin{aligned}
			&\big\|U^{\top}\nabla^{2}P_{\MM}(y_{k})
			[P_{T_{k}}\zeta_{k+1},x_{k}-y_{k}]
			\mathbf{1}(\mathcal{E}_{1}\setminus\mathcal{E}_{2})
			\big\|_{2}\\
			&\quad\leq CC_{\MM}\|\zeta_{k+1}\|_{6}
			\|x_{k}-y_{k}\|_{6}
			\P(\mathcal{E}_{2}^{c})^{1/6}
			\leq CC_{\MM}\tau_{1}^{2}k^{-3\beta/2}.
		\end{aligned}
	\end{align}
	The same bound holds for the replaced trajectory.
	Combining \eqref{eq-qq-17}, \eqref{eq-qq-18.1}, and
	\eqref{eq-qq-19} with conditional Jensen's inequality, we obtain
	\eqref{eq-qa-02}.

	\smallskip
	\noindent
		{\it (iii) Proof of \eqref{eq-qa-03} and \eqref{eq-qq-30}.}
		We first give two bounds for the increment $x_{k+1}-y_{k}$.
		By the update rule \eqref{eq-aa-02}, \eqref{eq-qq-05}--\eqref{eq-qq-06},
		and Lemma~\ref{lem-3}, we have
		\begin{align*}
			\|x_{k+1}-y_{k}\|_{6}
			&\leq \|x_{k}-y_{k}\|_{6}
			+\alpha_{k+1}
			\|G_{\alpha_{k+1}}(x_{k},\zeta_{k+1})\|_{6}\leq C\tau_{1}k^{-\beta/2}.
		\end{align*}
		Using instead the sixth moment bound in Lemma~\ref{lem-A6},
		we also obtain
		\begin{align*}
			\|x_{k+1}-y_{k}\|_{6}
			&\leq C\tau_{2}^{1/6}k^{1/6-\beta}
			+C(1+\tau_{0}+\tau)k^{-\beta}\leq C\tau_{1}^{5/3}k^{1/6-\beta},
		\end{align*}
		where $\tau_{2}$ is defined in Lemma~\ref{lem-A6} and satisfies
		$\tau_{2}\leq\tau_{1}^{10}$.
		Both bounds hold for the replaced trajectory as well.
		By \eqref{eq-ui-t-diff} and H\"older's inequality, we have
		\begin{align}\label{eq-qq-21}
			\begin{aligned}
				&\|(\Pi_{k,3}-\Pi_{k,3}^{\{i\}})
				\mathbf{1}(\mathcal{E}_{2})\|_{2}\\
				&\quad\leq CC_{\MM}^{2}\alpha_{k+1}^{-1}
				\{\|x_{k+1}-y_{k}\|_{6}^{2}
				+\|x_{k+1}^{\{i\}}-y_{k}^{\{i\}}\|_{6}^{2}\}\\
				&\qquad\quad\times\Big\{
				\|(y_{k}-y_{k}^{\{i\}})
				\mathbf{1}(\mathcal{E}_{2})\|_{6}+
				\big\|\{(x_{k+1}-y_{k})
				-(x_{k+1}^{\{i\}}-y_{k}^{\{i\}})\}
				\mathbf{1}(\mathcal{E}_{2})\big\|_{6}\Big\}.
			\end{aligned}
		\end{align}
		Using the first increment bound and \eqref{eq-qq-08}, we obtain
		\begin{align*}
			\|(\Pi_{k,3}-\Pi_{k,3}^{\{i\}})
			\mathbf{1}(\mathcal{E}_{2})\|_{2}
			&\leq CC_{\MM}^{3}k^{\beta}
			(\tau_{1}^{2}k^{-\beta})(\tau_{1}k^{-\beta})=CC_{\MM}^{3}\tau_{1}^{3}k^{-\beta}.
		\end{align*}
		Since $\tau_{1}\geq1$, this proves \eqref{eq-qa-03}.

		To prove \eqref{eq-qq-30}, we retain the exponential decay
		of the replacement difference. By \eqref{eq:loo-exponential}
	 in the proof of
		Lemma~\ref{lem-UpperB-xn-xni-fourth}, there exists $c>0$ such that
		\begin{align*}
			\|x_{k}-x_{k}^{\{i\}}\|_{6}
			\leq C\tau_{1}\alpha_{i}
			\exp\bigg\{-c\sum_{j=i+1}^{k}\alpha_{j}\bigg\},
			\qquad k\geq i.
		\end{align*}
		Using the local derivative bound for $P_{\MM}$ and the identity
		following \eqref{eq-qq-08}, we therefore have
		\begin{align}\label{eq-qq-31}
			\begin{aligned}
				&\|(y_{k}-y_{k}^{\{i\}})
				\mathbf{1}(\mathcal{E}_{2})\|_{6}+
				\big\|\{(x_{k+1}-y_{k})
				-(x_{k+1}^{\{i\}}-y_{k}^{\{i\}})\}
				\mathbf{1}(\mathcal{E}_{2})\big\|_{6}\\
				&\qquad\leq CC_{\MM}\tau_{1}\alpha_{i}
				\exp\bigg\{-c\sum_{j=i+1}^{k}\alpha_{j}\bigg\}.
			\end{aligned}
		\end{align}
		Substituting \eqref{eq-qq-31} and each of the two increment
		bounds into \eqref{eq-qq-21}, we obtain
		\begin{align}\label{eq-qq-32}
			\begin{aligned}
				&\|(\Pi_{k,3}-\Pi_{k,3}^{\{i\}})
				\mathbf{1}(\mathcal{E}_{2})\|_{2}\leq CC_{\MM}^{3}\alpha_{i}
				\exp\bigg\{-c\sum_{j=i+1}^{k}\alpha_{j}\bigg\}
				\min\big\{\tau_{1}^{3},
				\tau_{1}^{13/3}k^{1/3-\beta}\big\}.
			\end{aligned}
		\end{align}

		We next sum the exponential factor. For $i\leq k<2i$, we have
		$\sum_{j=i+1}^{k}\alpha_{j}\geq2^{-\beta}\alpha_{i}(k-i)$,
		whereas for $k\geq2i$,
		$\sum_{j=i+1}^{k}\alpha_{j}\geq (\alpha_{0}/2)k^{1-\beta}$.
		Consequently,
		\begin{align}\label{eq-qq-33}
			\begin{aligned}
				&\alpha_{i}\sum_{k=i}^{\infty}
				\exp\bigg\{-c\sum_{j=i+1}^{k}\alpha_{j}\bigg\}\\
				&\quad\leq C\alpha_{i}\sum_{r=0}^{\infty}
				\exp\{-c2^{-\beta}\alpha_{i}r\}
				+C\alpha_{i}\sum_{k=2i}^{\infty}
				\exp\{-(c\alpha_{0}/2)k^{1-\beta}\}\\
				&\quad\leq C.
			\end{aligned}
		\end{align}
		Here $\alpha_{i}/(1-\exp\{-c2^{-\beta}\alpha_{i}\})\leq C$
		because $0<\alpha_{i}\leq\alpha_{0}$, and
		$\sum_{k=1}^{\infty}\exp\{-(c\alpha_{0}/2)k^{1-\beta}\}<\infty$.
		Since $k^{1/3-\beta}\leq i^{1/3-\beta}$ for $k\geq i$,
		combining \eqref{eq-qq-32} and \eqref{eq-qq-33} gives
		\begin{align*}
			\sum_{k=i}^{\infty}
			\|(\Pi_{k,3}-\Pi_{k,3}^{\{i\}})
			\mathbf{1}(\mathcal{E}_{2})\|_{2}&\leq CC_{\MM}^{3}
			\min\big\{\tau_{1}^{3},
			\tau_{1}^{13/3}i^{1/3-\beta}\big\}\\
			&\leq CC_{\MM}^{3}
			(\tau_{1}^{3})^{1/4}
			(\tau_{1}^{13/3}i^{1/3-\beta})^{3/4}\\
			&=CC_{\MM}^{3}\tau_{1}^{4}i^{(1-3\beta)/4}
			\leq CC_{\MM}^{3}\tau_{1}^{4}i^{1-3\beta/2},
		\end{align*}
		where the last inequality follows from $\beta<1$.
		This proves \eqref{eq-qq-30}.

	\smallskip
	\noindent
	{\it (iv) Proof of \eqref{eq-qa-04}.}
	It  follows from the  definitions of $\Pi_{k,2}$, $\Pi_{k,2}^{\{i\}}$ and
	\eqref{eq-qq-06} that
	\begin{align}\label{eq-qq-22}
		\begin{aligned}
			\|\Pi_{k,2}\|_{2}+\|\Pi_{k,2}^{\{i\}}\|_{2}
			&\leq CC_{\MM}\alpha_{k+1}
			\|G_{\alpha_{k+1}}(x_{k},\zeta_{k+1})\|_{4}^{2}\leq CC_{\MM}\tau_{1}^{2}k^{-\beta}.
		\end{aligned}
	\end{align}
	This proves \eqref{eq-qa-04}.

	\smallskip
	\noindent
	{\it (v) Proof of \eqref{eq-qa-05}.}
	By the definitions of $\Pi_{k,2}$ and $\Pi_{k,2}^{\{i\}},$ we have
	\begin{align*}
		\Pi_{k,2}-\Pi_{k,2}^{\{i\}}
		&=J_{21}+J_{22}+J_{23},
	\end{align*}
	where
	\begin{align*}
		J_{21}&=\alpha_{k+1}U^{\top}
		\big(
		\nabla^{2}P_{\MM}(y_{k})
		-\nabla^{2}P_{\MM}(y_{k}^{\{i\}})
		\big)\Big[
		P_{T_{k}}G_{\alpha_{k+1}}(x_{k},\zeta_{k+1}),\\
		&\hspace{61mm}
		(I_{d}-P_{T_{k}})
		G_{\alpha_{k+1}}(x_{k},\zeta_{k+1})
		\Big],\\
		J_{22}&=	\alpha_{k+1}U^{\top}
		\nabla^{2}P_{\MM}(y_{k}^{\{i\}})
		\Big[
		P_{T_{k}}G_{\alpha_{k+1}}(x_{k},\zeta_{k+1})\\
		&\hspace{31mm}
		-
		P_{T_{k}^{\{i\}}}
		G_{\alpha_{k+1}}
		(x_{k}^{\{i\}},\zeta_{k+1}^{\{i\}}),
		(I_{d}-P_{T_{k}})
		G_{\alpha_{k+1}}(x_{k},\zeta_{k+1})
		\Big]\\
		J_{23}&=	\alpha_{k+1}U^{\top}
		\nabla^{2}P_{\MM}(y_{k}^{\{i\}})
		\Big[
		P_{T_{k}^{\{i\}}}
		G_{\alpha_{k+1}}
		(x_{k}^{\{i\}},\zeta_{k+1}^{\{i\}}),\\
		&\hspace{18mm}
		(I_{d}-P_{T_{k}})
		G_{\alpha_{k+1}}(x_{k},\zeta_{k+1})
		-
		(I_{d}-P_{T_{k}^{\{i\}}})
		G_{\alpha_{k+1}}
		(x_{k}^{\{i\}},\zeta_{k+1}^{\{i\}})
		\Big].
	\end{align*}
	For the following estimates, we first establish a conditional moment
	bound. On $\mathcal{E}_{1}$, both states at time $k$ are within
	$\delta$ of $x_{*}$. By Assumptions~\ref{assu-3}--\ref{assu-4} and
	the independence of $\xi_{k+1}$ from $\mathscr{F}_{k}^{\{i\}}$,
	\begin{align}\label{eq-pa-11}
		\begin{aligned}
			&\mathbf{1}(\mathcal{E}_{1})
			\E\Big[
			|G_{\alpha_{k+1}}(x_{k},\zeta_{k+1})|^{2}
			+|G_{\alpha_{k+1}}
			(x_{k}^{\{i\}},\zeta_{k+1}^{\{i\}})|^{2}
			\mid\mathscr{F}_{k}^{\{i\}}
			\Big]\\
			&\quad\leq C(1+\delta+\tau)^{2}
			\mathbf{1}(\mathcal{E}_{1})
			\leq C\tau_{1}^{2}\mathbf{1}(\mathcal{E}_{1}).
		\end{aligned}
	\end{align}
	For $J_{21}$, by the local Lipschitz properties of
	$\nabla^{2}P_{\MM}$ and $P_{\MM}$, we have on
	$\mathcal{E}_{1}$,
	\begin{align*}
		|J_{21}|
		&\leq
		CC_{\MM}\alpha_{k+1}
		|y_{k}-y_{k}^{\{i\}}|
		|G_{\alpha_{k+1}}(x_{k},\zeta_{k+1})|^{2}\\
		&\leq
		CC_{\MM}^{2}\alpha_{k+1}
		|x_{k}-x_{k}^{\{i\}}|
		|G_{\alpha_{k+1}}(x_{k},\zeta_{k+1})|^{2}.
	\end{align*}
	Using this inequality, \eqref{eq-pa-11}
	and Lemma~\ref{lem-UpperB-xn-xni-fourth}, we have
	\begin{align}\label{eq-qa-06}
		\big\|\E\big[J_{21}
		\mathbf{1}(\mathcal{E}_{1})
		\mid\mathscr{F}_{k}^{\{i\}}\big]\big\|_{2}\leq CC_{\MM}^{2}\tau_{1}^{2}\alpha_{k+1}	\|x_{k}-x_{k}^{\{i\}}\|_{2}\leq  CC_{\MM}^{2}\tau_{1}^{3}k^{-2\beta}.
	\end{align}
	For $J_{22}$, note that
	\begin{align*}
		&P_{T_{k}}G_{\alpha_{k+1}}(x_{k},\zeta_{k+1})
		-
		P_{T_{k}^{\{i\}}}
		G_{\alpha_{k+1}}(x_{k}^{\{i\}},\zeta_{k+1}^{\{i\}})\\
		&\quad=
		(P_{T_{k}}-P_{T_{k}^{\{i\}}})
		G_{\alpha_{k+1}}(x_{k},\zeta_{k+1})+
		P_{T_{k}^{\{i\}}}
		\big\{
		G_{\alpha_{k+1}}(x_{k},\zeta_{k+1})
		-
		G_{\alpha_{k+1}}
		(x_{k}^{\{i\}},\zeta_{k+1}^{\{i\}})
		\big\}.
	\end{align*}
	Thus, on $\mathcal{E}_{1}$,
	\begin{align*}
		|J_{22}|
		&\leq
		CC_{\MM}^{3}\alpha_{k+1}
		|G_{\alpha_{k+1}}(x_{k},\zeta_{k+1})|^{2}
		|x_{k}-x_{k}^{\{i\}}|\\
		&\quad+
		CC_{\MM}\alpha_{k+1}
		|G_{\alpha_{k+1}}(x_{k},\zeta_{k+1})|
		\big|
		G_{\alpha_{k+1}}(x_{k},\zeta_{k+1})
		-
		G_{\alpha_{k+1}}
		(x_{k}^{\{i\}},\zeta_{k+1}^{\{i\}})
		\big|.
	\end{align*}
	Using \eqref{eq-pa-11}, conditional
	Cauchy--Schwarz, and Lemma~\ref{lem-UpperB-xn-xni-fourth}, we have
	\begin{align}\label{eq-qa-07}
		\begin{aligned}
			&	\big\|\E\big[J_{22}
			\mathbf{1}(\mathcal{E}_{1})
			\mid\mathscr{F}_{k}^{\{i\}}\big]\big\|_{2}\\
			&\quad\leq   CC_{\MM}^{3}\tau_{1}^{3}k^{-2\beta}
			+CC_{\MM}\tau_{1}\alpha_{k+1}
			\big\|G_{\alpha_{k+1}}(x_{k},\zeta_{k+1})
			-G_{\alpha_{k+1}}
			(x_{k}^{\{i\}},\zeta_{k+1}^{\{i\}})\big\|_{2}.
		\end{aligned}
	\end{align}
	For $J_{23}$, we have
	\begin{align*}
		&(I_{d}-P_{T_{k}})
		G_{\alpha_{k+1}}(x_{k},\zeta_{k+1})
		-
		(I_{d}-P_{T_{k}^{\{i\}}})
		G_{\alpha_{k+1}}
		(x_{k}^{\{i\}},\zeta_{k+1}^{\{i\}})\\
		&\quad=
		-(P_{T_{k}}-P_{T_{k}^{\{i\}}})
		G_{\alpha_{k+1}}(x_{k},\zeta_{k+1})+
		(I_{d}-P_{T_{k}^{\{i\}}})
		\big\{
		G_{\alpha_{k+1}}(x_{k},\zeta_{k+1})
		-
		G_{\alpha_{k+1}}
		(x_{k}^{\{i\}},\zeta_{k+1}^{\{i\}})
		\big\}.
	\end{align*}
	Hence, on $\mathcal{E}_{1}$,
	\begin{align*}
		|J_{23}|
		&\leq
		CC_{\MM}^{3}\alpha_{k+1}
		|G_{\alpha_{k+1}}(x_{k},\zeta_{k+1})|
		|G_{\alpha_{k+1}}
		(x_{k}^{\{i\}},\zeta_{k+1}^{\{i\}})|
		|x_{k}-x_{k}^{\{i\}}|\\
		&\quad+
		CC_{\MM}\alpha_{k+1}
		|G_{\alpha_{k+1}}
		(x_{k}^{\{i\}},\zeta_{k+1}^{\{i\}})|
		\big|
		G_{\alpha_{k+1}}(x_{k},\zeta_{k+1})
		-
		G_{\alpha_{k+1}}
		(x_{k}^{\{i\}},\zeta_{k+1}^{\{i\}})
		\big|.
	\end{align*}
	Using \eqref{eq-pa-11}, conditional
	Cauchy--Schwarz, Lemma~\ref{lem-UpperB-xn-xni-fourth}, and
	$ab\leq(a^{2}+b^{2})/2$ for $a,b\geq0$, we have
	\begin{align}\label{eq-qa-08}
		\begin{aligned}
			&	\big\|\E\big[J_{23}
			\mathbf{1}(\mathcal{E}_{1})
			\mid\mathscr{F}_{k}^{\{i\}}\big]\big\|_{2}\\
			&\quad\leq   CC_{\MM}^{3}\tau_{1}^{3}k^{-2\beta}
			+CC_{\MM}\tau_{1}\alpha_{k+1}
			\big\|G_{\alpha_{k+1}}(x_{k},\zeta_{k+1})
			-G_{\alpha_{k+1}}
			(x_{k}^{\{i\}},\zeta_{k+1}^{\{i\}})\big\|_{2}.
		\end{aligned}
	\end{align}
	Combining \eqref{eq-qa-06}, \eqref{eq-qa-07} and \eqref{eq-qa-08}, we have
	\begin{align*}
		&\big\|\E\big[(\Pi_{k,2}-\Pi_{k,2}^{\{i\}})
		\mathbf{1}(\mathcal{E}_{1})
		\mid\mathscr{F}_{k}^{\{i\}}\big]\big\|_{2}\\
		&\quad\leq CC_{\MM}^{3}\tau_{1}^{3}k^{-2\beta}+CC_{\MM}\tau_{1}\alpha_{k+1}
		\big\|G_{\alpha_{k+1}}(x_{k},\zeta_{k+1})
		-G_{\alpha_{k+1}}
		(x_{k}^{\{i\}},\zeta_{k+1}^{\{i\}})\big\|_{2}.
	\end{align*}
	To replace $\mathcal{E}_{1}$ by $\mathcal{E}_{2}$, use
	\eqref{eq-qq-09} and H\"older's inequality to obtain
	\begin{align*}
		&\|(\Pi_{k,2}-\Pi_{k,2}^{\{i\}})
		\mathbf{1}(\mathcal{E}_{1}\setminus\mathcal{E}_{2})\|_{2}\\
		&\quad\leq CC_{\MM}\alpha_{k+1}
		\big\{\|G_{\alpha_{k+1}}(x_{k},\zeta_{k+1})\|_{6}^{2}
		+\|G_{\alpha_{k+1}}
		(x_{k}^{\{i\}},\zeta_{k+1}^{\{i\}})\|_{6}^{2}\big\}
		\P(\mathcal{E}_{2}^{c})^{1/6}\\
		&\quad\leq CC_{\MM}\tau_{1}^{2}k^{-2\beta}.
	\end{align*}
	Hence, for every $k\geq i$, we have
	\begin{align}\label{eq-qq-25}
		\begin{aligned}
			&\big\|\E\big[(\Pi_{k,2}-\Pi_{k,2}^{\{i\}})
			\mathbf{1}(\mathcal{E}_{2})
			\mid\mathscr{F}_{k}^{\{i\}}\big]\big\|_{2}\\
			&\quad\leq CC_{\MM}^{3}\tau_{1}^{3}k^{-2\beta}+CC_{\MM}\tau_{1}\alpha_{k+1}
			\big\|G_{\alpha_{k+1}}(x_{k},\zeta_{k+1})
			-G_{\alpha_{k+1}}
			(x_{k}^{\{i\}},\zeta_{k+1}^{\{i\}})\big\|_{2}.
		\end{aligned}
	\end{align}

	We now sum the last term in \eqref{eq-qq-25} using
	\eqref{eq:coupled-residual} in Assumption~\ref{assu-5}.
	Choose a fixed integer $k_{0}\geq1$ sufficiently large that
	\begin{align}\label{eq-qq-26}
		\alpha_{k_{0}}\leq\bar{\alpha},\qquad
		c_{13}\alpha_{k_{0}}\leq c_{12}/2,\qquad
		c_{12}\alpha_{k_{0}}\leq1.
	\end{align}
	For $k\geq\max\{i,k_{0}\}$, apply
	\eqref{eq:coupled-residual} conditionally on
	$\mathscr{F}_{k}^{\{i\}}$ and then take expectations. This gives
	\begin{align}\label{eq-qq-27}
		\begin{aligned}
			&\|x_{k+1}-x_{k+1}^{\{i\}}\|_{2}^{2}+c_{11}\alpha_{k+1}^{2}
			\big\|G_{\alpha_{k+1}}(x_{k},\zeta_{k+1})
			-G_{\alpha_{k+1}}
			(x_{k}^{\{i\}},\zeta_{k+1}^{\{i\}})\big\|_{2}^{2}\\
			&\qquad\leq
			(1-c_{12}\alpha_{k+1}/2)
			\|x_{k}-x_{k}^{\{i\}}\|_{2}^{2}
			\leq\|x_{k}-x_{k}^{\{i\}}\|_{2}^{2}.
		\end{aligned}
	\end{align}
	Therefore, for all integers $s\geq r\geq\max\{i,k_{0}\}$,
	\begin{align}\label{eq-qq-28}
		\sum_{k=r}^{s}\alpha_{k+1}^{2}
		\big\|G_{\alpha_{k+1}}(x_{k},\zeta_{k+1})
		-G_{\alpha_{k+1}}
		(x_{k}^{\{i\}},\zeta_{k+1}^{\{i\}})\big\|_{2}^{2}
		\leq C\|x_{r}-x_{r}^{\{i\}}\|_{2}^{2}.
	\end{align}
	Let $r=\max\{i,k_{0}\}$. Applying Cauchy--Schwarz on the
	intervals $\{2^{\ell}r,\ldots,2^{\ell+1}r-1\}$ and then using
	\eqref{eq-qq-28} and
	Lemma~\ref{lem-UpperB-xn-xni-fourth}, we have
	\begin{align*}
		&\sum_{k=r}^{\infty}\alpha_{k+1}
		\big\|G_{\alpha_{k+1}}(x_{k},\zeta_{k+1})
		-G_{\alpha_{k+1}}
		(x_{k}^{\{i\}},\zeta_{k+1}^{\{i\}})\big\|_{2}\\
		&\quad\leq C\sum_{\ell=0}^{\infty}
		(2^{\ell}r)^{1/2}
		\|x_{2^{\ell}r}-x_{2^{\ell}r}^{\{i\}}\|_{2}\\
		&\quad\leq C\tau_{1}r^{1/2-\beta}
		\sum_{\ell=0}^{\infty}2^{\ell(1/2-\beta)}
		\leq C\tau_{1}r^{1/2-\beta},
	\end{align*}
	where the last inequality follows from $\beta>1/2$.
	If $i<k_{0}$, the terms with $i\leq k<k_{0}$ have total
	contribution at most $C\tau_{1}$ by the growth bound and
	\eqref{eq-qq-06}. Since $k_{0}$ depends only on the structural
	quantities and $1\leq i<k_{0}$, this contribution is at most
	$C\tau_{1}i^{1/2-\beta}$ after enlarging $C$.
	The entire tail, including that caused by any such early
	replacement, is covered by the preceding estimate with $r=k_{0}$.
	Hence
	\begin{align}\label{eq-qq-29}
		\sum_{k=i}^{\infty}\alpha_{k+1}
		\big\|G_{\alpha_{k+1}}(x_{k},\zeta_{k+1})
		-G_{\alpha_{k+1}}
		(x_{k}^{\{i\}},\zeta_{k+1}^{\{i\}})\big\|_{2}
		\leq C\tau_{1}i^{1/2-\beta}.
	\end{align}
		Combining \eqref{eq-qq-25} and \eqref{eq-qq-29}, we obtain
		\begin{align*}
			\sum_{k=i}^{\infty}
			\big\|\E\big[(\Pi_{k,2}-\Pi_{k,2}^{\{i\}})
			\mathbf{1}(\mathcal{E}_{2})
			\mid\mathscr{F}_{k}^{\{i\}}\big]\big\|_{2}
			&\leq CC_{\MM}^{3}\tau_{1}^{3}i^{1-2\beta}
			+CC_{\MM}\tau_{1}^{2}i^{1/2-\beta}\\
			&\leq CC_{\MM}^{3}\tau_{1}^{4}i^{1-3\beta/2},
		\end{align*}
		where the last inequality follows from $\beta\in(2/3,1)$,
		$\tau_{1}\geq1$, and $C_{\MM}\geq1$.
		This proves \eqref{eq-qa-05} and completes the proof.

\end{proof}

\subsubsection{Completion of the proof}

We now complete the proof of Lemma~\ref{lem-ap_nonsmooth-1}.
For clarity, the argument is divided into two steps.
We first prove Lemmas~\ref{lem-D1-bound}--\ref{lem-D4-bound}, which provide separate
bounds for the terms in the decomposition of $D_n$ and their
corresponding replacement differences. We then combine these bounds
to obtain the conclusion of Lemma~\ref{lem-ap_nonsmooth-1}.

\begin{lemma}\label{lem-D1-bound}
	Suppose that Assumptions~\ref{assu-1}--\ref{assu-5} hold.
	Then, there exists an integer $n_{0}\geq 1$ such that, for every
	$n\geq n_{0}$ and every $1\leq i\leq n-1$,
	\begin{align}\label{eq-pa-01}
		\|D_{1,n}\|_{2}
		\leq CC_{\MM}\tau_{1}^{2}n^{1/2-\beta},
		\qquad
		\|D_{1,n}-D_{1,n}^{\{i\}}\|_{2}
		\leq CC_{\MM}\tau_{1}^{3}n^{-1/2}i^{1-3\beta/2},
	\end{align}
	where $\tau_{1}=(1+\delta+\tau_{0}+\tau)
	(1+\delta^{-1}(1+\tau_{0}+\tau))$ and $C_{\MM}$ is defined
	by \eqref{eq-defi-cm}.
\end{lemma}
\begin{proof}
	We first prove the first inequality of \eqref{eq-pa-01}.
	By the definition of $D_{1,n}$ and Lemma~\ref{lem-A4}, we have
	\begin{align}\label{eq-pa-02}
		\|D_{1,n}\|_{2}
		\leq \frac{C}{\sqrt{n}}\sum_{k=0}^{n-1}
		\|U^{\top}(x_{k}-y_{k})\|_{2}.
	\end{align}
	For $x\in B_{\delta_{1}}(x_{*})$,
	let $h(x)=U^{\top}(x-P_{\MM}(x))$.
	Since $\nabla P_{\MM}(x_{*})=P_{T_{\MM}(x_{*})}$
	and $U^{\top}U=I_{m}$, it follows that
	\begin{align*}
		\nabla h(x_{*})
		=U^{\top}(I_{d}-UU^{\top})
		=(I_{m}-U^{\top}U)U^{\top}=0.
	\end{align*}
		By the definition of $C_{\MM}$ and the fundamental theorem of calculus,
		for every $x\in B_{\delta}(x_{*})$, we have
		\begin{align}\label{eq-pa-03}
			\|\nabla h(x)\|_{\op}
			=\|\nabla h(x)-\nabla h(x_{*})\|_{\op}
			\leq C_{\MM}|x-x_{*}|.
		\end{align}
	By \eqref{eq-pa-03}, for any $x,x'\in B_{\delta}(x_{*})$, we have
	\begin{align}\label{eq-pa-04}
		\begin{aligned}
			|h(x)-h(x')|
			&\leq \int_{0}^{1}
		\|\nabla h(x'+t(x-x'))\|_{\op}
			|x-x'|\,dt\\
			&\leq C_{\MM}(|x-x_{*}|+|x'-x_{*}|)|x-x'|.
		\end{aligned}
	\end{align}
	Note that $h(x_{*})=0$.
		Also, by the definition of $y_{k}$ and the fact that $x_{*}\in\MM$,
		we have
	$
			|x_{k}-y_{k}|\leq |x_{k}-x_{*}|,$
		$k\geq 0.$
	Combining this with \eqref{eq-pa-04} and Lemma~\ref{lem-3},
	we have,  for every $k\geq 1$,
	\begin{align}\label{eq-pa-05}
		\begin{aligned}
			\E|U^{\top}(x_{k}-y_{k})|^{2}
			&=\E|h(x_{k})|^{2}
			\mathbf{1}(x_{k}\in B_{\delta}(x_{*}))+\E|U^{\top}(x_{k}-y_{k})|^{2}
			\mathbf{1}(x_{k}\notin B_{\delta}(x_{*}))\\
			&\leq \E|h(x_{k})-h(x_{*})|^{2}
			\mathbf{1}(x_{k}\in B_{\delta}(x_{*}))
			+\delta^{-2}\E|x_{k}-x_{*}|^{4}\\
			&\leq (C_{\MM}^{2}+\delta^{-2})
				\E|x_{k}-x_{*}|^{4}\\
			&\leq CC_{\MM}^{2}\tau_{1}^{4}k^{-2\beta}.
		\end{aligned}
	\end{align}
		For the initial term, Assumption~\ref{assu-4}(i) gives
		\begin{align}\label{eq-pa-5.1}
			\|U^{\top}(x_{0}-y_{0})\|_{2}
			\leq \|x_{0}-x_{*}\|_{2}
			\leq \tau_{0}.
		\end{align}
	Substituting \eqref{eq-pa-05} and \eqref{eq-pa-5.1} into \eqref{eq-pa-02}, we obtain
	\begin{align*}
		\|D_{1,n}\|_{2}
		&\leq \frac{C\tau_{0}}{\sqrt{n}}+
		\frac{CC_{\MM}\tau_{1}^{2}}{\sqrt{n}}
		\sum_{k=1}^{n-1}k^{-\beta}
		\leq CC_{\MM}\tau_{1}^{2}n^{1/2-\beta},
	\end{align*}
	where the last inequality follows from $\beta<1$, $C_{\MM}\geq 1$, and $\tau_{0}\leq \tau_{1}^{2}$.

	For the second inequality of \eqref{eq-pa-01},
	by Lemma~\ref{lem-A4}, we have
	\begin{align*}
		\|D_{1,n}-D_{1,n}^{\{i\}}\|_{2}
		&\leq \frac{C}{\sqrt{n}}
		\sum_{k=i}^{n-1}
		\big\|U^{\top}(x_{k}-y_{k})
		-U^{\top}(x_{k}^{\{i\}}-y_{k}^{\{i\}})\big\|_{2}\\
		&\leq J_{1}+J_{2}+J_{3},
	\end{align*}
	where
	\begin{align*}
		J_{1}
		&=\frac{C}{\sqrt{n}}\sum_{k=i}^{n-1}
		\big\|\big|U^{\top}(x_{k}-y_{k})
		-U^{\top}(x_{k}^{\{i\}}-y_{k}^{\{i\}})\big|
		\mathbf{1}(x_{k},x_{k}^{\{i\}}\in B_{\delta}(x_{*}))\big\|_{2},\\
		J_{2}
		&=\frac{C}{\sqrt{n}}\sum_{k=i}^{n-1}
		\big\|\big|U^{\top}(x_{k}-y_{k})
		-U^{\top}(x_{k}^{\{i\}}-y_{k}^{\{i\}})\big|
		\mathbf{1}(x_{k}\notin B_{\delta}(x_{*}))\big\|_{2},\\
		J_{3}
		&=\frac{C}{\sqrt{n}}\sum_{k=i}^{n-1}
		\big\|\big|U^{\top}(x_{k}-y_{k})
		-U^{\top}(x_{k}^{\{i\}}-y_{k}^{\{i\}})\big|
		\mathbf{1}(x_{k}^{\{i\}}\notin B_{\delta}(x_{*}))\big\|_{2}.
	\end{align*}
	Note that
		$x_{k}^{\{i\}}\overset{d}{=}x_{k}$ for every $k\geq 0$.
	For $J_{1}$, by \eqref{eq-pa-04},
	and Lemmas~\ref{lem-3} and~\ref{lem-UpperB-xn-xni-fourth}, we have
	\begin{align*}
		J_{1}
		&\leq \frac{CC_{\MM}}{\sqrt{n}}
		\sum_{k=i}^{n-1}
		\|x_{k}-x_{*}\|_{4}\cdot\|x_{k}-x_{k}^{\{i\}}\|_{4}\\
		&\leq \frac{CC_{\MM}\tau_{1}^{2}}{\sqrt{n}}
		\sum_{k=i}^{n-1}k^{-3\beta/2}
		\leq CC_{\MM}\tau_{1}^{2}n^{-1/2}i^{1-3\beta/2}.
	\end{align*}
		For $J_{2}$, using the bounds
		$|x_{k}-y_{k}|\leq |x_{k}-x_{*}|$ and
		$|x_{k}^{\{i\}}-y_{k}^{\{i\}}|\leq |x_{k}^{\{i\}}-x_{*}|$,
		together with H\"older's inequality, we have
		\begin{align*}
			&\big\|\big|U^{\top}(x_{k}-y_{k})
			-U^{\top}(x_{k}^{\{i\}}-y_{k}^{\{i\}})\big|
			\mathbf{1}(x_{k}\notin B_{\delta}(x_{*}))\big\|_{2}\\
			&\quad\leq
			(\|x_{k}-x_{*}\|_{6}+\|x_{k}^{\{i\}}-x_{*}\|_{6})
			\P(x_{k}\notin B_{\delta}(x_{*}))^{1/3}\\
			&\quad\leq 2\delta^{-2}\|x_{k}-x_{*}\|_{6}^{3}.
		\end{align*}
	Combining this with  Lemma~\ref{lem-3}, we have
	\begin{align*}
		J_{2}
		&\leq \frac{C\delta^{-2}}{\sqrt{n}}
		\sum_{k=i}^{n-1}\|x_{k}-x_{*}\|_{6}^{3}\\
		&\leq \frac{C\tau_{1}^{3}}{\sqrt{n}}
		\sum_{k=i}^{n-1}k^{-3\beta/2}
		\leq C\tau_{1}^{3}n^{-1/2}i^{1-3\beta/2}.
	\end{align*}
	Similarly, we have the same upper bound for $J_{3}$. Hence
	\begin{align*}
		\|D_{1,n}-D_{1,n}^{\{i\}}\|_{2}
		\leq CC_{\MM}\tau_{1}^{3}n^{-1/2}i^{1-3\beta/2}.
	\end{align*}
	This proves the second inequality of \eqref{eq-pa-01}.
\end{proof}

\begin{lemma}\label{lem-D3-bound}
	Suppose that Assumptions~\ref{assu-1}--\ref{assu-5} hold.
	Then there exists an integer $n_{0}\geq 2$ such that, for every
	$n\geq n_{0}$ and every $1\leq i\leq n-1$,
	\begin{align}\label{eq-mm-01}
		\begin{aligned}
			\|D_{3,n}\|_{2}
			&\leq CC_{\MM}^{2}\tau_{1}^{4}n^{1/2-\beta},\qquad
			\|D_{3,n}-D_{3,n}^{\{i\}}\|_{2}
			\leq CC_{\MM}^{3}\tau_{1}^{4}
			n^{-1/2}i^{1-3\beta/2},
		\end{aligned}
	\end{align}
	where $\tau_{1}=(1+\delta+\tau_{0}+\tau)
	\{1+\delta^{-1}(1+\tau_{0}+\tau)\}$ and $C_{\MM}$ is defined
	by \eqref{eq-defi-cm}.
\end{lemma}

\begin{proof}
	We first prove the first inequality of \eqref{eq-mm-01}.
	By the definition of $D_{3,n}$ and Lemma~\ref{lem-A4}, we have
	\begin{align}\label{eq-mm-02}
		\|D_{3,n}\|_{2}
		\leq \frac{C}{\sqrt{n}}
		\sum_{k=0}^{n-2}\|U^{\top}E_{k}\|_{2},
	\end{align}
	where we recall that
	$E_{k}=\alpha_{k+1}^{-1}\{y_{k+1}-y_{k}
	+\alpha_{k+1}F_{\MM}(y_{k})
	+\alpha_{k+1}P_{T_{\MM}(y_{k})}\zeta_{k+1}\}$.
		We next derive an upper bound for $\|U^{\top}E_{k}\|_{2}$.
		We use $T_{k}$ and $\Pi_{k}$ as defined before
		Lemma~\ref{lem-pi-replacement}, and the properties of the shadow
		sequence in \eqref{eq-qq-05}.
	Recall that, in Lemma~\ref{lem-gamma-local}, for any
	$y,z\in\MM\cap B_{4\delta}(x_{*})$, we defined
	\begin{align*}
		\Gamma(y,z)
		=
		U^{\top}\big\{
		z-y-P_{T_{\MM}(y)}(z-y)
		\big\}.
	\end{align*}
	Hence, by the definition of $E_{k}$, we can write
	\begin{align}\label{eq-mm-04}
		U^{\top}E_{k}
		&=
		\alpha_{k+1}^{-1}\Gamma(y_{k},y_{k+1})+
		U^{\top}\big\{
		\alpha_{k+1}^{-1}P_{T_{k}}(y_{k+1}-y_{k})
		+F_{\MM}(y_{k})
		+P_{T_{k}}\zeta_{k+1}
		\big\}.
	\end{align}

		For the first term on the right-hand side of \eqref{eq-mm-04},
		consider the event $\{x_{k},x_{k+1}\in B_{\delta}(x_{*})\}$.
		On this event, the segment joining
		$x_{k}$ and $x_{k+1}$ lies in $B_{\delta}(x_{*})$, and
		$y_{k}=P_{\MM}(x_{k})$, $y_{k+1}=P_{\MM}(x_{k+1})$.
		Differentiating the identity $P_{\MM}\circ P_{\MM}=P_{\MM}$
		at $x_{k}$ and using $\nabla P_{\MM}(y_{k})=P_{T_{k}}$, we obtain
		\begin{align*}
			P_{T_{k}}\nabla P_{\MM}(x_{k})
			&=\nabla P_{\MM}(x_{k}),
			\qquad
			(I_{d}-P_{T_{k}})\nabla P_{\MM}(x_{k})=0.
		\end{align*}
		Thus Taylor's formula at $x_{k}$ gives
		\begin{align*}
			&\Gamma(y_{k},y_{k+1})\\
			&\quad=U^{\top}(I_{d}-P_{T_{k}})
			\int_{0}^{1}(1-t)
			\nabla^{2}P_{\MM}\big(x_{k}+t(x_{k+1}-x_{k})\big)
			[x_{k+1}-x_{k},x_{k+1}-x_{k}]\,dt.
		\end{align*}
		Moreover, since $U^{\top}P_{T_{\MM}(x_{*})}=U^{\top}$, we have
		\begin{align*}
			\|U^{\top}(I_{d}-P_{T_{k}})\|_{\op}
			&\leq
			\|\nabla P_{\MM}(x_{*})-\nabla P_{\MM}(y_{k})\|_{\op}\\
			&\leq C_{\MM}|y_{k}-x_{*}|
			\leq 2C_{\MM}|x_{k}-x_{*}|.
		\end{align*}
		Combining these identities with the update rule \eqref{eq-aa-02},
		we have, on the same event,
		\begin{align*}
			|\Gamma(y_{k},y_{k+1})|
			&\leq CC_{\MM}^{2}|x_{k}-x_{*}|\,|x_{k+1}-x_{k}|^{2}\\
			&\leq CC_{\MM}^{2}\alpha_{k+1}^{2}|x_{k}-x_{*}|
			|G_{\alpha_{k+1}}(x_{k},\zeta_{k+1})|^{2}.
		\end{align*}
		Hence, by H\"older's inequality, \eqref{eq-qq-06}, and
		Lemma~\ref{lem-3}, we have, for $k\geq1$,
		\begin{align*}
			&\alpha_{k+1}^{-1}
			\big\|\Gamma(y_{k},y_{k+1})
			\mathbf{1}\{x_{k},x_{k+1}\in B_{\delta}(x_{*})\}\big\|_{2}\\
			&\quad\leq CC_{\MM}^{2}\alpha_{k+1}
			\|x_{k}-x_{*}\|_{6}
			\|G_{\alpha_{k+1}}(x_{k},\zeta_{k+1})\|_{6}^{2}\\
			&\quad\leq CC_{\MM}^{2}(1+\tau_{0}+\tau)^{3}k^{-3\beta/2}.
		\end{align*}
		On the complementary event, \eqref{eq-kb-00} in
		Lemma~\ref{lem-gamma-local} and \eqref{eq-qq-05} give
		\begin{align*}
			&|\Gamma(y_{k},y_{k+1})|
			\mathbf{1}\{x_{k}\notin B_{\delta}(x_{*})
			\text{ or }x_{k+1}\notin B_{\delta}(x_{*})\}\\
			&\quad\leq CC_{\MM}^{2}
			\{|x_{k}-x_{*}|+|x_{k+1}-x_{*}|\}^{3}
			\mathbf{1}\{x_{k}\notin B_{\delta}(x_{*})
			\text{ or }x_{k+1}\notin B_{\delta}(x_{*})\}\\
			&\quad\leq CC_{\MM}^{2}\delta^{-3}
			\{|x_{k}-x_{*}|+|x_{k+1}-x_{*}|\}^{6}.
		\end{align*}
		Consequently, Lemma~\ref{lem-3} implies
		\begin{align*}
			&\alpha_{k+1}^{-1}
			\big\|\Gamma(y_{k},y_{k+1})
			\mathbf{1}\{x_{k}\notin B_{\delta}(x_{*})
			\text{ or }x_{k+1}\notin B_{\delta}(x_{*})\}\big\|_{2}\\
			&\quad\leq CC_{\MM}^{2}\delta^{-3}
			(1+\tau_{0}+\tau)^{6}k^{-2\beta}
			\leq CC_{\MM}^{2}\tau_{1}^{3}k^{-2\beta}.
		\end{align*}
		For $k=0$, the same global cubic bound from
		\eqref{eq-kb-00}, together with Assumption~\ref{assu-4} and
		Lemma~\ref{lem-3}, gives
		\begin{align*}
			\alpha_{1}^{-1}\|\Gamma(y_{0},y_{1})\|_{2}
			&\leq CC_{\MM}^{2}\alpha_{1}^{-1}
			\{\|x_{0}-x_{*}\|_{6}+\|x_{1}-x_{*}\|_{6}\}^{3}\leq CC_{\MM}^{2}\tau_{1}^{3}.
		\end{align*}
		Combining the preceding bounds, we obtain
		\begin{align}\label{eq-mm-05}
			\alpha_{k+1}^{-1}\|\Gamma(y_{k},y_{k+1})\|_{2}
			\leq CC_{\MM}^{2}\tau_{1}^{3}(k+1)^{-3\beta/2},
			\qquad k\geq0.
		\end{align}

		We next estimate the second term on the right-hand side of
		\eqref{eq-mm-04}. For $k\geq0$, define
		\begin{align*}
			\Theta_{k}
			=U^{\top}\big\{
			P_{T_{k}}\big(G_{\alpha_{k+1}}(x_{k},\zeta_{k+1})
			-\zeta_{k+1}\big)-F_{\MM}(y_{k})\big\}.
		\end{align*}
		By the definition of $\Pi_{k}$ and the update rule \eqref{eq-aa-02}, we have
		\begin{align}\label{eq-mm-06}
			U^{\top}\big\{
			\alpha_{k+1}^{-1}P_{T_{k}}(y_{k+1}-y_{k})
			+F_{\MM}(y_{k})+P_{T_{k}}\zeta_{k+1}\big\}
			=\Pi_{k}-\Theta_{k},
		\end{align}
		where $\Pi_{k}$ is defined in Lemma \ref{lem-pi-replacement}.
		For $\Pi_{k}$,using the identity
	$
		P_{T_{k}}
		(y_{k+1}-x_{k+1}+x_{k}-y_{k})
		=
		\mathcal{Q}(y_{k},x_{k+1}-y_{k})
$
		established in the proof of Lemma~\ref{lem-pi-replacement},
		together with \eqref{eq-ui-01}, we have
		\begin{align}\label{eq-cc-01}
			\begin{aligned}
				&|\Pi_{k}|\,
				\mathbf{1}\{x_{k},x_{k+1}\in B_{\delta}(x_{*})\}\\
				&\quad\leq CC_{\MM}\alpha_{k+1}^{-1}
				|x_{k+1}-y_{k}|^{2}
				\mathbf{1}\{x_{k},x_{k+1}\in B_{\delta}(x_{*})\}\\
				&\quad\leq CC_{\MM}\alpha_{k+1}
				|G_{\alpha_{k+1}}(x_{k},\zeta_{k+1})|^{2}
				+CC_{\MM}\alpha_{k+1}^{-1}|x_{k}-y_{k}|^{2}.
			\end{aligned}
		\end{align}
		The calculation for a single trajectory used to prove \eqref{eq-qq-10}
		also gives, for $k\geq1$,
		\begin{align*}
			&\big\|\Pi_{k}\,
			\mathbf{1}\{x_{k}\notin B_{\delta}(x_{*})
			\text{ or }x_{k+1}\notin B_{\delta}(x_{*})\}\big\|_{2}\leq C\tau_{1}^{4}k^{-3\beta/2}.
		\end{align*}
		Using this bound and \eqref{eq-qq-06} in \eqref{eq-cc-01}, we obtain
		\begin{align}\label{eq-mm-07}
			\|\Pi_{k}\|_{2}
			\leq CC_{\MM}\tau_{1}^{4}k^{-\beta}
			+CC_{\MM}k^{\beta}\|x_{k}-y_{k}\|_{4}^{2},
			\qquad k\geq1.
		\end{align}

		For $\Theta_{k}$, Assumptions~\ref{assu-1},
		\ref{assu-3}--\ref{assu-4} and \eqref{eq-qq-05} imply the global bound
		\begin{align}\label{eq-cc-03}
			|\Theta_{k}|
			\leq C\{1+|x_{k}-x_{*}|+|\xi_{k+1}|\}.
		\end{align}
		On $\{x_{k}\in B_{\delta}(x_{*})\}$, we have
		$\Theta_{k}=U^{\top}\mathcal{R}_{\alpha_{k+1}}(x_{k},\zeta_{k+1})$.
		On its complement, \eqref{eq-cc-03} gives
		\begin{align*}
			|\Theta_{k}|\,
			\mathbf{1}\{x_{k}\notin B_{\delta}(x_{*})\}
			\leq C(1+\delta^{-2})(1+|\xi_{k+1}|)|x_{k}-x_{*}|^{2}.
		\end{align*}
		By Assumption~\ref{assu-2}(i), the independence of $\xi_{k+1}$
		from the states at time $k$, and \eqref{eq-qq-05}, we therefore have
		\begin{align*}
			\|\Theta_{k}\|_{2}
			&\leq C(1+\tau^{2})\|x_{k}-y_{k}\|_{2}
			+C\|x_{k}-x_{*}\|_{6}^{3}\\
			&\quad+C\alpha_{k+1}
			\|1+|x_{k}-x_{*}|^{2}+|\zeta_{k+1}|^{2}\|_{2}\\
			&\quad+C(1+\delta^{-2})
			\|(1+|\xi_{k+1}|)|x_{k}-x_{*}|^{2}\|_{2}.
		\end{align*}
		Using \eqref{eq-qq-06} and Lemma~\ref{lem-3}, we obtain
		\begin{align}\label{eq-cc-04}
			\|\Theta_{k}\|_{2}
			\leq C\tau_{1}^{3}k^{-\beta}
			+C(1+\tau^{2})\|x_{k}-y_{k}\|_{2},
			\qquad k\geq1.
		\end{align}
		Combining \eqref{eq-mm-07} and \eqref{eq-cc-04}, we have
		\begin{align}\label{eq-cc-05}
			\begin{aligned}
				\|\Pi_{k}-\Theta_{k}\|_{2}
				&\leq CC_{\MM}\tau_{1}^{4}k^{-\beta}
				+CC_{\MM}k^{\beta}\|x_{k}-y_{k}\|_{4}^{2}
				+C(1+\tau^{2})\|x_{k}-y_{k}\|_{2}\\
				&\leq CC_{\MM}\tau_{1}^{4}k^{-\beta}
				+CC_{\MM}\tau_{1}^{-6}
				k^{4\beta}\E|x_{k}-y_{k}|^{5},
			\end{aligned}
		\end{align}
	where the last inequality follows from the monotonicity of $L^{p}$-norms,
	the fact that $1+\tau^{2}\leq\tau_{1}^{2}$, and Young's inequality,
	which gives, for every $u\geq0$,
	\begin{align*}
		k^{\beta}u^{2}+\tau_{1}^{2}u
		\leq
		C\tau_{1}^{-6}k^{4\beta}u^{5}
		+C\tau_{1}^{4}k^{-\beta}.
	\end{align*}
		Combining \eqref{eq-mm-04}--\eqref{eq-mm-06} and \eqref{eq-cc-05},
		we obtain, for every $k\geq1$,
		\begin{align}\label{eq-cc-06}
			\|U^{\top}E_{k}\|_{2}
			\leq CC_{\MM}^{2}\tau_{1}^{4}k^{-\beta}
			+CC_{\MM}^{2}\tau_{1}^{-6}
			k^{4\beta}\E|x_{k}-y_{k}|^{5}.
		\end{align}
		For $k=0$, by the definition of $E_{0}$,
		Assumptions~\ref{assu-1} and~\ref{assu-4}, and \eqref{eq-qq-05}, we have
		\begin{align}\label{eq-cc-07}
			\|U^{\top}E_{0}\|_{2}
			\leq C(1+\delta+\tau_{0}+\tau)
			\leq C\tau_{1}.
		\end{align}
		Substituting \eqref{eq-cc-06} and \eqref{eq-cc-07} into
		\eqref{eq-mm-02}, and then using
		\eqref{eq-ya-02}, we obtain
		\begin{align*}
			\|D_{3,n}\|_{2}
			&\leq
			\frac{C\tau_{1}}{\sqrt{n}}
			+\frac{CC_{\MM}^{2}\tau_{1}^{4}}{\sqrt{n}}
			\sum_{k=1}^{n-2}k^{-\beta}+
			\frac{CC_{\MM}^{2}\tau_{1}^{-6}}{\sqrt{n}}
			\sum_{k=1}^{n-2}
			k^{4\beta}\E|x_{k}-y_{k}|^{5}\\
			&\leq
			\frac{C\tau_{1}}{\sqrt{n}}
			+CC_{\MM}^{2}
			\big\{\tau_{1}^{4}
			+\tau_{1}^{-6}\tau_{2}\big\}
			n^{1/2-\beta}\\
			&\leq
			CC_{\MM}^{2}\tau_{1}^{4}n^{1/2-\beta},
		\end{align*}
		where the last inequality follows from
		$\tau_{2}\leq\tau_{1}^{10}$, $\beta<1$,
		$C_{\MM}\geq1$, and $\tau_{1}\geq1$.

		We next prove the second inequality of \eqref{eq-mm-01}.
		Fix $1\leq i\leq n-1$, and use the replacement notation and
		$\mathscr{F}_{k}^{\{i\}}$ defined before
		Lemma~\ref{lem-pi-replacement}.
	Since only the innovation $\xi_{i}$ is replaced, the two trajectories
	coincide up to time $i-1$. In particular,
	$E_{k}=E_{k}^{\{i\}}$ for $0\leq k\leq i-2$. Hence, by the definition
	of $D_{3,n}$, we have
	\begin{align}\label{eq-cc-08}
		D_{3,n}-D_{3,n}^{\{i\}}
		=
		\frac{\Sigma_{n}^{-1/2}}{\sqrt{n}}
		\sum_{k=i-1}^{n-2}
		Q_{k+1}U^{\top}
		\big(E_{k}-E_{k}^{\{i\}}\big).
	\end{align}

		Using \eqref{eq-mm-04}, \eqref{eq-mm-06}, and Lemma~\ref{lem-A4},
		we have
	\begin{align}\label{eq-cc-09}
		\|D_{3,n}-D_{3,n}^{\{i\}}\|_{2}
		\leq J_{1}+J_{2}+J_{3},
	\end{align}
	where
	\begin{align*}
		J_{1}
		&=
		\frac{C}{\sqrt{n}}
		\sum_{k=i-1}^{n-2}
		\alpha_{k+1}^{-1}
		\big\|
		\Gamma(y_{k},y_{k+1})
		-\Gamma(y_{k}^{\{i\}},y_{k+1}^{\{i\}})
		\big\|_{2},\\
		J_{2}
		&=
		\frac{C}{\sqrt{n}}
		\bigg\|
		\sum_{k=i-1}^{n-2}
		Q_{k+1}
		(\Pi_{k}-\Pi_{k}^{\{i\}})
		\bigg\|_{2},\quad
		J_{3}
		=
		\frac{C}{\sqrt{n}}
		\sum_{k=i-1}^{n-2}
		\|\Theta_{k}-\Theta_{k}^{\{i\}}\|_{2}.
	\end{align*}
	In what follows, we derive upper bounds for $J_{1}$, $J_{3}$, and $J_{2}$ in turn.

		\smallskip
		\noindent
		{\it (i) Upper bound for $J_{1}$.}
		By the triangle inequality, equality in distribution of the two
		trajectories, and \eqref{eq-mm-05}, we have, for $k\geq i-1$,
		\begin{align}\label{eq-cc-10}
			\begin{aligned}
				&\alpha_{k+1}^{-1}
				\|\Gamma(y_{k},y_{k+1})
				-\Gamma(y_{k}^{\{i\}},y_{k+1}^{\{i\}})\|_{2}\\
				&\quad\leq 2\alpha_{k+1}^{-1}
				\|\Gamma(y_{k},y_{k+1})\|_{2}
				\leq CC_{\MM}^{2}\tau_{1}^{3}(k+1)^{-3\beta/2}.
			\end{aligned}
		\end{align}
		Consequently, since $3\beta/2>1$,
		\begin{align}\label{eq-cc-11}
			J_{1}
			\leq \frac{CC_{\MM}^{2}\tau_{1}^{3}}{\sqrt{n}}
			\sum_{k=i-1}^{n-2}(k+1)^{-3\beta/2}
			\leq CC_{\MM}^{2}\tau_{1}^{3}
			n^{-1/2}i^{1-3\beta/2}.
		\end{align}

		\smallskip
		\noindent
		{\it (ii) Upper bound for $J_{3}$.}
		For $k\geq i$, let $\mathcal{E}_{1}$ and $\mathcal{E}_{2}$ be
		as in the proof of Lemma~\ref{lem-pi-replacement}.
		On $\mathcal{E}_{1}$, we have
		\begin{align*}
			\Theta_{k}-\Theta_{k}^{\{i\}}
			=U^{\top}\big\{
			\mathcal{R}_{\alpha_{k+1}}(x_{k},\zeta_{k+1})
			-\mathcal{R}_{\alpha_{k+1}}
			(x_{k}^{\{i\}},\zeta_{k+1}^{\{i\}})\big\}.
		\end{align*}
		Using Assumption~\ref{assu-2}(ii), the bound
		$|\zeta_{k+1}-\zeta_{k+1}^{\{i\}}|
		\leq c_{8}|x_{k}-x_{k}^{\{i\}}|$, H\"older's inequality, and
		Lemmas~\ref{lem-3} and~\ref{lem-UpperB-xn-xni-fourth}, we obtain
		\begin{align}\label{eq-cc-12}
			\begin{aligned}
				&\|(\Theta_{k}-\Theta_{k}^{\{i\}})
				\mathbf{1}(\mathcal{E}_{1})\|_{2}\\
				&\quad\leq
				C\|1+|\zeta_{k+1}|^{2}
				+|\zeta_{k+1}^{\{i\}}|^{2}\|_{6}
				\|\alpha_{k+1}+|x_{k}-x_{*}|
				+|x_{k}^{\{i\}}-x_{*}|\|_{6}
				\|x_{k}-x_{k}^{\{i\}}\|_{6}\\
				&\quad\leq C\tau_{1}^{4}k^{-3\beta/2}.
			\end{aligned}
		\end{align}
		By \eqref{eq-cc-03} and Lemma~\ref{lem-3},
		$\|\Theta_{k}\|_{6}+\|\Theta_{k}^{\{i\}}\|_{6}
		\leq C(1+\tau_{0}+\tau)$.
		Thus H\"older's inequality, $\mathcal{E}_{1}^{c}\subseteq
		\mathcal{E}_{2}^{c}$, and \eqref{eq-qq-09} give
		\begin{align}\label{eq-tt-01}
			\begin{aligned}
			\|(\Theta_{k}-\Theta_{k}^{\{i\}})
				\mathbf{1}(\mathcal{E}_{1}^{c})\|_{2}&\leq
				(\|\Theta_{k}\|_{6}+\|\Theta_{k}^{\{i\}}\|_{6})
				\P(\mathcal{E}_{2}^{c})^{1/3}\\
				&\leq
				C\delta^{-4}(1+\tau_{0}+\tau)^{5}k^{-2\beta}
				\leq C\tau_{1}^{4}k^{-2\beta}.
			\end{aligned}
		\end{align}
		Combining \eqref{eq-cc-12} and \eqref{eq-tt-01}, we obtain
		\begin{align}\label{eq-cc-13}
			\|\Theta_{k}-\Theta_{k}^{\{i\}}\|_{2}
			\leq C\tau_{1}^{4}k^{-3\beta/2},
			\qquad k\geq i.
		\end{align}
		It remains to consider $k=i-1$.
		Since the two trajectories agree at time $i-1$,
		Assumption~\ref{assu-2}(ii) and H\"older's inequality give, for $i\geq2$,
		\begin{align*}
			&\big\|(\Theta_{i-1}-\Theta_{i-1}^{\{i\}})
			\mathbf{1}\{x_{i-1}\in B_{\delta}(x_{*})\}\big\|_{2}\\
			&\quad\leq C\|1+|\zeta_{i}|^{2}
			+|\zeta_{i}^{\{i\}}|^{2}\|_{6}
			\|\alpha_{i}+|x_{i-1}-x_{*}|\|_{6}
			\|\zeta_{i}-\zeta_{i}^{\{i\}}\|_{6}\\
			&\quad\leq C\tau_{1}^{4}i^{-\beta/2}.
		\end{align*}
		The same H\"older and Markov calculation as in \eqref{eq-tt-01},
		now applied to $\{x_{i-1}\notin B_{\delta}(x_{*})\}$, gives
		\begin{align*}
			&\big\|(\Theta_{i-1}-\Theta_{i-1}^{\{i\}})
			\mathbf{1}\{x_{i-1}\notin B_{\delta}(x_{*})\}\big\|_{2}
			\leq C\tau_{1}^{4}i^{-2\beta}.
		\end{align*}
		For $i=1$, \eqref{eq-cc-03} and the initial moment assumption give
		$\|\Theta_{0}-\Theta_{0}^{\{1\}}\|_{2}
		\leq C(1+\tau_{0}+\tau)$ directly. Therefore,
		\begin{align}\label{eq-cc-14}
			\|\Theta_{i-1}-\Theta_{i-1}^{\{i\}}\|_{2}
			\leq C\tau_{1}^{4}i^{-\beta/2},
			\qquad i\geq1.
		\end{align}
		Combining \eqref{eq-cc-13} and \eqref{eq-cc-14}, we have
		\begin{align*}
			J_{3}
			&\leq \frac{C\tau_{1}^{4}}{\sqrt{n}}
			\bigg\{i^{-\beta/2}+\sum_{k=i}^{n-2}k^{-3\beta/2}\bigg\}
			\leq C\tau_{1}^{4}n^{-1/2}i^{1-3\beta/2},
		\end{align*}
		where we used $3\beta/2>1$ and $\beta<1$.

		\smallskip
		\noindent
		{\it (iii) Upper bound for $J_{2}$.}
		For $k\geq i$, define
	\begin{align*}
		\widetilde{\Pi}_{k,i}
		&=
		\Pi_{k}-\Pi_{k}^{\{i\}}
		-\E\big(
		\Pi_{k}-\Pi_{k}^{\{i\}}
		\mid\mathscr{F}_{k}^{\{i\}}
		\big).
	\end{align*}
	By the definition of $J_{2}$, we have
	\begin{align}\label{eq-cc-j2-02}
		J_{2}
		&\leq J_{21}+J_{22}+J_{23},
	\end{align}
	where
	\begin{align*}
		J_{21}
		&=
		\frac{C}{\sqrt{n}}
		\|\Pi_{i-1}-\Pi_{i-1}^{\{i\}}\|_{2},\quad
		J_{22}
		=
		\frac{C}{\sqrt{n}}
		\bigg\|
		\sum_{k=i}^{n-2}
		Q_{k+1}\widetilde{\Pi}_{k,i}
		\bigg\|_{2},\\
		J_{23}
		&=
		\frac{C}{\sqrt{n}}
		\bigg\|
		\sum_{k=i}^{n-2}
		Q_{k+1}
		\E\big(
		\Pi_{k}-\Pi_{k}^{\{i\}}
		\mid\mathscr{F}_{k}^{\{i\}}
		\big)
		\bigg\|_{2}.
	\end{align*}
	We estimate $J_{21}$, $J_{22}$, and $J_{23}$ separately.
		We first consider $J_{21}$. The trajectories have the same
		$x_{i-1}$ and $y_{i-1}$. On
		$\{x_{i-1},x_{i},x_{i}^{\{i\}}\in B_{\delta}(x_{*})\}$,
		the projection identity used in the proof of
		Lemma~\ref{lem-pi-replacement} gives
		\begin{align*}
			\Pi_{i-1}
			&=\alpha_{i}^{-1}U^{\top}
			\mathcal{Q}(y_{i-1},x_{i}-y_{i-1}),\quad
			\Pi_{i-1}^{\{i\}}
			=\alpha_{i}^{-1}U^{\top}
			\mathcal{Q}(y_{i-1},x_{i}^{\{i\}}-y_{i-1}).
		\end{align*}
		By \eqref{eq-ui-02}, H\"older's inequality, and
		Lemmas~\ref{lem-3} and~\ref{lem-UpperB-xn-xni-fourth}, we obtain,
		for $i\geq2$,
		\begin{align*}
			&\big\|(\Pi_{i-1}-\Pi_{i-1}^{\{i\}})
			\mathbf{1}\{x_{i-1},x_{i},x_{i}^{\{i\}}
			\in B_{\delta}(x_{*})\}\big\|_{2}\\
			&\quad\leq CC_{\MM}\alpha_{i}^{-1}
			\{\|x_{i}-y_{i-1}\|_{4}
			+\|x_{i}^{\{i\}}-y_{i-1}\|_{4}\}
			\|x_{i}-x_{i}^{\{i\}}\|_{4}\\
			&\quad\leq CC_{\MM}\tau_{1}^{2}i^{-\beta/2}.
		\end{align*}
		For the complementary event, the common terms at time $i-1$
		cancel in the definition of $\Pi_{i-1}-\Pi_{i-1}^{\{i\}}$, so that
		\begin{align*}
			|\Pi_{i-1}-\Pi_{i-1}^{\{i\}}|
			\leq\alpha_{i}^{-1}
			\{|x_{i}-x_{*}|+|x_{i}^{\{i\}}-x_{*}|\}.
		\end{align*}
		Applying the H\"older and Markov calculation from
		\eqref{eq-qq-10} to these two state terms and the three possible
		exits, we have, for $i\geq2$,
		\begin{align*}
			&\big\|(\Pi_{i-1}-\Pi_{i-1}^{\{i\}})
			\mathbf{1}\{x_{i-1}\notin B_{\delta}(x_{*})
			\text{ or }x_{i}\notin B_{\delta}(x_{*})
			\text{ or }x_{i}^{\{i\}}\notin B_{\delta}(x_{*})\}\big\|_{2}\\
			&\quad\leq
			C\alpha_{i}^{-1}(1+\tau_{0}+\tau)i^{-\beta/2}
			\big\{\delta^{-12}(1+\tau_{0}+\tau)^{12}i^{-6\beta}\big\}^{1/3}\\
			&\quad\leq
			C\delta^{-4}(1+\tau_{0}+\tau)^{5}i^{-3\beta/2}
			\leq C\tau_{1}^{4}i^{-3\beta/2}.
		\end{align*}
		For $i=1$, the same global bound and Lemma~\ref{lem-3} give
		$\|\Pi_{0}-\Pi_{0}^{\{1\}}\|_{2}
		\leq C(1+\tau_{0}+\tau)$.
		Consequently, we have
		\begin{align}\label{eq-cc-j2-01}
			\|\Pi_{i-1}-\Pi_{i-1}^{\{i\}}\|_{2}
			\leq CC_{\MM}\tau_{1}^{4}i^{-\beta/2},
			\qquad i\geq1.
		\end{align}
		Hence
		\begin{align}\label{eq-cc-j2-03}
			J_{21}\leq CC_{\MM}\tau_{1}^{4}
			n^{-1/2}i^{-\beta/2}.
		\end{align}

		It remains to estimate $J_{22}$ and $J_{23}$.
		If $i=n-1$, both sums are zero. For $1\leq i\leq n-2$,
		$\widetilde{\Pi}_{k,i}$ is
		$\mathscr{F}_{k+1}^{\{i\}}$-measurable and
		$\E(\widetilde{\Pi}_{k,i}\mid\mathscr{F}_{k}^{\{i\}})=0$.
		Since $Q_{k+1}$ is deterministic, orthogonality,
		Lemma~\ref{lem-A4}, and \eqref{eq-qq-03} in
		Lemma~\ref{lem-pi-replacement} give
		\begin{align}\label{eq-sa-01}
			\begin{aligned}
				J_{22}
				&\leq\frac{C}{\sqrt{n}}
				\bigg(\sum_{k=i}^{n-2}
				\|\widetilde{\Pi}_{k,i}\|_{2}^{2}\bigg)^{1/2}
				\leq\frac{C}{\sqrt{n}}
				\bigg(\sum_{k=i}^{n-2}
				\|\Pi_{k}-\Pi_{k}^{\{i\}}\|_{2}^{2}\bigg)^{1/2}\\
				&\leq CC_{\MM}^{3}\tau_{1}^{4}n^{-1/2}
				\bigg(\sum_{k=i}^{n-2}k^{-2\beta}\bigg)^{1/2}\\
				&\leq CC_{\MM}^{3}\tau_{1}^{4}
				n^{-1/2}i^{1/2-\beta}\\
				&\leq CC_{\MM}^{3}\tau_{1}^{4}
				n^{-1/2}i^{1-3\beta/2},
			\end{aligned}
		\end{align}
		where  the last two inequalities use $2\beta>1$ and $\beta<1$.
		Finally, by Lemma~\ref{lem-A4} and \eqref{eq-qq-04} in
		Lemma~\ref{lem-pi-replacement},
		\begin{align}\label{eq-cc-j2-29}
			\begin{aligned}
				J_{23}
				&\leq\frac{C}{\sqrt{n}}\sum_{k=i}^{n-2}
				\big\|\E(\Pi_{k}-\Pi_{k}^{\{i\}}
				\mid\mathscr{F}_{k}^{\{i\}})\big\|_{2}
				\leq CC_{\MM}^{3}\tau_{1}^{4}
				n^{-1/2}i^{1-3\beta/2}.
			\end{aligned}
		\end{align}
		Combining \eqref{eq-cc-j2-03}, \eqref{eq-sa-01}, and
		\eqref{eq-cc-j2-29}, and using $\beta<1$, we obtain
		\begin{align}\label{eq-cc-j2-30}
			J_{2}\leq CC_{\MM}^{3}\tau_{1}^{4}
			n^{-1/2}i^{1-3\beta/2}.
		\end{align}
		Combining this bound with \eqref{eq-cc-09}, \eqref{eq-cc-11}, and
		the bound for $J_{3}$, we have
		\begin{align*}
			\|D_{3,n}-D_{3,n}^{\{i\}}\|_{2}
			&\leq CC_{\MM}^{2}\tau_{1}^{3}n^{-1/2}i^{1-3\beta/2}
			+CC_{\MM}^{3}\tau_{1}^{4}n^{-1/2}i^{1-3\beta/2}\\
			&\quad+C\tau_{1}^{4}n^{-1/2}i^{1-3\beta/2}\\
			&\leq CC_{\MM}^{3}\tau_{1}^{4}n^{-1/2}i^{1-3\beta/2}.
		\end{align*}
	This proves the second inequality of \eqref{eq-mm-01} and completes
	the proof.
\end{proof}

\begin{lemma}\label{lem-D4-bound}
		Suppose that Assumptions~\ref{assu-1}--\ref{assu-5} hold.
		Then there exist an integer $n_{0}\geq2$ and a constant $C>0$
		such that, for every $n\geq n_{0}$ and every $1\leq i\leq n-1$,
		\begin{align}\label{eq-pz-01}
			\begin{aligned}
				\|D_{4,n}\|_{2}
				&\leq CC_{\MM}^{2}\tau_{1}^{2}n^{1/2-\beta},\quad
				\|D_{4,n}-D_{4,n}^{\{i\}}\|_{2}
				\leq CC_{\MM}^{3}\tau_{1}^{3}
				n^{-1/2}i^{1-3\beta/2},
			\end{aligned}
		\end{align}
		where $\tau_{1}=(1+\delta+\tau_{0}+\tau)
		\{1+\delta^{-1}(1+\tau_{0}+\tau)\}$ and $C_{\MM}$ is defined
		by \eqref{eq-defi-cm}.
\end{lemma}

\begin{proof}
	Recall that
	\begin{align*}
		D_{4,n}
		=
		\frac{\Sigma_{n}^{-1/2}}{\sqrt n}
		\sum_{k=1}^{n-1}Q_{k}
		\big\{
		U^{\top}\eta_{k}+R(y_{k-1})+G_{k-1}
		\big\},
	\end{align*}
	where
	$R(y)=U^{\top}F_{\MM}(y)-HU^{\top}(y-x_{*})$
	and
	$G_{k-1}=U^{\top}
	\{P_{T_{\MM}(y_{k-1})}-P_{T_{\MM}(x_{*})}\}\zeta_{k}$.

		Choose $n_{0}\geq2$ as in Lemma~\ref{lem-A4}.
		Throughout the proof, $n\geq n_{0}$.
		We first bound $\|D_{4,n}\|_{2}$.
		Using Assumption~\ref{assu-4}, \eqref{eq-qq-05}--\eqref{eq-qq-06},
		Lemma~\ref{lem-3}, and the local Lipschitz bound for the tangent
		projections, we have, for every $k\geq1$,
		\begin{align*}
			\|\eta_{k}\|_{2}
			&\leq c_{8}\|x_{k-1}-x_{*}\|_{2}
			\leq C\tau_{1}k^{-\beta/2},\\
			\|G_{k-1}\|_{2}
			&\leq CC_{\MM}
			\|y_{k-1}-x_{*}\|_{4}\|\zeta_{k}\|_{4}
			\leq CC_{\MM}\tau_{1}^{2}k^{-\beta/2}.
		\end{align*}
		By Assumption~\ref{assu-4},
		$\{U^{\top}\eta_{k}+G_{k-1}\}_{k\geq1}$ is a martingale
		difference sequence with respect to
		$\{\sigma(x_{0},\allowbreak\xi_{1},\allowbreak\ldots,\allowbreak\xi_{k})\}_{k\geq0}$.
		Since $Q_{k}$ is deterministic, orthogonality and
		Lemma~\ref{lem-A4} give
		\begin{align}\label{eq-ua-01}
			\begin{aligned}
						&\bigg\|\frac{\Sigma_{n}^{-1/2}}{\sqrt n}
			\sum_{k=1}^{n-1}Q_{k}
			\{U^{\top}\eta_{k}+G_{k-1}\}\bigg\|_{2}\\
			&\quad\leq
			C\bigg(\frac1n\sum_{k=1}^{n-1}
			\{\E|\eta_{k}|^{2}+\E|G_{k-1}|^{2}\}\bigg)^{1/2}\\
			&\quad\leq CC_{\MM}\tau_{1}^{2}
			\bigg(\frac1n\sum_{k=1}^{n-1}k^{-\beta}\bigg)^{1/2}
			\leq CC_{\MM}\tau_{1}^{2}n^{-\beta/2}.
			\end{aligned}
		\end{align}

		We next bound $R$ and its derivative by using the natural extension
		$\widehat F_{\MM}=F_{\MM}\circ P_{\MM}$ from
		Assumption~\ref{assu-1}.
		For every $z\in B_{4\delta}(x_{*})$, the choice of $\delta$ gives
		\begin{align*}
			|P_{\MM}(z)-x_{*}|
			&\leq |P_{\MM}(z)-z|+|z-x_{*}|
			\leq 2|z-x_{*}|<8\delta\leq\delta_{2},
		\end{align*}
		and $4\delta\leq\delta_{1}/2$.
		Thus $P_{\MM}$ is smooth along every segment used below, and
		$\widehat F_{\MM}$ is well-defined and $C^{2}$ on
		$B_{4\delta}(x_{*})$ by its defining composition.
		Define
		\begin{align}\label{eq-pz-02}
			\widehat R(z)
			:=U^{\top}\widehat F_{\MM}(z)-HU^{\top}(z-x_{*}),
			\qquad z\in B_{4\delta}(x_{*}).
		\end{align}
		For $y\in\MM\cap B_{4\delta}(x_{*})$, we have
		$\widehat R(y)=R(y)$.

		We first derive a second derivative bound on this Euclidean ball.
		For every $y\in\MM\cap B_{8\delta}(x_{*})$,
		Assumption~\ref{assu-1} and
		$\nabla P_{\MM}(y)=P_{T_{\MM}(y)}$ imply
		\begin{align}\label{eq-pz-3}
			\begin{aligned}
				\nabla\widehat F_{\MM}(y)
				&=\nabla_{\MM}F_{\MM}(y)P_{T_{\MM}(y)},\\
				\|\nabla\widehat F_{\MM}(y)\|_{\op}
				&\leq L_{\MM},\qquad
				\|\nabla^{2}\widehat F_{\MM}(y)\|_{\op}\leq M_{2}.
			\end{aligned}
		\end{align}
		Since $\widehat F_{\MM}\circ P_{\MM}=\widehat F_{\MM}$,
		the second derivative chain rule gives, for
		$z\in B_{4\delta}(x_{*})$ and $u,v\in\R^{d}$,
		\begin{align*}
			\nabla^{2}\widehat F_{\MM}(z)[u,v]
			&=\nabla^{2}\widehat F_{\MM}(P_{\MM}(z))
			[\nabla P_{\MM}(z)u,\nabla P_{\MM}(z)v]\\
			&\quad+
			\nabla\widehat F_{\MM}(P_{\MM}(z))
			\nabla^{2}P_{\MM}(z)[u,v].
		\end{align*}
		In this identity, the derivatives of $\widehat F_{\MM}$ on the
		right-hand side are evaluated at the manifold point
		$P_{\MM}(z)\in\MM\cap B_{8\delta}(x_{*})$.
		Using \eqref{eq-defi-cm} and \eqref{eq-pz-3}, we obtain
		\begin{align}\label{eq-pz-04}
			\begin{aligned}
				\sup_{z\in B_{4\delta}(x_{*})}
				\|\nabla^{2}\widehat R(z)\|_{\op}
				&\leq
				\sup_{z\in B_{4\delta}(x_{*})}
				\|\nabla^{2}\widehat F_{\MM}(z)\|_{\op}\\
			&	\leq M_{2}C_{\MM}^{2}+L_{\MM}C_{\MM}
				\leq CC_{\MM}^{2}.
			\end{aligned}
		\end{align}
		Since $F_{\MM}(x_{*})=0$,
		$P_{T_{\MM}(x_{*})}=UU^{\top}$, and
		$H=U^{\top}\nabla_{\MM}F_{\MM}(x_{*})U$, we have
		\begin{align}\label{eq-pz-05}
			\begin{aligned}
				\widehat R(x_{*})&=0,\quad
				\nabla\widehat R(x_{*})
				=U^{\top}\nabla_{\MM}F_{\MM}(x_{*})
				P_{T_{\MM}(x_{*})}-HU^{\top}=0.
			\end{aligned}
		\end{align}
		For every $z\in B_{4\delta}(x_{*})$, the segment joining
		$x_{*}$ and $z$ lies in the same ball. Hence, by
		\eqref{eq-pz-04}--\eqref{eq-pz-05}, we have
		\begin{align*}
			\|\nabla\widehat R(z)\|_{\op}
			&\leq CC_{\MM}^{2}|z-x_{*}|,\quad \text{and}\quad
			|\widehat R(z)|
			\leq CC_{\MM}^{2}|z-x_{*}|^{2}.
		\end{align*}
		Applying the second inequality at $z=y_{k-1}$ and using
		\eqref{eq-qq-05} and Lemma~\ref{lem-3}, with the initial
		moment assumption when $k=1$, we obtain
		\begin{align}\label{eq-ua-02}
			\|R(y_{k-1})\|_{2}
			&\leq CC_{\MM}^{2}\|y_{k-1}-x_{*}\|_{4}^{2}
			\leq CC_{\MM}^{2}\tau_{1}^{2}k^{-\beta}.
		\end{align}
		Combining \eqref{eq-ua-01}, \eqref{eq-ua-02},  Lemma~\ref{lem-A4} and
		Minkowski's inequality, we have
		\begin{align*}
			\|D_{4,n}\|_{2}
			&\leq CC_{\MM}\tau_{1}^{2}n^{-\beta/2}
			+\frac{C}{\sqrt n}\sum_{k=1}^{n-1}\|R(y_{k-1})\|_{2}\\
			&\leq CC_{\MM}\tau_{1}^{2}n^{-\beta/2}
			+\frac{CC_{\MM}^{2}\tau_{1}^{2}}{\sqrt n}
			\sum_{k=1}^{n-1}k^{-\beta}\\
			&\leq CC_{\MM}^{2}\tau_{1}^{2}n^{1/2-\beta},
		\end{align*}
		where we used $\beta<1$ and $C_{\MM}\geq1$.
		This proves the first inequality of \eqref{eq-pz-01}.

	We now prove the second inequality of \eqref{eq-pz-01}.
		Fix $1\leq i\leq n-1$. The two trajectories coincide up to time
		$i-1$, so the definitions and Lemma~\ref{lem-A4} give
	\begin{align}\label{eq-pz-06}
		\begin{aligned}
			\|D_{4,n}-D_{4,n}^{\{i\}}\|_{2}
			&\leq\frac{C}{\sqrt n}
			\bigg\|\sum_{k=i}^{n-1}Q_{k}
			\big\{
			U^{\top}\eta_{k}+G_{k-1}
			-U^{\top}\eta_{k}^{\{i\}}-G_{k-1}^{\{i\}}
			\big\}\bigg\|_{2}\\
			&\quad+\frac{C}{\sqrt n}
			\bigg\|\sum_{k=i+1}^{n-1}Q_{k}
			\big\{
			R(y_{k-1})-R(y_{k-1}^{\{i\}})
			\big\}\bigg\|_{2}.
		\end{aligned}
	\end{align}
		We first estimate the first term on the right-hand side of
		\eqref{eq-pz-06}. For $k>i$, the two trajectories use the same
		current innovation, and hence
		\begin{align*}
			|\eta_{k}-\eta_{k}^{\{i\}}|
			&=|g(x_{k-1},\xi_{k})-g(x_{k-1}^{\{i\}},\xi_{k})|
			\leq c_{8}|x_{k-1}-x_{k-1}^{\{i\}}|.
		\end{align*}
		For $k=i$, we instead use
		\begin{align*}
			|\eta_{i}-\eta_{i}^{\{i\}}|
			&=|g(x_{i-1},\xi_{i})-g(x_{i-1},\xi_{i}')|
			\leq2c_{8}|x_{i-1}-x_{*}|.
		\end{align*}
		Thus Lemmas~\ref{lem-3} and~\ref{lem-UpperB-xn-xni-fourth},
		together with the initial moment assumption when $i=1$, imply
		\begin{align}\label{eq-ua-03}
			\|\eta_{k}-\eta_{k}^{\{i\}}\|_{2}
			&\leq C\tau_{1}
			\begin{cases}
				i^{-\beta/2},&k=i,\\
				k^{-\beta},&k>i.
			\end{cases}
		\end{align}
	For $G_{k-1}$, observe that
	\begin{align*}
		G_{k-1}-G_{k-1}^{\{i\}}
		&=U^{\top}
		\big\{P_{T_{\MM}(y_{k-1})}
		-P_{T_{\MM}(y_{k-1}^{\{i\}})}\big\}\zeta_{k}+U^{\top}
		\big\{P_{T_{\MM}(y_{k-1}^{\{i\}})}
		-P_{T_{\MM}(x_{*})}\big\}
		(\zeta_{k}-\zeta_{k}^{\{i\}}).
	\end{align*}

		To estimate this difference, we first control
		$\|y_{k-1}-y_{k-1}^{\{i\}}\|_{4}$ for $k>i$.
		The localized part is bounded by \eqref{eq-qq-08}, with
		$k-1$ in place of $k$, and the monotonicity of $L^{p}$ norms.
		On the complementary event, \eqref{eq-qq-05} gives
		\begin{align*}
			&|y_{k-1}-y_{k-1}^{\{i\}}|
			\mathbf{1}\big\{
			x_{k-1}\notin B_{\delta}(x_{*})
			\text{ or }x_{k-1}^{\{i\}}\notin B_{\delta}(x_{*})
			\big\}\leq C\delta^{-1}
			\big\{|x_{k-1}-x_{*}|
			+|x_{k-1}^{\{i\}}-x_{*}|\big\}^{2}.
		\end{align*}
		Taking $L^{4}$ norms, and applying Lemma~\ref{lem-3}, we obtain
		\begin{align}\label{eq-pz-07}
			\begin{aligned}
				\|y_{k-1}-y_{k-1}^{\{i\}}\|_{4}
				&\leq
				\big\|(y_{k-1}-y_{k-1}^{\{i\}})
				\mathbf{1}\{x_{k-1},x_{k-1}^{\{i\}}
				\in B_{\delta}(x_{*})\}\big\|_{4}\\
				&\quad+C\delta^{-1}
				\|x_{k-1}-x_{*}\|_{8}^{2}\\
				&\leq CC_{\MM}\tau_{1}k^{-\beta}
				+C\delta^{-1}(1+\tau_{0}+\tau)^{2}k^{-\beta}\\
				&\leq CC_{\MM}\tau_{1}k^{-\beta},
				\qquad k>i.
			\end{aligned}
		\end{align}
		For $k=i$, we have $y_{i-1}=y_{i-1}^{\{i\}}$.
		Moreover, Assumption~\ref{assu-4} and the preceding moment bounds
		give
		\begin{align*}
			\|\zeta_{i}-\zeta_{i}^{\{i\}}\|_{4}
			&\leq C\tau_{1},\quad
			\|\zeta_{k}-\zeta_{k}^{\{i\}}\|_{4}
			=\|\eta_{k}-\eta_{k}^{\{i\}}\|_{4}
			\leq C\tau_{1}k^{-\beta},\qquad k>i,
		\end{align*}
		where the second inequality uses
		Lemma~\ref{lem-UpperB-xn-xni-fourth}.
		By the local Lipschitz bound for the tangent projections and
		H\"older's inequality, we have
		\begin{align}\label{eq-ua-04}
			\begin{aligned}
						\|G_{k-1}-G_{k-1}^{\{i\}}\|_{2}
			&\leq CC_{\MM}
			\|y_{k-1}-y_{k-1}^{\{i\}}\|_{4}\|\zeta_{k}\|_{4}\\
			&\quad+CC_{\MM}
			\|y_{k-1}^{\{i\}}-x_{*}\|_{4}
			\|\zeta_{k}-\zeta_{k}^{\{i\}}\|_{4}\\
			&\leq CC_{\MM}^{2}\tau_{1}^{2}
			\begin{cases}
				i^{-\beta/2},&k=i,\\
				k^{-\beta},&k>i.
			\end{cases}
			\end{aligned}
		\end{align}
		For $k>i$, by  Assumption \ref{assu-4}, we have
		\begin{align*}
			\E\big[
			U^{\top}(\eta_{k}-\eta_{k}^{\{i\}})
			+G_{k-1}-G_{k-1}^{\{i\}}
			\mid\mathscr{F}_{k-1}^{\{i\}}\big]=0.
		\end{align*}
		The difference is measurable with respect to
		$\mathscr{F}_{k}^{\{i\}}$.
		We treat the term $k=i$ separately and use orthogonality for the
		terms $k>i$. By \eqref{eq-ua-03}, \eqref{eq-ua-04} and Lemma~\ref{lem-A4}, we have
		\begin{align}\label{eq-pz-08}
			\begin{aligned}
				&\frac{C}{\sqrt n}
				\bigg\|\sum_{k=i}^{n-1}Q_{k}
				\big\{
				U^{\top}\eta_{k}+G_{k-1}
				-U^{\top}\eta_{k}^{\{i\}}-G_{k-1}^{\{i\}}
				\big\}\bigg\|_{2}\\
				&\quad\leq CC_{\MM}^{2}\tau_{1}^{2}n^{-1/2}
				\bigg\{i^{-\beta/2}
				+\bigg(\sum_{k=i+1}^{n-1}k^{-2\beta}\bigg)^{1/2}
				\bigg\}\\
				&\quad\leq CC_{\MM}^{2}\tau_{1}^{2}
				n^{-1/2}i^{1/2-\beta},
			\end{aligned}
		\end{align}
		where we used $2\beta>1$ and
		$i^{-\beta/2}\leq i^{1/2-\beta}$ for $\beta<1$.

		We next estimate the second term on the right-hand side of
		\eqref{eq-pz-06}. By \eqref{eq-qq-05}, both $y_{k-1}$ and
		$y_{k-1}^{\{i\}}$ lie in $\MM\cap B_{4\delta}(x_{*})$.
		The segment joining them lies in the Euclidean ball
		$B_{4\delta}(x_{*})$, on which the bound for
		$\nabla\widehat R$ obtained above is valid. Hence,
		\begin{align*}
			|R(y_{k-1})-R(y_{k-1}^{\{i\}})|
			&=|\widehat R(y_{k-1})
			-\widehat R(y_{k-1}^{\{i\}})|\\
			&\leq\int_{0}^{1}
			\big\|\nabla\widehat R\big(
			y_{k-1}^{\{i\}}+t(y_{k-1}-y_{k-1}^{\{i\}})
			\big)\big\|_{\op}
			|y_{k-1}-y_{k-1}^{\{i\}}|\,dt\\
			&\leq CC_{\MM}^{2}
			\big\{|y_{k-1}-x_{*}|+|y_{k-1}^{\{i\}}-x_{*}|\big\}
			|y_{k-1}-y_{k-1}^{\{i\}}|.
		\end{align*}
		Using \eqref{eq-pz-07}, Lemma~\ref{lem-3}, and
		H\"older's inequality, we obtain, for $k>i$,
		\begin{align}\label{eq-pz-09}
			\begin{aligned}
				&\|R(y_{k-1})-R(y_{k-1}^{\{i\}})\|_{2}\\
				&\quad\leq CC_{\MM}^{2}
				\big\{\|y_{k-1}-x_{*}\|_{4}
				+\|y_{k-1}^{\{i\}}-x_{*}\|_{4}\big\}
				\|y_{k-1}-y_{k-1}^{\{i\}}\|_{4}\\
				&\quad\leq CC_{\MM}^{3}\tau_{1}^{2}k^{-3\beta/2}
				\leq CC_{\MM}^{3}\tau_{1}^{3}k^{-3\beta/2},
			\end{aligned}
		\end{align}
		where the last inequality uses $\tau_{1}\geq1$.
		By Minkowski's inequality and \eqref{eq-pz-09}, we have
		\begin{align}\label{eq-pz-10}
			\begin{aligned}
				&\frac{C}{\sqrt n}\bigg\|
				\sum_{k=i+1}^{n-1}Q_{k}
				\big\{R(y_{k-1})-R(y_{k-1}^{\{i\}})\big\}
				\bigg\|_{2}\\
				&\quad\quad\leq
				\frac{CC_{\MM}^{3}\tau_{1}^{3}}{\sqrt n}
				\sum_{k=i+1}^{n-1}k^{-3\beta/2}
				\leq CC_{\MM}^{3}\tau_{1}^{3}
				n^{-1/2}i^{1-3\beta/2}.
			\end{aligned}
		\end{align}
		Combining \eqref{eq-pz-06}, \eqref{eq-pz-08}, and
		\eqref{eq-pz-10}, and using
		$i^{1/2-\beta}\leq i^{1-3\beta/2}$, $C_{\MM}\geq1$, and
		$\tau_{1}\geq1$, we obtain

	\begin{align*}
		\|D_{4,n}-D_{4,n}^{\{i\}}\|_{2}
		\leq CC_{\MM}^{3}\tau_{1}^{3}
		n^{-1/2}i^{1-3\beta/2}.
	\end{align*}
		This proves the second inequality of \eqref{eq-pz-01} and
		completes the proof.
\end{proof}

Now, we are ready to prove Lemma \ref{lem-ap_nonsmooth-1}.
\begin{proof}[Proof of Lemma \ref{lem-ap_nonsmooth-1}]
	It  follows from the definition of $\varepsilon_{k}$
	and Lemma \ref{lem-A4} that  $\E|\varepsilon_{k}|^{3}\leq Cn^{-3/2}\tau^{3}$. Hence,
	\eqref{eq:X-third-sum} holds. For \eqref{eq:D-norm}, by Minkowski's inequality, we have
	\begin{align}
		\|D_{n}\|_{3/2}\leq \|D_{1,n}\|_{2}+\|D_{2,n}\|_{2}+\|D_{3,n}\|_{2}+\|D_{4,n}\|_{2}.
	\end{align}
	For $D_{2,n}$, by Lemma \ref{lem-A4}, we have $\|D_{2,n}\|_{2}\leq Cn^{-1/2}\tau_{0}.$ Combining this
	with Lemmas \ref{lem-D1-bound}, \ref{lem-D3-bound} and \ref{lem-D4-bound}, we obtain \eqref{eq:D-norm}.
	As for \eqref{eq:D-LOO}, note that by the definition of $D_{2,n}^{\{i\}}$, we have $D_{2,n}-D_{2,n}^{\{i\}}=0$.
	Hence, by Lemmas \ref{lem-D1-bound}, \ref{lem-D3-bound} and \ref{lem-D4-bound},  we get \eqref{eq:D-LOO}.
\end{proof}
\section{Further details for additional applications}\label{app:additional-applications}
This section provides further details for the applications to
temporal-difference learning with linear function approximation and
multivariate \(U\)-statistics. For temporal-difference learning, we
prove Lemma~\ref{lem:TD-remainder-summary}, show that the two additional
terms arising in the proof of Theorem~\ref{thm-app-TD} can be absorbed
into the right-hand side of \eqref{eq-app2-thm-01}, and verify the
dimension-only rates. For multivariate \(U\)-statistics, we prove
Lemma~\ref{lem:U-remainder-general}.
\subsection{Proof of Lemma \ref{lem:TD-remainder-summary}}
\label{sec:appendix-TD}

In this subsection, we provide the proof of Lemma \ref{lem:TD-remainder-summary}.
Throughout the section, $C$ and $c$
denote positive constants that may change from line to line and depend only
on $\alpha$ and $K_\phi$. In particular, they never depend on $d$, $T$,
$\eta_0$, $\|\Delta_{0}\|_{2}$, $\nu_3$, $\operatorname{Tr}(\Gamma)$, or
$\underline\lambda_T$.

We begin with a lemma collecting the basic matrix inequalities for $A$ and its
stochastic counterpart $A_t$, which will be used in the proof of
Lemma~\ref{lem:TD-remainder-summary}.
\begin{lemma}\label{lem:TD-feature-algebra}
	Let $H=A+A^\top$. Then
	\begin{align}
		2(1-\gamma)\Sigma_\phi
		\preceq H\preceq2(1+\gamma)\Sigma_\phi,
		\label{eq-kk-01}
	\end{align}
	and
	\begin{align}
		A^\top A\preceq\lambda_\phi H,
		\qquad
		\E(A_t^\top A_t)\preceq B_\phi^2H.
		\label{eq-kk-02}
	\end{align}
	Consequently, $A$ is nonsingular and, for
	$0\leq\eta\leq(2\lambda_\phi)^{-1}$,
	\begin{align}
		\|I_d-\eta A\|_{\op}
		\leq1-\frac12(1-\gamma)\lambda_0\eta.
		\label{eq-kk-03}
	\end{align}
	Moreover,
	\begin{align}
		\E\{(A_t-A)^\top(A_t-A)\}
		\preceq4B_\phi^2\lambda_\phi I_d.
		\label{eq-kk-04 }
	\end{align}
\end{lemma}

\begin{proof}
By the definitions of $A$ and $\Sigma_\phi$, we have
\[
H
=2\Sigma_\phi-
\gamma\E\bigl\{
\phi(s_t)\phi(s_t')^\top+
\phi(s_t')\phi(s_t)^\top
\bigr\}.
\]
For arbitrary $x,y\in\mathbb{R}^d$,
$-xx^\top-yy^\top\preceq xy^\top+yx^\top\preceq xx^\top+yy^\top$.
Applying this pointwise with $x=\phi(s_t)$ and $y=\phi(s_t')$,
and using the fact that $s_t$ and $s_t'$ have the same marginal
distribution, we obtain \eqref{eq-kk-01}.  If $Ax=0$, then
	$x^\top Hx=2x^\top Ax=0$, so $x=0$. Hence $A$ is nonsingular.

	Fix $x\in\R^d$ and define
	$
	f(s)=\E[
	\{\phi(s)-\gamma\phi(s_t')\}^\top x
	\,|\,s_t=s
	].
	$
	For every unit vector $u$, by  Cauchy--Schwarz  inequality, we have
	\[
	|u^\top Ax|^2
	=\big|\E\{u^\top\phi(s_t)f(s_t)\}\big|^2
	\leq u^\top\Sigma_\phi u\,\E f^{2}(s_t)
	\leq\lambda_\phi\E f^{2}(s_t).
	\]
	Taking the supremum over $|u|=1$ yields
	$|Ax|^2\leq\lambda_\phi\E f^{2}(s_t)$. By conditional Jensen's inequality and \eqref{eq-kk-01}, we have
	\[
	\E f^{2}(s_t)
	\leq x^\top
	\bigl\{H-(1-\gamma^2)\Sigma_\phi\bigr\}x
	\leq x^\top Hx.
	\]
	Thus $A^\top A\preceq\lambda_\phi H$.
	Similarly, for every $x\in\R^d$,
	\begin{align*}
		x^\top\E(A_t^\top A_t)x
		&=\E\bigl[
		|\phi(s_t)|^2
		\bigl\{\{\phi(s_t)-\gamma\phi(s_t')\}^\top x\bigr\}^2
		\bigr]
		\\
		&\leq B_\phi^2
		x^\top\bigl\{H-(1-\gamma^2)\Sigma_\phi\bigr\}x
		\leq B_\phi^2x^\top Hx,
	\end{align*}
	which proves the second inequality in
	\eqref{eq-kk-02}.

	If $\eta\leq(2\lambda_\phi)^{-1}$,  by \eqref{eq-kk-01}, \eqref{eq-kk-02} and the fact that $\Sigma_{\phi}\succeq \lambda_{0}I_{d}$, we have
	\begin{align*}
		(I_d-\eta A)^\top(I_d-\eta A)
		&\preceq I_d-\eta(1-\eta\lambda_\phi)H\preceq I_d-\eta(1-\gamma)\Sigma_{\phi}\preceq\{1-(1-\gamma)\lambda_0\eta\}I_d.
	\end{align*}
	Note that $(1-\gamma)\lambda_0\eta\leq1/2$ and
	$\sqrt{1-x}\leq1-x/2$ for $x\in[0,1]$. Combining this with the above inequality, we have
	\eqref{eq-kk-03}. Finally, observe that
	\[
	\E\{(A_t-A)^\top(A_t-A)\}
	=\E(A_t^\top A_t)-A^\top A
	\preceq B_\phi^2H
	\preceq4B_\phi^2\lambda_\phi I_d.
	\]
	This	proves \eqref{eq-kk-04 }.
\end{proof}

For integers $j\geq i-1$ and $0<r\leq1/2$, set
\[
P_{i,j}(r)=\prod_{k=i}^j(1-rk^{-\alpha}),
\qquad
P_{i,i-1}(r)=1.
\]
The following lemma establishes elementary bounds for the scalar kernel
$P_{i,j}(r)$ generated by the polynomially decaying step size.
\begin{lemma}\label{lem:TD-scalar-kernel}
	There are constants $C,c>0$ such that, for all integers $i,t\geq1$ and
	all $0<r\leq1/2$,
	\begin{align}
		\sum_{j=i}^{\infty}P_{i+1,j}(r)
		&\leq C\bigl\{r^{-1/(1-\alpha)}+r^{-1}i^\alpha\bigr\},
		\label{eq-kk-05}
		\\
		\sum_{i=1}^tP_{i+1,t}(r)
		&\leq Cr^{-1}t^\alpha,
		\qquad
		P_{1,t}(r)\leq\exp\{-cr t^{1-\alpha}\}.
		\label{eq-kk-06}
	\end{align}
\end{lemma}

\begin{proof}
	Using the inequality $1-x\leq e^{-x}$, we have  for every integer $m\geq0$,
	\[
	P_{i+1,i+m}(r)
	\leq\exp\bigg\{-r\sum_{k=i+1}^{i+m}k^{-\alpha}\bigg\}.
	\]
	If $0\leq m\leq i$, then
	$\sum_{k=i+1}^{i+m}k^{-\alpha}\geq m(2i)^{-\alpha}$, and hence
	\[
	\sum_{m=0}^iP_{i+1,i+m}(r)
	\leq\sum_{m=0}^{\infty}\exp(-r(2i)^{-\alpha}m)
	\leq Cr^{-1}i^\alpha.
	\]
	If $m>i$, at least $\lfloor m/2\rfloor$ summands have indices at most
	$2m$, so the exponent is bounded below by $cr m^{1-\alpha}$. Hence,
	\[
	\sum_{m=i+1}^{\infty}P_{i+1,i+m}(r)
	\leq\sum_{m=1}^{\infty}\exp(-cr m^{1-\alpha})
	\leq Cr^{-1/(1-\alpha)}.
	\]
	This proves \eqref{eq-kk-05}.

	To prove the first inequality in \eqref{eq-kk-06}, split
	the sum at $\lfloor t/2\rfloor$. If $i\leq\lfloor t/2\rfloor$, then
	$P_{i+1,t}(r)\leq \exp(-cr t^{1-\alpha})$, and therefore
	\[
	\sum_{i=1}^{\lfloor t/2\rfloor}P_{i+1,t}(r)
	\leq t \exp(-cr t^{1-\alpha})
	\leq Cr^{-1}t^\alpha.
	\]
	If $i>\lfloor t/2\rfloor$, then
	$\sum_{k=i+1}^tk^{-\alpha}\geq(t-i)t^{-\alpha}$. Thus
	\[
	\sum_{i=\lfloor t/2\rfloor+1}^tP_{i+1,t}(r)
	\leq\sum_{m=0}^{\infty}e^{-rmt^{-\alpha}}
	=\frac1{1-e^{-rt^{-\alpha}}}
	\leq Cr^{-1}t^\alpha,
	\]
	where the last inequality follows from the inequality  $1-e^{-x}\geq x/2$ for $x\in[0,1]$. Finally,
	$\sum_{k=1}^tk^{-\alpha}\geq ct^{1-\alpha}$, which proves the product
	bound in \eqref{eq-kk-06}.
\end{proof}

For convenience, we let
$
\rho=\frac12(1-\gamma)\lambda_0\eta_0.
$
The feature regularity condition and the stepsize restriction imply
\[
C^{-1}\eta_0\leq\rho\leq C\eta_0,
\qquad
0<\rho\leq\frac18,
\qquad
B_\phi\leq\frac1{2 \sqrt{\eta_{0}}},
\qquad
\lambda_\phi\leq C.
\]

The following lemma is useful in the proof of Lemma \ref{lem:TD-remainder-summary}.
\begin{lemma}\label{lem:TD-kernel-trajectory-bounds}
	There are constants $C,c>0$ such that, for $1\leq t\leq T$,
	\begin{align}
		\|AQ_t\|_{\op}
		&\leq C\bigl(1+b_\eta^\alpha t^{-\alpha}\bigr),
		\qquad
		\frac{\|AQ_0\|_{\op}}{\eta_0}\leq Cb_\eta,
		\label{eq-ww-01}
		\\
		\E|\Delta_t|^2
		&\leq C\Big[
		\bigl\{|\Delta_{0}|^2+\eta_0\operatorname{Tr}(\Gamma)\bigr\}
		e^{-c\eta_0t^{1-\alpha}}
		+\eta_0\operatorname{Tr}(\Gamma)t^{-\alpha}
		\Big],
		\label{eq-ww-02}
	\end{align}
where $b_\eta=1+\eta_0^{-1/(1-\alpha)}$.	Moreover, for $m\in\{2,3\}$,
	\begin{align}
		\sum_{t=1}^T\|AQ_t\|_{\op}^m
		\leq C\bigl(T+b_\eta^{m\alpha}\bigr),
		\label{eq-ww-03}
	\end{align}
	and
	\begin{align}
		\sum_{t=1}^T\|AQ_t\|_{\op}^2\E|\Delta_{t-1}|^2
		&\leq C\Big[
		\eta_0\operatorname{Tr}(\Gamma)
		\bigl\{T^{1-\alpha}+b_\eta^{2\alpha}\bigr\}
		+|\Delta_{0}|^2b_\eta^{2\alpha}
		\Big],
		\label{eq-ww-04}
		\\
		\sum_{t=1}^T\|AQ_t\|_{\op}^2\|\Delta_{t-1}\|_2
		&\leq C\Big[
		\bigl\{\eta_0\operatorname{Tr}(\Gamma)\bigr\}^{1/2}
		\bigl\{T^{1-\alpha/2}+b_\eta^{2\alpha}\bigr\}
		+|\Delta_{0}| b_\eta^{2\alpha}
		\Big].
		\label{eq-ww-05}
	\end{align}
\end{lemma}

\begin{proof}
	By
	Lemma~\ref{lem:TD-feature-algebra} and the fact that
	$\eta_t\leq(2\lambda_\phi)^{-1}$, we have
	$
	\|I_d-\eta_tA\|_{\op}
	\leq1-\rho t^{-\alpha}.
	$
	Combining this with Lemma \ref{lem:TD-feature-algebra}, the definition of $Q_{t}$ and  \eqref{eq-kk-05}, we have
	\begin{align*}
		\|AQ_t\|_{\op}
		&\leq2\lambda_\phi\eta_0t^{-\alpha}
		\sum_{j=t}^TP_{t+1,j}(\rho)
		\leq C\eta_0t^{-\alpha}
		\bigl\{\rho^{-1/(1-\alpha)}+\rho^{-1}t^\alpha\bigr\}
		\leq C\bigl(1+b_\eta^\alpha t^{-\alpha}\bigr),
	\end{align*}
	where the last inequality
	used $\rho\asymp\eta_0$ and
	$\eta_0\rho^{-1/(1-\alpha)}\leq Cb_\eta^\alpha$.
	Similarly, we have
	\[
	\frac{\|AQ_0\|_{\op}}{\eta_0}
	\leq2\lambda_\phi\sum_{j=1}^TP_{1,j}(\rho)
	\leq C\rho^{-1/(1-\alpha)}
	\leq Cb_\eta.
	\]
	This proves \eqref{eq-ww-01}.

	From
	$\Delta_t=\Delta_{t-1}-\eta_t(A_t\Delta_{t-1}+\varepsilon_t)$,
	the independence of the current sample from $\Delta_{t-1}$, and
	$\E\varepsilon_t=0$, we obtain
	\begin{align*}
		\E|\Delta_t|^2
		&=\E|\Delta_{t-1}|^2
		-\eta_t\E(\Delta_{t-1}^\top H\Delta_{t-1})+
		\eta_t^2\E|A_t\Delta_{t-1}+\varepsilon_t|^2
		\\
		&\leq\E|\Delta_{t-1}|^2
		-\eta_t(1-2\eta_tB_\phi^2)
		\E(\Delta_{t-1}^\top H\Delta_{t-1})
		+2\eta_t^2\operatorname{Tr}(\Gamma)
		\\
		&\leq(1-2\rho t^{-\alpha})\E|\Delta_{t-1}|^2
		+2\eta_0^2\operatorname{Tr}(\Gamma)t^{-2\alpha},
	\end{align*}
	where
	the second inequality  follows from
	$|x+y|^2\leq2|x|^2+2|y|^2$ and
	\eqref{eq-kk-02}; the last inequality uses
	$\eta_tB_\phi^2\leq1/4$ and \eqref{eq-kk-01}.
	Iteration yields
	\[
	\E|\Delta_t|^2
	\leq P_{1,t}(2\rho)|\Delta_{0}|^2
	+2\eta_0^2\operatorname{Tr}(\Gamma)
	\sum_{i=1}^ti^{-2\alpha}P_{i+1,t}(2\rho).
	\]
	For $i\leq t/2$, the product is at most
	$e^{-c\rho t^{1-\alpha}}$ and
	$\sum_{i\geq1}i^{-2\alpha}<\infty$. For $i>t/2$, use
	$i^{-2\alpha}\leq Ct^{-2\alpha}$ and
	\eqref{eq-kk-06}.
The assumptions imply $B_\phi^2\geq\lambda_0\geq K_\phi^{-1}$
	and hence $\eta_0\leq K_\phi/4$, so $\eta_0^2\leq C\eta_0$.
	Together with $\rho\asymp\eta_0$, this proves \eqref{eq-ww-02}.

	To prove the sum bounds, write
	$g_t=1+b_\eta^\alpha t^{-\alpha}$. Split every sum at
	$t=b_\eta$. If $t\leq b_\eta$, then
	$g_t\leq2b_\eta^\alpha t^{-\alpha}$; if $t>b_\eta$, then $g_t\leq2$.
	Since $m\alpha>1$ for $m\in\{2,3\}$, we have
	\[
	\sum_{t=1}^Tg_t^m\leq C\bigl(T+b_\eta^{m\alpha}\bigr).
	\]
	Similarly, we have
	\[
	\sum_{t=1}^Tg_t^2t^{-\alpha}
	\leq C\bigl\{T^{1-\alpha}+b_\eta^{2\alpha}\bigr\},
	\qquad
	\sum_{t=1}^Tg_t^2t^{-\alpha/2}
	\leq C\bigl\{T^{1-\alpha/2}+b_\eta^{2\alpha}\bigr\}.
	\]
	Finally,
	\[
	\sum_{t=1}^Tg_t^2e^{-c\eta_0t^{1-\alpha}}
	\leq Cb_\eta^{2\alpha}.
	\]
	For the last estimate, the part $t\leq b_\eta$ is summable after using
	$g_t\leq2b_\eta^\alpha t^{-\alpha}$, while the tail is of order
	$b_\eta$ and is bounded by $b_\eta^{2\alpha}$ because $2\alpha>1$.
	These estimates give \eqref{eq-ww-03} directly. For the sums involving
	$\Delta_{t-1}$, the $t=1$ terms are bounded directly using
	\eqref{eq-ww-01}. For $t\geq2$, we apply \eqref{eq-ww-02}, taking
	square roots when needed, and use $t-1\geq t/2$, with the constants
	adjusted in the polynomial and exponential bounds. This proves
	\eqref{eq-ww-04} and \eqref{eq-ww-05}.
\end{proof}

Now, we are ready to prove Lemma \ref{lem:TD-remainder-summary}.
\begin{proof}[Proof of Lemma \ref{lem:TD-remainder-summary}.]
Recall that
	\[
	X_t=-\frac1{\sqrt T}AQ_t\varepsilon_t,\qquad
	D_T=-\frac1{\sqrt T}\sum_{t=1}^TAQ_t(A_t-A)\Delta_{t-1}.
	\]
	Observe that
	$
	\|X_t\|_3
	\leq \nu_3\|AQ_t\|_{\op}/\sqrt{T}.
	$
	Taking $m=3$ and using \eqref{eq-ww-03}, we obtain
	\[
	\sum_{t=1}^T\|X_t\|_3^3
	\leq C\nu_3^3
	\bigg\{
	\frac1{T^{1/2}}+
	\frac{b_\eta^{3\alpha}}{T^{3/2}}
	\bigg\}.
	\]
	This proves \eqref{eq-uu-01}.

We next estimate $D_T$. Its summands are martingale differences and therefore pairwise
	orthogonal in $L^2$. By \eqref{eq-kk-04 }, we have
	\[
	\E|D_T|^2
	\leq\frac{CB_\phi^2}{T}
	\sum_{t=1}^T\|AQ_t\|_{\op}^2\E|\Delta_{t-1}|^2.
	\]
	Combining this with \eqref{eq-ww-04} and the fact that $B_\phi\leq(2\sqrt{\eta_{0}})^{-1}$ , we have
	\begin{align*}
		\|D_T\|_2
		&\leq CB_\phi\bigg\{
		\frac{\{\eta_0\operatorname{Tr}(\Gamma)\}^{1/2}}{T^{\alpha/2}}
		+\frac{
			\bigl\{|\Delta_{0}|+
			(\eta_0\operatorname{Tr}(\Gamma))^{1/2}\bigr\}
			b_\eta^\alpha}{T^{1/2}}
		\bigg\}\\
		&	\leq C\bigg\{
		\frac{\operatorname{Tr}(\Gamma)^{1/2}}{T^{\alpha/2}}
		+\frac{\eta_0^{-1/2}
			\bigl\{|\Delta_{0}|+
			(\eta_0\operatorname{Tr}(\Gamma))^{1/2}\bigr\}
			b_\eta^\alpha}{T^{1/2}}
		\bigg\}.
	\end{align*}
This proves \eqref{eq-uu-02}.

	It remains to prove \eqref{eq-uu-03}. Fix $1\leq i\leq T$.
	Replace $(A_i,b_i)$ by an independent copy $(A_i',b_i')$, and let
	$\theta_j^{\{i\}}$ be the resulting trajectory. For $j\geq0$, define
	\[
	\mathcal F_j^{\{i\}}
	:=
	\begin{cases}
		\sigma\bigl((A_k,b_k):1\leq k\leq j\bigr), & j<i,\\
		\sigma\bigl((A_k,b_k):1\leq k\leq j,\,(A'_{i},b'_i)\bigr), & j\geq i.
	\end{cases}
	\]
	For the remainder of the
	argument, write
	$\delta_j=\theta_j-\theta_j^{\{i\}}$. The two trajectories agree up to
	time $i-1$, and
	\[
	\delta_i
	=-\eta_i\bigl\{
	(A_i-A_i')\Delta_{i-1}+\varepsilon_i-\varepsilon_i'
	\bigr\},
	\qquad
	\varepsilon_i'=A_i'\theta^\star-b_i'.
	\]
	Conditioning on the history through time $i-1$ and using
	\eqref{eq-kk-02} gives
	\begin{align}
		\E|\delta_i|^2
		\leq C\eta_0^2i^{-2\alpha}
		\bigl\{
		B_\phi^2\E|\Delta_{i-1}|^2+
		\operatorname{Tr}(\Gamma)
		\bigr\}.
		\label{eq-we-01}
	\end{align}
	For $j>i$, both trajectories use the same sample, so
	$\delta_j=(I_d-\eta_jA_j)\delta_{j-1}$. Conditional on the two
	trajectories through time $j-1$, Lemma~\ref{lem:TD-feature-algebra} and
	$\eta_jB_\phi^2\leq1/4$ imply
	\[
	\E\bigl(|\delta_j|^2\mid\mathcal F_{j-1}^{\{i\}}\bigr)
	\leq(1-2\rho j^{-\alpha})|\delta_{j-1}|^2.
	\]
	Iterating gives
	\begin{align}
		\E|\delta_j|^2
		\leq P_{i+1,j}(2\rho)\E|\delta_i|^2,
		\qquad j>i.
		\label{eq-we-02}
	\end{align}
	By
	\eqref{eq-ww-01},
	$\|AQ_j\|_{\op}^2\leq
	C\{1+b_\eta^{2\alpha}j^{-2\alpha}\}$. Therefore,
	\eqref{eq-kk-05} and
	$\sum_{j=i+1}^\infty j^{-2\alpha}\leq Ci^{1-2\alpha}$ give
	\begin{align*}
		&\sum_{j=i+1}^T
		\|AQ_j\|_{\op}^2P_{i+1,j-1}(2\rho)\leq C\big\{
		\rho^{-1/(1-\alpha)}+\rho^{-1}i^\alpha
		+b_\eta^{2\alpha}i^{1-2\alpha}
		\big\}.
	\end{align*}
	If $i\leq b_\eta$, then
	$b_\eta^{2\alpha}i^{1-2\alpha}\geq b_\eta$; if $i>b_\eta$, then
	$\rho^{-1}i^\alpha\geq cb_\eta$. Consequently, we have
	\begin{align}
		\sum_{j=i+1}^T
		\|AQ_j\|_{\op}^2P_{i+1,j-1}(2\rho)
		\leq C\bigl\{
		\rho^{-1}i^\alpha+
		b_\eta^{2\alpha}i^{1-2\alpha}
		\bigr\}.
		\label{eq-we-03}
	\end{align}
	By the definitions of $D_{T}$ and $D_{T}^{\{i\}}$, we have
	\[
	D_T-D_T^{\{i\}}
	=-\frac1{\sqrt T}AQ_i(A_i-A_i')\Delta_{i-1}
	-\frac1{\sqrt T}\sum_{j=i+1}^T
	AQ_j(A_j-A)\delta_{j-1}.
	\]
With respect to this filtration,
	the summands are martingale differences and hence pairwise orthogonal in
	$L^2$.  By \eqref{eq-kk-04 },
	\eqref{eq-we-01},
	\eqref{eq-we-02}, and
	\eqref{eq-we-03}, we have
	\begin{align}
		\begin{aligned}
			\|D_T-D_T^{\{i\}}\|_2^2
			&\leq\frac{CB_\phi^2}{T}\Big[
			\|AQ_i\|_{\op}^2\E|\Delta_{i-1}|^2\\
		&\quad\qquad\,\,	+
			\eta_0^2i^{-2\alpha}
			\bigl\{B_\phi^2\E|\Delta_{i-1}|^2+
			\operatorname{Tr}(\Gamma)\bigr\}
			\times
			\bigl\{\rho^{-1}i^\alpha+
			b_\eta^{2\alpha}i^{1-2\alpha}\bigr\}
			\Big].
		\end{aligned}
		\label{eq-we-04}
	\end{align}
	Taking square roots in \eqref{eq-we-04} and using
	$\sqrt{x+y}\leq\sqrt x+\sqrt y$, we have
	\begin{align*}
		&\sum_{i=1}^T	\|AQ_i\|_{\op}\|D_T-D_T^{\{i\}}\|_2\\
		&\quad\leq\frac{CB_\phi}{T^{1/2}}\Bigg[
		\sum_{i=1}^T	\|AQ_i\|_{\op}^2\|\Delta_{i-1}\|_2
		\\
		&\qquad+\eta_0B_\phi\rho^{-1/2}
		\bigg\{
		\sum_{i=1}^T	\|AQ_i\|_{\op}\|\Delta_{i-1}\|_2i^{-\alpha/2}
		+\rho^{1/2}b_\eta^\alpha
		\sum_{i=1}^T	\|AQ_i\|_{\op}\|\Delta_{i-1}\|_2i^{1/2-2\alpha}
		\bigg\}
		\\
		&\qquad+\eta_0\operatorname{Tr}(\Gamma)^{1/2}
		\rho^{-1/2}
		\bigg\{
		\sum_{i=1}^T	\|AQ_i\|_{\op} i^{-\alpha/2}
		+\rho^{1/2}b_\eta^\alpha
		\sum_{i=1}^T	\|AQ_i\|_{\op} i^{1/2-2\alpha}
		\bigg\}
		\Bigg].
	\end{align*}
	We claim that
	\begin{align}
		&\sum_{i=1}^T\|AQ_i\|_{\op} i^{-\alpha/2}
		+\rho^{1/2}b_\eta^\alpha
		\sum_{i=1}^T\|AQ_i\|_{\op} i^{1/2-2\alpha}
		\leq C\bigl\{T^{1-\alpha/2}+b_\eta^{2\alpha}\bigr\},
		\label{eq-we-05}
		\end{align}
		and
		\begin{align}		\label{eq-we-06}
			\begin{aligned}
					&\sum_{i=1}^T\|AQ_i\|_{\op}\|\Delta_{i-1}\|_2i^{-\alpha/2}
			+\rho^{1/2}b_\eta^\alpha
			\sum_{i=1}^T\|AQ_i\|_{\op}\|\Delta_{i-1}\|_2i^{1/2-2\alpha}
			\\
			&\qquad\leq C\Big[
			\{\eta_0\operatorname{Tr}(\Gamma)\}^{1/2}T^{1-\alpha/2}
			+\bigl\{|\Delta_{0}|+
			(\eta_0\operatorname{Tr}(\Gamma))^{1/2}\bigr\}
			b_\eta^{2\alpha}
			\Big].
			\end{aligned}
	\end{align}
	We first prove  these two estimates. For $i\leq b_\eta$,
	$\|AQ_i\|_{\op}\leq Cb_\eta^\alpha i^{-\alpha}$, whereas $\|AQ_i\|_{\op}\leq C$ for
	$i>b_\eta$. Combining this with
	$
	\rho^{1/2}b_\eta^\alpha
	\leq Cb_\eta^{(3\alpha-1)/2},
	$
	we have
	\begin{align*}
		b_\eta^\alpha
		\sum_{i\leq b_\eta}i^{-3\alpha/2}
		&\leq Cb_\eta^{2\alpha},\quad \text{and}\quad
		\rho^{1/2}b_\eta^{2\alpha}
		\sum_{i\leq b_\eta}i^{1/2-3\alpha}
		\leq Cb_\eta^{2\alpha};
	\end{align*}
	For $i>b_\eta$, the displayed relation above gives
	$\rho^{1/2}b_\eta^\alpha i^{1/2-2\alpha}
	\leq Ci^{-\alpha/2}$. These estimates prove
	\eqref{eq-we-05}.

	For \eqref{eq-we-06},
	by \eqref{eq-ww-02},
	\[
	\|\Delta_{i-1}\|_2
	\leq C\Big[
	\bigl\{|\Delta_{0}|+
	(\eta_0\operatorname{Tr}(\Gamma))^{1/2}\bigr\}
	e^{-c\eta_0i^{1-\alpha}}
	+(\eta_0\operatorname{Tr}(\Gamma))^{1/2}i^{-\alpha/2}
	\Big].
	\]
	The polynomial part is bounded by
	\eqref{eq-we-05}, since its additional factor is at
	most one. For the exponential part, comparison with an integral at the
	scale $b_\eta$ gives
	\[
	\sum_{i=1}^T
	\bigl\{\|AQ_i\|_{\op} i^{-\alpha/2}
	+\rho^{1/2}b_\eta^\alpha \|AQ_i\|_{\op} i^{1/2-2\alpha}\bigr\}
	e^{-c\eta_0i^{1-\alpha}}
	\leq Cb_\eta^{2\alpha}.
	\]
	This proves \eqref{eq-we-06}.

	Note that
	$
	\eta_0B_\phi\rho^{-1/2}\leq C,$
	and
	$\eta_0\operatorname{Tr}(\Gamma)^{1/2}\rho^{-1/2}
	\leq C\{\eta_0\operatorname{Tr}(\Gamma)\}^{1/2}.
	$
	Combining this with \eqref{eq-ww-05},
	\eqref{eq-we-05}, and
	\eqref{eq-we-06},
	we obtain
	\begin{align}		\label{eq-we-07}
		\begin{aligned}
				&\sum_{i=1}^T  \|AQ_i\|_{\op} \|D_T-D_T^{\{i\}}\|_2\\
		&\quad\quad	\leq\frac{CB_\phi}{T^{1/2}}
		\Big[
		\{\eta_0\operatorname{Tr}(\Gamma)\}^{1/2}T^{1-\alpha/2}
		+\bigl\{|\Delta_{0}|+
		(\eta_0\operatorname{Tr}(\Gamma))^{1/2}\bigr\}
		b_\eta^{2\alpha}
		\Big].
		\end{aligned}
	\end{align}
	Combining \eqref{eq-we-07}, $B_\phi\leq(2\eta_0^{1/2})^{-1}$ and
	$\|X_i\|_3\leq\nu_3\|AQ_i\|_{\op}/T^{1/2}$, we have
	\[
	\sum_{i=1}^T\|X_i\|_3\|D_T-D_T^{\{i\}}\|_2
	\leq C\nu_3\bigg\{
	\frac{\operatorname{Tr}(\Gamma)^{1/2}}{T^{\alpha/2}}
	+\frac{\eta_0^{-1/2}
		\bigl\{|\Delta_{0}|+
		(\eta_0\operatorname{Tr}(\Gamma))^{1/2}\bigr\}
		b_\eta^{2\alpha}}{T}
	\bigg\}.
	\]
	This proves \eqref{eq-uu-03}.
\end{proof}
\subsection{Absorption of the additional terms in the proof of
	Theorem~\ref{thm-app-TD}}\label{sec:appen-lemmac4}

In this subsection, we complete the absorption step in the proof of
Theorem~\ref{thm-app-TD}.
\begin{lemma}\label{lem:TD-absorption}
	Under the assumptions of Theorem~\ref{thm-app-TD}, the bound obtained
	from Corollary~\ref{coro-main-1} and
	Lemma~\ref{lem:TD-remainder-summary}, together with $d_c\leq1$,
	implies \eqref{eq-app2-thm-01}.
\end{lemma}
\begin{proof}
	For brevity, denote the transient third-moment term by
	\[
	R_T^\circ
	:=\frac{d^{1/4}\nu_3^3b_\eta^{3\alpha}}
	{\underline\lambda_T^{3/2}T^{3/2}}.
	\]
	If $R_T^\circ\leq1$, then
	\[
	\frac{(1+\log d)\nu_3b_\eta^\alpha}
	{\underline\lambda_T^{1/2}T^{1/2}}
	=(1+\log d)d^{-1/12}(R_T^\circ)^{1/3}
	\leq C.
	\]
	Thus the first additional term in the proof of
	Theorem~\ref{thm-app-TD} is bounded by the last term on the right-hand side of
	\eqref{eq-app2-thm-01}, up to a constant. If $R_T^\circ>1$,
	the desired conclusion follows directly from $d_c\leq1$.

	To handle the second additional term, \eqref{eq-ww-03} and
	$\operatorname{Tr}(\Gamma)^{1/2}\leq\nu_3$ give
	\[
	\underline\lambda_T^{1/2}
	\leq
	\bigg\{\frac{\operatorname{Tr}(\Gamma)}{T}
	\sum_{t=1}^T\|AQ_t\|_{\op}^2\bigg\}^{1/2}
	\leq C_0\nu_3\bigg(1+\frac{b_\eta^\alpha}{T^{1/2}}\bigg).
	\]
	If $\underline\lambda_T^{1/2}\leq2C_0\nu_3$, then
	\[
	\frac{\operatorname{Tr}(\Gamma)^{1/2}}
	{\underline\lambda_T^{1/2}T^{\alpha/2}}
	\leq\frac{2C_0\nu_3\operatorname{Tr}(\Gamma)^{1/2}}
	{\underline\lambda_TT^{\alpha/2}},
	\]
	which is bounded by the second term on the right-hand side of
	\eqref{eq-app2-thm-01}. Otherwise, the preceding bound implies
	\[
	\frac{\nu_3b_\eta^\alpha}
	{\underline\lambda_T^{1/2}T^{1/2}}\geq\frac1{2C_0},
	\qquad
	R_T^\circ\geq(2C_0)^{-3}.
	\]
	The trivial bound $d_c\leq1$ again gives the desired conclusion
	after enlarging the constant.
\end{proof}

\subsection{Verification of the dimension-only rates}
\label{app:TD-rate-verification}

	Assume the conditions of Theorem~\ref{thm-app-TD} and
	\eqref{eq-ab-06}, and recall $s_\alpha=(1-\alpha)^{-1}$.
	Throughout this subsection, $c_K,C_K>0$ denote constants that may
	change from line to line and depend only on $\alpha$, $K_\phi$, and $K$.
Then
\begin{align}
	\begin{aligned}
		b_\eta&\leq C_{\alpha,K}d^{s_\alpha},
	&
	\eta_0^{-1/2}&\leq C_Kd^{1/2},
	&
	\operatorname{Tr}(\Gamma)^{1/2}&\leq C_Kd^{1/2},
\\
	(\eta_0\operatorname{Tr}(\Gamma))^{1/2}&\leq C_K,
	&
	|\Delta_{0}|&\leq Kd^{1/2},
	&
	\nu_3&\leq Kd^{1/2}.
	\end{aligned}
	\label{eq:TD-rate-inputs}
\end{align}
As explained in Remark~\ref{rem:TD-high-dimensional-CLT}, applying
Lemma~6.12 of \citet{WuEtAl2026} to the rescaled feature map
$\widetilde\phi=\phi/B_\phi$ and the corresponding stepsizes
$\widetilde\eta_t=B_\phi^2\eta_t$ gives
$\underline\lambda_T\geq c_{K,K_\phi}$ under the stated sample-size
condition.

	If $T<M_Kd^{2s_\alpha}$, then
	$1\leq M_K^{1/2}d^{s_\alpha}T^{-1/2}$.
	Since $d_c\leq1$ and $d\geq1$, both \eqref{eq-hh-01} and
	\eqref{eq-ab-08} follow in this case after enlarging $C_K$.
	We may therefore assume that $T\geq M_Kd^{2s_\alpha}$.
	This implies $T\geq M_Kd^{s_\alpha}$ and hence
	$\underline\lambda_T\geq c_K$.
Substituting \eqref{eq:TD-rate-inputs} into
\eqref{eq-app2-thm-01}, we have
\begin{align}
	\begin{aligned}
		&\dc\bigl(\sqrt T\,(\bar\theta_T-\E\bar\theta_T),
		\overline\Lambda_T^{1/2}Z\bigr)
		\leq C_K\bigg\{
		\frac{d^{7/4}}{T^{1/2}}
		+\frac{d^{3s_\alpha-5/4}}{T^{3/2}}
		+\frac{d(1+\log d)}{T^{\alpha/2}}
		+\frac{d^{s_\alpha}}{T^{1/2}}
		\bigg\}.
	\end{aligned}
	\label{eq:TD-third-expanded-rate}
\end{align}
	Since $s_\alpha>2$ and
	$T\geq M_Kd^{2s_\alpha}\geq d^{2s_\alpha}$,
	the first and second terms on the right-hand side of
	\eqref{eq:TD-third-expanded-rate} are bounded by
	$Cd^{s_\alpha}T^{-1/2}$.
This proves \eqref{eq-hh-01}.

For the original estimation error, the proof of
Corollary~\ref{coro-app-TD-uncentered} adds at most
\[
\frac{C|\Delta_0|b_\eta}
{\underline\lambda_T^{1/2}\sqrt T}
\leq C_K\frac{d^{s_\alpha+1/2}}{\sqrt T}.
\]
Since $d^{s_\alpha}\leq d^{s_\alpha+1/2}$, this proves
\eqref{eq-ab-08}.

\subsection{Proof of Lemma \ref{lem:U-remainder-general}}\label{sec-appendix-pr-U}
For each $1\leq i\leq d$, $D_{i,n}$ is the degenerate component in
the Hoeffding decomposition of the normalized $U$-statistic.
Applying Lemma 10.1 of  \citet{Chen2011} with
$m=m_i$, $\sigma^2=\sigma_i^2$, and $\sigma_1^2=\lambda_i^2$, we have for each $1\leq i\leq d$,
\[
\E|D_{i,n}|^2
\leq
\frac{(m_i-1)^2\sigma_i^2}
{m_i(n-m_i+1)\lambda_i^2}.
\]
Therefore,
\[
\E|D|^2
=
\sum_{i=1}^{d}\E|D_{i,n}|^2
\leq
\sum_{i=1}^{d}
\frac{(m_i-1)^2\sigma_i^2}
{m_i(n-m_i+1)\lambda_i^2}.
\]
For each $1\leq j\leq n$, let $D_{i,n}^{(-j)}$ be obtained
from $D_{i,n}$ by omitting all summands involving $X_j$,
while keeping the original normalization unchanged.
Lemma~10.1 of \citet{Chen2011} gives
\[
\E\big|D_{i,n}-D_{i,n}^{(-j)}\big|^2
\leq
2\frac{(m_i-1)^2\sigma_i^2}
{nm_i(n-m_i+1)\lambda_i^2}.
\]
Let $\mathcal F_{-j}=\sigma(X_k:k\ne j)$.
Conditionally on $\mathcal F_{-j}$, $D_{i,n}$ and
$D_{i,n}^{\{j\}}$ are independent and identically distributed.
Since $D_{i,n}^{(-j)}$ is $\mathcal F_{-j}$-measurable,
we have
\[
\begin{aligned}
	\E\big|D_{i,n}-D_{i,n}^{\{j\}}\big|^2
	&=
	2\E\big\{\operatorname{Var}(D_{i,n}\mid\mathcal F_{-j})\big\}
	\leq
	2\E\big|D_{i,n}-D_{i,n}^{(-j)}\big|^2\\
	&\leq
	4\frac{(m_i-1)^2\sigma_i^2}
	{nm_i(n-m_i+1)\lambda_i^2}.
\end{aligned}
\]
Hence,
\[
\E|D-D^{\{j\}}|^2
=
\sum_{i=1}^{d}
\E|D_{i,n}-D_{i,n}^{\{j\}}|^2
\leq
4\sum_{i=1}^{d}
\frac{(m_i-1)^2\sigma_i^2}
{nm_i(n-m_i+1)\lambda_i^2}.
\]
This proves the result.

% BEGIN UNIFIED BIBLIOGRAPHY

% END UNIFIED BIBLIOGRAPHY
\end{document}